\documentclass[numbook,envcountsame,smallcondensed]{amsart}

\usepackage[foot]{amsaddr}
\usepackage{amssymb,amsmath,amsfonts,mathrsfs,cite,mathptmx,graphicx,gastex,rotating,url}

\DeclareMathOperator{\con}{con}
\DeclareMathOperator{\simple}{sim}
\DeclareMathOperator{\mul}{mul}
\DeclareMathOperator{\occ}{occ}
\DeclareMathOperator{\ini}{ini}
\DeclareMathOperator{\head}{h}
\DeclareMathOperator{\tail}{t}
\DeclareMathOperator{\var}{var}
\DeclareMathOperator{\FIC}{FIC}

\newtheorem{theorem}{Theorem}[section]
\newtheorem{proposition}[theorem]{Proposition}
\newtheorem{lemma}[theorem]{Lemma}
\newtheorem{corollary}[theorem]{Corollary}
\newtheorem{example}[theorem]{Example}

\numberwithin{equation}{section}

\makeatletter

\renewcommand*\subjclass[2][2010]{\def\@subjclass{#2}\@ifundefined{subjclassname@#1}{\ClassWarning{\@classname}{Unknown edition (#1) of Mathematics Subject Classification; using '2010'.}}{\@xp\let\@xp\subjclassname\csname subjclassname@#1\endcsname}}

\renewcommand{\subjclassname}{\textup{2010} Mathematics Subject Classification}

\makeatother

\begin{document}

\title[Varieties of monoids whose subvariety lattice is distributive]{Varieties of aperiodic monoids with commuting idempotents whose subvariety lattice is distributive}
\thanks{Supported by the Ministry of Science and Higher Education of the Russian Federation (project FEUZ-2023-0022).}

\author{Sergey V. Gusev}

\address{Ural Federal University, Institute of Natural Sciences and Mathematics, Lenina 51, Ekaterinburg 620000, Russia}

\email{sergey.gusb@gmail.com}

\begin{abstract}
We completely classify all varieties of aperiodic monoids with commuting idempotents whose subvariety lattice is distributive.
\end{abstract}


\subjclass{20M07}

\maketitle

\section{Background and overview}
\label{Sec: introduction}

A variety $\mathbf V$ is \textit{distributive} if its lattice $\mathfrak L(\mathbf V)$ of subvarieties is distributive.

In 1979, Shevrin~\cite[Problem~2.60a]{sverdlovsk-tetrad} posed the problem of classifying all distributive varieties of semigroups. 
This problem includes the problem of identifying all distributive varieties of periodic groups.
In view of the result by Kozhevnikov~\cite{Kozhevnikov-12}, there exist uncountably many group varieties whose lattice of subvarieties is isomorphic to the 3-element chain.
Thus, the latter problem seems to be extremely difficult and it is natural to speak about the classification of distributive varieties of semigroups modulo group varieties.
In the early 1990s, in a series of papers, Volkov solved Shevrin's problem in a very wide partial case, resulting in an almost complete description modulo group varieties.
In particular, Volkov completely classified all distributive varieties of \textit{aperiodic} semigroups, i.e., semigroups all whose subgroups are trivial (see~\cite[Section~11]{Shevrin-Vernikov-Volkov-09} for more details).

The present article is concerned with the distributive varieties of \textit{monoids}, i.e., semigroups with an identity element.
In comparison to the semigroup case, distributive varieties of monoids have not been systematically examined until recently, although non-trivial examples of such varieties have long been known.
More information and many references can be found in a very recent survey~\cite{Gusev-Lee-Vernikov-22}.

As in the semigroup case, in view of the above-mentioned result of~\cite{Kozhevnikov-12}, the general problem of classifying distributive varieties of monoids seems to be extremely difficult.
Thus, it is natural to begin the study monoid varieties with the mentioned property within the class of aperiodic monoids.
Nevertheless, experience suggests that even the problem of classifying distributive varieties of aperiodic monoids remains quite difficult.
So, it is natural at first to try solving this problem within some subclass of the class of all aperiodic monoids. 

The first step was taken in~\cite{Gusev-23}, where distributive varieties of aperiodic monoids with central idempotents were classified.
Here we continue this investigation.
In the survey~\cite{Gusev-Lee-Vernikov-22}, some reasonable problems were suggested, the solutions of which would further contribute toward solving problem of a complete classification of distributive varieties of aperiodic monoids.
One of them is the following: describe distributive varieties within the class $\mathbf A_\mathsf{com}$ of aperiodic monoids with commuting idempotents~\cite[Problem~6.12a)]{Gusev-Lee-Vernikov-22}.

The goal of the present paper is to solve this problem.
We present 5 countably infinite series of varieties and 27 ``sporadic'' varieties such that every distributive subvariety of $\mathbf A_\mathsf{com}$ is contained in one of them. Notice that the proof this result implies that set of all distributive subvarieties of $\mathbf A_\mathsf{com}$ is countably infinite, although the set of all distributive varieties of monoids is uncountably infinite~\cite{Kozhevnikov-12}.

This paper is structured as follows.
In Section~\ref{Sec: main result}, we formulate and discuss our main result.
Some background results are first given in Section~\ref{Sec: preliminaries}. 
In Section~\ref{Sec: construction}, a general construction of Olga Sapir is defined that allows one to construct monoids of a special form.  
In Section~\ref{Sec: separating identities}, separating identities are found for a number of certain varieties of monoids.
Section~\ref{Sec: non-distributive} contains a series of examples of non-distributive varieties of monoids.
Results in Section~\ref{Sec: identities} provide explicit sets of identities which define subvarieties of our 27 ``sporadic'' distributive varieties. 
Finally, Section~\ref{Sec: proof} is devoted to the proof of our main result.

Many identities will be introduced and used throughout this article.
For the reader's convenience, some of these identities are collected in the appendix for quick referencing.

\section{The main result}
\label{Sec: main result}

Let us briefly recall a few notions that we need to formulate our main result.
Let $\mathfrak X$ be a countably infinite set called an \textit{alphabet}. 
As usual, let~$\mathfrak X^\ast$ denote the free monoid over the alphabet~$\mathfrak X$. 
Elements of~$\mathfrak X$ are called \textit{letters} and elements of~$\mathfrak X^\ast$ are called \textit{words}.
We treat the identity element of~$\mathfrak X^\ast$ as \textit{the empty word}, which is denoted by~$1$.  
Words and letters are denoted by small Latin letters. 
However, words unlike letters are written in bold. 
An identity is written as $\mathbf u \approx \mathbf v$, where $\mathbf u,\mathbf v \in \mathfrak X^\ast$; it is \textit{non-trivial} if $\mathbf u \ne \mathbf v$.

As usual, $\mathbb N$ denote the set of all natural numbers. 
For any $n\in\mathbb N$, we denote by $S_n$ the full symmetric group on the set $\{1,\dots,n\}$. 
For convenience, we put $S_0=S_1$. 
Let $\mathbb N_0:=\mathbb N\cup\{0\}$. 
For any $n,m,k\in\mathbb N_0$, $\rho\in S_{n+m}$ and $\tau\in S_{n+m+k}$, we define the words:
\[
\begin{aligned}
\mathbf a_{n,m}[\rho]&:=\biggl(\prod_{i=1}^n z_it_i\biggr)x\biggl(\prod_{i=1}^{n+m-1} z_{i\rho}y_i^2\biggr)z_{(n+m)\rho}x\biggl(\prod_{i=n+1}^{n+m} t_iz_i\biggr),\\
\mathbf a_{n,m}^\prime[\rho]&:=\biggl(\prod_{i=1}^n z_it_i\biggr)\biggl(\prod_{i=1}^{n+m-1} z_{i\rho}y_i^2\biggr)z_{(n+m)\rho}x^2\biggl(\prod_{i=n+1}^{n+m} t_iz_i\biggr),\\
\overline{\mathbf a}_{n,m}[\rho]&:=\biggl(\prod_{i=1}^n z_it_i\biggr)x\biggl(\prod_{i=1}^{n+m-1} z_{i\rho}y_i^2x\biggr)z_{(n+m)\rho}x\biggl(\prod_{i=n+1}^{n+m} t_iz_i\biggr),\\
\mathbf c_{n,m,k}[\tau]&:=\biggl(\prod_{i=1}^n z_it_i\biggr)xyt\biggl(\prod_{i=n+1}^{n+m} z_it_i\biggr)x\biggl(\prod_{i=1}^{n+m+k-1} z_{i\tau}y_i^2\biggr)z_{(n+m+k)\tau}y\biggl(\prod_{i=n+m+1}^{n+m+k} t_iz_i\biggr).
\end{aligned}
\]
Let $\mathbf c_{n,m,k}^\prime[\tau]$ denote the word obtained from $\mathbf c_{n,m,k}[\tau]$ by interchanging the first occurrences of $x$ and $y$.
We denote also by $\mathbf d_{n,m,k}[\tau]$ and $\mathbf d_{n,m,k}^\prime[\tau]$ the words obtained from the words $\mathbf c_{n,m,k}[\tau]$ and $\mathbf c_{n,m,k}^\prime[\tau]$, respectively, when reading the last words from right to left.
We fix notation for the following three identities:
\begin{align*}
\sigma_1:\enskip xyzxty\approx yxzxty,\ \ \ \ 
\sigma_2:\enskip xzytxy\approx xzytyx,\ \ \ \ 
\sigma_3:\enskip xzxyty\approx xzyxty.
\end{align*}
Let
\[
\begin{aligned}
&\Phi:=\{x^2\approx x^3,\,x^2y^2\approx y^2x^2\},\\
&\Phi_1:=\left\{
\mathbf c_{k,\ell,m}[\rho]\approx\mathbf c_{k,\ell,m}^\prime[\rho],\,\mathbf d_{k,\ell,m}[\rho]\approx\mathbf d_{k,\ell,m}^\prime[\rho]
\mid
k,\ell,m\in\mathbb N,\,
\rho\in S_{k+\ell+m}
\right\},\\
&\Phi_2:=\left\{
\mathbf a_{k,\ell}[\rho] \approx \overline{\mathbf a}_{k,\ell}[\rho]
\mid
k,\ell\in\mathbb N,\,
\rho\in S_{k+\ell}
\right\},\\
&\Phi_3:=\left\{
\mathbf a_{k,\ell}[\rho] \approx \mathbf a_{k,\ell}^\prime[\rho]
\mid
k,\ell\in\mathbb N,\,
\rho\in S_{k+\ell}
\right\}.
\end{aligned}
\]
Let $\var\,\Sigma$ denote the monoid variety given by a set $\Sigma$ of identities. 
Given a class $\mathbf V$ of monoids, the class \textit{dual} to $\mathbf V$, denoted by $\mathbf V^\delta$, is the class consisting of monoids dual to members of $\mathbf V$.

Our main result is the following

\begin{theorem}
\label{T: A_com}
A subvariety of $\mathbf A_\mathsf{com}$ is distributive if and only if it is contained in one of the varieties 
\[
\begin{aligned}
&\mathbf D_1:=\var
\left\{
\Phi,\,xyx\approx xyx^2,\,x^2y\approx x^2yx \right\},
\\
&\mathbf D_2:=\var
\left\{
\Phi,\,\Phi_1,\,\Phi_2,\,xyx\approx xyx^2 \right\},
\\
&\mathbf D_3:=\var
\left\{
\Phi,\,\sigma_2,\,\sigma_3,\,x^2y\approx x^2yx\approx xyx^2,\,xyxzx\approx xyxzx^2
\right\}
,\\
&\mathbf D_4:=\var
\left\{
\Phi,\,\sigma_2,\,\sigma_3,\,x^2y\approx x^2yx,\,x^2yty\approx yx^2ty,\,xyzx^2ty^2\approx yxzx^2ty^2,\,xyxzx\approx xyx^2zx
\right\}
,\\
&\mathbf D_5:=\var
\left\{
\Phi,\,\sigma_2,\,\sigma_3,\,xyx^2\approx x^2yx,\,xyzx^2y\approx yxzx^2y,\,xyxzx\approx xyxzx^2
\right\}
,\\
&\mathbf D_6:=\var
\left\{
\Phi,\,\sigma_2,\,\sigma_3,\,x^2yx\approx x^2yx^2,\,x^2yty\approx yx^2ty,\,xyzx^2ty\approx yxzx^2ty,\,xyzytx^2\approx yxzytx^2
\right\}
,\\
&\mathbf D_7:=\var
\left\{\!\!\!\!
\begin{array}{l}
\Phi,\,\sigma_2,\,\sigma_3,\,x^2yx\approx x^2yx^2,\,x^2yty\approx yx^2ty,\,xyzx^2ty^2\approx yxzx^2ty^2,\\
xyzx^2y\approx yxzx^2y,\,xyxzx\approx xyx^2zx
\end{array}
\!\!\!\!\right\}
,\\
&\mathbf D_8:=\var
\left\{\!\!\!\!
\begin{array}{l}
\Phi,\,\sigma_2,\,\sigma_3,\,x^2yx\approx x^2yx^2,\,x^2yty\approx yx^2ty,\,xyzx^2ty^2\approx yxzx^2ty^2,\\
xyzx^2y\approx yxzx^2y,\,xyzx^2tysx\approx yxzx^2tysx,\,xyxzx\approx xyxzx^2
\end{array}
\!\!\!\!\right\}
,\\
&\mathbf D_9:=\var
\left\{
\Phi,\,\sigma_1,\,\sigma_3,\,x^2yx\approx x^2yx^2,\, ytx^2y\approx ytyx^2
\right\}
,\\
&\mathbf D_{10}:=\var
\left\{
\Phi,\,\sigma_1,\,\sigma_3,\,x^2yx\approx x^2yx^2,\, xyxzx\approx xyxzx^2,\,x^2zytxy\approx x^2zytyx
\right\}
,\\
&\mathbf D_{11}:=\var
\left\{
\Phi,\,\Phi_1,\,\Phi_2,\,xyx^2\approx x^2yx,\, xyxzx\approx xyxzx^2
\right\}
,\\
&\mathbf D_{12}:=\var
\left\{\!\!\!\!
\begin{array}{l}
\Phi,\,\Phi_1,\,\Phi_2,\,x^2yx\approx x^2yx^2,\,ytx^2y\approx ytyx^2,\, x^2yty\approx yx^2ty,\\
xyzx^2ty\approx yxzx^2ty,\, xyzytx^2\approx yxzytx^2,\, yzyxtx^2\approx yzxytx^2
\end{array}
\!\!\!\!\right\}
,\\
&\mathbf D_{13}:=\var
\left\{\!\!\!\!
\begin{array}{l}
\Phi,\,\Phi_1,\,\Phi_2,\,x^2yx\approx x^2yx^2,\,yx^2ty\approx xyx^2ty,\,ytyx^2\approx ytxyx^2,\,xyxzx\approx xyx^2zx\\
x^2yty^2\approx yx^2ty^2,\,x^2yzytx\approx yx^2zytx,\,x^2yzxty\approx yx^2zxty,\,x^2zytxy\approx x^2zytyx,\\ 
yzx^2ytx\approx yzyx^2tx,\,xyzx^2ty^2\approx yxzx^2ty^2
\end{array}
\!\!\!\!\right\}
,\\
&\mathbf D_{14}:=\var
\left\{\!\!\!\!
\begin{array}{l}
\Phi,\,\Phi_1,\,\Phi_2,\,x^2yx\approx x^2yx^2,\,yx^2ty\approx xyx^2ty,\,ytyx^2\approx ytxyx^2,\,xyxzx\approx xyxzx^2\\
x^2yty^2\approx yx^2ty^2,\,x^2yzytx\approx yx^2zytx,\,x^2yzxty\approx yx^2zxty,\,x^2zytxy\approx x^2zytyx,\\ 
yzx^2ytx\approx yzyx^2tx,\,xyzx^2ty^2\approx yxzx^2ty^2,\,xyzx^2tysx\approx yxzx^2tysx
\end{array}
\!\!\!\!\right\}
,\\
&\mathbf P_n:=\var
\left\{
\Phi_1,\,\Phi_3,\,x^n\approx x^{n+1},\,x^2y\approx yx^2
\right\}
,\\
&\mathbf Q_n:=\var\left\{x^n\approx x^{n+1},\,x^ny\approx yx^n,\,x^2y\approx xyx\right\},\\
&\mathbf R_n:=\var\left\{\sigma_1,\,\sigma_2,\,x^n\approx x^{n+1},\,x^2y\approx yx^2\right\},\ \ n \in \mathbb N,
\end{aligned}
\]
or the dual ones.
\end{theorem}

A variety $\mathbf V$ is \textit{self-dual} if $\mathbf V=\mathbf V^\delta$.
Notice that only the varieties $\mathbf D_{11}$ and $\mathbf P_n$ are self-dual among the varieties listed in Theorem~\ref{T: A_com}.
Thus, Theorem~\ref{T: A_com} shows that every distributive subvariety of $\mathbf A_\mathsf{com}$ is contained in either a variety in 5 countably infinite series or 27 ``sporadic'' varieties. 

Analysing the proof of Theorem~\ref{T: A_com} in Section~\ref{Sec: proof}, one can note that if $\mathbf X$ is one of the varieties listed in Theorem~\ref{T: A_com}, then each subvariety of $\mathbf X$ may be given within $\mathbf X$ by a finite set of identities.
Consequently, the set of all distributive subvarieties of $\mathbf A_\mathsf{com}$ is countably infinite.

\section{Preliminaries}
\label{Sec: preliminaries}

\subsection{Deduction}

A variety $\mathbf V$ \textit{satisfies} an identity $\mathbf u \approx \mathbf v$, if for any monoid $M\in \mathbf V$ and any substitution $\phi\colon \mathfrak X \to M$, the equality $\phi(\mathbf u)=\phi(\mathbf v)$ holds in $M$.
An identity $\mathbf u \approx \mathbf v$ is \textit{directly deducible} from an identity $\mathbf s \approx \mathbf t$ if there exist some words $\mathbf a,\mathbf b \in \mathfrak X^\ast$ and substitution $\phi\colon \mathfrak X \to \mathfrak X^\ast$ such that $\{ \mathbf u, \mathbf v \} = \{ \mathbf a\phi(\mathbf s)\mathbf b,\mathbf a\phi(\mathbf t)\mathbf b \}$.
A non-trivial identity $\mathbf u \approx \mathbf v$ is \textit{deducible} from a set $\Sigma$ of identities if there exists some finite sequence $\mathbf u = \mathbf w_0, \dots, \mathbf w_m = \mathbf v$ of words such that each identity $\mathbf w_i \approx \mathbf w_{i+1}$ is directly deducible from some identity in $\Sigma$.

\begin{proposition}[Birkhoff's Completeness Theorem for Equational Logic; see {\cite[Theorem~1.4.6]{Almeida-94}}]
\label{P: deduction}
Let $\mathbf V$ be the variety defined by some set $\Sigma$ of identities.
Then $\mathbf V$ satisfies an identity $\mathbf u \approx \mathbf v$ if and only if $\mathbf u \approx \mathbf v$ is deducible from $\Sigma$.\qed
\end{proposition}

\subsection{Rees quotient monoids}
\label{Subsec: Rees quotient monoids}

The following construction was introduced by Perkins~\cite{Perkins-69} to build the first example of a finite semigroup generating non-finitely based variety.
We use $W^\le$ to denote the closure of a set of words $W\subset\mathfrak X^\ast$ under taking factors. 
For any set $W$ of words, let $M(W)$ denote the Rees quotient monoid of $\mathfrak X^\ast$ over the ideal $\mathfrak X^\ast \setminus W^{\le}$ consisting of all words that are not factors of any word in $W$.
A word $\mathbf w$ is an \textit{isoterm} for a variety $\mathbf V$ if $\mathbf V$ violates any non-trivial identity of the form $\mathbf w \approx \mathbf w^\prime$.

\begin{lemma}[\mdseries{\!\cite[Lemma~3.3]{Jackson-05}}]
\label{L: M(W) in V}
Let $\mathbf V$ be a monoid variety and $W$ a set of words. 
Then $M(W)$ lies in $\mathbf V$ if and only if each word in $W$ is an isoterm for $\mathbf V$.\qed
\end{lemma}

\subsection{Decomposition of words}

The \textit{content} of a word $\mathbf w$, i.e., the set of all letters occurring in $\mathbf w$, is denoted by $\con(\mathbf w)$. 
A letter is called \textit{simple} [\textit{multiple}] \textit{in a word} $\mathbf w$ if it occurs in $\mathbf w$ once [at least twice]. 
The set of all simple [multiple] letters in a word $\mathbf w$ is denoted by $\simple(\mathbf w)$ [respectively $\mul(\mathbf w)$]. 
If $\mathbf w$ is a word and $X\subseteq\con(\mathbf w)$, then we denote by $\mathbf w(X)$ the word obtained from $\mathbf w$ by deleting all letters except letters from $X$. 
If $X=\{x_1,\dots,x_k\}$, then we write $\mathbf w(x_1,\dots,x_k)$ rather than $\mathbf w(\{x_1,\dots,x_k\})$. 

Let $\mathbf w$ be a word with $\simple(\mathbf w)=\{t_1,\dots,t_m\}$. 
We will assume without loss of generality that $\mathbf w(t_1,\dots,t_m)=t_1\cdots t_m$. 
Then $\mathbf w=\mathbf w_0t_1\mathbf w_1\cdots t_m\mathbf w_m$ for some words $\mathbf w_0,\dots,\mathbf w_m$. 
The words $\mathbf w_0,\dots, \mathbf w_m$ are called \textit{blocks} of the word $\mathbf w$. 
The representation of the word $\mathbf w$ as a product of alternating simple letters in $\mathbf w$ and blocks is called \textit{decomposition} of $\mathbf w$.

\begin{lemma}[\mdseries{\!\cite[Lemma~2.17]{Gusev-Vernikov-21}}]
\label{L: identities of M(xy)}
Let $\mathbf u\approx\mathbf v$ be an identity of $M(xy)$. 
If $\mathbf u_0t_1\mathbf u_1\cdots t_m\mathbf u_m$ is the decomposition of the word $\mathbf u$, then $\con(\mathbf u)=\con(\mathbf v)$ and the decomposition of the word $\mathbf v$ has the form $\mathbf v_0t_1\mathbf v_1\cdots t_m\mathbf v_m$.\qed
\end{lemma}

For a word $\mathbf w$ and a letter $x$, let $\occ_x(\mathbf w)$ denote the number of occurrences of $x$ in $\mathbf w$.
A non-empty word $\mathbf w$ is called \emph{linear} if $\occ_x(\mathbf w)\le 1$ for each letter $x$. 
Let $\mathbf u$ and $\mathbf v$ be words and $\mathbf u_0t_1\mathbf u_1\cdots t_m\mathbf u_m$ and $\mathbf v_0t_1\mathbf v_1\cdots t_m\mathbf v_m$ decompositions of $\mathbf u$ and $\mathbf v$, respectively. 
A letter $x$ is called \emph{linear-balanced in the identity} $\mathbf u\approx\mathbf v$ if $x$ is multiple in $\mathbf u$ and $\occ_x(\mathbf u_i)=\occ_x(\mathbf v_i)\le 1$ for all $i=0,1,\dots,m$; the identity $\mathbf u\approx\mathbf v$ is called \emph{linear-balanced} if any letter $x\in\mul(\mathbf u)\cup\mul(\mathbf v)$ is linear-balanced in this identity. 

The next statement directly follows from Lemma~\ref{L: M(W) in V} and Lemma~3.1 in~\cite{Gusev-Vernikov-21}.

\begin{lemma}
\label{L: identities of M(xt_1x...t_kx)}
Let $\mathbf u\approx\mathbf v$ be an identity of $M(xt_1x\cdots t_kx)$. 
If all blocks of $\mathbf u$ are linear words and $\occ_x(\mathbf u)\le k+1$ for every letter $x$, then the identity $\mathbf u\approx\mathbf v$ is linear-balanced.\qed
\end{lemma}

\section{Construction of Olga Sapir}
\label{Sec: construction}

In~\cite{Sapir-18,Sapir-21}, Olga Sapir introduced a generalization of $M(W)$ construction.
It played a critical role in the classification of limit varieties of $J$-trivial monoids~\cite{Gusev-Sapir-22} and finding new examples of limit varieties generated by finite non-$J$-trivial aperiodic monoids~\cite{Sapir-23}.
The construction of Olga Sapir is also a key tool in the present article.

\subsection{Definition}

Let $\alpha$ be a congruence on the free monoid $\mathfrak X^\ast$.  
The elements of the quotient monoid $\mathfrak X^\ast/\alpha$ are called $\alpha$-\textit{classes} and written using lowercase letters in the typewriter style. 
The factor relation on $\mathfrak X^\ast$ can be naturally extended to $\alpha$-classes as follows: given two $\alpha$-classes $\mathtt u, \mathtt v \in \mathfrak X^\ast/\alpha$  we write $\mathtt v \le_\alpha\mathtt u$ if $\mathtt u = \mathtt p\mathtt v\mathtt s$ for some $\mathtt p,\mathtt s \in \mathfrak X^\ast/\alpha$.
Given a set $\mathtt W$ of $\alpha$-classes we define $\mathtt W^{\le_\alpha}$ as closure of $\mathtt W$ in quasi-order $\le_\alpha$.
If $\mathtt W$ is a set of $\alpha$-classes, then $M_\alpha(\mathtt W)$ denotes the Rees quotient of $\mathfrak X^\ast /\alpha$ over the ideal $(\mathfrak X^\ast/ \alpha) \setminus  \mathtt W^{\le_\alpha}$.
Evidently, if $\alpha$ is the trivial congruence on  $\mathfrak X^\ast$, then $M_\alpha(\mathtt W)$ is nothing but $M(\mathtt W)$.

The following statement shows how to calculate the relation $\le_\alpha$.

\begin{lemma}[\mdseries{\!\cite[Lemma~2.1]{Sapir-21}}]
\label{L: le_alpha}
For $\mathtt u, \mathtt v \in \mathfrak X^\ast/\alpha$ the following are equivalent:
\begin{itemize}
\item[\textup{(i)}] $\mathtt v \le_\alpha\mathtt u$;
\item[\textup{(ii)}] every word $\mathbf v \in\mathtt v$ is a factor of a word $\mathbf u \in\mathtt u$;
\item[\textup{(iii)}] some word $\mathbf v \in\mathtt v$ is a factor of a word $\mathbf u \in\mathtt u$.
\end{itemize}
\end{lemma}

A set $W$ of words is \textit{stable} with respect to a monoid variety $\mathbf V$ if for each $\mathbf w\in W$, we have $\mathbf v\in W$ whenever $\mathbf V$ satisfies $\mathbf u \approx \mathbf v$.
Notice that a word $\mathbf w$ is an isoterm for $\mathbf V$ if and only if the singleton set $\{\mathbf w\}$ is stable with respect to $\mathbf V$.
The following lemma generalizes Lemma~\ref{L: M(W) in V}.

\begin{lemma}[\mdseries{\!\cite[Proposition~2.3]{Sapir-21}}]
\label{L: M_alpha(W) in V} 
Let $\alpha$ be a congruence on $\mathfrak X^\ast$ such that the empty
word $1$ forms a singleton $\alpha$-class, $\mathbf V$ a monoid variety and $\mathtt W$ a set of $\alpha$-classes.
Then $M_\alpha(\mathtt W)$ lies in $\mathbf V$ if and only if each $\alpha$-class in $\mathtt W^{\le_\alpha}$ is stable with respect to $\mathbf V$.\qed
\end{lemma} 

Given any set $\mathtt W$ of $\alpha$-classes, let $\mathbf M_\alpha(\mathtt W)$ denote the variety generated by the monoid $M_\alpha(\mathtt W)$.
For brevity, if $\mathtt w_1,\dots,\mathtt w_k\in\mathfrak X^\ast/ \alpha$, then we write $M_\alpha(\mathtt w_1,\dots,\mathtt w_k)$ [respectively, $\mathbf M_\alpha(\mathtt w_1,\dots,\mathtt w_k)$] rather than $M_\alpha(\{\mathtt w_1,\dots,\mathtt w_k\})$ [respectively, $\mathbf M_\alpha(\{\mathtt w_1,\dots,\mathtt w_k\})$].

\begin{lemma}[\mdseries{\!\cite[Corollary~2.5]{Sapir-21}}]
\label{L: M_alpha(W_1) vee M_alpha(W_2)} 
Let $\alpha$ be a congruence on $\mathfrak X^\ast$ such that the empty
word $1$ forms a singleton $\alpha$-class, while $\mathtt W_1$ and $\mathtt W_2$ two sets of $\alpha$-classes.
Then $\mathbf M_\alpha(\mathtt W_1)\vee \mathbf M_\alpha(\mathtt W_2)=\mathbf M_\alpha(\mathtt W_1\cup\mathtt W_2)$.\qed
\end{lemma} 

\subsection{Several certain congruences on $\mathfrak X^\ast$}

In~\cite{Sapir-21}, several certain congruences on $\mathfrak X^\ast$ were introduced.
Two of them, the congruences $\gamma$ and $\lambda$, will play a critical role in the present paper. 
They are defined as follows: for every $\mathbf u,\mathbf v \in \mathfrak X^\ast$, 
\begin{itemize}
\item $\mathbf u\mathrel{\gamma}\mathbf v$ if and only if $\simple(\mathbf u)=\simple(\mathbf v)$ and $\mathbf u$ can be obtained from $\mathbf v$ by changing the individual exponents of letters;
\item $\mathbf u\mathrel{\lambda}\mathbf v$ if and only if $\mathbf u\mathrel{\gamma}\mathbf v$ and the first two occurrences of each multiple letter are adjacent in $\mathbf u$ if and only if these occurrences are adjacent in $\mathbf v$.
\end{itemize}
For any monoid variety $\mathbf V$, let $\FIC(\mathbf V)$ denote the fully invariant congruence on $\mathfrak X^\ast$ corresponding to $\mathbf V$.
We partition the congruences $\gamma$ and $\lambda$ as follows:
\begin{itemize}
\item $\gamma^\prime:=\gamma\wedge \FIC(\mathbf M(xyx))$;
\item $\lambda^\prime:=\lambda\wedge \FIC(\mathbf M(xyx))$.
\end{itemize}

If $\mathbf w \in \mathfrak X^\ast$ and $x \in \con(\mathbf w)$, then an \textit{island} formed by $x$ in $\mathbf w$ is a maximal factor of $\mathbf w$, which is a power of $x$. 
We say that a word $\mathbf w \in \mathfrak X^\ast$ is $2$-\textit{island-limited} if each letter forms at most two islands in $\mathbf w$.
A set $W$ of words is $2$-\textit{island-limited} if each word in $W$ is $2$-island-limited.

\begin{lemma}
\label{L: subclasses are stable}
Let $\alpha$ be a congruence on $\mathfrak X^\ast$ and $\mathtt u$ an $\alpha$-class. 
Assume that the $\alpha$-class $\mathtt u$ is stable with respect to a monoid variety $\mathbf V$ and one of the following holds:
\begin{itemize}
\item[\textup{(i)}] $\alpha=\gamma$;
\item[\textup{(ii)}] $\alpha=\lambda$ and the $\alpha$-class $\mathtt u$ is $2$-island-limited;
\item[\textup{(iii)}] $\alpha=\FIC(\mathbf M(xyx))$ and 
\begin{equation}
\label{eq: condition for FIC(M(xyx))}
\text{there are } \mathbf u\in\mathtt u \text{ and } x,y\in \mathfrak X\text{ such that }\mathbf u(x,y)=xyx;
\end{equation}
\item[\textup{(iv)}] $\alpha=\gamma^\prime$ and the condition~\eqref{eq: condition for FIC(M(xyx))} holds;
\item[\textup{(v)}] $\alpha=\lambda^\prime$, the $\alpha$-class $\mathtt u$ is $2$-island-limited and the condition~\eqref{eq: condition for FIC(M(xyx))} holds.
\end{itemize}
Then for each $\mathtt v\in\{\mathtt u\}^{\le_{\alpha}}$ the $\alpha$-class $\mathtt v$ is stable with respect to $\mathbf V$. 
\end{lemma}

\begin{proof}
Parts~(i) and~(ii) are proved in~\cite[Corollary~3.5]{Sapir-21} and~\cite[Lemma~6.3]{Sapir-21}, respectively.

\smallskip

(iii) Take an arbitrary $\mathtt v\in\{\mathtt u\}^{\le_{\FIC(\mathbf M(xyx))}}$.
If the $\FIC(\mathbf M(xyx))$-class $\mathtt v$ is not stable with respect to $\mathbf V$,  then $M(xyx)\notin\mathbf V$.
According to Lemma~\ref{L: M(W) in V}, $\mathbf V$ satisfies a non-trivial identity $xyx\approx \mathbf w$ for some $\mathbf w\in\mathfrak X^\ast$.
In view of the condition~\eqref{eq: condition for FIC(M(xyx))}, there are $\mathbf u\in\mathtt u$ and $a,b\in\mathfrak X^\ast$ such that $\mathbf u(a,b)=aba$.
It is easy to see that the identity $xyx\approx \mathbf w$ implies an identity $\mathbf u\approx \mathbf u^\prime$ with $\mathbf u^\prime(a,b)\ne aba$.
Hence $\mathbf u^\prime\notin \mathtt u$, contradicting the assumption that the $\FIC(\mathbf M(xyx))$-class $\mathtt u$ is stable with respect to $\mathbf V$.
Therefore, the $\FIC(\mathbf M(xyx))$-class $\mathtt v$ is stable with respect to $\mathbf V$.

\smallskip

Parts~(iv) and~(v) readily follow from Parts~(i)--(iii).
\end{proof}

The following statement directly follows from Lemmas~\ref{L: M_alpha(W) in V} and~\ref{L: subclasses are stable}.

\begin{corollary}
\label{C: M_alpha(W) in V} 
Let $\alpha$ be a congruence on $\mathfrak X^\ast$, $\mathbf V$ a monoid variety and $\mathtt W$ a set of $\alpha$-classes.
Assume that one of the following holds:
\begin{itemize}
\item[\textup{(i)}] $\alpha=\gamma$;
\item[\textup{(ii)}] $\alpha=\lambda$ and every $\alpha$-class in $\mathtt W$ is $2$-island-limited;
\item[\textup{(iii)}] $\alpha=\FIC(\mathbf M(xyx))$ and for each $\mathtt u\in\mathtt W$ the condition~\eqref{eq: condition for FIC(M(xyx))} holds;
\item[\textup{(iv)}] $\alpha=\gamma^\prime$ and for each $\mathtt u\in\mathtt W$ the condition~\eqref{eq: condition for FIC(M(xyx))} holds;
\item[\textup{(v)}] $\alpha=\lambda^\prime$, each $\alpha$-class in $\mathtt W$ is $2$-island-limited and for each $\mathtt u\in\mathtt W$ the condition~\eqref{eq: condition for FIC(M(xyx))} holds.
\end{itemize}
Then $M_\alpha(\mathtt W)$ lies in $\mathbf V$ if and only if every $\alpha$-class in $\mathtt W$ is stable with respect to $\mathbf V$.\qed
\end{corollary} 

We notice that Parts~(i) and~(ii) of Corollary~\ref{C: M_alpha(W) in V} are Corollaries~3.6 and~6.4 in~\cite{Sapir-21}, respectively.

Define one more relation relative to the congruence $\lambda$: for every $\mathbf u,\mathbf v \in \mathfrak X^\ast$, 
\begin{itemize}
\item $\mathbf u\mathrel{\beta}\mathbf v$ if and only if $\mathbf u\mathrel{\gamma}\mathbf v$ and the first two occurrences of each multiple letter lie in the same block in $\mathbf u$ if and only if these occurrences lie in the same block in $\mathbf v$.
\end{itemize}
It is easy to see that the relation $\beta$ is actually a congruence on $\mathfrak X^\ast$.
Indeed, first notice that two $\gamma$-related words begin and end with the same letters. 
Using this fact, it is straightforward to verify that $\beta$ relation is stable under multiplication in $\mathfrak X^\ast$.

If $\alpha$ is an equivalence relation on the free monoid $\mathfrak X^\ast$, then a
word $\mathbf u$ is said to be a $\alpha$-\textit{term} for a variety $\mathbf V$ if $\mathbf u\mathrel{\alpha}\mathbf v$ whenever $\mathbf V$ satisfies $\mathbf u \approx \mathbf v$.

\begin{lemma}
\label{L: M_beta(W) in V} 
Let $\mathbf V$ be a monoid variety and $\mathtt W$ a set of $\beta$-classes such that, for each $\mathtt u\in\mathtt W$, there are $\mathbf u\in\mathtt u$ and $x,y\in \mathfrak X$ such that $\mathbf u(x,y)= xyx^p$ for some $p\in\mathbb N$.
Then $M_\beta(\mathtt W)$ lies in $\mathbf V$ if and only if each $\beta$-class in $\mathtt W$ is stable with respect to $\mathbf V$.
\end{lemma}

\begin{proof}
In view of Lemma~\ref{L: M_alpha(W) in V}, it remains to verify that if $\mathtt u\in\mathtt W$ is stable with respect to $\mathbf V$ and $\mathtt v\in\{\mathtt u\}^{\le_{\beta}}$, then $\mathtt v$ is also stable with respect to $\mathbf V$.
Take an arbitrary $\mathbf v \in \mathtt v$. 
Then the word $\mathbf v$ is a factor of some word $\mathbf u \in \mathtt u$ by Lemma~\ref{L: le_alpha}. 
Since the word $\mathbf u$ is a $\gamma$-term for the variety $\mathbf V$ by~\cite[Observation~2.4]{Sapir-21}, the word $\mathbf v$ is a $\gamma$-term for this variety by~\cite[Lemma~3.3 and Fact~3.4]{Sapir-21}.
If $\mathbf v$ is not a $\beta$-term for $\mathbf V$, then $\mathbf V$ satisfies an identity $\mathbf v \approx \mathbf v^\prime$ such that $\mathbf v\mathrel{\gamma}\mathbf v^\prime$ but $\mathbf v^\prime\notin\mathtt v$.
In this case, there are $x,y\in\mathfrak X$ and $p,r\in\mathbb N$, $q\ge2$ such that $\mathbf v(x,y)$ coincides with one of the words $xyx^p$ or $xx^qyx^r$, while $\mathbf v^\prime(x,y)$ coincides with the other one.
This implies that $\mathbf V$ satisfies $xyx^p\approx x^qyx^r$.
In view of the condition of the lemma, one can find $\mathbf u^\prime\in\mathtt u$ and $z\in\con(\mathbf u^\prime)$ such that $\mathbf u^\prime=\mathbf pz\mathbf qz^p\mathbf r$, $z\notin\con(\mathbf p\mathbf q)$ and $\mathbf q$ contains a letter which is simple in $\mathbf u^\prime$.
Clearly, $\mathbf u^\prime\approx \mathbf pz^q\mathbf qz^r\mathbf r$ is directly deducible from $xyx^p\approx x^qyx^r$ and so holds in $\mathbf V$.
However, $\mathbf px^q\mathbf qx^r\mathbf r\notin\mathtt u$, contradicting the assumption that the $\beta$-class $\mathtt u$ is stable with respect to the variety $\mathbf V$.
\end{proof}

Let
\[
\gamma^{\prime\prime}:=\gamma\wedge\FIC(\mathbf M(xyxty,ytxyx)).
\]
A $2$-island-limited word $\mathbf w \in \mathfrak X^\ast$ is $2$-\textit{island-rigid} if $\occ_x(\mathbf w)=2$ for each letter forming two islands in $\mathbf w$.
A set $W$ of words is $2$-\textit{island-rigid} if each word in $W$ is $2$-island-rigid.

\begin{lemma}
\label{L: M_{gamma''}(W) in V} 
Let $\mathbf V$ be a monoid variety and $\mathtt W$ a set of $2$-island-rigid $\gamma^{\prime\prime}$-classes.
Then $M_{\gamma^{\prime\prime}}(\mathtt W)$ lies in $\mathbf V$ if and only if each $\gamma^{\prime\prime}$-class in $\mathtt W$ is stable with respect to $\mathbf V$.
\end{lemma}

\begin{proof}
In view of Lemma~\ref{L: M_alpha(W) in V}, it remains to verify that if $\mathtt u\in\mathtt W$ is stable with respect to $\mathbf V$ and $\mathtt v\in\{\mathtt u\}^{\le_{\gamma^{\prime\prime}}}$, then $\mathtt v$ is also stable with respect to $\mathbf V$.
Take an arbitrary $\mathbf v \in \mathtt v$. 
Then the word $\mathbf v$ is a factor of some word $\mathbf u \in \mathtt u$ by Lemma~\ref{L: le_alpha}, i.e., $\mathbf u=\mathbf p\mathbf v\mathbf q$ for some $\mathbf p,\mathbf q\in\mathfrak X^\ast$. 
Since the word $\mathbf u$ is a $\gamma$-term for the variety $\mathbf V$ by~\cite[Observation~2.4]{Sapir-21}, the word $\mathbf v$ is also a $\gamma$-term for this variety by~\cite[Lemma~3.3 and Fact~3.4]{Sapir-21}.
If $\mathbf v$ is not a $\gamma^{\prime\prime}$-term for $\mathbf V$, then $\mathbf V$ satisfies an identity $\mathbf v \approx \mathbf v^\prime$ such that $\mathbf v\mathrel{\gamma}\mathbf v^\prime$ but $\mathbf v \approx\mathbf v^\prime$ does not hold in $M(xyxty,ytxyx)$.
In this case, we may assume without any loss that there are $x,y,t\in\mathfrak X$ and $X\subseteq\{x,y,t\}$ such that one of the words $\mathbf v(X)$ or $\mathbf v^\prime(X)$ belong to the set $\{xyxty\}^\le$, while the other one does not.
Since $\mathbf v$ is 2-island-rigid and $\mathbf v\mathrel{\gamma}\mathbf v^\prime$, this is only possible when $\mathbf v(X)\in\{xyx,xyxt,yxty,xyxty\}$ and either $\occ_x(\mathbf v^\prime(X))>2$ or $\occ_y(\mathbf v^\prime(X))>2$.
The identity $\mathbf u\approx \mathbf p\mathbf v^\prime\mathbf q$ holds in $\mathbf V$.
However, the $\gamma^{\prime\prime}$-class $\mathtt u$ is 2-island-rigid.
Hence $\mathbf p\mathbf v^\prime\mathbf r\notin\mathtt u$, contradicting the assumption that the $\gamma^{\prime\prime}$-class $\mathtt u$ is stable with respect to $\mathbf V$.
\end{proof}

Define four more congruences on $\mathfrak X^\ast$, which will be used in the present paper: 
\begin{itemize}
\item $\alpha_1:=\FIC(\var\{xyx^2\approx x^2yx,\,x^2y^2\approx y^2x^2,\,\sigma_3\})$;
\item $\eta:=\FIC(\var\{xyxz\approx xyxzx,\,\sigma_2\}\vee\mathbf M(xyx))$;
\item $\mu:=\FIC(\var\{x^2\approx x^3,\,xyxzx\approx xyxzx^2\})$;
\item $\nu:=\FIC(\var\{xy\approx xyx\}\vee\mathbf M(xyx))$.
\end{itemize}
Given a congruence $\alpha$ on $\mathfrak X^\ast$ and a word $\mathbf u\in\mathfrak X^\ast$, let $[\mathbf u]^\alpha$ denote the $\alpha$-class containing $\mathbf u$.

\subsection{Defining congruence classes by regular expressions}

We use regular expressions to describe sets of words, in particular, congruence classes. 
Given a letter $x\in\mathfrak X^\ast$, we let $x^+:=\{x^n\mid n\in\mathbb N\}$ and $x^\ast:=\{x^n\mid n\in\mathbb N_0\}$.
Using this notation, one can represent many concrete congruence classes as words in the alphabet $\{x,\,x^+,\,x^\ast\mid x\in\mathfrak X^\ast\}$.
For example, $[xyx]^\lambda=\{xyx^n\mid n\in\mathbb N\}$.
Using regular expressions, we write $[xyx]^\lambda=xyx^+$.
The following routinely checked example describes all $\mu$-classes in $\{[xyzxtysx]^{\mu}\}^{\le_\mu}$.
\begin{example} 
\label{E: xyzxtysx^+}
\[
\begin{aligned}
\{[xyzxtysx]^{\mu}\}^{\le_\mu}=\{xyzxtysx^+\}^{\le_\mu} = \{&1,  x, xx^+, y, z, t, s,\\
&xy, yz, zx, xt, ty, ys, sx, sxx^+,\\
&xyz, yzx,zxt, xty, tys, ysx, ysxx^+\\
&xyzx, yzxt, zxty, xtys, tysx, tysxx^+,\\
&xyzxt, yzxty, zxtys, xtysx, xtysxx^+,\\
&xyzxty, yzxtys, zxtysx, zxtysxx^+,\\
&xyzxtys,yzxtysx,yzxtysxx^+,\\
&xyzxtysx^+\}.
\end{aligned}
\]
\end{example}

\subsection{Lattices of subvarieties of two varieties}

The subvariety of a variety $\mathbf V$ defined by a set $\Sigma$ of identities is denoted by $\mathbf V \Sigma$.
Let 
\[
\mathbf H:=\mathbf D_1\{xyxty\approx yx^2ty\}.
\]
Let $\mathbf T$ and $\mathbf{SL}$ denote the trivial variety of monoids and the variety of all semilattice monoids, respectively.

\begin{lemma}[\mdseries{\!\cite[Proposition~6.1]{Gusev-Vernikov-18},~\cite[Theorem~7.2]{Sapir-21}}]
\label{L: L(D1)}
The lattice $\mathfrak L(\mathbf D_1)$ is a chain.
The bottom seven elements of this chain are $\mathbf T\subset \mathbf{SL}\subset \mathbf M(x)\subset \mathbf M(xy)\subset \mathbf M_\gamma(yxx^+)\subset \mathbf M_\lambda(xyx^+)\subset \mathbf H$.\qed
\end{lemma}

\begin{lemma}[\mdseries{\!\cite[Proposition~3.1]{Gusev-20},~\cite[Theorem~7.2]{Sapir-21}}]
\label{L: L(M(xzyx^+ty^+))}
The lattice $\mathfrak L\left(\mathbf M_\lambda(xzyx^+ty^+)\right)$ is given in Fig.~\ref{F: L(J)}.\qed
\end{lemma}

\begin{figure}[htb]
\unitlength=1mm
\linethickness{0.4pt}
\begin{center}
\begin{picture}(55,95)
\put(35,5){\circle*{1.33}}
\put(35,15){\circle*{1.33}}
\put(35,25){\circle*{1.33}}
\put(35,35){\circle*{1.33}}
\put(45,45){\circle*{1.33}}
\put(25,45){\circle*{1.33}}
\put(35,55){\circle*{1.33}}
\put(15,55){\circle*{1.33}}
\put(25,65){\circle*{1.33}}
\put(25,75){\circle*{1.33}}
\put(25,85){\circle*{1.33}}
\put(25,95){\circle*{1.33}}

\put(35,5){\line(0,1){30}}
\put(35,35){\line(-1,1){20}}
\put(35,35){\line(1,1){10}}
\put(25,45){\line(1,1){10}}
\put(25,45){\line(1,1){10}}
\put(15,55){\line(1,1){10}}
\put(45,45){\line(-1,1){20}}
\put(25,65){\line(0,1){30}}

\put(35,2){\makebox(0,0)[cc]{\textbf T}}
\put(37,15){\makebox(0,0)[lc]{\textbf{SL}}}
\put(37,25){\makebox(0,0)[lc]{$\mathbf M(x)$}}
\put(37,35){\makebox(0,0)[lc]{$\mathbf M(xy)$}}
\put(23,45){\makebox(0,0)[rc]{$\mathbf M_\gamma(yxx^+)$}}
\put(13,55){\makebox(0,0)[rc]{$\mathbf M_\lambda(xyx^+)$}}
\put(47,45){\makebox(0,0)[lc]{$\mathbf M_\gamma(xx^+y)$}}
\put(37,55){\makebox(0,0)[lc]{$\mathbf M_\gamma(yxx^+)\vee\mathbf M_\gamma(xx^+y)$}}
\put(27,65){\makebox(0,0)[lc]{$\mathbf M_\lambda(xyx^+)\vee\mathbf M_\gamma(xx^+y)$}}
\put(27,75){\makebox(0,0)[lc]{$\mathbf M_\lambda(xyzx^+ty^+)$}}
\put(27,85){\makebox(0,0)[lc]{$\mathbf M_\lambda(yxx^+ty^+)$}}
\put(27,95){\makebox(0,0)[lc]{$\mathbf M_\lambda(xzyx^+ty^+)$}}
\end{picture}
\end{center}
\caption{The lattice $\mathfrak L\left(\mathbf M_\lambda(xzyx^+ty^+)\right)$}
\label{F: L(J)}
\end{figure}
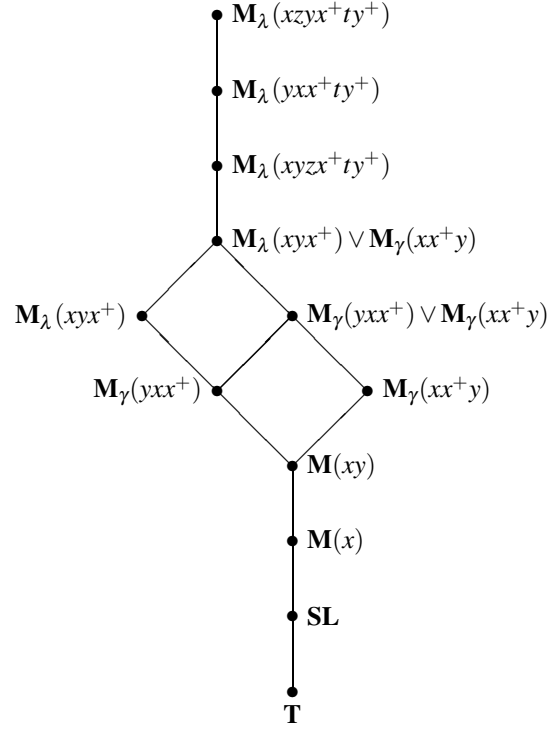

\section{Separating identities for varieties}
\label{Sec: separating identities}

Here we provide separating identities which hold in varieties not containing given monoid varieties.

\subsection{Many certain identities}

The following statement was established in the proof of Lemma~3.5 in~\cite{Gusev-Vernikov-21}.
It also can be readily deduced from Lemmas~\ref{L: M(W) in V} and~\ref{L: identities of M(xt_1x...t_kx)}.

\begin{lemma}
\label{L: swapping in linear-balanced}
Let $\mathbf V$ be a monoid variety such that $M(xt_1x\cdots t_nx)\in\mathbf V$. 
If $M(\mathbf p\,xy\,\mathbf q)\notin\mathbf V$, where $\mathbf p:=a_1t_1\cdots a_kt_k$ and $\mathbf q:=t_{k+1}a_{k+1}\cdots t_{k+\ell}a_{k+\ell}$ for some $k,\ell\in\mathbb N_0$ and $a_1,\dots,a_{k+\ell}$ are letters such that $\{a_1,\dots,a_{k+\ell}\}=\{x,y\}$ and $\occ_x(\mathbf p\mathbf q),\occ_y(\mathbf p\mathbf q)\le n$, then $\mathbf V$ satisfies the identity $\mathbf p\,xy\,\mathbf q\approx\mathbf p\,yx\,\mathbf q$.\qed
\end{lemma}

\begin{lemma}
\label{L: nsub M(xy)}
Let $\mathbf V$ be a monoid variety satisfying the identities $x^n\approx x^{n+1}$ and $x^ny^n\approx y^nx^n$ for some $n\in\mathbb N$.
If $M(xy)\notin\mathbf V$, then $\mathbf V$ is commutative.
\end{lemma}

\begin{proof}
In view of Lemma~2.7 in~\cite{Gusev-Sapir-22} and Lemma~\ref{L: M(W) in V}, the variety $\mathbf V$ is either commutative or idempotent.
If $\mathbf V$ is idempotent, then $\mathbf V$ is commutative as well because $x^ny^n\approx y^nx^n$ together with $x\approx x^2$ imply the commutative law.
Thus, $\mathbf V$ is commutative in any case.
\end{proof}

The following lemma can be proved by the arguments similar to ones from the proof of Lemma~4.1 in~\cite{Gusev-Sapir-22}.

\begin{lemma}
\label{L: nsub M(yxx^+)}
Let $\mathbf V$ be a monoid variety satisfying the identity $x^n\approx x^{n+1}$ for some $n\in\mathbb N$ such that $M(xy)\in \mathbf V$.
If $M_\gamma(yxx^+)\notin \mathbf V$, then $\mathbf V$ satisfies the identity $yx^n\approx xyx^n$.\qed
\end{lemma}

\begin{lemma}
\label{L: nsub M(xyx^+)}
Let $\mathbf V$ be a monoid variety satisfying the identity $x^2\approx x^3$ such that $M(xy)\in \mathbf V$.
If $M_\lambda(xyx^+)\notin \mathbf V$, then $\mathbf V$ satisfies the identity
\begin{equation}
\label{xyxx=xxyxx}
xyx^2\approx x^2yx^2.
\end{equation}
\end{lemma}

\begin{proof}
In view of Corollary~\ref{C: M_alpha(W) in V}(ii), the $\lambda$-class $xyx^+$ is not stable with respect to $\mathbf V$.
This means that $\mathbf V$ satisfies an identity $\mathbf u\approx \mathbf v$ such that $\mathbf u\in xyx^+$ and $\mathbf v\notin xyx^+$.
Since $M(xy)\in \mathbf V$, Lemma~\ref{L: identities of M(xy)} implies that $\mathbf v\in yxx^+\cup xx^+yx^\ast$.
If $\mathbf v\in yxx^+$, then the identity $x\mathbf ux\approx x\mathbf vx$ is equivalent modulo $x^2\approx x^3$ to~\eqref{xyxx=xxyxx}.
If $\mathbf v\in xx^+yx^\ast$, then the identity $\mathbf ux^2\approx \mathbf vx^2$ is equivalent modulo $x^2\approx x^3$ to~\eqref{xyxx=xxyxx} again. 
\end{proof}

The expression $_{i\mathbf w}x$ means the $i$th occurrence of a letter $x$ in a word $\mathbf w$. 
We use $_{\ell\mathbf w}x$ to refer to the last occurrence of $x$ in $\mathbf w$.  
If the $i$th occurrence of $x$ precedes the $j$th occurrence of $y$ in a word $\mathbf w$, then we write $({_{i\mathbf w}x}) < ({_{j\mathbf w}y})$.

\begin{lemma}
\label{L: nsub M([yx^2zy]),M(xx^+yty)}
Let $\mathbf V$ be a monoid variety satisfying the identity $x^2\approx x^3$ such that $M(xyx)\in \mathbf V$.
\begin{itemize}
\item[\textup{(i)}] If $M_\nu([yx^2ty]^\nu)\notin \mathbf V$, then $\mathbf V$ satisfies the identity
\begin{equation}
\label{yxxtxxyxx=xxyxxtxxyxx}
yx^2tx^2yx^2\approx x^2yx^2tx^2yx^2.
\end{equation}
\item[\textup{(ii)}] If $M_{\gamma^\prime}(xx^+yty)\notin \mathbf V$, then $\mathbf V$ satisfies the identity
\begin{equation} 
\label{xxyty=xxyxty}
x^2yty\approx x^2yxty.
\end{equation} 
\end{itemize}
\end{lemma}

\begin{proof}
(i) If $M_\gamma(yxx^+)\notin \mathbf V$, then $\mathbf V$ satisfies the identity
\begin{equation}
\label{yxx=xyxx}
yx^2\approx xyx^2
\end{equation}
by Lemma~\ref{L: nsub M(yxx^+)}.
Evidently, the latter identity implies the identity~\eqref{yxxtxxyxx=xxyxxtxxyxx}.
So, we may further assume that $M_\gamma(yxx^+)\in \mathbf V$.
Using Lemma~\ref{L: le_alpha}, it is routine to check that
\[
\begin{aligned}
\{[yx^2ty]^\nu\}^{\le_\nu}= \{&1,  x,  y, t,[x^2]^\nu,\\
&xy, xt, yx,  tx, ty, [x^2y]^\nu, [x^2t]^\nu, [yx^2]^\nu,  [tx^2]^\nu,\\
&xyx, xtx, xty, yxt, txy, tyx, [x^2ty]^\nu, [yx^2t]^\nu,  [tx^2y]^\nu, [tyx^2]^\nu,\\ 
&xtxy, xtyx, yxtx, yxty, txyx, yxtxy, yxtyx, [yx^2ty]^\nu\}.
\end{aligned}
\]
Since $M(xyx)\in\mathbf V$, Lemma~\ref{L: M(W) in V} implies that all singleton $\nu$-classes in $\{[yx^2ty]^\nu\}^{\le_\nu}$ together with $[x^2]^\nu$ are stable with respect to $\mathbf V$.
Further, consider an identity $\mathbf w\approx \mathbf w^\prime$ of $\mathbf V$ with $\mathbf w\in[x^2y]^\nu=xx^+yx^\ast\cup xyxx^+$.
Since $\mathbf M(xyx)\vee\mathbf M_\gamma(yxx^+)\subseteq \mathbf V$, Lemmas~\ref{L: M(W) in V} and~\ref{L: M_alpha(W) in V} imply that $\simple(\mathbf w^\prime)=\{y\}$ and $\mul(\mathbf w^\prime)=\{x\}$ but $\mathbf w^\prime\notin yxx^+\cup\{xyx\}$.
Hence $\mathbf w^\prime\in[x^2y]^\nu$ and so $[x^2y]^\nu$ is stable with respect to $\mathbf V$.
By a similar argument we can show that the other non-singleton $\nu$-classes in $\{[yx^2ty]^\nu\}^{\le_\nu}$ except $[yx^2ty]^\nu$ are stable with respect to $\mathbf V$.
This fact and Lemma~\ref{L: M_alpha(W) in V} imply that the $\nu$-class $[yx^2ty]^\nu$ is not stable with respect to $\mathbf V$.
This means that $\mathbf V$ satisfies an identity $\mathbf u\approx \mathbf v$ such that $\mathbf u\in [yx^2ty]^\nu$ and $\mathbf v\notin [yx^2ty]^\nu$.
Since $\mathbf M(xyx)\vee\mathbf M_\gamma(yxx^+)\subseteq \mathbf V$, Lemmas~\ref{L: M(W) in V} and~\ref{L: M_alpha(W) in V} imply that $\mathbf v(y,t)=yty$, $\mathbf v(x,t)\ne xtx$ and $(_{1\mathbf v}x)<(_{1\mathbf v}t)$.
Hence $\mathbf v\in x^+yx^\ast tx^\ast yx^\ast$ because $\mathbf v\notin [yx^2ty]^\nu$.
It follows that $\mathbf u\approx \mathbf v$ together with $x^2\approx x^3$ imply~\eqref{yxxtxxyxx=xxyxxtxxyxx}. 

\smallskip

(ii) If $M_\gamma(xx^+y)\notin \mathbf V$, then $\mathbf V$ satisfies the identity
\begin{equation}
\label{xxy=xxyx}
x^2y\approx x^2yx
\end{equation}
by the dual to Lemma~\ref{L: nsub M(yxx^+)}.
Evidently, the latter identity implies the identity~\eqref{xxyty=xxyxty}.
So, we may further assume that $M_\gamma(xx^+y)\in \mathbf V$.
In view of Corollary~\ref{C: M_alpha(W) in V}(iv), the $\gamma^\prime$-class $xx^+yty$ is not stable with respect to $\mathbf V$.
This means that $\mathbf V$ satisfies an identity $\mathbf u\approx \mathbf v$ such that $\mathbf u\in xx^+yty$ and $\mathbf v\notin xx^+yty$.
Since $\mathbf M(xyx)\vee \mathbf M_\gamma(xx^+y)\subseteq \mathbf V$, Lemmas~\ref{L: M(W) in V} and~\ref{L: M_alpha(W) in V} imply that $\mathbf v(x,t)\in xx^+t$ and $\mathbf v(y,t)=yty$.
Hence $(_{1\mathbf v}y)<(_{\ell\mathbf v}x)$ because $\mathbf v\notin xx^+yty$.
It follows that $\mathbf u\approx \mathbf v$ together with $x^2\approx x^3$ imply~\eqref{xxyty=xxyxty}.
\end{proof}

\begin{lemma}
\label{L: nsub M(yxx^+ty)}
Let $\mathbf V$ be a monoid variety satisfying the identity $x^2\approx x^3$ such that $\mathbf M(xyx)\vee \mathbf M_\gamma(xx^+y)\subseteq \mathbf V$.
If $M_{\gamma^\prime}(yxx^+ty)\notin \mathbf V$, then $\mathbf V$ satisfies the identity
\begin{equation} 
\label{yxxty=xyxxty}
yx^2ty\approx xyx^2ty.
\end{equation}
\end{lemma}

\begin{proof}
In view of Corollary~\ref{C: M_alpha(W) in V}(iv), the $\gamma^\prime$-class $yxx^+ty$ is not stable with respect to $\mathbf V$.
This means that $\mathbf V$ satisfies an identity $\mathbf u\approx \mathbf v$ such that $\mathbf u\in yxx^+ty$ and $\mathbf v\notin yxx^+ty$.
Since $\mathbf M(xyx)\vee \mathbf M_\gamma(xx^+y)\subseteq \mathbf V$, Lemmas~\ref{L: M(W) in V} and~\ref{L: M_alpha(W) in V} imply that $\mathbf v(x,t)\in xx^+t$ and $\mathbf v(y,t)=yty$.
Hence $(_{1\mathbf v}x)<(_{1\mathbf v}y)$ because $\mathbf v\notin yxx^+ty$.
It follows that $\mathbf u\approx \mathbf v$ together with $x^2\approx x^3$ imply~\eqref{yxxty=xyxxty}. 
\end{proof}

\begin{corollary}
\label{C: nsub M(yxx^+ty)}
Let $\mathbf V$ be a monoid variety satisfying the identities $x^2\approx x^3$ and
\begin{equation}
\label{xxytxy=yxxtxy}
x^2ytxy\approx yx^2txy
\end{equation}
such that $M(xyx)\in\mathbf V$.
If $M_{\gamma^\prime}(yxx^+ty)\notin \mathbf V$, then $\mathbf V$ satisfies the identity~\eqref{yxxty=xyxxty}.
\end{corollary}

\begin{proof}
If $M_\gamma(xx^+y)\notin \mathbf V$, then $\mathbf V$ satisfies~\eqref{xxy=xxyx} by the dual to Lemma~\ref{L: nsub M(yxx^+)}.
In this case, the identity~\eqref{yxxty=xyxxty} is also satisfied by $\mathbf V$ because
\[
yx^2ty\stackrel{\eqref{xxy=xxyx}}\approx yx^2txy\stackrel{\eqref{xxytxy=yxxtxy}}\approx x^2ytxy\stackrel{x^2\approx x^3}\approx x^3ytxy\stackrel{\eqref{xxytxy=yxxtxy}}\approx xyx^2txy\stackrel{\eqref{xxy=xxyx}}\approx xyx^2ty.
\]
If $M_\gamma(xx^+y)\in \mathbf V$, then the required claim follows from Lemma~\ref{L: nsub M(yxx^+ty)}.
\end{proof}

Let $\mathbf A$ denote the variety defined by the identities $x^2\approx x^3$ and
\begin{equation}
\label{xxyx=xxyxx}
x^2yx\approx x^2yx^2.
\end{equation}

\begin{lemma}
\label{L: nsub M(x^+yzx^+)}
Let $\mathbf V$ be a subvariety of $\mathbf A$ such that $M(xy)\in \mathbf V$.
If $M_\gamma(x^+yzx^+)\notin \mathbf V$, then $\mathbf V$ satisfies the identity
\begin{equation}
\label{xxyzx=xxyxzx}
x^2yzx\approx x^2yxzx.
\end{equation}
\end{lemma}

\begin{proof}
In view of Corollary~\ref{C: M_alpha(W) in V}(i), the $\gamma$-class $x^+yzx^+$ is not stable with respect to $\mathbf V$.
This means that $\mathbf V$ satisfies an identity $\mathbf u\approx \mathbf v$ such that $\mathbf u\in x^+yzx^+$ and $\mathbf v\notin x^+yzx^+$.
Since $M(xy)\in \mathbf V$, Lemma~\ref{L: identities of M(xy)} implies that $\mathbf v(y,z)=yz$ and $\mul(\mathbf v)=\{x\}$.
Then it is routine to check that
\begin{itemize}
\item if $(_{1\mathbf v}y)<(_{1\mathbf v}x)$, then the identity $x^2s\mathbf u(x,y)x\approx x^2s\mathbf v(x,y)x$ is equivalent modulo $\{x^2\approx x^3,\,\eqref{xxyx=xxyxx}\}$ to~\eqref{xxyzx=xxyxzx};
\item if $(_{\ell\mathbf v}x)<(_{1\mathbf v}z)$, then the identity $x^2\mathbf u(x,z)sx\approx x^2\mathbf v(x,z)sx$ is equivalent modulo $\{x^2\approx x^3,\,\eqref{xxyx=xxyxx}\}$ to~\eqref{xxyzx=xxyxzx};
\item if $(_{1\mathbf v}x)<(_{1\mathbf v}y)<(_{j\mathbf v}x)<(_{1\mathbf v}z)<(_{\ell\mathbf v}x)$ for some $1<j<\occ_x(\mathbf v)$, then the identity $x\mathbf u\approx x\mathbf v$ is equivalent modulo $\{x^2\approx x^3,\,\eqref{xxyx=xxyxx}\}$ to~\eqref{xxyzx=xxyxzx}.
\end{itemize}
Hence $\mathbf V$ satisfies the identity~\eqref{xxyzx=xxyxzx} in any case. 
\end{proof}

\begin{lemma}
\label{L: nsub M([x^2yzytx^2]),M([x^2yzx^2ty])}
Let $\mathbf V$ be a subvariety of $\mathbf A$ satisfying the identity~\eqref{yxxty=xyxxty} such that $M(xyx)\in \mathbf V$.
\begin{itemize}
\item[\textup{(i)}] If $M_{\gamma^\prime}([x^2yzytx^2]^{\gamma^\prime})\notin \mathbf V$, then $\mathbf V$ satisfies the identity
\begin{equation} 
\label{xxyzytx=yxxzytx}
x^2yzytx\approx yx^2zytx.
\end{equation} 
\item[\textup{(ii)}] If $M_{\gamma^\prime}([x^2yzx^2ty]^{\gamma^\prime})\notin \mathbf V$, then $\mathbf V$ satisfies the identity
\begin{equation} 
\label{xxyzxty=yxxzxty}
x^2yzxty\approx yx^2zxty.
\end{equation} 
\end{itemize}
\end{lemma}

\begin{proof}
(i) If $M_\gamma(x^+yzx^+)\notin \mathbf V$, then $\mathbf V$ satisfies the identity~\eqref{xxyzx=xxyxzx} by Lemma~\ref{L: nsub M(x^+yzx^+)}.
Evidently, the latter identity together with~\eqref{yxxty=xyxxty} imply the identity~\eqref{xxyzytx=yxxzytx}.
So, we may further assume that $M_\gamma(x^+yzx^+)\in \mathbf V$.
In view of Corollary~\ref{C: M_alpha(W) in V}(iv), the $\gamma^\prime$-class $[x^2yzytx^2]^{\gamma^\prime}$ is not stable with respect to $\mathbf V$.
This means that $\mathbf V$ satisfies an identity $\mathbf u\approx \mathbf v$ such that $\mathbf u\in [x^2yzytx^2]^{\gamma^\prime}$ and $\mathbf v\notin [x^2yzytx^2]^{\gamma^\prime}$.
Since $\mathbf M(xyx)\vee \mathbf M_\gamma(x^+yzx^+)\subseteq \mathbf V$, Lemmas~\ref{L: M(W) in V} and~\ref{L: M_alpha(W) in V} imply that $\mathbf v(x,z,t)\in x^+ztx^+$, $\occ_x(\mathbf v)\ge3$ and $\mathbf v(y,z,t)=yzyt$.
Hence there is an occurrence of $x$ between $_{1\mathbf v}y$ and $_{1\mathbf v}z$ in $\mathbf v$ because $\mathbf v\notin [x^2yzytx^2]^{\gamma^\prime}$.
It follows that $x^2\mathbf u\approx x^2\mathbf v$ is equivalent modulo $\{\eqref{yxxty=xyxxty},\,\eqref{xxyx=xxyxx}\}$ to~\eqref{xxyzytx=yxxzytx}.

\smallskip

(ii) The proof is quite similar to the proof of Part~(i).
\end{proof}

The proof of the following lemma is similar to the proof of Lemma~\ref{L: nsub M([x^2yzytx^2]),M([x^2yzx^2ty])} and we omit it.

\begin{lemma}
\label{L: nsub M([yzx^2ytx^2])}
Let $\mathbf V$ be a subvariety of $\mathbf A$ satisfying the identity
\begin{equation}
\label{ytyxx=ytxyxx}
ytyx^2\approx ytxyx^2
\end{equation}
such that $M(xyx)\in \mathbf V$.
If $M_{\gamma^\prime}([yzx^2ytx^2]^{\gamma^\prime})\notin \mathbf V$, then the identity
\begin{equation} 
\label{yzxxytx=yzyxxtx}
yzx^2ytx\approx yzyx^2tx
\end{equation}
holds in $\mathbf V$.\qed
\end{lemma}

\begin{lemma}
\label{L: nsub M([x^2zytx^2y])}
Let $\mathbf V$ be a subvariety of $\mathbf A$ satisfying~\eqref{ytyxx=ytxyxx} and
\begin{equation}
\label{xxytxy=xxytyx}
x^2ytxy\approx x^2ytyx
\end{equation}
such that $M(xyx)\in\mathbf V$.
If $M_{\gamma^\prime}([x^2zytx^2y]^{\gamma^\prime})\notin \mathbf V$, then $\mathbf V$ satisfies the identity
\begin{equation} 
\label{xxzytxy=xxzytyx}
x^2zytxy\approx x^2zytyx.
\end{equation}
\end{lemma}

\begin{proof}
If $M_\gamma(x^+yzx^+)\notin \mathbf V$, then $\mathbf V$ satisifies~\eqref{xxyzx=xxyxzx} by Lemma~\ref{L: nsub M(x^+yzx^+)}.
In this case, the identity~\eqref{xxzytxy=xxzytyx} holds in $\mathbf V$ because
\[
x^2zytxy\stackrel{\eqref{xxyzx=xxyxzx}}\approx
x^2zx^2ytxy\stackrel{\eqref{xxytxy=xxytyx}}\approx
x^2zx^2ytyx\stackrel{\eqref{xxyzx=xxyxzx}}\approx
x^2zytyx.
\]
So, we may further assume that $M_\gamma(x^+yzx^+)\in \mathbf V$.
In view of Corollary~\ref{C: M_alpha(W) in V}(iv), the $\gamma^\prime$-class $[x^2zytx^2y]^{\gamma^\prime}$ is not stable with respect to $\mathbf V$.
This means that $\mathbf V$ satisfies an identity $\mathbf u\approx \mathbf v$ such that $\mathbf u\in [x^2zytx^2y]^{\gamma^\prime}$ and $\mathbf v\notin [x^2zytx^2y]^{\gamma^\prime}$.
Since $\mathbf M(xyx)\vee \mathbf M_\lambda(x^+yzx^+)\subseteq \mathbf V$, Lemmas~\ref{L: M(W) in V} and~\ref{L: M_alpha(W) in V} imply that $\mathbf v(x,z,t)\in x^+ztx^+$, $\occ_x(\mathbf v)\ge3$ and $\mathbf v(y,z,t)=zyty$.
Hence $(_{\ell\mathbf v}y)<(_{\ell\mathbf v}x)$ because $\mathbf v\notin [x^2zytx^2y]^{\gamma^\prime}$.
It follows that $\mathbf u\approx \mathbf v$ together with $\{x^2\approx x^3,\,\eqref{xxyx=xxyxx},\,\eqref{ytyxx=ytxyxx}\}$ imply~\eqref{xxzytxy=xxzytyx}.
\end{proof}

\begin{lemma}
\label{L: nsub M([xyzx^2ty]),M(xyzytxx^+),M(yzyxtxx^+)}
Let $\mathbf V$ be a subvariety of $\mathbf A$ such that $\mathbf M(xyx)\vee \mathbf M_\lambda(xyx^+)\subseteq \mathbf V$.
\begin{itemize}
\item[\textup{(i)}] If $M_\eta([xyzx^2ty]^\eta)\notin \mathbf V$, then $\mathbf V$ satisfies the identity
\begin{equation} 
\label{xyzxxtxyx=yxzxxtxyx}
xyzx^2txyx\approx yxzx^2txyx.
\end{equation} 
\item[\textup{(ii)}] If $M_{\lambda^\prime}(xyzytxx^+)\notin \mathbf V$, then $\mathbf V$ satisfies the identity
\begin{equation} 
\label{xyzytxx=yxzytxx}
xyzytx^2\approx yxzytx^2.
\end{equation} 
\item[\textup{(iii)}] If $M_{\lambda^\prime}(yzyxtxx^+)\notin \mathbf V$, then $\mathbf V$ satisfies the identity
\begin{equation} 
\label{yzxytxx=yzyxtxx}
yzxytx^2\approx yzyxtx^2.
\end{equation} 
\end{itemize}
\end{lemma}

\begin{proof}
(i) If $M(xyzxty)\notin \mathbf V$, then $\mathbf V$ satisfies the identity $\sigma_1$ by Lemma~\ref{L: swapping in linear-balanced}.
Evidently, the latter identity implies the identity~\eqref{xyzxxtxyx=yxzxxtxyx}.
So, we may further assume that $M(xyzxty)\in \mathbf V$.
Using Lemma~\ref{L: le_alpha}, it is routine to check that
\[
\begin{aligned}
\{&[xyzx^2ty]^\eta\}^{\le_\eta}= \{1,  x,  y, z, t, [x^2]^\eta,\\
&xy, xt, yx, yz, zx, tx, ty, [x^2y]^\eta,[x^2t]^\eta, [yx^2]^\eta, [zx^2]^\eta,  [tx^2]^\eta,\\
&xyx, xyz, xtx, xty, yzx, zxt, txy, tyx, [xyx^2]^\eta,[xtx^2]^\eta,[x^2ty]^\eta, [yzx^2]^\eta, [zx^2t]^\nu, [tx^2y]^\eta,[tyx^2]^\eta,\\ 
&xyzx,xtxy,xtyx,yzxt,zxtx,zxty,txyx, [xyzx^2]^\eta, [xtx^2y]^\eta,[xtyx^2]^\eta,[yzx^2t]^\eta, [zxtx^2]^\eta, [zx^2ty]^\eta, \\ 
&xyzxt, yzxtx, yzxty, zxtxy, zxtyx, [xyzx^2t]^\eta,[yzxtx^2]^\eta, [yzx^2ty]^\eta, [zxtx^2y]^\eta, [zxtyx^2]^\eta,[txyx^2]^\eta,\\
&xyzxty, \{yzxtxy, yzxtyx\}, [xyzx^2ty]^\eta, [yzxtx^2y]^\eta\}.
\end{aligned}
\]
Since $M(xyzxty)\in\mathbf V$, Lemma~\ref{L: M(W) in V} implies that all singleton $\eta$-classes in $\{[xyzx^2ty]^\eta\}^{\le_\eta}$ are stable with respect to $\mathbf V$.
Further, consider an identity $\mathbf w\approx \mathbf w^\prime$ of $\mathbf V$ with $\mathbf w\in[yzxtx^2y]^\eta=[yzxtyx^2]^\eta$.
Since $\mathbf M(xyx)\vee\mathbf M_\gamma(xyx^+)\subseteq \mathbf V$, Lemmas~\ref{L: M(W) in V} and~\ref{L: M_alpha(W) in V} imply that $\mathbf w^\prime(x,z,t)\in zxtxx^+$ and $\mathbf w^\prime(y,z,t)\in yzty$.
Hence $\mathbf w^\prime\in[yzxtx^2y]^\eta$ and so $[yzxtx^2y]^\eta$ is stable with respect to $\mathbf V$.
By a similar argument we can show that the other non-singleton $\eta$-classes in $\{[xyzx^2ty]^\eta\}^{\le_\eta}$ except $[xyzx^2ty]^\eta$ are stable with respect to $\mathbf V$.
This fact and Lemma~\ref{L: M_alpha(W) in V} imply that the $\eta$-class $[xyzx^2ty]^\eta$ is not stable with respect to $\mathbf V$.
This means that $\mathbf V$ satisfies an identity $\mathbf u\approx \mathbf v$ such that $\mathbf u\in [xyzx^2ty]^\eta$ and $\mathbf v\notin [xyzx^2ty]^\eta$.
Since $\mathbf M(xyx)\vee\mathbf M_\lambda(xyx^+)\subseteq \mathbf V$, Lemmas~\ref{L: M(W) in V} and~\ref{L: M_alpha(W) in V} imply that $\mathbf v(x,z)\in xzxx^+$, $\mathbf v(y,z,t)=yzty$ and $\mathbf v(x,t)\in xx^+tx^\ast$.
Hence $\mathbf v\in yxz x^+ tx^\ast yx^\ast$ because $\mathbf v\notin [xyzx^2ty]^\eta$.
It follows that $\mathbf u\approx \mathbf v$ together with $\{x^2\approx x^3,\,\eqref{xxyx=xxyxx}\}$ imply~\eqref{xyzxxtxyx=yxzxxtxyx}. 

\smallskip

(ii) In view of Corollary~\ref{C: M_alpha(W) in V}(v), the $\lambda^\prime$-class $xyzytxx^+$ is not stable with respect to $\mathbf V$.
This means that $\mathbf V$ satisfies an identity $\mathbf u\approx \mathbf v$ such that $\mathbf u\in xyzytxx^+$ and $\mathbf v\notin xyzytxx^+$.
Since $\mathbf M(xyx)\vee \mathbf M_\lambda(xyx^+)\subseteq \mathbf V$, Lemmas~\ref{L: M(W) in V} and~\ref{L: M_alpha(W) in V} imply that $\mathbf v(x,z,t)\in xztxx^+$ and $\mathbf v(y,z,t)=yzyt$.
Hence $\mathbf v\in yxzytxx^+$ because $\mathbf v\notin xyzytxx^+$.
It follows that $\mathbf u\approx \mathbf v$ is equivalent modulo $x^2\approx x^3$ to~\eqref{xyzytxx=yxzytxx}. 

\smallskip

(iii) In view of Corollary~\ref{C: M_alpha(W) in V}(v), the $\lambda^\prime$-class $yzyxtxx^+$ is not stable with respect to $\mathbf V$.
This means that $\mathbf V$ satisfies an identity $\mathbf u\approx \mathbf v$ such that $\mathbf u\in yzyxtxx^+$ and $\mathbf v\notin yzyxtxx^+$.
Since $\mathbf M(xyx)\vee \mathbf M_\lambda(xyx^+)\subseteq \mathbf V$, Lemmas~\ref{L: M(W) in V} and~\ref{L: M_alpha(W) in V} imply that $\mathbf v(x,z,t)\in zxtxx^+$ and $\mathbf v(y,z,t)=yzyt$.
Hence $\mathbf v\in yzxytxx^+$ because $\mathbf v\notin yzyxtxx^+$.
It follows that $\mathbf u\approx \mathbf v$ is equivalent modulo $x^2\approx x^3$ to~\eqref{yzxytxx=yzyxtxx}. 
\end{proof}

\begin{lemma}
\label{L: nsub H}
Let $\mathbf V$ be a monoid variety satisfying $\Phi$ such that $M_\lambda(xyx^+)\in \mathbf V$.
If $\mathbf H\nsubseteq \mathbf V$, then $\mathbf V$ satisfies the identity
\begin{equation} 
\label{xyzxxyy=yxzxxyy}
xyzx^2y^2\approx yxzx^2y^2.
\end{equation} 
\end{lemma}

\begin{proof}
One can easlily deduce from Propositions~6.9(i) and~6.12 in~\cite{Gusev-Vernikov-18} that $\mathbf V$ satisfies an identity $xyz\mathbf a\approx yxz\mathbf b$ with $\con(\mathbf a)=\con(\mathbf b)=\{x,y\}$.
Clearly, the identity $xyz\mathbf ax^2y^2\approx yxz\mathbf bx^2y^2$ is equivalent modulo $\Phi$ to~\eqref{xyzxxyy=yxzxxyy}, and we are done.
\end{proof}

\begin{lemma}
\label{L: nsub M(xx^+yty^+)}
Let $\mathbf V$ be a variety satisfying the identity $x^2\approx x^3$ such that $M_\lambda(xyx^+)\in \mathbf V$.
If $M_\lambda(xx^+yty^+)\notin \mathbf V$, then $\mathbf V$ satisfies the identity
\begin{equation} 
\label{xxytyy=xxyxtyy}
x^2yty^2\approx x^2yxty^2.
\end{equation} 
\end{lemma}

\begin{proof}
If $M_\gamma(xx^+y)\notin \mathbf V$, then $\mathbf V$ satisfies~\eqref{xxy=xxyx} by the dual to Lemma~\ref{L: nsub M(yxx^+)}.
Clearly, this identity implies~\eqref{xxytyy=xxyxtyy}.
So, we may further assume that $M_\gamma(xx^+y)\in \mathbf V$.
In view of Corollary~\ref{C: M_alpha(W) in V}(ii), the $\lambda$-class $xx^+yty^+$ is not stable with respect to $\mathbf V$.
This means that $\mathbf V$ satisfies an identity $\mathbf u\approx \mathbf v$ such that $\mathbf u\in xx^+yty^+$ and $\mathbf v\notin xx^+yty^+$.
Since $\mathbf M_\lambda(xyx^+)\vee \mathbf M_\lambda(xx^+y)\subseteq \mathbf V$, Lemma~\ref{L: M_alpha(W) in V} imply that $\mathbf v(x,t)\in xx^+t$ and $\mathbf v(y,t)\in yty^+$.
Hence $(_{1\mathbf v}y)<(_{\ell\mathbf v}x)$ because $\mathbf v\notin xx^+yty^+$.
It follows that $\mathbf u\approx \mathbf v$ together with $x^2\approx x^3$ imply~\eqref{xxytyy=xxyxtyy}.
\end{proof}

\begin{corollary}
\label{C: nsub M(xx^+yty^+)}
Let $\mathbf V$ be a variety satisfying the identities $x^2\approx x^3$ and~\eqref{yxxty=xyxxty} such that $M_\lambda(xyx^+)\in \mathbf V$.
If $M_\lambda(xx^+yty^+)\notin \mathbf V$, then $\mathbf V$ satisfies the identity
\begin{equation} 
\label{xxytyy=yxxtyy}
x^2yty^2\approx yx^2ty^2.
\end{equation} 
\end{corollary}

\begin{proof}
By Lemma~\ref{L: nsub M(xx^+yty^+)}, $\mathbf V$ satisfies the identity~\eqref{xxytyy=xxyxtyy}.
Then the identities $x^2yty^2\stackrel{\eqref{xxytyy=xxyxtyy}}\approx x^2yx^2ty^2 \stackrel{\eqref{yxxty=xyxxty}}\approx yx^2ty^2$ hold in $\mathbf V$.
\end{proof}

\begin{lemma}
\label{L: nsub M(xyzx^+ty^+)}
Let $\mathbf V$ be a variety satisfying the identity $x^2\approx x^3$ such that $\mathbf M_\lambda(xyx^+)\vee \mathbf M_\gamma(xx^+y)\subseteq \mathbf V$.
If $M_\lambda(xyzx^+ty^+)\notin \mathbf V$, then $\mathbf V$ satisfies the identity
\begin{equation} 
\label{xyzxxtyy=yxzxxtyy}
xyzx^2ty^2\approx yxzx^2ty^2.
\end{equation} 
\end{lemma}

\begin{proof}
In view of Corollary~\ref{C: M_alpha(W) in V}(ii), the $\lambda$-class $xyzx^+ty^+$ is not stable with respect to $\mathbf V$.
This means that $\mathbf V$ satisfies an identity $\mathbf u\approx \mathbf v$ such that $\mathbf u\in xyzx^+ty^+$ and $\mathbf v\notin xyzx^+ty^+$.
Since $\mathbf M_\lambda(xyx^+)\vee \mathbf M_\gamma(xx^+y)\subseteq \mathbf V$, Lemma~\ref{L: M_alpha(W) in V} imply that $\mathbf v(x,z)\in xzx^+$, $\mathbf v(x,t)\in xx^+t$, $\mathbf v(y,z)\in yzy^+$ and $\mathbf v(y,t)\in yty^+$.
Hence $\mathbf v\in yxzx^+ty^+$ because $\mathbf v\notin xyzx^+ty^+$.
It follows that $\mathbf u\approx \mathbf v$ together with $x^2\approx x^3$ imply~\eqref{xyzxxtyy=yxzxxtyy}.
\end{proof}

\begin{corollary}
\label{C: nsub M(xyzx^+ty^+)}
Let $\mathbf V$ be a variety satisfying the identities $x^2\approx x^3$ and
\begin{equation}
\label{xyzxxy=yxzxxy}
xyzx^2y\approx yxzx^2y
\end{equation} 
such that $M_\lambda(xyx^+)\in \mathbf V$.
If $M_\lambda(xyzx^+ty^+)\notin \mathbf V$, then $\mathbf V$ satisfies the identity~\eqref{xyzxxtyy=yxzxxtyy}.
\end{corollary}

\begin{proof}
If $M_\gamma(xx^+y)\notin \mathbf V$, then $\mathbf V$ satisfies~\eqref{xxy=xxyx} by the dual to Lemma~\ref{L: nsub M(yxx^+)}.
In this case, the identity~\eqref{xyzxxtyy=yxzxxtyy} is satisfied by $\mathbf V$ because
\[
xyzx^2ty^2\stackrel{\eqref{xxy=xxyx}}\approx xyzx^2ty^2x\stackrel{\eqref{xyzxxy=yxzxxy}}\approx yxzx^2ty^2x\stackrel{\eqref{xxy=xxyx}}\approx yxzx^2ty^2.
\]
If $M_\gamma(xx^+y)\in \mathbf V$, then the required claim follows from Lemma~\ref{L: nsub M(xyzx^+ty^+)}.
\end{proof}

\begin{lemma}
\label{L: nsub M(xyzxx^+ty)}
Let $\mathbf V$ be a monoid variety satisfying the identity $x^2\approx x^3$ such that $\mathbf M(xyx)\vee \mathbf M_\lambda(xyx^+)\vee \mathbf M_\gamma(xx^+y)\subseteq \mathbf V$.
If $M_{\lambda^\prime}(xyzxx^+ty)\notin \mathbf V$, then $\mathbf V$ satisfies the identity
\begin{equation} 
\label{xyzxxty=yxzxxty}
xyzx^2ty\approx yxzx^2ty.
\end{equation}
\end{lemma}

\begin{proof}
In view of Corollary~\ref{C: M_alpha(W) in V}(v), the $\lambda^\prime$-class $xyzxx^+ty$ is not stable with respect to $\mathbf V$.
This means that $\mathbf V$ satisfies an identity $\mathbf u\approx \mathbf v$ such that $\mathbf u\in xyzxx^+ty$ and $\mathbf v\notin xyzxx^+ty$.
Since $\mathbf M(xyx)\vee \mathbf M_\lambda(xyx^+)\vee \mathbf M_\gamma(xx^+y)\subseteq \mathbf V$, Lemmas~\ref{L: M(W) in V} and~\ref{L: M_alpha(W) in V} imply that $\mathbf v(x,z)\in xzxx^+$, $\mathbf v(x,t)\in xx^+t$ and $\mathbf v(y,z,t)=yzyt$.
Hence $\mathbf v\in yxzxx^+ty$ because $\mathbf v\notin xyzxx^+ty$.
It follows that $\mathbf u\approx \mathbf v$ is equivalent modulo $x^2\approx x^3$ to~\eqref{xyzxxty=yxzxxty}.
\end{proof}

\begin{corollary}
\label{C: nsub M(xyzxx^+ty)}
Let $\mathbf V$ be a monoid variety satisfying the identities $x^2\approx x^3$ and~\eqref{xyzxxy=yxzxxy} such that $\mathbf M(xyx)\vee \mathbf M_\lambda(xyx^+)\subseteq \mathbf V$.
If $M_{\lambda^\prime}(xyzxx^+ty)\notin \mathbf V$, then $\mathbf V$ satisfies the identity~\eqref{xyzxxty=yxzxxty}.
\end{corollary}

\begin{proof}
If $M_\gamma(xx^+y)\notin \mathbf V$, then $\mathbf V$ satisfies~\eqref{xxy=xxyx} by the dual to Lemma~\ref{L: nsub M(yxx^+)}.
In this case, the identity~\eqref{xyzxxty=yxzxxty} is satisfied by $\mathbf V$ because
\[
xyzx^2ty\stackrel{\eqref{xxy=xxyx}}\approx xyzx^2tx^2y\stackrel{\eqref{xyzxxy=yxzxxy}}\approx yxzx^2tx^2y\stackrel{\eqref{xxy=xxyx}}\approx yxzx^2ty.
\]
If $M_\gamma(xx^+y)\in \mathbf V$, then the required claim follows from Lemma~\ref{L: nsub M(xyzxx^+ty)}.
\end{proof}

\begin{lemma}
\label{L: nsub M(xyzx^+tysx^+)}
Let $\mathbf V$ be a subvariety of $\mathbf A$ satisfying the identity~\eqref{xyzxxy=yxzxxy} such that $\mathbf M(xyx)\vee \mathbf M_\lambda(xyx^+)\subseteq \mathbf V$. 
If $M_{\lambda^\prime}(xyzx^+tysx^+)\notin \mathbf V$, then $\mathbf V$ satisfies the identity
\begin{equation} 
\label{xyzxxtysx=yxzxxtysx}
xyzx^2tysx\approx yxzx^2tysx.
\end{equation} 
\end{lemma}

\begin{proof}
If $M_\gamma(x^+yzx^+)\notin \mathbf V$, then $\mathbf V$ satisfies~\eqref{xxyzx=xxyxzx} by Lemma~\ref{L: nsub M(x^+yzx^+)}.
In this case, the identity~\eqref{xyzxxtysx=yxzxxtysx} holds in $\mathbf V$ because
\[
xyzx^2tysx\stackrel{\eqref{xxyzx=xxyxzx}}\approx
xyzx^2tx^2ysx\stackrel{\eqref{xyzxxy=yxzxxy}}\approx
yxzx^2tx^2ysx\stackrel{\eqref{xxyzx=xxyxzx}}\approx yxzx^2tysx.
\]
If $M_{\lambda^\prime}(xyzxx^+ty)\notin \mathbf V$, then $\mathbf V$ satisfies the identity~\eqref{xyzxxty=yxzxxty} by Lemma~\ref{L: nsub M(xyzxx^+ty)}.
Clearly, the identity~\eqref{xyzxxty=yxzxxty} implies the identity~\eqref{xyzxxtysx=yxzxxtysx}.
So, we may further assume that $\mathbf M_{\lambda^\prime}(xyzxx^+ty)\vee \mathbf M_\gamma(x^+yzx^+)\subseteq \mathbf V$.
Using Lemma~\ref{L: le_alpha}, it is routine to check that
\[
\begin{aligned}
\{xyzx^+tysx^+\}^{\le_{\lambda^\prime}} \!=\! \{&1,  x, xx^+, y, z, t, s,\\
&xy, yz, zx, zxx^+, xt, xx^+t, ty, ys, sx, sxx^+,\\
&xyz, yzx, yzxx^+,zxt, zxx^+t, xty,xx^+ty, tys, ysx, ysxx^+\\
&xyzx,xyzxx^+, yzxt,yzxx^+t, zxty,zxx^+ty, xtys,xx^+tys, tysx, tysxx^+,\\
&xyzxt,xyzxx^+t, yzxty,yzxx^+ty, zxtys, zxx^+tys,xtysx,xtysxx^+,xx^+tysx^+,\\
&xyzxty, xyzxx^+ty, yzxtys,yzxx^+tys, zxtysx, zxtysxx^+,zxx^+tysx^+,\\
&xyzxtys,xyzxx^+tys,yzxtysx,yzxtysxx^+,yzxx^+tysx^+,\\
&xyzx^+tysx^+\}.
\end{aligned}
\]
Since $\mathbf M(xyzxty)\subset \mathbf M_{\lambda^\prime}(xyzxx^+ty)\subseteq\mathbf V$, Lemma~\ref{L: M(W) in V} implies that all singleton $\lambda^\prime$-classes in $\{xyzx^+tysx^+\}^{\le_{\lambda^\prime}}$ are stable with respect to $\mathbf V$.
Since $\mathbf M_{\lambda^\prime}(xyzxx^+ty)\vee\mathbf M_\gamma(x^+yzx^+)\subseteq \mathbf V$, one can deduce from Lemma~\ref{L: M_alpha(W) in V} that the non-singleton $\lambda^\prime$-classes in $\{xyzx^+tysx^+\}^{\le_{\lambda^\prime}}$ except $xyzx^+tysx^+$ also are stable with respect to $\mathbf V$.
This fact and Lemma~\ref{L: M_alpha(W) in V} imply that the $\lambda^\prime$-class $xyzx^+tysx^+$ is not stable with respect to $\mathbf V$.
This means that $\mathbf V$ satisfies an identity $\mathbf u\approx \mathbf v$ such that $\mathbf u\in xyzx^+tysx^+$ and $\mathbf v\notin xyzx^+tysx^+$.
Since $\mathbf M(xyx)\vee \mathbf M_\lambda(xyx^+)\vee \mathbf M_\gamma(x^+yzx^+)\subseteq \mathbf V$, Lemmas~\ref{L: M(W) in V} and~\ref{L: M_alpha(W) in V} imply that $\mathbf v(x,z)\in xzxx^+$,  $\mathbf v(x,t,s)\in x^+tsx^+$ and $\mathbf v(y,z,t,s)=yztys$.
Hence $\mathbf v\in yxzx^+tysx^+$ because $\mathbf v\notin xyzx^+tysx^+$.
It follows that $\mathbf u\approx \mathbf v$ together with $\{x^2\approx x^3,\,\eqref{xxyx=xxyxx}\}$ imply~\eqref{xyzxxtysx=yxzxxtysx}. 
\end{proof}

\subsection{Identities in $\Phi_1$}

\begin{lemma}
\label{L: nsub M([c_{0,0,n}[rho]]^lambda)}
Let $\mathbf V$ be a variety satisfying the identities
\begin{align}
\label{xyx=xyxx}
xyx\approx xyx^2,\\
\label{xxyy=yyxx}
x^2y^2\approx y^2x^2,\\
\label{xyzxy=yxzxy}
xyzxy\approx yxzxy
\end{align}
such that $\mathbf M_\lambda(xyx^+)\vee \mathbf M_\gamma(xx^+y)\subseteq\mathbf V$.
Suppose that the variety $\mathbf V$ does not contain the monoid $M_\lambda([\mathbf c_{0,0,p}[\pi]]^\lambda)$ for all $p\in\mathbb N$ and $\pi\in S_p$. 
Then $\mathbf V$ satisfies the identity $\mathbf c_{n,m,k}[\rho] \approx\mathbf c_{n,m,k}^\prime[\rho]$ for any $n,m,k\in\mathbb N$ and $\rho\in S_{n+m+k}$.
\end{lemma}

\begin{proof}
Take arbitrary $n,m,k\in \mathbb N$ and $\rho\in S_{n+m+k}$.
For brevity, set
\[
\mathbf h:=z_1t_1\cdots z_nt_n,\ \mathbf s:=tz_{n+1}t_{n+1}\cdots z_{n+m}t_{n+m},\ \mathbf t:=t_{n+m+1}z_{n+m+1}\cdots t_{n+m+k}z_{n+m+k}.
\]
Then one can find sufficiently large $r$, say $r:=2(n+m+k)$, and $\rho^\prime\in S_r$ such that the identity $\mathbf c_{0,0,r}[\rho^\prime]\approx \mathbf c_{0,0,r}^\prime[\rho^\prime]$ implies the identity
\[
\mathbf h\,xy\,\mathbf s\,x\biggl(\prod_{i=1}^{n+m+k-1} z_{i\rho}^{g_i}y_i^2\biggr)z_{(n+m+k)\rho}^{g_{(n+m+k)}}y\,\mathbf t\approx \mathbf h\,yx\,\mathbf s\,x\biggl(\prod_{i=1}^{n+m+k-1} z_{i\rho}^{g_i}y_i^2\biggr)z_{(n+m+k)\rho}^{g_{(n+m+k)}}y\,\mathbf t,
\]
where
\[
g_i:=
\begin{cases} 
2 & \text{if $1\le i\rho\le n+m$}, \\
1 & \text{if $n+m<i\rho\le n+m+k$}.
\end{cases} 
\]
Since
\[
\begin{aligned}
&\mathbf c_{n,m,k}[\rho]\stackrel{\{\eqref{xyx=xyxx}}\approx \mathbf h\,xy\,\mathbf s\,x\biggl(\prod_{i=1}^{n+m+k-1} z_{i\rho}^{g_i}y_i^2\biggr)z_{(n+m+k)\rho}^{g_{(n+m+k)}}y\,\mathbf t,\\
&\mathbf c_{n,m,k}^\prime[\rho]\stackrel{\eqref{xyx=xyxx}}\approx \mathbf h\,yx\,\mathbf s\,x\biggl(\prod_{i=1}^{n+m+k-1} z_{i\rho}^{g_i}y_i^2\biggr)z_{(n+m+k)\rho}^{g_{(n+m+k)}}y\,\mathbf t,
\end{aligned}
\]
it remains to show that $\mathbf V$ satisfies the identity $\mathbf c_{0,0,p}[\rho] \approx \mathbf c_{0,0,p}^\prime[\rho]$ for any $p\in\mathbb N$ and $\pi\in S_p$.

If $M_\lambda(xzyx^+ty^+)\notin\mathbf V$, then $\mathbf V$ satisfies the identity
\begin{equation}
\label{xzyxty=xzxyxty}
xzyxty\approx xzxyxty
\end{equation} 
by Lemma~4.2 in~\cite{Gusev-Sapir-22}.
In this case, the variety $\mathbf V$ satisfies the identities
\[
\begin{aligned}
\mathbf c_{0,0,p}[\pi]&{}\stackrel{\{\eqref{xxyy=yyxx},\,\eqref{xzyxty=xzxyxty}\}}\approx xytxy\biggl(\prod_{i=1}^p z_{i\pi}y_i^2\biggr)z_{n\pi}y\biggl(\prod_{i=1}^n t_iz_i\biggr)\\
&{}\stackrel{\{\eqref{xxyy=yyxx},\,\eqref{xzyxty=xzxyxty}\}}\approx yxtxy\biggl(\prod_{i=1}^p z_{i\pi}y_i^2\biggr)z_{n\pi}y\biggl(\prod_{i=1}^n t_iz_i\biggr)\stackrel{\eqref{xzyxty=xzxyxty}}\approx \mathbf c_{0,0,p}^\prime[\pi]
\end{aligned}
\]
for all $p\in\mathbb N$ and $\pi\in S_p$ as required.
Using Lemma~\ref{L: nsub M(xx^+yty^+)} instead of Lemma~4.2 in~\cite{Gusev-Sapir-22}, by a similar argument we can show that the identity $\mathbf c_{0,0,p}[\pi] \approx \mathbf c_{0,0,p}^\prime[\pi]$ holds in $\mathbf V$ whenever $M_\lambda(xx^+yty^+)\notin\mathbf V$.
Thus, we may further assume that $\mathbf M_\lambda(xzyx^+ty^+)\vee\mathbf M_\lambda(xx^+yty^+) \subseteq\mathbf V$ and, therefore, the $\lambda$-classes $xx^+yty^+$, $xyzx^+ty^+$ and $yxx^+ty^+$ are stable with respect to $\mathbf V$ by Lemmas~\ref{L: M_alpha(W) in V} and~\ref{L: L(M(xzyx^+ty^+))}.
Since $M_\lambda([\mathbf c_{0,0,n}[\pi]]^\lambda)\notin \mathbf V$, Corollary~\ref{C: M_alpha(W) in V}(ii) implies that $\mathbf V$ satisfies an identity $\mathbf u\approx\mathbf v$ such that $\mathbf u\in[\mathbf c_{0,0,n}[\pi]]^\lambda$ and $\mathbf v\notin[\mathbf c_{0,0,n}[\pi]]^\lambda$.
It follows from the fact that $xx^+yty^+$, $xyzx^+ty^+$ and $yxx^+ty^+$ are stable with respect to $\mathbf V$ that
\[
\mathbf v\in yxtx^+\biggl(\prod_{i=1}^n z_{i\pi}y_iy_i^+\biggr)z_{k\pi}y^+\biggl(\prod_{i=1}^n t_iz_i^+\biggr)=[\mathbf c_{0,0,n}^\prime[\pi]]^\lambda.
\]
Then $\mathbf u\approx \mathbf v$ is equivalent modulo~\eqref{xyx=xyxx} to $\mathbf c_{0,0,n}[\pi] \approx \mathbf c_{0,0,n}^\prime[\pi]$, and we are done.
\end{proof}

Let 
\[
\hat{\mathbb N}_0^2 :=\{(n,m)\in \mathbb N_0\times \mathbb N_0\mid |n-m|\le 1\}.
\]
For $n,m\in\mathbb N_0$, a permutation $\rho$ from $S_{n+m}$ is a $(n,m)$-\textit{permutation} if, for all $i=1,2,\dots,n+m-1$, one of the following holds:
\begin{itemize}
\item $1\le i\rho\le n$ and $n<(i+1)\rho\le n+m$;
\item $1\le (i+1)\rho\le n$ and $n<i\rho\le n+m$.
\end{itemize}
Evidently, if $\rho$ is a $(n,m)$-permutation, then $(n,m)\in \hat{\mathbb N}_0^2$.
The set of all $(n,m)$-permutations is denoted by $S_{n,m}$.

For any $n,m,k\in \mathbb N_0$ and $\rho\in S_{n+m+k}$, we denote by $\hat{\mathbf c}_{n,m,k}[\rho]$, $\hat{\mathbf c}_{n,m,k}^\prime[\rho]$, $\hat{\mathbf d}_{n,m,k}[\rho]$ and $\hat{\mathbf d}_{n,m,k}^\prime[\rho]$ the words obtained from $\mathbf c_{n,m,k}[\rho]$, $\mathbf c_{n,m,k}^\prime[\rho]$, $\mathbf d_{n,m,k}[\rho]$ and $\mathbf d_{n,m,k}^\prime[\rho]$, respectively, by deleting all the occurrences of letters $y_1,\dots,y_{n+m+k-1}$.
Let
\[
\mathbf N:=\var\{x^2\approx x^3,\,x^2y\approx yx^2,\,xyxzx\approx x^2yz,\,\sigma_2,\,\sigma_3\}.
\]

\begin{lemma}
\label{L: nsub M(c_{n,m,k}[rho]) 1}
Let $\mathbf V$ be a variety satisfying $\{\eqref{yxxty=xyxxty},\,\eqref{ytyxx=ytxyxx}\}$ such that $M(xyx)\in\mathbf V$.
\begin{itemize}
\item[\textup{(i)}] If $\mathbf N,\mathbf M(\hat{\mathbf c}_{p,q,0}[\pi_1]),\mathbf M(\hat{\mathbf c}_{p,q,p+q+1}[\pi_2])\nsubseteq\mathbf V$ for all $p,q\in\mathbb N_0$, $\pi_1\in S_{p+q}$ and $\pi_2\in S_{p+q,p+q+1}$, then $\mathbf V$ satisfies the identity $\mathbf c_{n,m,k}[\rho] \approx\mathbf c_{n,m,k}^\prime[\rho]$ for any $n,m,k\in\mathbb N_0$ and $\rho\in S_{n+m+k}$.
\item[\textup{(ii)}] If $\mathbf N^\delta,\mathbf M(\hat{\mathbf d}_{p,q,0}[\pi_1]),\mathbf M(\hat{\mathbf d}_{p,q,p+q+1}[\pi_2])\nsubseteq\mathbf V$ for all $p,q\in\mathbb N_0$, $\pi_1\in S_{p+q}$ and $\pi_2\in S_{p+q,p+q+1}$, then $\mathbf V$ satisfies the identity $\mathbf d_{n,m,k}[\rho] \approx\mathbf d_{n,m,k}^\prime[\rho]$ for any $n,m,k\in\mathbb N_0$ and $\rho\in S_{n+m+k}$.
\end{itemize}
\end{lemma}

\begin{proof}
(i) Take arbitrary $n,m,k\in\mathbb N_0$ and $\rho\in S_{n+m+k}$.
For brevity, put $r:=n+m+k$.
It is easy to see that there is $\tau\in S_{3r-2}$ such that $\hat{\mathbf c}_{n,2(r-1)+m,k}[\tau] \approx\hat{\mathbf c}_{n,2(r-1)+m,k}^\prime[\tau]$ implies $\mathbf h\,xy\,\mathbf t\approx\mathbf h\,yx\,\mathbf t$, where
\[
\begin{aligned}
\mathbf h:= \biggl(\prod_{i=1}^n z_it_i\biggr),\ 
\mathbf t:=t\biggl(\prod_{i=n+1}^{n+m} z_it_i\biggr)\biggl(\prod_{i=1}^{r-1} y_i^2\biggr)x\biggl(\prod_{i=1}^{r-1} z_{i\rho}y_i^2\biggr)z_{r\rho}y\biggl(\prod_{i=n+m+1}^r t_iz_i\biggr).
\end{aligned}
\]
In view of Lemma~4.9 in~\cite{Gusev-23}, the variety $\mathbf V$ satisfies $\hat{\mathbf c}_{n,2(r-1)+m,k}[\tau] \approx\hat{\mathbf c}_{n,2(r-1)+m,k}^\prime[\tau]$.
Since $\mathbf c_{n,m,k}[\rho] \stackrel{\{\eqref{yxxty=xyxxty},\,\eqref{ytyxx=ytxyxx}\}}\approx\mathbf h\,xy\,\mathbf t$ and $\mathbf c_{n,m,k}^\prime[\rho] \stackrel{\{\eqref{yxxty=xyxxty},\,\eqref{ytyxx=ytxyxx}\}}\approx\mathbf h\,yx\,\mathbf t$, it follows that $\mathbf V$ satisfies the identity $\mathbf c_{n,m,k}[\rho] \approx\mathbf c_{n,m,k}^\prime[\rho]$.

\smallskip

(ii) The proof is quite similar to the proof of Part~(i).
\end{proof}

If $\mathbf w$ is a word and $X\subseteq\con(\mathbf w)$, then we denote by $\mathbf w_X$ the word obtained from $\mathbf w$ by deleting all letters from $X$. 
If $X=\{x\}$, then we write $\mathbf w_x$ rather than $\mathbf w_{\{x\}}$.

\begin{lemma}
\label{L: nsub M(c_{n,m,k}[rho]) 1b}
Let $\mathbf V$ be a variety satisfying $\{x^2\approx x^3,\,\sigma_3\}$ such that $M(xyx)\in\mathbf V$. 
If $\mathbf N,\mathbf M(\hat{\mathbf c}_{p,q,0}[\pi_1]),\mathbf M(\hat{\mathbf c}_{p,q,p+q+1}[\pi_2])\nsubseteq\mathbf V$ for all $p,q\in\mathbb N_0$, $\pi_1\in S_{p+q}$ and $\pi_2\in S_{p+q,p+q+1}$, then $\mathbf V$ satisfies the identity $\mathbf c_{n,m,k}[\rho] \approx\mathbf c_{n,m,k}^\prime[\rho]$ for any $n,m,k\in\mathbb N_0$ and $\rho\in S_{n+m+k}$.
\end{lemma}

\begin{proof}
First, notice that $\mathbf V$ satisfies 
\[
\Sigma:=\{\hat{\mathbf c}_{p,q,r}[\pi] \approx\hat{\mathbf c}_{p,q,r}^\prime[\pi]\mid p,q,r\in\mathbb N_0,\ \pi\in S_{p+q+r}\}
\] 
by Lemma~4.9 in~\cite{Gusev-23}.
Further, for any $n,m,k\in\mathbb N_0$ and $\rho\in S_{n+m+k}$, put
\[
\mathbf h_{n,m,k}[\rho]:=(\mathbf c_{n,m,k}[\rho])_{\{y_i\mid i\rho>n+m\}}\ \text{ and }\ \mathbf h_{n,m,k}^\prime[\rho]:=(\mathbf c_{n,m,k}^\prime[\rho])_{\{y_i\mid i\rho>n+m\}}.
\]
It is easy to see that if $\mathbf V$ satisfies $\{\mathbf h_{n,m,k}[\rho]\approx \mathbf h_{n,m,k}^\prime[\rho]\mid n,m,k\in\mathbb N_0,\ \rho\in S_{n+m+k}\}$, then $\mathbf V$ also satisfies $\{\mathbf c_{n,m,k}[\rho]\approx \mathbf c_{n,m,k}^\prime[\rho]\mid n,m,k\in\mathbb N_0,\ \rho\in S_{n+m+k}\}$.
Therefore, it suffices to show that $\mathbf h_{n,m,k}[\rho]\approx \mathbf h_{n,m,k}^\prime[\rho]$ holds in $\mathbf V$ for any $n,m,k\in\mathbb N_0$ and $\rho\in S_{n+m+k}$.
We will use induction on $k$.
For brevity, we put
\[
\mathbf p:=\biggl(\prod_{i=1}^n z_it_i\biggr),\,
\mathbf q:=t\biggl(\prod_{i=n+1}^{n+m} z_it_i\biggr),\,
\mathbf r:=x\biggl(\prod_{i=1}^{n+m+k-1} z_{i\rho}y_i^2\biggr)z_{(n+m+k)\rho}y,\,
\mathbf s:=\biggl(\prod_{i=n+m+1}^{n+m+k} t_iz_i\biggr).
\]

\smallskip

\textbf{Induction base}: $k=0$.
In this case, $\mathbf V$ satisfies
\[
\mathbf h_{n,m,k}[\rho]\stackrel{\{x^2\approx x^3,\,\sigma_3\}}\approx\mathbf pxy\mathbf qy_1^2\cdots y_{n+m-1}^2\mathbf r\mathbf s\stackrel{\Sigma}\approx\mathbf pyx\mathbf qy_1^2\cdots y_{n+m-1}^2\mathbf r\mathbf s\stackrel{\{x^2\approx x^3,\,\sigma_3\}}\approx \mathbf h_{n,m,k}^\prime[\rho]
\]
as required.

\smallskip

\textbf{Induction step}: $k>0$.
In this case, there exists the least $j$ such that $j\rho>n+m$.
Then $\mathbf V$ satisfies the identities
\[
\mathbf h_{n,m,k}[\rho]\stackrel{\{x^2\approx x^3,\,\sigma_3\}}\approx \mathbf pxy\mathbf qy_1\cdots y_{j-1}z_{j\rho}\mathbf r^\prime\mathbf s, \ 
\mathbf h_{n,m,k}^\prime[\rho]\stackrel{\{x^2\approx x^3,\,\sigma_3\}}\approx\mathbf pyx\mathbf qy_1\cdots y_{j-1}z_{j\rho}\mathbf r^\prime\mathbf s,
\]
where $\mathbf r^\prime:=\mathbf r_{\{z_{j\rho},y_i\mid i\rho>n+m\}}$.
By the induction assumption, the identity
\[
\mathbf pxy\mathbf qy_1\cdots y_{j-1}z_{j\rho}\mathbf r^\prime\mathbf s\approx\mathbf pyx\mathbf qy_1\cdots y_{j-1}z_{j\rho}\mathbf r^\prime\mathbf s
\] 
holds in $\mathbf V$.
Hence $\mathbf h_{n,m,k}[\rho]\approx\mathbf h_{n,m,k}^\prime[\rho]$ is satisfied by $\mathbf V$ as required.
\end{proof}

For $n\in\mathbb N_0$, $\pi_1\in S_{4n+1}$, $\pi_2\in S_{n+1}$, $\pi_3,\pi_4\in S_{2n+1}$ and $\tau\in S_{2n}$, let
\[
\begin{aligned}
&\mathbf c_n^{(1)}[\pi_1,\tau]:=\biggl(\prod_{i=1}^n z_i^\prime t_i^\prime \biggr)xyt\biggl(\prod_{i=n+1}^{2n} z_i^\prime t_i^\prime \biggr)x\biggl(\prod_{i=1}^{2n} z_{(2i-1)\pi_1}y_i^2z_{(2i)\pi_1}z_{i\tau}^\prime\biggr)z_{(4n+1)\pi_1}y\biggl(\prod_{i=1}^{4n+1} t_iz_i\biggr),\\
&\mathbf c_n^{(2)}[\pi_2,\tau]:=\biggl(\prod_{i=1}^n z_i^\prime t_i^\prime \biggr)xyt\biggl(\prod_{i=n+1}^{2n} z_i^\prime t_i^\prime \biggr)x\biggl(\prod_{i=1}^n z_{i\pi_2}z_{(2i-1)\tau}^\prime y_i^2z_{(2i)\tau}^\prime\biggr)z_{(n+1)\pi_2}y\biggl(\prod_{i=1}^{n+1} t_iz_i\biggr),\\
&\mathbf c_n^{(3)}[\pi_3,\tau]:=\biggl(\prod_{i=1}^n z_i^\prime t_i^\prime \biggr)xyt\biggl(\prod_{i=n+1}^{2n} z_i^\prime t_i^\prime \biggr)x\biggl(\prod_{i=1}^{2n} z_{i\pi_3}z_{i\tau}^\prime y_i^2\biggr)z_{(2n+1)\pi_3}y\biggl(\prod_{i=1}^{2n+1} t_iz_i\biggr),\\
&\mathbf c_n^{(4)}[\pi_4,\tau]:=\biggl(\prod_{i=1}^n z_i^\prime t_i^\prime \biggr)xyt\biggl(\prod_{i=n+1}^{2n} z_i^\prime t_i^\prime \biggr)x\biggl(\prod_{i=1}^{2n} z_{i\pi_4}y_i^2z_{i\tau}^\prime \biggr)z_{(2n+1)\pi_4}y\biggl(\prod_{i=1}^{2n+1} t_iz_i\biggr).
\end{aligned}
\]
Denote by $\overline{\mathbf c}_n^{(i)}[\pi,\tau]$ the word obtained from $\mathbf c_n^{(i)}[\pi,\tau]$ by interchanging the first occurrences of $x$ and $y$.

\begin{lemma}
\label{L: nsub M(c_{n,m,k}[rho]) 2}
Let $\mathbf V$ be a variety satisfying $\Phi$ such that $M(xyx)\in\mathbf V$ and $\mathbf N\nsubseteq\mathbf V$.
Suppose that $\mathbf V$ does not contain the monoids 
\[
\begin{aligned}
&M(\hat{\mathbf c}_{p,q,0}[\rho_1]),\ 
M(\hat{\mathbf c}_{p,q,p+q+1}[\rho_2]),\\ 
&M_{\gamma^\prime}([\mathbf c_r^{(1)}[\pi_1,\tau]]^{\gamma^\prime}),\ M_{\gamma^\prime}([\mathbf c_r^{(2)}[\pi_2,\tau]]^{\gamma^\prime}),\ 
M_{\gamma^\prime}([\mathbf c_r^{(3)}[\pi_3,\tau]]^{\gamma^\prime}),\ M_{\gamma^\prime}([\mathbf c_r^{(4)}[\pi_4,\tau]]^{\gamma^\prime})
\end{aligned}
\]
for all $p,q\in\mathbb N_0$, $r\in\mathbb N$, $\rho_1\in S_{p+q}$, $\rho_2\in S_{p+q,p+q+1}$, $\pi_1\in S_{4r+1}$, $\pi_2\in S_{r+1}$, $\pi_3,\pi_4\in S_{2r+1}$ and $\tau\in S_{2r}$. 
Then $\mathbf V$ satisfies the identity $\mathbf c_{n,m,k}[\rho] \approx\mathbf c_{n,m,k}^\prime[\rho]$ for any $n,m,k\in\mathbb N_0$ and $\rho\in S_{n+m+k}$.
\end{lemma}

\begin{proof}
Take arbitrary $n,m,k\in\mathbb N_0$ and $\rho\in S_{n+m+k}$.
If $M_\gamma(yxx^+)\notin \mathbf V$, then $\mathbf V$ satisfies~\eqref{yxx=xyxx} by Lemma~\ref{L: nsub M(yxx^+)}.
In this case, Lemma~\ref{L: nsub M(c_{n,m,k}[rho]) 1}(i) implies that the identity $\mathbf c_{n,m,k}[\rho] \approx\mathbf c_{n,m,k}^\prime[\rho]$ holds in $\mathbf V$ because $yx^2ty\stackrel{\eqref{yxx=xyxx}}\approx xyx^2ty$ and $ytyx^2\stackrel{\eqref{yxx=xyxx}}\approx ytxyx^2$.
By the dual arguments we can show that $\mathbf V$ satisfies $\mathbf c_{n,m,k}[\rho] \approx\mathbf c_{n,m,k}^\prime[\rho]$ whenever $M_\gamma(xx^+y)\notin \mathbf V$.
If $M(xzxyty)\notin\mathbf V$, then the required claim follows from Lemmas~\ref{L: swapping in linear-balanced} and~\ref{L: nsub M(c_{n,m,k}[rho]) 1b}.
Thus, we may further assume that $\mathbf M(xzxyty)\vee\mathbf M_\gamma(xx^+y)\vee \mathbf M_\gamma(yxx^+)\subseteq \mathbf V$.

If $M_{\gamma^\prime}(yxx^+ty),M_{\gamma^\prime}(ytyxx^+)\notin \mathbf V$, then $\mathbf V$ satisfies the identities~\eqref{yxxty=xyxxty} and~\eqref{ytyxx=ytxyxx} by Lemma~\ref{L: nsub M(yxx^+ty)} and the dual to Lemma~\ref{L: nsub M([yx^2zy]),M(xx^+yty)}(ii).
Now Lemma~\ref{L: nsub M(c_{n,m,k}[rho]) 1}(i) applies, yielding that the identity $\mathbf c_{n,m,k}[\rho] \approx\mathbf c_{n,m,k}^\prime[\rho]$ holds in $\mathbf V$.
By the similar arguments we can show that $\mathbf V$ satisfies $\mathbf c_{n,m,k}[\rho] \approx\mathbf c_{n,m,k}^\prime[\rho]$ whenever $M_{\gamma^\prime}(xx^+yty),M_{\gamma^\prime}(ytxx^+y)\notin \mathbf V$.
Thus, we may further assume that the following two claims hold:
\begin{itemize}
\item either $M_{\gamma^\prime}(yxx^+ty)\in\mathbf V$ or $M_{\gamma^\prime}(ytyxx^+)\in \mathbf V$;
\item either $M_{\gamma^\prime}(xx^+yty)\in\mathbf V$ or $M_{\gamma^\prime}(ytxx^+y)\in \mathbf V$.
\end{itemize}
In other words, four cases are possible:
\begin{itemize}
\item $M_{\gamma^\prime}(yxx^+ty),M_{\gamma^\prime}(xx^+yty)\in\mathbf V$;
\item $M_{\gamma^\prime}(ytyxx^+),M_{\gamma^\prime}(ytxx^+y)\in\mathbf V$;
\item $M_{\gamma^\prime}(ytyxx^+),M_{\gamma^\prime}(xx^+ yty)\in\mathbf V$;
\item $M_{\gamma^\prime}(ytxx^+y),M_{\gamma^\prime}(yxx^+ty)\in\mathbf V$.
\end{itemize}
We consider only the first case, assuming that $M_{\gamma^\prime}(yxx^+ty),M_{\gamma^\prime}(xx^+yty)\in\mathbf V$, because the other three cases are quite similar.

It is easy to see that one can find sufficiently large $r\in\mathbb N$, $\pi\in S_{4r+1}$ and $\tau\in S_{2r}$ such that the identity $\mathbf c_{n,m,k}[\rho] \approx\mathbf c_{n,m,k}^\prime[\rho]$ follows from the identity $\mathbf c^{(1)}_r[\pi,\tau]\approx \overline{\mathbf c}^{(1)}_r[\pi,\tau]$.
By the condition of the lemma, $M_{\gamma^\prime}([\mathbf c^{(1)}_r[\pi,\tau]]^{\gamma^\prime})\notin\mathbf V$.
According to Corollary~\ref{C: M_alpha(W) in V}(iv), the $\gamma^\prime$-class $[\mathbf c^{(1)}_r[\pi,\tau]]^{\gamma^\prime}$ is not stable with respect to $\mathbf V$.
This implies that $\mathbf V$ satisfies an identity $\mathbf u\approx \mathbf v$ such that $\mathbf u\in [\mathbf c^{(1)}_r[\pi,\tau]]^{\gamma^\prime}$ and $\mathbf v\notin [\mathbf c^{(1)}_r[\pi,\tau]]^{\gamma^\prime}$.
In view of Lemma~\ref{L: identities of M(xt_1x...t_kx)},
\[
\mathbf v_{\{y_1,\dots,y_{2r}\}}=\biggl(\prod_{i=1}^r z_i^\prime t_i^\prime \biggr)\mathbf at\biggl(\prod_{i=r+1}^{2r} z_i^\prime t_i^\prime \biggr)\mathbf b\biggl(\prod_{i=1}^{4r+1} t_iz_i\biggr),
\]
where $\mathbf a\in\{xy,yx\}$ and $\mathbf b$ is a linear word depending on the letters $x$, $y$, $z_1$, $\dots$, $z_{4r+1}$, $z_1^\prime$, $\dots$, $z_{2r}^\prime$.
Then, since $\mathbf M(xzxyty)\vee\mathbf M_{\gamma^\prime}(yxx^+ty)\vee \mathbf M_{\gamma^\prime}(xx^+yty)\subseteq \mathbf V$, Lemmas~\ref{L: M(W) in V} and~\ref{L: M_alpha(W) in V} imply that 
\[
\begin{aligned}
&\mathbf v(x,t,z_{1\pi},t_{1\pi})=xtxz_{1\pi}t_{1\pi}z_{1\pi},\\
&\mathbf v(z_{(2i)\pi},t_{(2i)\pi},y_i)\in y_iy_i^+z_{(2i)\pi}t_{(2i)\pi}z_{(2i)\pi},\\
&\mathbf v(z_{(2i-1)\pi},t_{(2i-1)\pi},y_i)\in z_{(2i-1)\pi}y_iy_i^+t_{(2i-1)\pi}z_{(2i-1)\pi},\\
&\mathbf v(z_{(2i)\pi},t_{(2i)\pi},z_{i\tau}^\prime,t_{i\tau}^\prime)=z_{i\tau}^\prime t_{i\tau}^\prime z_{(2i)\pi}z_{i\tau}^\prime t_{(2i)\pi} z_{(2i)\pi},\\
&\mathbf v(z_{(2i+1)\pi},t_{(2i+1)\pi},z_{i\tau}^\prime,t_{i\tau}^\prime)=z_{i\tau}^\prime t_{i\tau}^\prime z_{i\tau}^\prime z_{(2i+1)\pi} t_{(2i+1)\pi} z_{(2i+1)\pi},\\
&\mathbf v(z_{(4r+1)\pi},t_{(4r+1)\pi},y,t)=yt z_{(4r+1)\pi}y t_{(4r+1)\pi} z_{(4r+1)\pi},
\end{aligned}
\]
$i=1,\dots,2r$.
Hence 
\[
\mathbf b\in x\biggl(\prod_{i=1}^{2r} z_{(2i-1)\pi}y_iy_i^+z_{(2i)\pi}z_{i\tau}^\prime\biggr)z_{(4r+1)\pi}y.
\]
It follows that $\mathbf v\in[\overline{\mathbf c}^{(1)}_r[\pi,\tau]]^{\gamma^\prime}$.
Then the identity $\mathbf u\approx \mathbf v$ is equivalent modulo $x^2\approx x^3$ to $\mathbf c^{(1)}_r[\pi,\tau]\approx \overline{\mathbf c}^{(1)}_r[\pi,\tau]$, whence  $\mathbf c_{n,m,k}[\rho] \approx\mathbf c_{n,m,k}^\prime[\rho]$ holds in $\mathbf V$ as required.
\end{proof}

\subsection{Identities in $\Phi_2$}

For $n,m\in\mathbb N_0$ and $\rho\in S_{n+m}$, set 
\[
\hat{\mathbf a}_{n,m}[\rho]:=(\mathbf a_{n,m}[\rho])_{\{y_1,\dots,y_{n+m-1}\}}\ \text{ and }\ \hat{\overline{\mathbf a}}_{n,m}[\rho]:=(\overline{\mathbf a}_{n,m}[\rho])_{\{y_1,\dots,y_{n+m-1}\}}.
\]

\begin{lemma}
\label{L: nsub M([hat{a}_{0,n}[rho]]^lambda),M([a_{0,n}[rho]]^lambda),M([a_{0,n}[rho]]^beta)}
Let $\mathbf V$ be a monoid variety satisfying the identities~\eqref{xyx=xyxx} and~\eqref{xxyy=yyxx} such that $\mathbf M_\lambda(xyx^+)\vee \mathbf M_\gamma(xx^+y)\subseteq\mathbf V$.
Suppose that $\mathbf V$ does not contain the monoids
\[
M_\lambda([\hat{\mathbf a}_{0,k}[\pi]]^\lambda),\ M_\lambda([\mathbf a_{0,k}[\pi]]^\lambda) \ \text{ and }\ M_\beta([\mathbf a_{0,k}[\pi]]^\beta)
\] 
for all $k\ge2$ and $\pi\in S_k$. 
Then $\mathbf V$ satisfies $\Phi_2$.
\end{lemma}

\begin{proof}
The same arguments as in the first paragraph of the proof of Lemma~\ref{L: nsub M([c_{0,0,n}[rho]]^lambda)} imply that if $n,m\in\mathbb N$ and $\rho\in S_{n+m}$, then 
one can find sufficiently large $r$, say $r:=2(n+m)$, and $\rho^\prime\in S_r$ such that $\mathbf a_{n,m}[\rho] \approx \overline{\mathbf a}_{n,m}[\rho]$ follows from $\{\eqref{xyx=xyxx},\,\eqref{xxyy=yyxx},\,\mathbf a_{0,r}[\rho^\prime] \approx \overline{\mathbf a}_{0,r}[\rho^\prime]\}$.
Thus, it suffices to show that $\mathbf V$ satisfies the identity $\mathbf a_{0,n}[\rho] \approx \overline{\mathbf a}_{0,n}[\rho]$ for any $n\in\mathbb N$ and $\rho\in S_n$.

Assume that $M_\lambda(xzyx^+ty^+)\notin\mathbf V$. 
Then $\mathbf V$ satisfies the identity~\eqref{xzyxty=xzxyxty} by~\cite[Lemma~4.2]{Gusev-Sapir-22}.
Clearly, this identity together with~\eqref{xxyy=yyxx} imply $\mathbf a_{0,n}[\rho] \approx \overline{\mathbf a}_{0,n}[\rho]$ for any $n\in\mathbb N$ and $\rho\in S_n$.
Thus, we may further assume that $M_\lambda(xzyx^+ty^+)\in\mathbf V$ and, in particular, the $\lambda$-classes $yxx^+ty^+$ and $xyzx^+ty^+$ are stable with respect to $\mathbf V$ by Lemmas~\ref{L: M_alpha(W) in V} and~\ref{L: L(M(xzyx^+ty^+))}.

There are two cases.

\smallskip

\textbf{Case 1:} $M_\lambda(xx^+yty^+)\in\mathbf V$.
Lemma~\ref{L: M_alpha(W) in V} implies that the $\lambda$-class $xx^+yty^+$ is stable with respect to $\mathbf V$.
We will use induction on $n$.
Let $\varepsilon$ denote the trivial permutation from $S_1$.
If $n=1$, then $\mathbf a_{0,1}[\varepsilon]=\overline{\mathbf a}_{0,1}[\varepsilon]$ and there is nothing to prove.
Let now $n>1$ and take an arbitrary $\rho\in S_n$.
Since $M_\lambda([\mathbf a_{0,n}[\rho]]^\lambda)\notin\mathbf V$, Corollary~\ref{C: M_alpha(W) in V}(ii) implies that the $\lambda$-class $[\mathbf a_{0,n}[\rho]]^\lambda$ is not stable with respect to $\mathbf V$.
This means that $\mathbf V$ satisfies an identity $\mathbf u\approx \mathbf v$ such that $\mathbf u\in [\mathbf a_{0,n}[\rho]]^\lambda$ and $\mathbf v\notin [\mathbf a_{0,n}[\rho]]^\lambda$.
Since the sets $xx^+yty^+$ and $yxx^+ty^+$ are stable with respect to $\mathbf V$, we have 
\[
\begin{aligned}
&\mathbf v(x,z_{1\rho},t_{1\rho})\in x^+z_{1\rho}x^+t_{1\rho}z_{1\rho}^+,\\
&\mathbf v(x,z_{n\rho},t_{n\rho})\in x^+z_{n\rho}x^+t_{n\rho}z_{n\rho}^+,\\
&\mathbf v_x\in z_{1\rho}y_1y_1^+z_{2\rho}y_2y_2^+\cdots z_{n\rho}t_1z_1^+\cdots t_nz_n^+.
\end{aligned}
\]
This is only possible when $(_{2\mathbf v}x)<(_{1\mathbf v}z_{n\rho})$ because $\mathbf v\notin [\mathbf a_{0,n}[\rho]]^\lambda$.
By the induction assumption, $\mathbf V$ satisfies all the identities in the set $\Sigma_n:=\{\mathbf a_{0,i}[\tau]\approx \overline{\mathbf a}_{0,i}[\tau]\mid 1\le i<n,\ \tau\in S_i\}$.
Clearly, if there is an occurrence of $x$ between ${_{1\mathbf v}}z_{1\rho}$ and ${_{1\mathbf v}}z_{n\rho}$ in $\mathbf v$, then one can choose some identities in $\Sigma_n$ which together with $\{\eqref{xyx=xyxx},\eqref{xxyy=yyxx}\}$ imply the identity $\mathbf v\approx \overline{\mathbf a}_{0,n}[\rho]$.
Since $\mathbf a_{0,n}[\rho]\stackrel{\eqref{xyx=xyxx}}\approx\mathbf u$, this implies that $\mathbf V$ satisfies $\mathbf a_{0,n}[\rho]\approx\overline{\mathbf a}_{0,n}[\rho]$.
So, it remains to consider the case when there are no $x$ between ${_{1\mathbf v}}z_{1\rho}$ and ${_{1\mathbf v}}z_{n\rho}$ in $\mathbf v$.
Taking into account that $(_{2\mathbf v}x)<(_{1\mathbf v}z_{n\rho})$, we have $\mathbf v\in xx^+z_{1\rho}\cdots z_{n\rho}x^+ t_1z_1^+\cdots t_nz_n^+$.
In particular, $\mathbf u\approx \mathbf v$ together with~\eqref{xyx=xyxx} imply
\begin{equation}
\label{xxyxty=xyxty}
x^2yxty\approx xyxty.
\end{equation}
Further, since $M_\beta([\mathbf a_{0,n}[\rho]]^\beta)\notin\mathbf V$, Lemma~\ref{L: M_beta(W) in V} implies that  $\mathbf V$ satisfies an identity $\mathbf u^\prime\approx \mathbf v^\prime$ such that $\mathbf u^\prime\in [\mathbf a_{0,n}[\rho]]^\beta$ and $\mathbf v^\prime\notin [\mathbf a_{0,n}[\rho]]^\beta$.
Since the sets $xx^+yty^+$ and $yxx^+ty^+$ are stable with respect to $\mathbf V$, we have 
\[
\begin{aligned}
&\mathbf v^\prime(x,z_{1\rho},t_{1\rho})\in x^+z_{1\rho}x^+t_{1\rho}z_{1\rho}^+,\\
&\mathbf v^\prime(x,z_{n\rho},t_{n\rho})\in x^+z_{n\rho}x^+t_{n\rho}z_{n\rho}^+,\\
&\mathbf v^\prime_x\in z_{1\rho}y_1y_1^+z_{2\rho}y_2y_2^+\cdots z_{n\rho}t_1z_1^+\cdots t_nz_n^+.
\end{aligned}
\]
This is only possible when there is an occurrence of $x$ between ${_{1\mathbf v^\prime}}z_{1\rho}$ and ${_{1\mathbf v^\prime}}z_{n\rho}$ in $\mathbf v^\prime$ because $\mathbf v^\prime\notin [\mathbf a_{0,n}[\rho]]^\beta$.
Hence $\mathbf V$ satisfies 
\[
\mathbf a_{0,n}[\rho]\stackrel{\eqref{xyx=xyxx}}\approx\mathbf u\approx \mathbf v\stackrel{\eqref{xyx=xyxx}}\approx x\mathbf u^\prime\approx x\mathbf v^\prime\stackrel{\{\Sigma_n,\,\eqref{xyx=xyxx},\,\eqref{xxyy=yyxx},\,\eqref{xxyxty=xyxty}\}}\approx \overline{\mathbf a}_{0,n}[\rho]
\]
as required.

\smallskip

\textbf{Case 2:} $M_\lambda(xx^+yty^+)\notin\mathbf V$.
Corollary~\ref{C: M_alpha(W) in V}(ii) implies that $\mathbf V$ satisfies an identity $\mathbf u\approx \mathbf v$ such that $\mathbf u\in xx^+yty^+$ and $\mathbf v\notin xx^+yty^+$.
Since the $\lambda$-class $yxx^+ty^+$ is stable with respect to $\mathbf V$, we have $\mathbf v\in x^+yx^+ty^+$.
Then $\mathbf u\approx \mathbf v$ together with~\eqref{xyx=xyxx} imply~\eqref{xxyty=xxyxty}.
Further, since
\[
\begin{aligned}
&\mathbf a_{0,n}[\rho]\stackrel{\{\eqref{xxyty=xxyxty},\,\eqref{xyx=xyxx},\,\eqref{xxyy=yyxx}\}}\approx x\biggl(\prod_{i=1}^{n-1} z_{i\rho}y_i^2\biggr)z_{n\rho}x\biggl(\prod_{i=1}^{n-1} y_i^2\biggr)\biggl(\prod_{i=1}^n t_iz_i\biggr),\\
&\overline{\mathbf a}_{0,n}[\rho]\stackrel{\{\eqref{xxyty=xxyxty},\,\eqref{xyx=xyxx},\,\eqref{xxyy=yyxx}\}}\approx x\biggl(\prod_{i=1}^{n-1} z_{i\rho}x(y_ix)^2\biggr)z_{n\rho}x\biggl(\prod_{i=1}^{n-1} y_i^2\biggr)\biggl(\prod_{i=1}^n t_iz_i\biggr),
\end{aligned}
\]
there is $\rho^\prime\in S_{3n-2}$ such that $\mathbf a_{0,n}[\rho] \approx \overline{\mathbf a}_{0,n}[\rho]$ follows from  
\[
\{\eqref{xxyty=xxyxty},\,\eqref{xyx=xyxx},\,\eqref{xxyy=yyxx},\,\hat{\mathbf a}_{0,3n-2}[\rho^\prime] \approx \hat{\overline{\mathbf a}}_{0,3n-2}[\rho^\prime]\}.
\]
Thus, it suffices to show that $\mathbf V$ satisfies the identity $\hat{\mathbf a}_{0,n}[\rho] \approx \hat{\overline{\mathbf a}}_{0,n}[\rho]$ for any $n\in\mathbb N$ and $\rho\in S_n$.

We will use induction on $n$.
If $n=1$, then $\hat{\mathbf a}_{0,1}[\varepsilon]=\hat{\overline{\mathbf a}}_{0,1}[\varepsilon]$ and there is nothing to prove.
Let now $n>1$ and take an arbitrary $\rho\in S_n$.
Since $M_\lambda([\hat{\mathbf a}_{0,n}[\rho]]^\lambda)\notin\mathbf V$, Corollary~\ref{C: M_alpha(W) in V}(ii) implies that $\mathbf V$ satisfies an identity $\mathbf u\approx \mathbf v$ such that $\mathbf u\in [\hat{\mathbf a}_{0,n}[\rho]]^\lambda$ and $\mathbf v\notin [\hat{\mathbf a}_{0,n}[\rho]]^\lambda$.
Since the set $xyzx^+ty^+$ is stable with respect to $\mathbf V$, we have $\mathbf v(x,t_1)\in xx^+t_1$ and $\mathbf v_x\in z_{1\rho}\cdots z_{n\rho}t_1z_1^+\cdots t_nz_n^+$.
Further, $\mathbf v(x,z_{1\rho},t_{1\rho})\notin z_{1\rho}xx^+tz_{1\rho}^+$ because the set $yxx^+ty^+$  is stable with respect to $\mathbf V$, whence $(_{1\mathbf v}x)<(_{1\mathbf v}z_{1\rho})$.
Since $\mathbf v\notin [\hat{\mathbf a}_{0,n}[\rho]]^\lambda$, this is only possible when $(_{2\mathbf v}x)<(_{1\mathbf v}z_{n\rho})$.
Then the identity $\mathbf u(x,z_{n\rho},t_{n\rho})\approx \mathbf v(x,z_{n\rho},t_{n\rho})$ is equivalent modulo $\{\eqref{xxyty=xxyxty},\,\eqref{xyx=xyxx}\}$ to 
\begin{equation}
\label{xxyty=xyxty}
x^2yty\approx xyxty.
\end{equation}
By the induction assumption, $\mathbf V$ satisfies all the identities in $\hat{\Sigma}_n:=\{\hat{\mathbf a}_{0,i}[\tau]\approx \hat{\overline{\mathbf a}}_{0,i}[\tau]\mid 1\le i<n,\ \tau\in S_i\}$.
Clearly, if there is an occurrence of $x$ between ${_{1\mathbf v}}z_{1\rho}$ and ${_{1\mathbf v}}z_{n\rho}$ in $\mathbf v$, then one can choose some identities in $\hat{\Sigma}_n$ which together with $\{\eqref{xyx=xyxx},\,\eqref{xxyty=xyxty}\}$ imply the identity $\mathbf v\approx \hat{\overline{\mathbf a}}_{0,n}[\rho]$.
Since $\hat{\mathbf a}_{0,n}[\rho]\stackrel{\eqref{xyx=xyxx}}\approx\mathbf u$, this implies that $\mathbf V$ satisfies $\hat{\mathbf a}_{0,n}[\rho]\approx\hat{\overline{\mathbf a}}_{0,n}[\rho]$.
Let now there are no $x$ between ${_{1\mathbf v}}z_{1\rho}$ and ${_{1\mathbf v}}z_{n\rho}$ in $\mathbf v$.
In this case, taking into account that $(_{2\mathbf v}x)<(_{1\mathbf v}z_{n\rho})$, we have $\mathbf v\in xx^+z_{1\rho}\cdots z_{n\rho}x^\ast t_1z_1^+\cdots t_nz_n^+$.
Then $\hat{\mathbf a}_{0,n}[\rho]\approx \hat{\overline{\mathbf a}}_{0,n}[\rho]$ holds in $\mathbf V$ because $\mathbf v\stackrel{\{\eqref{xyx=xyxx},\,\eqref{xxyty=xyxty}\}}\approx\hat{\overline{\mathbf a}}_{0,n}[\rho]$ and $\mathbf u\stackrel{\eqref{xyx=xyxx}}\approx\hat{\mathbf a}_{0,n}[\rho]$, and we are done.
\end{proof}

Let $\mathbf K$ denote the class of all monoids $M$ such that:
\begin{itemize}
\item $M$ satisfies the identities~\eqref{yxxty=xyxxty},~\eqref{xxyx=xxyxx},~\eqref{ytyxx=ytxyxx},~\eqref{xyzxy=yxzxy} and
\begin{equation}
\label{xyzxy=xyzyx}
xyzxy\approx xyzyx;
\end{equation}
\item $M$ violates the identity $\hat{\mathbf a}_{n,m}[\pi]\approx \hat{\mathbf a}_{n,m}^\prime[\pi]$ for some $(n,m)\in\hat{\mathbb N}_0^2$ and $\pi\in S_{n,m}$.
\end{itemize}
If $\mathbf w:=a_1\cdots a_k$, then put $\chi(\mathbf w):=a_1xa_2x\cdots xa_k$.

\begin{lemma}
\label{L: nsub M(a_{n,m}[rho]) 1}
Let $\mathbf V$ be a monoid variety satisfying the identities~\eqref{yxxty=xyxxty},~\eqref{xxyx=xxyxx},~\eqref{ytyxx=ytxyxx},~\eqref{xyzxy=yxzxy} and~\eqref{xyzxy=xyzyx}. 
If $\mathbf V$ does not contain all the monoids in $\mathbf K$, then $\mathbf V$ satisfies $\Phi_2$.
\end{lemma}

\begin{proof}
Take arbitrary $p,q\in \mathbb N$ and $\pi\in S_{p+q}$.
Then one can find $\tau\in S_{3(p+q)-2}$ such that  
\[ 
\mathbf a_{p,q}[\pi]\stackrel{\{\eqref{yxxty=xyxxty},\,\eqref{ytyxx=ytxyxx}\}}\approx\mathbf h\,x\mathbf px\,\mathbf t\stackrel{\hat{\mathbf a}_{3p+2q-2,q}[\tau] \approx\hat{\mathbf a}_{3p+2q-2,q}^\prime[\tau]}\approx\mathbf h\,x\chi(\mathbf p)x\,\mathbf t\stackrel{\{\eqref{yxxty=xyxxty},\,\eqref{ytyxx=ytxyxx}\}}\approx \overline{\mathbf a}_{p,q}[\pi],
\]
where
\[
\mathbf h:= \biggl(\prod_{i=1}^p z_it_i\biggr)\biggl(\prod_{i=1}^{p+q-1} y_i^2\biggr),\ 
\mathbf p:=\biggl(\prod_{i=1}^{p+q-1} z_{i\pi}y_i^2\biggr)z_{(p+q)\pi},\ 
\mathbf t:=\biggl(\prod_{i=p+1}^{p+q} z_it_i\biggr).
\]
Further, it is easy to see that there exist $(n,m)\in\hat{\mathbb N}_0^2$ and $\rho\in S_{n,m}$ such that $\hat{\mathbf a}_{n,m}[\rho] \approx\hat{\mathbf a}_{n,m}^\prime[\rho]$ implies $\hat{\mathbf a}_{3p+2q-2,q}[\tau] \approx\hat{\mathbf a}_{3p+2q-2,q}^\prime[\tau]$.
Since $\mathbf V$ does not contain all the monoids in $\mathbf K$ and $\mathbf V$ satisfies \{\eqref{yxxty=xyxxty},\,\eqref{xxyx=xxyxx},\,\eqref{ytyxx=ytxyxx},\,\eqref{xyzxy=yxzxy},\,\eqref{xyzxy=xyzyx}\}, the identity $\hat{\mathbf a}_{n,m}[\rho] \approx\hat{\mathbf a}_{n,m}^\prime[\rho]$ holds in $\mathbf V$.
Hence $\mathbf a_{p,q}[\pi]\approx \overline{\mathbf a}_{p,q}[\pi]$ holds in $\mathbf V$ as well.
\end{proof}

For $(n,m)\in\hat{\mathbb N}_0^2$ and $\rho\in S_{n,m}$, let
\[
\begin{aligned}
&\mathbf a_{n,m}^{(1)}[\rho]:=\biggl(\prod_{i=1}^n z_i t_i \biggr)x\biggl(\prod_{i=1}^{n+m} \mathbf z_{i\rho}^{(1)}\biggr)x\biggl(\prod_{i=n+1}^{n+m} t_i z_i t_i^{\prime}z_i^{\prime}\biggr),\\
&\mathbf a_{n,m}^{(2)}[\rho]:=\biggl(\prod_{i=1}^n z_i t_i z_i^{\prime}t_i^{\prime} \biggr)x\biggl(\prod_{i=1}^{n+m} \mathbf z_{i\rho}^{(2)}\biggr)x\biggl(\prod_{i=n+1}^{n+m} t_i z_i \biggr),\\
&\mathbf a_{n,m}^{(3)}[\rho]:=\biggl(\prod_{i=1}^n z_i t_i \biggr)x\biggl(\prod_{i=1}^{n+m-1} \mathbf z_{i\rho}^{(3)}\biggr)(\mathbf z_{(n+m)\rho}^{(3)})_{y_{(n+m)\rho}}x\biggl(\prod_{i=n+1}^{n+m} t_i z_i \biggr),\\
&\mathbf a_{n,m}^{(4)}[\rho]:=\biggl(\prod_{i=1}^n z_i t_i \biggr)x\biggl(\prod_{i=1}^{n+m-1} \mathbf z_{i\rho}^{(4)}\biggr)(\mathbf z_{(n+m)\rho}^{(4)})_{y_{(n+m)\rho}}x\biggl(\prod_{i=n+1}^{n+m} t_i z_i \biggr),
\end{aligned}
\]
where
\[
\begin{aligned}
&\mathbf z_i^{(1)}:=
\begin{cases} 
z_i & \text{if $1\le i\le n$}, \\
z_iy_i^2z_i^{\prime} & \text{if $n+1\le i\le n+m$},
\end{cases} 
\ 
&&\mathbf z_i^{(2)}:=
\begin{cases} 
z_iy_i^2z_i^{\prime} & \text{if $1\le i\le n$}, \\
z_i & \text{if $n+1\le i\le n+m$},
\end{cases} 
\\
&\mathbf z_i^{(3)}:=
\begin{cases} 
z_iy_i^2 & \text{if $1\le i\le n$}, \\
z_i & \text{if $n+1\le i\le n+m$},
\end{cases} 
\ 
&&\mathbf z_i^{(4)}:=
\begin{cases} 
z_i & \text{if $1\le i\le n$}, \\
z_iy_i^2 & \text{if $n+1\le i\le n+m$}.
\end{cases} 
\end{aligned}
\]
For $i=1,2,3,4$, denote by $\overline{\mathbf a}_{n,m}^{(i)}[\rho]$ the word obtained $\mathbf a_{n,m}^{(i)}[\rho]$ by replacing the factor lying between the first and the second occurrences of $x$ to its image under $\chi$.
Denote by $\head(\mathbf w)$ [respectively, $\tail(\mathbf w)$] the first [last] letter of a word $\mathbf w$.
For any $k\in\mathbb N$, let 
\[
S_{2k}^\sharp:=\{\pi\in S_{2k}\mid k+1\le1\pi,\dots,k\pi\le2k,\ 1\le(k+1)\pi,\dots,(2k)\pi\le k\}.
\]

\begin{lemma}
\label{L: nsub M(a_{n,m}[rho]) 2}
Let $\mathbf V$ be a variety satisfying the identities $x^2\approx x^3$,~\eqref{xxyy=yyxx},~\eqref{xyzxy=yxzxy},~\eqref{xyzxy=xyzyx} and
\begin{equation}
\label{xyxx=xxyx}
xyx^2\approx x^2yx
\end{equation}
such that $M(xyx)\in\mathbf V$.
Suppose that $\mathbf V$ does not contain all the monoids in $\mathbf K$ and $\mathbf K^\delta$ together with the monoids
\[
\begin{aligned}
&M_{\alpha_1}([\mathbf a_{k,k}[\tau]]^{\alpha_1}),\\
&M_{\gamma^{\prime}}([\mathbf a_{n,m}^{(1)}[\pi]]^{\gamma^{\prime}}),\ 
M_{\gamma^{\prime}}([\mathbf a_{n,m}^{(2)}[\pi]]^{\gamma^{\prime}}),\
M_{\gamma^{\prime}}([\mathbf a_{n,m}^{(3)}[\pi]]^{\gamma^{\prime}}),\  
M_{\gamma^{\prime}}([\mathbf a_{n,m}^{(4)}[\pi]]^{\gamma^{\prime}}),\\
&M_{\gamma^{\prime\prime}}([\mathbf a_{n,m}^{(1)}[\pi]]^{\gamma^{\prime\prime}}),\ 
M_{\gamma^{\prime\prime}}([\mathbf a_{n,m}^{(2)}[\pi]]^{\gamma^{\prime\prime}}),\ 
M_{\gamma^{\prime\prime}}([\mathbf a_{n,m}^{(3)}[\pi]]^{\gamma^{\prime\prime}}),\ 
M_{\gamma^{\prime\prime}}([\mathbf a_{n,m}^{(4)}[\pi]]^{\gamma^{\prime\prime}})
\end{aligned}
\]
for all $(n,m)\in\hat{\mathbb N}_0^2$, $\pi\in S_{n,m}$, $k\in\mathbb N$ and $\tau\in S_{2k}^\sharp$. 
Then $\mathbf V$ satisfies $\Phi_2$.
\end{lemma} 

To prove Lemma~\ref{L: nsub M(a_{n,m}[rho]) 2}, we need the following auxiliary result.

\begin{lemma}
\label{L: subclasses of [a_{n,n}[tau]]^{alpha_1}} 
Let $\mathbf V$ be a monoid variety, $n\in\mathbb N$ and $\tau\in S_{2n}^\sharp$.
Suppose that the $\alpha_1$-class $[\mathbf a_{n,n}[\tau]]^{\alpha_1}$ is stable with respect to $\mathbf V$.
Then each $\alpha_1$-class in $\{[\mathbf a_{n,n}[\tau]]^{\alpha_1}\}^{\le_{\alpha_1}}$ is stable with respect to $\mathbf V$.
\end{lemma}

\begin{proof}
Take arbitrary $\mathtt v\in\{[\mathbf a_{n,n}[\tau]]^{\alpha_1}\}^{\le_{\alpha_1}}$ and $\mathbf v \in \mathtt v$. 
Consider an identity $\mathbf v\approx\mathbf v^\prime$ of $\mathbf V$.
In view of Lemma~\ref{L: le_alpha}, the word $\mathbf v$ is a factor of some $\mathbf u \in [\mathbf a_{n,n}[\tau]]^{\alpha_1}$, i.e., there are $\mathbf p,\mathbf q\in\mathfrak X^\ast$ such that $\mathbf u=\mathbf p\mathbf v\mathbf q$. 
Clearly,  
\[
\mathbf u=\biggl(\prod_{i=1}^n z_it_i\biggr)x\biggl(\prod_{i=1}^{2n-1} z_{i\tau}\mathbf y_i\biggr)z_{(2n)\tau}x\biggl(\prod_{i=n+1}^{2n} t_iz_i\biggr),
\]
where $\occ_{y_1}(\mathbf y_1\cdots \mathbf y_{2n-1}),\dots,\occ_{y_{2n-1}}(\mathbf y_1\cdots \mathbf y_{2n-1}) \ge 2$ and
\[
\begin{aligned}
&\{y_i\}\subseteq\con(\mathbf y_i)\subseteq\{y_1,\dots,y_i\},\\
&\{y_{n+i}\}\subseteq\con(\mathbf y_{n+i})\subseteq\{y_{n+i},\dots,y_{2n-1}\},\\
&\{y_n\}\subseteq\con(\mathbf y_n)\subseteq\{y_1,\dots,y_{2n-1}\},
\end{aligned}
\]
$i=1,\dots,n-1$.
Since the $\alpha_1$-class $[\mathbf a_{n,n}[\tau]]^{\alpha_1}$ is stable with respect to $\mathbf V$, 
\[
\mathbf p\mathbf v^\prime\mathbf q=\biggl(\prod_{i=1}^n z_it_i\biggr)x\biggl(\prod_{i=1}^{2n-1} z_{i\tau}\mathbf y_i^\prime\biggr)z_{(2n)\tau}x\biggl(\prod_{i=n+1}^{2n} t_iz_i\biggr),
\]
where $\occ_{y_1}(\mathbf y_1^\prime\cdots \mathbf y_{2n-1}^\prime),\dots,\occ_{y_{2n-1}}(\mathbf y_1^\prime\cdots \mathbf y_{2n-1}^\prime) \ge 2$ and
\[
\begin{aligned}
&\{y_i\}\subseteq\con(\mathbf y_i^\prime)\subseteq\{y_1,\dots,y_i\},\\
&\{y_{n+i}\}\subseteq\con(\mathbf y_{n+i}^\prime)\subseteq\{y_{n+i},\dots,y_{2n-1}\},\\
&\{y_n\}\subseteq\con(\mathbf y_n^\prime)\subseteq\{y_1,\dots,y_{2n-1}\},
\end{aligned}
\]
$i=1,\dots,n-1$, and the $\alpha_1$-classes $yxx^+ty$ and $ytxx^+y$ are also stable with respect to $\mathbf V$.
One can easily deduce from these facts that $\mathbf v$ and $\mathbf v^\prime$ are $\alpha_1$-related.
Since the identity $\mathbf v\approx \mathbf v^\prime$ is arbitrary, it follows that $\mathtt v$ is stable with respect to $\mathbf V$.
\end{proof}

\begin{proof}[Proof of Lemma~\ref{L: nsub M(a_{n,m}[rho]) 2}]
If $M_\gamma(yxx^+)\notin \mathbf V$, then $\mathbf V$ satisfies~\eqref{yxx=xyxx} by Lemma~\ref{L: nsub M(yxx^+)}.
In this case, Lemma~\ref{L: nsub M(a_{n,m}[rho]) 1} implies that $\Phi_2$ holds in $\mathbf V$ because $yx^2ty\stackrel{\eqref{yxx=xyxx}}\approx xyx^2ty$ and $ytyx^2\stackrel{\eqref{yxx=xyxx}}\approx ytxyx^2$.
Since $xyx^2\stackrel{x^2\approx x^3}\approx xyx^3\stackrel{\eqref{xyxx=xxyx}}\approx x^2yx^2$, by the dual arguments we can show that $\mathbf V$ satisfies $\Phi_2$ whenever $M_\gamma(xx^+y)\notin \mathbf V$.

If $M_{\gamma^\prime}(yxx^+ty),M_{\gamma^\prime}(ytyxx^+)\notin \mathbf V$, then $\mathbf V$ satisfies the identities~\eqref{yxxty=xyxxty} and~\eqref{ytyxx=ytxyxx} by Lemma~\ref{L: nsub M(yxx^+ty)} and the dual to Lemma~\ref{L: nsub M([yx^2zy]),M(xx^+yty)}(ii).
Now Lemma~\ref{L: nsub M(a_{n,m}[rho]) 1} applies again, yielding that $\Phi_2$ holds in $\mathbf V$.
By the similar arguments we can show that $\mathbf V$ satisfies $\Phi_2$ whenever $M_{\gamma^\prime}(xx^+yty),M_{\gamma^\prime}(ytxx^+y)\notin \mathbf V$.
Thus, we may further assume that the following two claims hold:
\begin{itemize}
\item either $M_{\gamma^\prime}(yxx^+ty)\in\mathbf V$ or $M_{\gamma^\prime}(ytyxx^+)\in \mathbf V$;
\item either $M_{\gamma^\prime}(xx^+yty)\in\mathbf V$ or $M_{\gamma^\prime}(ytxx^+y)\in \mathbf V$.
\end{itemize}
Four cases are possible.

\smallskip

\textbf{Case 1}: $M_{\gamma^\prime}(yxx^+ty),M_{\gamma^\prime}(xx^+yty)\in\mathbf V$.
It is easy to see that every identity in $\Phi_2$ follows from an identity $\mathbf a^{(1)}_{n,m}[\rho]\approx \overline{\mathbf a}^{(1)}_{n,m}[\rho]$ for some $(n,m)\in\hat{\mathbb N}_0^2$ and $\rho\in S_{n,m}$.
Thus, it suffices to show that the identity $\mathbf a^{(1)}_{n,m}[\rho]\approx \overline{\mathbf a}^{(1)}_{n,m}[\rho]$ holds $\mathbf V$ for all $(n,m)\in\hat{\mathbb N}_0^2$ and $\rho\in S_{n,m}$.
To do this, we will show by induction on $n+m$ that $\mathbf V$ satisfies 
\[
\Sigma_{n+m}^{(1)}:=\{\mathbf a^{(1)}_{n,m}[\pi]\approx \overline{\mathbf a}^{(1)}_{n,m}[\pi]\mid \pi\in S_{n,m}\}.
\]

\smallskip

\textbf{Induction base}: $n+m=1$.
The identity $\mathbf a^{(1)}_{1,0}[\rho]\approx \overline{\mathbf a}^{(1)}_{1,0}[\rho]$ is trivial an so holds in $\mathbf V$.
Assume now that $n=0$ and $m=1$.
Evidently, $\rho$ in nothing but the trivial permutation $\varepsilon$ from $S_1$.
By the condition of the lemma, $M_{\gamma^{\prime\prime}}([\mathbf a^{(1)}_{0,1}[\varepsilon]]^{\gamma^{\prime\prime}})\notin\mathbf V$.
Then the $\gamma^{\prime\prime}$-class $[\mathbf a^{(1)}_{0,1}[\varepsilon]]^{\gamma^{\prime\prime}}=xz_1y_1y_1^+z_1^\prime xt_1z_1t_1^\prime z_1^\prime$ is not stable with respect to $\mathbf V$ by Lemma~\ref{L: M_{gamma''}(W) in V}.
This implies that $\mathbf V$ satisfies an identity $\mathbf u\approx \mathbf v$ such that $\mathbf u\in [\mathbf a^{(1)}_{0,1}[\varepsilon]]^{\gamma^{\prime\prime}}$ and $\mathbf v\notin [\mathbf a^{(1)}_{0,1}[\varepsilon]]^{\gamma^{\prime\prime}}$.
Since $\mathbf M_{\gamma^\prime}(yxx^+ty)\vee \mathbf M_{\gamma^\prime}(xx^+yty)\subseteq \mathbf V$, Lemma~\ref{L: M_alpha(W) in V} implies that $\mathbf v\in x^+z_1x^\ast y_1x^\ast y_1\,\{x,y_1\}^\ast \,z_1^\prime x^+t_1z_1t_1^\prime z_1^\prime$.
If there is an occurrence of $x$ between ${_{1\mathbf v}}z_1$ and ${_{1\mathbf v}}z_1^\prime$ in $\mathbf v$, then it is easy to see that $\mathbf u\approx \mathbf v$ together with $\Phi$ imply $\mathbf a^{(1)}_{0,1}[\varepsilon]\approx \overline{\mathbf a}^{(1)}_{0,1}[\varepsilon]$.
So, it remains to consider the case when there are no $x$ between ${_{1\mathbf v}}z_1$ and ${_{1\mathbf v}}z_1^\prime$ in $\mathbf v$.
Then 
\[
\mathbf v\in xx^+z_1y_1y_1^+ z_1^\prime x^+t_1z_1t_1^\prime z_1^\prime\cup x^+z_1y_1y_1^+ z_1^\prime xx^+t_1z_1t_1^\prime z_1^\prime=[\mathbf a^{(1)}_{0,1}[\varepsilon]]^{\gamma^\prime}\setminus[\mathbf a^{(1)}_{0,1}[\varepsilon]]^{\gamma^{\prime\prime}}.
\]
Further, since $M_{\gamma^\prime}([\mathbf a^{(1)}_{0,1}[\varepsilon]]^{\gamma^\prime})\notin\mathbf V$, Corollary~\ref{C: M_alpha(W) in V}(iv) implies that $\mathbf V$ satisfies an identity $\mathbf u^\prime\approx \mathbf v^\prime$ such that $\mathbf u^\prime\in [\mathbf a^{(1)}_{0,1}[\varepsilon]]^{\gamma^\prime}$ and $\mathbf v^\prime\notin [\mathbf a^{(1)}_{0,1}[\varepsilon]]^{\gamma^\prime}$.
Then $\mathbf V$ satisfies the identities $\mathbf a^{(1)}_{0,1}[\varepsilon]\stackrel{x^2\approx x^3}\approx\mathbf u\approx \mathbf v\stackrel{\eqref{xyxx=xxyx}}\approx \mathbf u^\prime\approx \mathbf v^\prime$.
Since $\mathbf M_{\gamma^\prime}(yxx^+ty)\vee \mathbf M_{\gamma^\prime}(xx^+yty)\subseteq \mathbf V$, Lemma~\ref{L: M_alpha(W) in V} implies that $\mathbf v^\prime\in x^+z_1x^\ast y_1x^\ast y_1\,\{x,y_1\}^\ast \,z_1^\prime x^+t_1z_1t_1^\prime z_1^\prime$.
Then the fact that $\mathbf v^\prime\notin [\mathbf a^{(1)}_{0,1}[\varepsilon]]^{\gamma^\prime}$ implies that there is an occurrence of $x$ between ${_{1\mathbf v^\prime}}z_1$ and ${_{1\mathbf v^\prime}}z_1^\prime$ in $\mathbf v^\prime$.
In this case, $\mathbf a^{(1)}_{0,1}[\varepsilon]\approx \mathbf v^\prime$ together with $\Phi$ imply $\mathbf a^{(1)}_{0,1}[\varepsilon]\approx \overline{\mathbf a}^{(1)}_{0,1}[\varepsilon]$.
We see that the identity $\mathbf a^{(1)}_{0,1}[\varepsilon]\approx \overline{\mathbf a}^{(1)}_{0,1}[\varepsilon]$ holds in $\mathbf V$ as required.

\smallskip

\textbf{Induction step}: $n+m>1$ and $\mathbf V$ satisfies $\Sigma_k^{(1)}$ for all $k<n+m$.
By the condition of the lemma, $M_{\gamma^{\prime\prime}}([\mathbf a^{(1)}_{n,m}[\rho]]^{\gamma^{\prime\prime}})\notin\mathbf V$.
Then the $\gamma^{\prime\prime}$-class $[\mathbf a^{(1)}_{n,m}[\rho]]^{\gamma^{\prime\prime}}$ is not stable with respect to $\mathbf V$ by Lemma~\ref{L: M_{gamma''}(W) in V}.
This implies that $\mathbf V$ satisfies an identity $\mathbf u\approx \mathbf v$ such that $\mathbf u\in [\mathbf a^{(1)}_{n,m}[\rho]]^{\gamma^{\prime\prime}}$ and $\mathbf v\notin [\mathbf a^{(1)}_{n,m}[\rho]]^{\gamma^{\prime\prime}}$.
Clearly, $
[\mathbf a^{(1)}_{n,m}[\rho]]^{\gamma^{\prime\prime}}=\mathbf hxZ_{1\rho}^{(1)}\cdots Z_{(n+m)\rho}^{(1)}x\mathbf t$, where
\[
\mathbf h:=\biggl(\prod_{i=1}^n z_i t_i \biggr),\ 
\mathbf t:=\biggl(\prod_{i=n+1}^{n+m} t_i z_i t_i^{\prime}z_i^{\prime}\biggr),\ 
Z_i^{(1)}:=
\begin{cases} 
\{z_i\} & \text{if $1\le i\le n$}, \\
z_iy_iy_i^+z_i^{\prime} & \text{if $n+1\le i\le n+m$}.
\end{cases} 
\]
In view of Lemma~\ref{L: identities of M(xt_1x...t_kx)}, $\mathbf v_{\{x,y_{n+1},\dots,y_{n+m}\}}=\mathbf h\mathbf a\mathbf t$, where $\mathbf a$ is a linear word depending on the letters $z_1,\dots,z_n,z_{n+1},z_{n+1}^\prime,\dots,z_{n+m},z_{n+m}^\prime$.
It is easy to see that $\mathbf M(xzxyty)\subseteq \mathbf M_{\gamma^\prime}(xx^+yty)$.
Then, since $\mathbf M_{\gamma^\prime}(yxx^+ty)\vee \mathbf M_{\gamma^\prime}(xx^+yty)\subseteq \mathbf V$, Lemmas~\ref{L: M(W) in V} and~\ref{L: M_alpha(W) in V} imply that the sets $yxx^+ty$, $xx^+yty$ and $\{xzxyty\}$ are stable with respect to $\mathbf V$.
Hence $\mathbf v= \mathbf h\mathbf z_{1\rho}\cdots \mathbf z_{(n+m)\rho}\mathbf t$, where $\hat{\mathbf z}_{i\rho}:=(\mathbf z_{i\rho})_x\in Z_{i\rho}^{(1)}$, $i=1,\dots,n+m$.
Moreover, since $\rho\in S_{n,m}$, either ${1\rho}\le n$ or ${2\rho}\le n$ and either ${(n+m-1)\rho}\le n$ or ${(n+m)\rho}\le n$, whence either $x\in\con(\mathbf z_{1\rho})$ or $\head(\mathbf z_{2\rho})=x$ and either $x\in\con(\mathbf z_{(n+m)\rho})$ or $\tail(\mathbf z_{(n+m-1)\rho})=x$.

Suppose that either $\tail(\mathbf z_{1\rho})=x$ or $\head(\mathbf z_{(n+m)\rho})=x$ or $x\in\con(\mathbf z_{2\rho}\cdots\mathbf z_{(n+m-1)\rho})$.
Then, using identities in $\Sigma_{n+m-1}^{(1)}$, we can put occurrences of $x$ in the factor 
\[
\tail(\mathbf z_{1\rho})\mathbf z_{2\rho}\cdots\mathbf z_{(n+m-1)\rho}\head(\mathbf z_{(n+m)\rho})
\] 
of $\mathbf v$.
In other words,
\[
\mathbf v\stackrel{\Sigma_{n+m-1}^{(1)}}\approx \mathbf h\mathbf z_{1\rho}\,x\,\chi(\hat{\mathbf z}_{2\rho}\cdots \hat{\mathbf z}_{(n+m-1)\rho})\,x\,\mathbf z_{(n+m)\rho}\mathbf t:=\hat{\mathbf v}.
\]
If $\head(\mathbf z_{1\rho})\ne x$, then ${1\rho}\le n$ and $\hat{\mathbf z}_{1\rho}=z_{1\rho}$ because the sets $yxx^+ty$ and $xx^+yty$ are stable with respect to $\mathbf V$.
In this case, $\mathbf z_{1\rho}\in z_{1\rho}x^\ast$ and the identity $\mathbf u(x,t_{1\rho},z_{1\rho})\approx \mathbf v(x,t_{1\rho},z_{1\rho})$ is equivalent modulo $x^2\approx x^3$ to
\begin{equation}
\label{ytyxx=ytxyx}
ytyx^2\approx ytxyx.
\end{equation}
Then $\mathbf V$ satisfies 
\[
\hat{\mathbf v}\stackrel{\Sigma_{n+m-1}^{(1)}}\approx \mathbf h\mathbf z_{1\rho}\,x^2\,\chi(\hat{\mathbf z}_{2\rho}\cdots \hat{\mathbf z}_{(n+m-1)\rho})\,x\,\mathbf z_{(n+m)\rho}\mathbf t\stackrel{\eqref{ytyxx=ytxyx}}\approx \mathbf h x\,\chi(\hat{\mathbf z}_{1\rho}\cdots \hat{\mathbf z}_{(n+m-1)\rho})\,x\,\mathbf z_{(n+m)\rho}\mathbf t.
\]
So, we may assume without any loss that $\head(\mathbf z_{1\rho})=x$ and, dually, $\tail(\mathbf z_{(n+m)\rho})=x$.
In this case, $\hat{\mathbf v}\stackrel{\Sigma_{n+m-1}^{(1)}}\approx \overline{\mathbf a}^{(1)}_{n,m}[\rho]$ and, therefore, $\mathbf V$ satisfies $\mathbf a^{(1)}_{n,m}[\rho]\approx \overline{\mathbf a}^{(1)}_{n,m}[\rho]$.

Thus, it remains to consider the case when $x\notin\con(\tail(\mathbf z_{1\rho})\mathbf z_{2\rho}\cdots\mathbf z_{(n+m-1)\rho}\head(\mathbf z_{(n+m)\rho}))$.
Then $x\in\con(\mathbf z_{1\rho})\cap\con(\mathbf z_{(n+m)\rho})$. 
If $\head(\mathbf z_{1\rho})\ne x$, then ${1\rho}\le n$ and $\hat{\mathbf z}_{1\rho}=z_{1\rho}$ because the sets $yxx^+ty$ and $xx^+yty$ are stable with respect to $\mathbf V$.
In this case, $\mathbf z_{1\rho}\in z_{1\rho}x^+$ contradicting the assumption that $\tail(\mathbf z_{1\rho})\ne x$.
Therefore, $\head(\mathbf z_{1\rho})=x$.
Analogously, $\tail(\mathbf z_{(n+m)\rho})=x$.

If there is an occurrence of $x$ succeeding $z_{1\rho}$ in $\mathbf z_{1\rho}$, then $1\rho>n$ and this occurrence of $x$ lies between $z_{1\rho}$ and $z_{1\rho}^\prime$ in $\mathbf z_{1\rho}$ because $\tail(\mathbf z_{1\rho})\ne x$.
In this case, since
\[
\mathbf u_{\{y_{1\rho},z_{1\rho}^\prime\}}\stackrel{x^2\approx x^3}\approx \mathbf hxz_{1\rho}\hat{\mathbf z}_{2\rho}\cdots \hat{\mathbf z}_{(n+m)\rho}x\mathbf t\ \text{ and }\ \mathbf v_{\{y_{1\rho},z_{1\rho}^\prime\}}\stackrel{\Sigma_{n+m-1}^{(1)}}\approx  \mathbf hxz_{1\rho}x\mathbf z_{2\rho}\cdots \mathbf z_{(n+m)\rho}x\mathbf t,
\]
the variety $\mathbf V$ satisfies the identities
\[
\mathbf a^{(1)}_{n,m}[\rho]\stackrel{x^2\approx x^3}\approx\mathbf u\approx \mathbf v\stackrel{\{\Phi,\,\Sigma_{n+m-1}^{(1)},\,\mathbf u\approx \mathbf v\}}\approx \mathbf hx z_{1\rho}(xy_{1\rho})^2x z_{1\rho}^\prime x \mathbf z_{2\rho}\cdots \mathbf z_{(n+m)\rho}\mathbf t\stackrel{\Sigma_{n+m-1}^{(1)}}\approx \overline{\mathbf a}^{(1)}_{n,m}[\rho].
\]
By a similar argument we can show that $\mathbf V$ satisfies $\mathbf a^{(1)}_{n,m}[\rho]\approx \overline{\mathbf a}^{(1)}_{n,m}[\rho]$ whenever $(n+m)\rho>n$ and there is an occurrence of $x$ between $z_{(n+m)\rho}$ and $z_{(n+m)\rho}^\prime$ in $\mathbf z_{(n+m)\rho}$.
So, it remains to consider the case when 
\[
\mathbf v\in \mathbf hxx^+Z_{1\rho}^{(1)}\cdots Z_{(n+m)\rho}^{(1)}x^+\mathbf t\cup \mathbf hx^+Z_{1\rho}^{(1)}\cdots Z_{(n+m)\rho}^{(1)}xx^+\mathbf t.
\]
Further, since $M_{\gamma^\prime}([\mathbf a^{(1)}_{n,m}[\rho]]^{\gamma^\prime})\notin\mathbf V$, Corollary~\ref{C: M_alpha(W) in V}(iv) implies that $\mathbf V$ satisfies an identity $\mathbf u^\prime\approx \mathbf v^\prime$ such that $\mathbf u^\prime\in [\mathbf a^{(1)}_{n,m}[\rho]]^{\gamma^\prime}$ and $\mathbf v\notin [\mathbf a^{(1)}_{n,m}[\rho]]^{\gamma^\prime}$.
Then $\mathbf V$ satisfies the identities $\mathbf a^{(1)}_{n,m}[\rho]\stackrel{x^2\approx x^3}\approx\mathbf u\approx \mathbf v\stackrel{\eqref{xyxx=xxyx}}\approx \mathbf u^\prime\approx \mathbf v^\prime$.
Since the sets $yxx^+ty$ and $xx^+yty$ are stable with respect to $\mathbf V$, we have $
\mathbf v^\prime=\mathbf hx^p\mathbf z_{1\rho}^\prime\cdots \mathbf z_{(n+m)\rho}^\prime x^q\mathbf t$, where $p,q\in\mathbb N_0$, $\head(\mathbf z_{1\rho}^\prime)\ne x$, $\tail(\mathbf z_{(n+m)\rho}^\prime)\ne x$ and $(\mathbf z_{i\rho}^\prime)_x\in Z_{i\rho}^{(1)}$, $i=1,\dots,n+m$.
Further, $x\in\con(\mathbf z_{1\rho}^\prime\cdots\mathbf z_{(n+m)\rho}^\prime)$ because $\mathbf v^\prime\notin [\mathbf a^{(1)}_{n,m}[\rho]]^{\gamma^\prime}$.
Then arguments similar to ones from the above imply that $\mathbf v^\prime\approx \overline{\mathbf a}^{(1)}_{n,m}[\rho]$ holds in $\mathbf V$.
Hence $\mathbf V$ satisfies the identity $\mathbf a^{(1)}_{n,m}[\rho]\approx\overline{\mathbf a}^{(1)}_{n,m}[\rho]$ as required.

\smallskip

\textbf{Case 2}: $M_{\gamma^\prime}(ytyxx^+),M_{\gamma^\prime}(ytxx^+y)\in\mathbf V$.
This case is dual to Case~1 and we omit the corresponding considerations.

\smallskip

\textbf{Case 3}: $M_{\gamma^\prime}(ytyxx^+),M_{\gamma^\prime}(xx^+ yty)\in\mathbf V$. 
This case is quite similar to Case~1 and we omit the corresponding considerations.

\smallskip

\textbf{Case 4}: $M_{\gamma^\prime}(ytxx^+y),M_{\gamma^\prime}(yxx^+ty)\in\mathbf V$.
If $M(xzxyty)\in\mathbf V$, then arguments similar to ones from Case~1 imply that $\Phi_2$ holds in $\mathbf V$.
Thus, it remains to consider the case when $M(xzxyty)\notin\mathbf V$.
Then $\sigma_3$ holds in $\mathbf V$ by Lemma~\ref{L: swapping in linear-balanced}.
It is easy to see that every identity in $\Phi_2$ follows from an identity $\mathbf a_{n,n}[\pi]\approx \overline{\mathbf a}_{n,n}[\pi]$ for some $n\in\mathbb N$ and $\pi\in S_{n,n}$ with
\[
n+1\le1\pi,\dots,(2n-1)\pi\le2n,\ 1\le2\pi,\dots,(2n)\pi\le n.
\]
Further, since
\[
\begin{aligned}
&\mathbf a_{n,n}[\pi]\stackrel{\{\Phi,\,\sigma_3\}}\approx\mathbf p x\biggl(\prod_{i=1}^{n-1} z_{(2i-1)\pi}(y_{2i-1}^2y_{2i}^2)^2\biggr)z_{(2n-1)\pi}\biggl(\prod_{i=1}^{n-1} z_{(2i)\pi}(y_{2i}^2y_{2i+1}^2)^2\biggr)z_{(2n)\pi}x\mathbf q,\\
&\overline{\mathbf a}_{n,n}[\pi]\stackrel{\{\Phi,\,\sigma_3\}}\approx\mathbf p x\biggl(\prod_{i=1}^{n-1} z_{(2i-1)\pi}(y_{2i-1}^2y_{2i}^2)^2x\biggr)z_{(2n-1)\pi}x\biggl(\prod_{i=1}^{n-1} z_{(2i)\pi}(y_{2i}^2y_{2i+1}^2)^2x\biggr)z_{(2n)\pi}x\mathbf q,
\end{aligned}
\]
where
\[
\mathbf p:=\biggl(\prod_{i=1}^n z_it_i\biggr) \ \text{ and }\ \mathbf q:=\biggl(\prod_{i=n+1}^{2n} t_iz_i\biggr),
\]
the identity $\mathbf a_{n,n}[\pi]\approx \overline{\mathbf a}_{n,n}[\pi]$ follows from $\{\Phi,\,\sigma_3,\,\mathbf a_{n,n}[\rho]\approx \overline{\mathbf a}_{n,n}[\rho]\}$ for some $\rho\in S_{2n}^\sharp$.
Thus, it suffices to show that  $\mathbf a_{n,n}[\rho]\approx \overline{\mathbf a}_{n,n}[\rho]$ holds $\mathbf V$.

By the condition of the lemma, $M_{\alpha_1}([\mathbf a_{n,n}[\rho]]^{\alpha_1})\notin\mathbf V$.
In view of Lemmas~\ref{L: M_alpha(W) in V} and~\ref{L: subclasses of [a_{n,n}[tau]]^{alpha_1}}, the $\alpha_1$-class $[\mathbf a_{n,n}[\rho]]^{\alpha_1}$ is not stable with respect to $\mathbf V$.
This implies that $\mathbf V$ satisfies an identity $\mathbf u\approx \mathbf v$ such that $\mathbf u\in [\mathbf a_{n,n}[\rho]]^{\alpha_1}$ and $\mathbf v\notin [\mathbf a_{n,n}[\rho]]^{\alpha_1}$.
In view of Lemma~\ref{L: identities of M(xt_1x...t_kx)},
\[
\mathbf v_{\{x,y_1,\dots,y_{2n-1}\}}=\biggl(\prod_{i=1}^n z_i t_i \biggr)\mathbf a\biggl(\prod_{i=n+1}^{2n} t_iz_i\biggr),
\]
where $\mathbf a$ is a linear word depending on the letters $z_1,\dots,z_{2n}$.
Then since $\mathbf M_{\gamma^\prime}(ytxx^+y)\vee \mathbf M_{\gamma^\prime}(yxx^+ty)\subseteq \mathbf V$, Lemma~\ref{L: M_alpha(W) in V} implies that 
\[
\begin{aligned}
&\mathbf v(x,z_{1\rho},t_{1\rho},z_{(2n)\rho},t_{(2n)\rho})\in z_{(2n)\rho}t_{(2n)\rho}x^+z_{1\rho}x^\ast z_{(2n)\rho}x^+ t_{1\rho}z_{1\rho},\\
&\mathbf v(y_i,z_{i\rho},t_{i\rho})\in z_{i\rho}y_iy_i^+ t_{i\rho}z_{i\rho},\\
&\mathbf v(y_j,z_{(j+1)\rho},t_{(j+1)\rho})\in y_j^+z_{(j+1)\rho}y_j^\ast t_{(j+1)\rho}z_{(j+1)\rho},\\
&\mathbf v(y_{n+i-1},z_{(n+i)\rho},t_{(n+i)\rho})\in z_{(n+i)\rho} t_{(n+i)\rho}y_{n+i-1}y_{n+i-1}^+z_{(n+i)\rho},\\
&\mathbf v(y_{n+j},z_{(n+j)\rho},t_{(n+j)\rho})\in z_{(n+j)\rho} t_{(n+j)\rho}y_{n+j}^\ast z_{(n+j)\rho}y_{n+j}^+,
\end{aligned}
\]
$i=1,\dots,n$ and $j=1,\dots,n-1$.
Then
\[
\mathbf v\in \biggl(\prod_{i=1}^n z_it_i\biggr)x^+\biggl(\prod_{i=1}^{2n-1} z_{i\rho}\mathbf y_i\biggr)z_{(2n)\rho}x^+\biggl(\prod_{i=n+1}^{2n} t_iz_i\biggr),
\]
where
\[
\begin{aligned}
&\{y_i\}\subseteq\con(\mathbf y_i)\subseteq\{x,y_1,\dots,y_i\},\\
&\{y_{n+i}\}\subseteq\con(\mathbf y_{n+i})\subseteq\{x,y_{n+i},\dots,y_{2n-1}\},\\
&\{y_n\}\subseteq\con(\mathbf y_n)\subseteq\{x,y_1,\dots,y_{2n-1}\},
\end{aligned}
\]
$i=1,\dots,n-1$.
Further, since $\mathbf v\notin[\mathbf a_{n,n}[\rho]]^{\alpha_1}$, we have $\occ_x(\mathbf v)\ge3$. 
In this case, $\mathbf a_{n,n}[\rho]\approx \overline{\mathbf a}_{n,n}[\rho]$ holds in $\mathbf V$ because $\mathbf u\stackrel{\{\Phi,\,\sigma_3\}}\approx \mathbf a_{n,n}[\rho]$ and $\mathbf v\stackrel{\{\Phi,\,\sigma_3\}}\approx \overline{\mathbf a}_{n,n}[\rho]$.
\end{proof}

\section{Certain varieties with a non-distributive subvariety lattice}
\label{Sec: non-distributive}

\subsection{Sporadic varieties}
\label{Subsec: sporadic}

For a congruence $\alpha$ on $\mathfrak X^\ast$, we denote by $\overline{\alpha}$ the congruence on $\mathfrak X^\ast$ dual to $\alpha$.

\begin{proposition}
\label{P: sporadic}
The following varieties have a non-modular lattice of subvarieties:
\begin{itemize}
\item[\textup{(i)}] $\mathbf M_\gamma(xx^+y)\vee\mathbf H$;
\item[\textup{(ii)}] $\mathbf M_\gamma(xx^+y)\vee\mathbf M_\nu([yx^2ty]^\nu)$;
\item[\textup{(iii)}] $\mathbf M_\gamma(xx^+y)\vee\mathbf M_\eta([xyzx^2ty]^\eta)$;
\item[\textup{(iv)}] $\mathbf M_\lambda(xyx^+)\vee\mathbf M_{\overline{\lambda}}(x^+yx)$;
\item[\textup{(v)}] $\mathbf M_\lambda(xyx^+)\vee\mathbf M_\nu([yx^2ty]^\nu)$;
\item[\textup{(vi)}] $\mathbf M_\lambda(xyx^+)\vee\mathbf M_{\gamma^\prime}(yxx^+ty)$;
\item[\textup{(vii)}] $\mathbf M_\lambda(xyx^+)\vee\mathbf M_{\gamma^\prime}(ytyxx^+)$;
\item[\textup{(viii)}] $\mathbf M_\lambda(xyx^+)\vee\mathbf M_{\gamma^\prime}([x^2yzytx^2]^{\gamma^\prime})$;
\item[\textup{(ix)}] $\mathbf M_\lambda(xyx^+)\vee\mathbf M_{\gamma^\prime}([x^2yzx^2ty]^{\gamma^\prime})$;
\item[\textup{(x)}] $\mathbf M_\lambda(xyx^+)\vee\mathbf M_{\gamma^\prime}([yzx^2ytx^2]^{\gamma^\prime})$;
\item[\textup{(xi)}] $\mathbf M_\lambda(xyx^+)\vee\mathbf M_{\gamma^\prime}([x^2zytx^2y]^{\gamma^\prime})$;
\item[\textup{(xii)}] $\mathbf M(xyx)\vee\mathbf H$;
\item[\textup{(xiii)}] $\mathbf M(xyx)\vee\mathbf M_\lambda(xyzx^+ty^+)$;
\item[\textup{(xiv)}] $\mathbf M(xyx)\vee\mathbf M_\lambda(xx^+yty^+)$;
\item[\textup{(xv)}] $\mathbf M(xyxzx)\vee\mathbf M_\nu([yx^2ty]^\nu)$;
\item[\textup{(xvi)}] $\mathbf M(xyxzx)\vee\mathbf M_{\gamma^\prime}(yxx^+ty)$;
\item[\textup{(xvii)}] $\mathbf M(xyxzx)\vee\mathbf M_{\gamma^\prime}(xx^+yty)$;
\item[\textup{(xviii)}] $\mathbf M(xyxzx)\vee\mathbf M_{\lambda^\prime}(xyzxx^+ty)$;
\item[\textup{(xix)}] $\mathbf M(xyxzx)\vee\mathbf M_{\lambda^\prime}(xyzytxx^+)$;
\item[\textup{(xx)}] $\mathbf M(xyxzx)\vee\mathbf M_{\lambda^\prime}(yzyxtxx^+)$;
\item[\textup{(xxi)}] $\mathbf M_\mu(xyxzx^+)\vee\mathbf M_{\lambda^\prime}(xyzx^+tysx^+)$;
\item[\textup{(xxii)}] $\mathbf M_\mu(xyxzx^+)\vee\mathbf M_\eta([xyzx^2ty]^\eta)$;
\item[\textup{(xxiii)}] $\mathbf M(xzxyty)\vee\mathbf N$;
\item[\textup{(xxiv)}] $\mathbf M(xzytxy)\vee\mathbf N$.
\end{itemize}
\end{proposition}

\begin{proof}
(i) Consider an arbitrary identity $\mathbf u\approx \mathbf v$ of $\mathbf M_\gamma(xx^+y)\vee\mathbf H$ with $\mathbf u\in xyzx^+ty^+$.
In view of Lemma~\ref{L: L(M(xzyx^+ty^+))}, $\mathbf M_\lambda(xyx^+)\subset \mathbf H$.
Since $\mathbf u(x,z)\in xzx^+$, $\mathbf u(x,t)\in xx^+t$ and $\mathbf u(y,z,t)\in yzty^+$, Lemma~\ref{L: M_alpha(W) in V} implies that $\mathbf v(x,z)\in xzx^+$, $\mathbf v(x,t)\in xx^+t$ and $\mathbf v(y,z,t)\in yzty^+$.
Hence $\mathbf v\in xyzx^+ty^+\cup yxzx^+ty^+$.
It follows from Proposition~6.12 in~\cite{Gusev-Vernikov-18} that $\mathbf H$ violates $xyzxty\approx yxzxty$, whence $\mathbf v\in xyzx^+ty^+$. 
We see that the $\lambda$-class $xyzx^+ty^+$ is stable with respect to $\mathbf M_\gamma(x^+y)\vee\mathbf H$.
Now Corollary~\ref{C: M_alpha(W) in V}(ii) applies, yielding that 
\[
\left(\mathbf M_\gamma(xx^+y)\vee\mathbf H\right)\wedge \mathbf M_\lambda(xyzx^+ty^+) = \mathbf M_\lambda(xyzx^+ty^+).
\]
Further, in view of Lemmas~\ref{L: L(D1)} and~\ref{L: L(M(xzyx^+ty^+))},
\[
\mathbf M_\gamma(xx^+y)\vee\left(\mathbf H\wedge\mathbf M_\lambda(xyzx^+ty^+)\right)=\mathbf M_\gamma(xx^+y)\vee\mathbf M_\lambda(xyx^+)\subset\mathbf M_\lambda(xyzx^+ty^+).
\]
Therefore, the lattice $\mathfrak L\left(\mathbf M_\gamma(xx^+y)\vee\mathbf H\right)$ is not modular.

\smallskip

(ii) Consider an arbitrary identity $\mathbf u\approx \mathbf v$ of $\mathbf M_\gamma(xx^+y)\vee\mathbf M_\nu([yx^2zy]^\nu)$ with $\mathbf u\in yxx^+ty$.
Since $\mathbf u(x,t)\in xx^+t$, $\mathbf v(y,t)=yty$ and $({_{1\mathbf u}y}) < ({_{1\mathbf u}x}) < ({_{1\mathbf u}t})$, Lemmas~\ref{L: M(W) in V} and~\ref{L: M_alpha(W) in V} imply $\mathbf v(x,t)\in xx^+t$, $\mathbf v(y,t)=yty$ and $({_{1\mathbf v}y}) < ({_{1\mathbf v}x}) < ({_{1\mathbf v}t})$.
Hence $\mathbf v\in yxx^+ty$.
We see that the $\gamma^\prime$-class $yxx^+ty$ is stable with respect to $\mathbf M_\gamma(xx^+y)\vee\mathbf M_\nu([yx^2ty]^\nu)$.
Now Corollary~\ref{C: M_alpha(W) in V}(iv) applies, yielding that 
\[
\left(\mathbf M_\gamma(xx^+y)\vee\mathbf M_\nu([yx^2ty]^\nu)\right)\wedge \mathbf M_{\gamma^\prime}(yxx^+ty)=\mathbf M_{\gamma^\prime}(yxx^+ty).
\]
Further, $\mathbf M_\nu([yx^2ty]^\nu)\wedge \mathbf M_{\gamma^\prime}(yxx^+ty)$ satisfies the identities
\[
yx^2ty\stackrel{\mathbf M_\nu([yx^2ty]^\nu)}\approx yx^2txy\stackrel{\mathbf M_{\gamma^\prime}(yxx^+ty)}\approx x^2ytxy \stackrel{\mathbf M_\nu([yx^2ty]^\nu)}\approx x^2yty.
\]
Evidently, the identity 
\begin{equation}
\label{xxyty=yxxty}
x^2yty\approx yx^2ty
\end{equation}
holds in $\mathbf M_\gamma(xx^+y)$ as well.
Since $\mathbf M_\gamma(xx^+y)\subset \mathbf M_{\gamma^\prime}(yxx^+ty)$, this implies that
\[
\mathbf M_\gamma(xx^+y)\vee\left(\mathbf M_\nu([yx^2ty]^\nu)\wedge\mathbf M_{\gamma^\prime}(yxx^+ty)\right)\subset\mathbf M_{\gamma^\prime}(yxx^+ty).
\]
The lattice $\mathfrak L\left(\mathbf M_\gamma(xx^+y)\vee\mathbf M_\nu([yx^2ty]^\nu)\right)$ is thus not modular.

\smallskip

(iii) Consider an arbitrary identity $\mathbf u\approx \mathbf v$ holding in $\mathbf M_\gamma(xx^+y)\vee\mathbf M_\eta([xyzx^2ty]^\eta)$ with $\mathbf u\in xyzxx^+ty$.
Since $\mathbf u(x,z,t)\in xzxx^+t$ and $\mathbf u(y,z,t)=yzty$, Lemmas~\ref{L: M(W) in V} and~\ref{L: M_alpha(W) in V} imply that $\mathbf v(x,z,t)\in xzxx^+t$ and $\mathbf v(y,z,t)=yzty$.
Hence $\mathbf v\in xyzxx^+ty\cup yxzxx^+ty$.
However, $M_\eta([xyzx^2ty]^\eta)$ violates the identity~\eqref{xyzxxty=yxzxxty}.
Therefore, $\mathbf v\in xyzxx^+ty$.
We see that the $\lambda^\prime$-class $xyzxx^+ty$ is stable with respect to $\mathbf M_\gamma(xx^+y)\vee\mathbf M_\eta([xyzx^2ty]^\eta)$.
Now Corollary~\ref{C: M_alpha(W) in V}(v) applies, yielding that 
\[
\left(\mathbf M_\lambda(xx^+y)\vee\mathbf M_\eta([xyzx^2ty]^\eta)\right)\wedge \mathbf M_{\lambda^\prime}(xyzxx^+ty)=\mathbf M_{\lambda^\prime}(xyzxx^+ty).
\]
Further, $\mathbf M_\eta([xyzx^2ty]^\eta)\wedge \mathbf M_{\lambda^\prime}(xyzxx^+ty)$ satisfies the identities
\[
xyzx^2ty\stackrel{\mathbf M_\eta([xyzx^2ty]^\eta)}\approx xyzx^2txy\stackrel{\mathbf M_{\lambda^\prime}(xyzxx^+ty)}\approx yxzx^2txy \stackrel{\mathbf M_\eta([xyzx^2ty]^\eta)}\approx yxzx^2ty.
\]
Clearly, the identity~\eqref{xyzxxty=yxzxxty} holds in $\mathbf M_\gamma(xx^+y)$ as well.
Since $\mathbf M_\gamma(xx^+y)\subset \mathbf M_{\lambda^\prime}(xyzxx^+ty)$, this implies that
\[
\mathbf M_\gamma(xx^+y)\vee\left(\mathbf M_\eta([xyzx^2ty]^\eta)\wedge \mathbf M_{\lambda^\prime}(xyzxx^+ty)\right)\subset\mathbf M_{\lambda^\prime}(xyzxx^+ty).
\]
The lattice $\mathfrak L\left(\mathbf M_\gamma(xx^+y)\vee\mathbf M_\eta([xyzx^2ty]^\eta)\right)$ is thus not modular.

\smallskip

(iv) Suppose that $\mathbf M_\lambda(xyx^+)\vee\mathbf M_{\overline{\lambda}}(x^+yx)$ satisfies an identity $xyx\approx \mathbf w$ for some $\mathbf w\in\mathfrak X^\ast$.
Then $\mathbf w=x^syx^t$ for some $s,t\in\mathbb N_0$ by Lemma~\ref{L: identities of M(xy)} and the evident inclusion $\mathbf M(xy)\subseteq \mathbf M_\lambda(xyx^+)\vee\mathbf M_{\overline{\lambda}}(x^+yx)$.
In view of Lemma~\ref{L: M_alpha(W) in V}, this is only possible when $s=t=1$ and so $xyx$ is an isoterm for $\mathbf M_\lambda(xyx^+)\vee\mathbf M_{\overline{\lambda}}(x^+yx)$.
Let $\mathbf V:=\mathbf M_{\overline{\lambda}}(x^+yx)\vee\mathbf M(xyx)$.
Then 
\[
\left(\mathbf M_\lambda(xyx^+)\vee\mathbf M_{\overline{\lambda}}(x^+yx)\right)\wedge \mathbf V = \mathbf V
\]
by Lemma~\ref{L: M(W) in V}.
Clearly, $\mathbf M_{\overline{\lambda}}(x^+yx)\subset\mathbf V$.
Further, it is routine to check that~\eqref{yxx=xyxx} holds in $\mathbf V$.
Now Lemma~\ref{L: M_alpha(W) in V} applies, yielding that $M_\gamma(yxx^+)\notin \mathbf V$.
Then $\mathbf M_\lambda(xyx^+)\wedge\mathbf V=\mathbf M(xy)$ by Lemma~\ref{L: L(M(xzyx^+ty^+))}.
Hence
\[
\mathbf M_{\overline{\lambda}}(x^+yx)\vee\left(\mathbf M_\lambda(xyx^+)\wedge\mathbf V\right)=\mathbf M_{\overline{\lambda}}(x^+yx)\subset\left(\mathbf M_\lambda(xyx^+)\vee\mathbf M_{\overline{\lambda}}(x^+yx)\right)\wedge \mathbf V =\mathbf V.
\]
It follows that the lattice $\mathfrak L\left(\mathbf M_\lambda(xyx^+)\vee\mathbf M_{\overline{\lambda}}(x^+yx)\right)$ is not modular.

\smallskip

(v),(viii) Let $\mathbf V\in\{\mathbf M_\nu([yx^2ty]^\nu),\mathbf M_{\gamma^\prime}([x^2yzytx^2]^{\gamma^\prime})\}$.
Consider an arbitrary identity $\mathbf u\approx \mathbf v$ of $\mathbf M_\lambda(xyx^+)\vee\mathbf V$ with $\mathbf u\in xyzytxx^+$.
Since $\mathbf u(x,z,t)\in xztxx^+$ and $\mathbf u(y,z,t)=yzyt$, Lemmas~\ref{L: M(W) in V} and~\ref{L: M_alpha(W) in V} imply that $\mathbf v(x,z,t)\in xztxx^+$ and $\mathbf v(y,z,t)=yzty$.
Hence $\mathbf v\in xyzytxx^+\cup yxzytxx^+$.
However, $\mathbf M_\nu([yx^2ty]^\nu)$ violates the identity~\eqref{xyzytxx=yxzytxx} because this identity together with~\eqref{xxy=xxyx} imply the identity~\eqref{xxyty=yxxty}.
Evidently, $\mathbf M_{\gamma^\prime}([x^2yzytx^2]^{\gamma^\prime})$ does not satisfy~\eqref{xyzytxx=yxzytxx}  as well.
Therefore, $\mathbf v\in xyzytxx^+$.
We see that the $\lambda^\prime$-class $xyzytxx^+$ is stable with respect to $\mathbf M_\lambda(xyx^+)\vee\mathbf V$.
Now Corollary~\ref{C: M_alpha(W) in V}(v) applies, yielding that 
\[
\left(\mathbf M_\lambda(xyx^+)\vee\mathbf V\right)\wedge \mathbf M_{\lambda^\prime}(xyzytxx^+)=\mathbf M_{\lambda^\prime}(xyzytxx^+).
\]
Further, $\mathbf V\wedge \mathbf M_{\lambda^\prime}(xyzytxx^+)$ satisfies the identities
\[
xyzytx^2\stackrel{\mathbf V}\approx x^2yzytx\stackrel{\mathbf M_{\lambda^\prime}(xyzytxx^+)}\approx yx^2zytx \stackrel{\mathbf V}\approx yxzytx^2.
\]
Clearly, the identity~\eqref{xyzytxx=yxzytxx} holds in $\mathbf M_\lambda(xyx^+)$ as well.
Since $\mathbf M_\lambda(xyx^+)\subset \mathbf M_{\lambda^\prime}(xyzxx^+ty)$, this implies that
\[
\mathbf M_\lambda(xyx^+)\vee\left(\mathbf V\wedge\mathbf M_{\lambda^\prime}(xyzytxx^+)\right)\subset\mathbf M_{\lambda^\prime}(xyzytxx^+).
\]
The lattices $\mathfrak L\left(\mathbf M_\lambda(xyx^+)\vee\mathbf M_\nu([yx^2zy]^\nu)\right)$ and $\mathfrak L\left(\mathbf M_\lambda(xyx^+)\vee\mathbf M_{\gamma^\prime}([x^2yzytx^2]^{\gamma^\prime})\right)$ are not modular.

\smallskip

(vi),(ix) Let $\mathbf V\in\{\mathbf M_{\gamma^\prime}(yxx^+ty),\mathbf M_{\gamma^\prime}([x^2yzx^2ty]^{\gamma^\prime})\}$.
Consider an arbitrary identity $\mathbf u\approx \mathbf v$ of $\mathbf M_\lambda(xyx^+)\vee\mathbf V$ with $\mathbf u\in xyzxx^+ty$.
Since $\mathbf u(x,z,t)\in xzxx^+t$ and $\mathbf u(y,z,t)=yzty$, Lemmas~\ref{L: M(W) in V} and~\ref{L: M_alpha(W) in V} imply that $\mathbf v(x,z,t)\in xzxx^+t$ and $\mathbf v(y,z,t)=yzty$.
Hence $\mathbf v\in xyzxx^+ty\cup yxzxx^+ty$.
However, $\mathbf M_{\gamma^\prime}(yxx^+ty)$ violates the identity~\eqref{xyzxxty=yxzxxty} because this identity together with $x^2\approx x^3$ imply the identity~\eqref{yxxty=xyxxty}.
Evidently, $\mathbf M_{\gamma^\prime}([x^2yzx^2ty]^{\gamma^\prime})$ does not satisfy~\eqref{xyzxxty=yxzxxty} as well.
Therefore, $\mathbf v\in xyzxx^+ty$.
We see that the $\lambda^\prime$-class $xyzxx^+ty$ is stable with respect to $\mathbf M_\lambda(xyx^+)\vee\mathbf V$.
Now Corollary~\ref{C: M_alpha(W) in V}(v) applies, yielding that 
\[
\left(\mathbf M_\lambda(xyx^+)\vee\mathbf V\right)\wedge \mathbf M_{\lambda^\prime}(xyzxx^+ty)=\mathbf M_{\lambda^\prime}(xyzxx^+ty).
\]
Further, $\mathbf V\wedge\mathbf M_{\lambda^\prime}(xyzxx^+ty)$ satisfies the identities
\[
xyzx^2ty\stackrel{\mathbf V}\approx x^2yzxty\stackrel{\mathbf M_{\lambda^\prime}(xyzxx^+ty)}\approx yx^2zxty \stackrel{\mathbf V}\approx yxzx^2ty.
\]
Clearly, the identity~\eqref{xyzxxty=yxzxxty} holds in $\mathbf M_\lambda(xyx^+)$ as well.
Since $\mathbf M_\lambda(xyx^+)\subset \mathbf M_{\lambda^\prime}(xyzxx^+ty)$, this implies that
\[
\mathbf M_\lambda(xyx^+)\vee\left(\mathbf V\wedge\mathbf M_{\lambda^\prime}(xyzytxx^+)\right)\subset\mathbf M_{\lambda^\prime}(xyzytxx^+).
\]
The lattices $\mathfrak L\left(\mathbf M_\lambda(xyx^+)\vee\mathbf M_{\gamma^\prime}(yxx^+ty)\right)$ and $\mathfrak L\left(\mathbf M_\lambda(xyx^+)\vee\mathbf M_{\gamma^\prime}([x^2yzx^2ty]^{\gamma^\prime})\right)$ are thus not modular.

\smallskip

(vii),(x) Let $\mathbf V\in\{\mathbf M_{\gamma^\prime}(ytyxx^+),\mathbf M_{\gamma^\prime}([yzx^2ytx^2]^{\gamma^\prime})\}$.
Consider an arbitrary identity $\mathbf u\approx \mathbf v$ of $\mathbf M_\lambda(xyx^+)\vee\mathbf V$ with $\mathbf u\in yzyxtxx^+$.
Since $\mathbf u(x,z,t)\in zxtxx^+$ and $\mathbf u(y,z,t)=yzyt$, Lemmas~\ref{L: M(W) in V} and~\ref{L: M_alpha(W) in V} imply that $\mathbf v(x,z,t)\in zxtxx^+$ and $\mathbf v(y,z,t)=yzyt$.
Hence $\mathbf v\in yzyxtxx^+\cup yzxytxx^+$.
However, $\mathbf M_{\gamma^\prime}(ytyxx^+)$ violates the identity~\eqref{yzxytxx=yzyxtxx} because this identity together with $x^2\approx x^3$ imply the identity~\eqref{ytyxx=ytxyxx}.
Evidently, $\mathbf M_{\gamma^\prime}([yzx^2ytx^2]^{\gamma^\prime})$ does not satisfy~\eqref{yzxytxx=yzyxtxx} as well.
Therefore, $\mathbf v\in yzyxtxx^+$.
We see that the $\lambda^\prime$-class $yzyxtxx^+$ is stable with respect to $\mathbf M_\lambda(xyx^+)\vee\mathbf V$.
Now Corollary~\ref{C: M_alpha(W) in V}(v) applies, yielding that 
\[
\left(\mathbf M_\lambda(xyx^+)\vee\mathbf V\right)\wedge \mathbf M_{\lambda^\prime}(yzyxtxx^+)=\mathbf M_{\lambda^\prime}(yzyxtxx^+).
\]
Further, $\mathbf V\wedge \mathbf M_{\lambda^\prime}(yzyxtxx^+)$ satisfies the identities
\[
yzyxtx^2\stackrel{\mathbf V}\approx yzyx^2tx\stackrel{\mathbf M_{\lambda^\prime}(yzyxtxx^+)}\approx yzx^2ytx \stackrel{\mathbf V}\approx yzxytx^2.
\]
Clearly, the identity~\eqref{yzxytxx=yzyxtxx} holds in $\mathbf M_\lambda(xyx^+)$ as well.
Since $\mathbf M_\lambda(xyx^+)\subset \mathbf M_{\lambda^\prime}(yzyxtxx^+)$, this implies that
\[
\mathbf M_\lambda(xyx^+)\vee\left(\mathbf V\wedge\mathbf M_{\lambda^\prime}(yzyxtxx^+)\right)\subset\mathbf M_{\lambda^\prime}(yzyxtxx^+).
\]
The lattices $\mathfrak L\left(\mathbf M_\lambda(xyx^+)\vee\mathbf M_{\gamma^\prime}(ytyxx^+)\right)$ and $\mathfrak L\left(\mathbf M_\lambda(xyx^+)\vee\mathbf M_{\gamma^\prime}([yzx^2ytx^2]^{\gamma^\prime})\right)$ are thus not modular.

\smallskip

(xi) Consider an arbitrary identity $\mathbf u\approx \mathbf v$ of $\mathbf M_\lambda(xyx^+)\vee\mathbf M_{\gamma^\prime}([x^2zytx^2y]^{\gamma^\prime})$ with $\mathbf u\in xzytxx^+y$.
Since $\mathbf u(x,z,t)\in xztxx^+$ and $\mathbf u(y,z,t)=zyty$, Lemmas~\ref{L: M(W) in V} and~\ref{L: M_alpha(W) in V} imply that $\mathbf v(x,z,t)\in xztxx^+$ and $\mathbf v(y,z,t)=zyty$.
Hence $\mathbf v\in xzytxx^+y\cup xzytyxx^+\cup xzytx^+yx^+$.
However, $\mathbf M_{\gamma^\prime}([x^2zytx^2y]^{\gamma^\prime})$ violates the identities 
\begin{align}
\label{xzytxxy=xzytyxx}
xzytx^2y&{}\approx xzytyx^2,\\
\label{xzytxxy=xzytxyx}
xzytx^2y&{}\approx xzytxyx.
\end{align}
Therefore, $\mathbf v\in xzytxx^+y$.
We see that the $\lambda^\prime$-class $xzytxx^+y$ is stable with respect to $\mathbf M_\lambda(xyx^+)\vee\mathbf M_{\gamma^\prime}([x^2zytx^2y]^{\gamma^\prime})$.
Now Corollary~\ref{C: M_alpha(W) in V}(v) applies, yielding that 
\[
\left(\mathbf M_\lambda(xyx^+)\vee \mathbf M_{\gamma^\prime}([x^2zytx^2y]^{\gamma^\prime})\right)\wedge \mathbf M_{\lambda^\prime}(xzytxx^+y)=\mathbf M_{\lambda^\prime}(xzytxx^+y).
\]
Further, $\mathbf M_{\gamma^\prime}([x^2zytx^2y]^{\gamma^\prime})\wedge \mathbf M_{\lambda^\prime}(xzytxx^+y)$ satisfies the identities
\[
xzytx^2y\stackrel{\mathbf M_{\gamma^\prime}([x^2zytx^2y]^{\gamma^\prime})}\approx x^2zytx^2y\stackrel{\mathbf M_{\lambda^\prime}(xzytxx^+y)}\approx x^2zytyx^2 \stackrel{\mathbf M_{\gamma^\prime}([x^2zytx^2y]^{\gamma^\prime})}\approx xzytyx^2.
\]
Clearly, the identity~\eqref{xzytxxy=xzytyxx} holds in $\mathbf M_\lambda(xyx^+)$ as well.
Since $\mathbf M_\lambda(xyx^+)\subset \mathbf M_{\lambda^\prime}(xzytxx^+y)$, this implies that
\[
\mathbf M_\lambda(xyx^+)\vee\left(\mathbf M_{\gamma^\prime}([x^2zytx^2y]^{\gamma^\prime})\wedge\mathbf M_{\lambda^\prime}(xzytxx^+y)\right)\subset\mathbf M_{\lambda^\prime}(xzytxx^+y).
\]
The lattice $\mathfrak L\left(\mathbf M_\lambda(xyx^+)\vee\mathbf M_{\gamma^\prime}([x^2zytx^2y]^{\gamma^\prime})\right)$ is thus not modular.

\smallskip

(xii),(xiii) Let $\mathbf V\in\{\mathbf H,\mathbf M_\lambda(xyzx^+ty^+)\}$. 
In view of Lemma~\ref{L: identities of M(xt_1x...t_kx)}, the set $\{xyzxty,yxzxty\}$ forms a $\FIC\left(\mathbf M(xyx)\right)$-class. 
It follows from Proposition~6.12 in~\cite{Gusev-Vernikov-18} that $\mathbf H$ does not satisfy the identity $\sigma_1$.
Clearly, $\mathbf M_\lambda(xyzx^+ty^+)$ violates $\sigma_1$ as well. 
Hence the word $xyzxty$ is an isoterm for $\mathbf M(xyx)\vee\mathbf V$ in any case.
Now Lemma~\ref{L: M(W) in V} applies, yielding that 
\[
\left(\mathbf M(xyx)\vee\mathbf V\right)\wedge \mathbf M(xyzxty)=\mathbf M(xyzxty).
\]
Further, in view of~\cite[Fig.~1]{Jackson-05} and Lemmas~\ref{L: L(D1)} and~\ref{L: L(M(xzyx^+ty^+))},
\[
\begin{aligned}
\mathbf M(xyx)\vee\left(\mathbf V\wedge\mathbf M(xyzxty)\right)=\mathbf M(xyx)\vee\mathbf M(xy)=\mathbf M(xyx)\subset\mathbf M(xyzxty).
\end{aligned}
\]
Therefore, the lattices $\mathfrak L\left(\mathbf M(xyx)\vee\mathbf H\right)$ and $\mathfrak L\left(\mathbf M(xyx)\vee\mathbf M_\lambda(xyzx^+ty^+)\right)$ are not modular.

\smallskip

(xiv) The proof is quite similar to the proof of Parts~(xii) and~(xiii).

\smallskip

(xv),(xvi),(xviii) Let 
\[
\mathbf V\in\{\mathbf M_\nu([yx^2ty]^\nu),\mathbf M_{\gamma^\prime}(yxx^+ty),\mathbf M_{\lambda^\prime}(xyzxx^+ty)\}.
\]
It follows from Lemma~\ref{L: identities of M(xt_1x...t_kx)} that the set $\{xyzxtxsy,yxzxtxsy\}$ forms a $\FIC\left(\mathbf M(xyxzx)\right)$-class.   
Clearly, the variety $\mathbf V$ violates the identity $xyzxtxsy\approx yxzxtxsy$. 
Hence the word $xyzxtxsy$ is an isoterm for $\mathbf M(xyxzx)\vee\mathbf V$.
Now Lemma~\ref{L: M(W) in V} applies, yielding that 
\[
\left(\mathbf M(xyxzx)\vee\mathbf V\right)\wedge \mathbf M(xyzxtxsy)=\mathbf M(xyzxtxsy).
\]
Further, the variety $\mathbf V\wedge\mathbf M(xyzxtxsy)$ satisfies the identities
\[
xyzxtxsy\stackrel{\mathbf V}\approx xyzxtx^2sy\stackrel{\mathbf M(xyzxtxsy)}\approx yxzxtx^2sy \stackrel{\mathbf V}\approx yxzxtxsy.
\]
Evidently, the identity $xyzxtxsy\approx yxzxtxsy$ holds in $\mathbf M(xyxzx)$ as well.
Since $\mathbf M(xyxzx)\subset \mathbf M(xyzxtxsy)$, this implies that
\[
\mathbf M(xyxzx)\vee\left(\mathbf V\wedge\mathbf M(xyzxtxsy)\right)\subset\mathbf M(xyzxtxsy).
\]
Therefore, the lattice $\mathfrak L\left(\mathbf M(xyxzx)\vee\mathbf V\right)$ is not modular.

\smallskip

(xvii),(xix),(xx) The proof is quite similar to the proof of Parts~(xv),~(xvi) and~(xviii).

\smallskip

(xxi),(xxii) Let $\mathbf V\in\{\mathbf M_{\lambda^\prime}(xyzx^+tysx^+),\mathbf M_\eta([xyzx^2ty]^\eta)\}$.
Consider an arbitrary identity $\mathbf u\approx \mathbf v$ of $\mathbf M_\mu(xyxzx^+)\vee\mathbf V$ with $\mathbf u\in xyzxtysx^+$.
Since $\mathbf u(x,z,t,s)\in xzxtsx^+$ and $\mathbf u(y,z,t,s)=yztys$, Lemmas~\ref{L: M(W) in V} and~\ref{L: M_alpha(W) in V} imply that $\mathbf v(x,z,t,s)\in xzxtsx^+$ and $\mathbf v(y,z,t,s)=yztys$.
Hence $\mathbf v\in xyzxtysx^+\cup yxzxtysx^+$.
However, the variety $\mathbf V$ violates the identity 
\begin{equation}
\label{xyzxtysx=yxzxtysx}
xyzxtysx\approx yxzxtysx.
\end{equation}
Therefore, $\mathbf v\in xyzxtysx^+$.
We see that the $\mu$-class $xyzxtysx^+$ is stable with respect to $\mathbf M_\mu(xyxzx^+)\vee\mathbf V$.
One can easily deduce from this fact and Example~\ref{E: xyzxtysx^+} that every $\mu$-class in $\{xyzxtysx^+\}^{\le_\mu}$ is stable with respect to $\mathbf M_\mu(xyxzx^+)\vee\mathbf V$.
Now Lemma~\ref{L: M_alpha(W) in V} applies, yielding that 
\[
\left(\mathbf M_\mu(xyxzx^+)\vee\mathbf V\right)\wedge \mathbf M_\mu(xyzxtysx^+)=\mathbf M_\mu(xyzxtysx^+).
\]
Further, $\mathbf V\wedge \mathbf M_\mu(xyzxtysx^+)$ satisfies the identities
\[
xyzxtysx\stackrel{\mathbf V}\approx xyzx^2tysx\stackrel{\mathbf M_\mu(xyzxtysx^+)}\approx yxzx^2tysx \stackrel{\mathbf V}\approx yxzxtysx.
\]
Evidently, the identity~\eqref{xyzxtysx=yxzxtysx} holds in $\mathbf M_\mu(xyxzx^+)$ as well.
Then since $\mathbf M_\mu(xyxzx^+)\subset \mathbf M_\mu(xyzxtysx^+)$, this implies that
\[
\mathbf M_\mu(xyxzx^+)\vee\left(\mathbf V\wedge \mathbf M_\mu(xyzxtysx^+)\right)\subset\mathbf M_\mu(xyzxtysx^+).
\]
The lattices $\mathfrak L\left(\mathbf M_\mu(xyxzx^+)\vee\mathbf M_{\lambda^\prime}(xyzx^+tysx^+)\right)$ and $\mathfrak L\left(\mathbf M_\mu(xyxzx^+)\vee\mathbf M_\eta([xyzx^2ty]^\eta)\right)$ are thus not modular.

\smallskip

(xxiii),(xxiv) The required claims were established in~\cite[Proposition~3.2]{Gusev-23} and~\cite[Theorem~1.1]{Gusev-19}, respectively.
\end{proof}

\subsection{Infinite series of varieties induced by $\mathbf a_{n,m}[\rho]$}
\label{Subsec: infinite series a_{n,m}[rho]}

\begin{proposition}
\label{P: non-dis L(M([hat{a}_{0,k}[pi]]^lambda))}
The lattices $\mathfrak L\left(\mathbf M_\lambda([\hat{\mathbf a}_{0,k}[\pi]]^\lambda)\right)$ is not distributive for any $k\in\mathbb N$ and $\pi\in S_k$.
\end{proposition}

To prove Proposition~\ref{P: non-dis L(M([hat{a}_{0,k}[pi]]^lambda))}, we need one auxiliary result.

\begin{lemma}
\label{L: stable M_lambda(hat{a}_{0,k}[pi])}
If $k,n\in\mathbb N$, $\pi\in S_k$ and $\rho\in S_n$, then the $\lambda$-class $xzyx^+ty^+$ and so the $\lambda$-classes $xyx^+$, $xx^+y$, $yxx^+ty^+$, $[(\hat{\mathbf a}_{0,n}[\rho])_x]^\lambda$ are stable with respect to $\mathbf M_\lambda([\hat{\mathbf a}_{0,k}[\pi]]^\lambda)$.
\end{lemma}

\begin{proof}
For brevity, let $\mathtt a_r[\tau]:=[\hat{\mathbf a}_{0,r}[\tau]]^\lambda$ for any $r\in\mathbb N$ and $\tau\in S_r$.
Evidently, $xy$ is an isoterm for $\mathbf M_\lambda(\mathtt a_k[\pi])$.
Notice also that $\mathbf M_\lambda(\mathtt a_k[\pi])$ satisfies the identity~\eqref{xyx=xyxx}.
If the set $xyx^+$ is not stable with respect to a variety satisfying~\eqref{xyx=xyxx}, then it follows from Lemmas~\ref{L: M_alpha(W) in V} and~\ref{L: nsub M(xyx^+)} that this variety satisfies the identity
\begin{equation}
\label{xyx=xxyx}
xyx\approx x^2yx.
\end{equation} 
Since 
\[
\hat{\mathbf a}_{0,k}[\pi]\stackrel{\eqref{xyx=xxyx}}\approx x^2z_{1\pi}\cdots z_{k\pi}xt_1z_1\cdots t_nz_k,
\] 
we see that the $\lambda$-class $xyx^+$ is stable with respect to $\mathbf M_\lambda(\mathtt a_k[\pi])$.
By similar arguments, we can show that the $\lambda$-class $xx^+y$ is stable with respect to $\mathbf M_\lambda(\mathtt a_k[\pi])$.

Now, take an arbitrary identity $\mathbf u \approx \mathbf v$ of $\mathbf M_\lambda(\mathtt a_k[\pi])$ with $\mathbf u\in xzyx^+ty^+$.
Since $xyx^+$ and $xx^+y$ are stable with respect to $\mathbf M_\lambda(\mathtt a_k[\pi])$, we have $\mathbf v\in xzyx^+ty^+\cup xzx^+y x^\ast ty^+$.
If $\mathbf v\in xzx^+y x^\ast ty^+$, then $\mathbf u\approx \mathbf v$ together with~\eqref{xyx=xyxx} imply~\eqref{xzyxty=xzxyxty}.
Since 
\[
\hat{\mathbf a}_{0,k}[\pi]\stackrel{\eqref{xzyxty=xzxyxty}}\approx xz_{1\pi}\cdots xz_{k\pi}xt_1z_1\cdots t_nz_k,
\] 
we see that the $\lambda$-class $xzyx^+ty^+$ must be stable with respect to $\mathbf M_\lambda(\mathtt a_k[\pi])$ by Lemma~\ref{L: M_alpha(W) in V}.
Then Lemmas~\ref{L: M_alpha(W) in V} and~\ref{L: L(M(xzyx^+ty^+))} imply that the $\lambda$-classes $yxx^+ty^+$ and $xyzx^+ty^+$ are stable with respect to $\mathbf M_\lambda(\mathtt a_k[\pi])$ as well.

Finally, take an arbitrary identity $\mathbf u \approx \mathbf v$ of $\mathbf M_\lambda(\mathtt a_k[\pi])$ with $\mathbf u\in [([\hat{\mathbf a}_{0,n}[\rho])_x]^\lambda$.
Since $xyzx^+ty^+$ is stable with respect to $\mathbf M_\lambda(\mathtt a_k[\pi])$, we have $\mathbf v(z_i,z_j,t_i,z_j)\in z_{i\rho}z_{j\rho}t_iz_i^+t_jz_j^+$ for any $1\le i<j\le n$.
Hence $\mathbf v\in [([\hat{\mathbf a}_{0,n}[\rho])_x]^\lambda$.
Since the identity $\mathbf u \approx \mathbf v$ is arbitrary, the $\lambda$-class $[([\hat{\mathbf a}_{0,n}[\rho])_x]^\lambda$ is also stable with respect to $\mathbf M_\lambda(\mathtt a_k[\pi])$.
\end{proof}

For any non-empty word $\mathbf w$ of length $\ell$, $0\le k \le \ell$ and $0\le m\le\ell-k$, let $\mathbf w[k;m]$ denote a factor of $\mathbf w$ of length $m$ directly succeeding the prefix of $\mathbf w$ of length $k$.
For any word $\mathbf w$, let $\ini_2(\mathbf w)$ denote the word obtained from $\mathbf w$ by retaining the first and second occurrences of each letter.

\begin{proof}[Proof of Proposition~\ref{P: non-dis L(M([hat{a}_{0,k}[pi]]^lambda))}]
For any $r\in\mathbb N$ and $\tau\in S_r$, let $\mathtt a_r[\tau]:=[\hat{\mathbf a}_{0,r}[\tau]]^\lambda$.
For brevity, put $n:=k+1$.
Define permutation $\rho\in S_n$ as follows:
\begin{equation}
\label{rho= a_{o,n}[rho]}
i\rho := 
\begin{cases} 
n & \text{if $i=1$}, \\
(i-1)\pi & \text{if $1<i\le n$}.
\end{cases} 
\end{equation}
Let $\mathbf a\approx \mathbf a^\prime$ be an identity of $\mathbf M_\lambda(\mathtt a_k[\pi])$ with $\mathbf a\in \mathtt a_n[\rho]$.
In view of Lemma~\ref{L: stable M_lambda(hat{a}_{0,k}[pi])}, the following $\lambda$-classes are stable with respect to $\mathbf M_\lambda(\mathtt a_k[\pi])$: $[(\hat{\mathbf a}_{0,n}[\rho])_x]^\lambda$, $xx^+y$ and $yxx^+ty^+$.
Since $[(\hat{\mathbf a}_{0,n}[\rho])_x]^\lambda$ is stable with respect to $\mathbf M_\lambda(\mathtt a_k[\pi])$, we have $\mathbf a_x^\prime\in [(\hat{\mathbf a}_{0,n}[\rho])_x]^\lambda$.
The fact that $xx^+y$ is stable with respect to $\mathbf M_\lambda(\mathtt a_k[\pi])$ implies that $({_{\ell\mathbf a^\prime}}x)<({_{\ell\mathbf a^\prime}}t_1)$.
Further, if $({_{1\mathbf a^\prime}}z_n)<({_{1\mathbf a^\prime}}x)$, then $\mathbf a^\prime(x,z_n,t_n)\in z_nxx^+t_nz_n^+$, while $\mathbf a(x,z_n,t_n)\in xz_nx^+t_nz_n^+$, contradicting the fact that $yxx^+ty^+$ is stable with respect to $\mathbf M_\lambda(\mathtt a_k[\pi])$.
Therefore, $({_{1\mathbf a^\prime}}x)<({_{1\mathbf a^\prime}}z_n)$.
Finally, since $\mathbf a_{\{z_n,t_n\}}\in \mathtt a_k[\pi]$, Lemma~\ref{L: M_alpha(W) in V} implies that $\mathbf a_{\{z_n,t_n\}}^\prime\in \mathtt a_k[\pi]$.
Hence $\mathbf a^\prime\in \mathtt a_n[\rho]$.
We see that the $\lambda$-class $\mathtt a_n[\rho]$ is stable with respect to $\mathbf M_\lambda(\mathtt a_k[\pi])$ and, therefore, $M_\lambda(\mathtt a_n[\rho])\in\mathbf M_\lambda(\mathtt a_k[\pi])$ by Corollary~\ref{C: M_alpha(W) in V}(ii).
Thus, it suffices to verify that the lattice $\mathfrak L\left(\mathbf M_\lambda(\mathtt a_n[\rho])\right)$ is not distributive.

Now, define the relation $\delta$ on $\mathfrak X^\ast$ as follows: for every $\mathbf u,\mathbf v \in \mathfrak X^\ast$, $\mathbf u\mathrel{\delta}\mathbf v$ if and only if
\begin{itemize}
\item $\mathbf u(\simple(\mathbf u))=\mathbf v(\simple(\mathbf u))$;
\item if $\mathbf u=\mathbf u_0t_1\mathbf u_1\cdots t_m\mathbf u_m$ and $\mathbf v=\mathbf v_0t_1\mathbf v_1\cdots t_m\mathbf v_m$ are the decompositions of $\mathbf u$ and $\mathbf v$, respectively, then $\con(\mathbf u_i)=\con(\mathbf v_i)$ for all $i=0,\dots,m$;
\item $\ini_2(\mathbf u)=\ini_2(\mathbf v)$.
\end{itemize}
It is routine to check that $\delta$ is, in fact, a congruence on $\mathfrak X^\ast$.
Let now
\[
\begin{aligned}
&\mathbf v_0 := x\,\mathbf q[0;3n-3]\,y\,\mathbf q[3n-3;2]\,x\,\mathbf q[3n-1;3n+1]\,y\,\mathbf r,\\
&\mathbf v_1 := x\,\mathbf q[0;3n-3]\,y\,\mathbf q[3n-3;1]\,x\,\mathbf q[3n-2;3n+2]\,y\,\mathbf r,\\
&\mathbf v_2 := x\,\mathbf q[0;3n-3]\,y\,\mathbf q[3n-3;2]\,x\,\mathbf q[3n-1;3n]\,y\,\mathbf q[6n-1;1]\,\mathbf r,
\end{aligned}
\]
where
\[
\mathbf q := \biggl(\prod_{i=1}^n  z_{i\rho}^{(1)}z_{i\rho}^{(2)}z_{i\rho}^{(3)}\biggr) \biggl(\prod_{i=1}^n z_{i\rho}^{(4)}z_{i\rho}^{(5)}z_{i\rho}^{(6)}\biggr),\ \mathbf r := \biggl(\prod_{j=1}^3\biggl(\prod_{i=1}^n  t_i^{(j)}z_i^{(j)}\biggr)\biggr)\biggl(\prod_{j=6}^4\biggl(\prod_{i=1}^n  t_i^{(j)}z_i^{(j)}\biggr)\biggr).
\]
Let $\mathtt v_j:=[\mathbf v_j]^\delta$, $j=0,1,2$.
Since $\mathbf M_\lambda(\mathtt a_n[\rho])$ satisfies the identities~\eqref{xxyty=xxyxty} and~\eqref{xyx=xyxx}, all the words in $\mathtt v_j$, $j=0,1,2$, are $\FIC(\mathbf M_\lambda(\mathtt a_n[\rho]))$-related.

Let $\mathbf u\approx \mathbf v$ be an identity of $\mathbf M_\lambda(\mathtt a_n[\rho])$ with $\mathbf u \in \mathtt v_0$.
It follows from Lemma~\ref{L: stable M_lambda(hat{a}_{0,k}[pi])} that $[\mathbf q\mathbf r]^\lambda$ is stable with respect to $\mathbf M_\lambda(\mathtt a_n[\rho])$.
Hence $\mathbf v_{\{x,y\}}\in[\mathbf q\mathbf r]^\lambda$.
Since $yxx^+ty^+$ is stable with respect to $\mathbf M_\lambda(\mathtt a_n[\rho])$ by Lemma~\ref{L: stable M_lambda(hat{a}_{0,k}[pi])}, we have $({_{1\mathbf v}}x)<({_{1\mathbf v}}z_{1\rho}^{(1)})$ and $({_{1\mathbf v}}z_{(n-1)\rho}^{(3)})<({_{1\mathbf v}}y)<({_{1\mathbf v}}z_{n\rho}^{(1)})$. 
Further, if $\phi\colon \mathfrak X \to \mathfrak X^\ast$ is the substitution given by
\[ 
\phi(v) := 
\begin{cases} 
x & \text{if }v=x, \\ 
z_i & \text{if }v=z_i^{(2)},\ i=1,\dots,n,\\
t_i & \text{if }v=t_i^{(2)},\ i=1,\dots,n,\\
1 & \text{otherwise},
\end{cases} 
\]
then $\phi(\mathbf u)\in\mathtt a_n[\rho]$.
Then $\phi(\mathbf v)\in \mathtt a_n[\rho]$ by Lemma~\ref{L: M_alpha(W) in V}. 
Hence $({_{1\mathbf v}}z_{n\rho}^{(2)})<({_{2\mathbf v}}x)$.
By a similar argument we can show that $({_{2\mathbf v}}x)<({_{1\mathbf v}}z_{n\rho}^{(3)})$ and $({_{1\mathbf v}}z_{n\rho}^{(6)})<({_{2\mathbf v}}y)$.
Finally, $({_{\ell\mathbf v}}x)<({_{1\mathbf v}}t_1^{(1)})$ and $({_{\ell\mathbf v}}y)<({_{1\mathbf v}}t_1^{(1)})$ because $xx^+y$ is stable with respect to $\mathbf M_\lambda(\mathtt a_n[\rho])$ by Lemma~\ref{L: stable M_lambda(hat{a}_{0,k}[pi])}.
We see that $\mathbf v\in\mathtt v_0$ and thus the $\delta$-class $\mathtt v_0$ is a $\FIC(\mathbf M_\lambda(\mathtt a_n[\rho]))$-class.
By similar arguments we can show that the $\delta$-classes $\mathtt v_1$ and $\mathtt v_2$ are $\FIC(\mathbf M_\lambda(\mathtt a_n[\rho]))$-classes.

Let 
\[
\begin{aligned}
&\mathbf X := \mathbf M_\lambda(\mathtt a_n[\rho])\wedge\var\{\mathbf v_0 \approx \mathbf v_1\},\\ 
&\mathbf Y := \mathbf M_\lambda(\mathtt a_n[\rho])\wedge\var\{\mathbf v_0 \approx \mathbf v_2\},\\
&\mathbf Z := \mathbf M_\lambda(\mathtt a_n[\rho])\wedge\var\{\mathbf v_1\approx\mathbf v_3, \mathbf v_2\approx \mathbf v_3\},
\end{aligned}
\]
where
\[
\mathbf v_3 := x^2\,\mathbf q[0;3n-3]\,y^2\,\mathbf q[3n-3;3n+3]\,\mathbf r.
\]
It is routine to check that one can take a sufficiently large integer $r$, say $r>10n$, such that   the identities $\mathbf v_0\approx\mathbf v_1\approx \mathbf v_2\approx \mathbf v_3$ hold in $\mathbf M_\lambda(\mathtt a_r[\theta])$ for any $\theta\in S_r$.
Arguments similar to ones from the first paragraph of this proof imply that there is $\theta^\prime\in S_r$ such that $M_\lambda(\mathtt a_r[\theta^\prime])\in \mathbf M_\lambda(\mathtt a_n[\rho])$.
Hence $M_\lambda(\mathtt a_r[\theta^\prime])\in\mathbf X\wedge\mathbf Y\wedge\mathbf Z$.

\smallskip

\textbf{1. The set $\mathtt v_0\cup \mathtt v_1$ forms a $\FIC(\mathbf X)$-class.} 
Consider an arbitrary identity $\mathbf u \approx \mathbf u^\prime$ of $\mathbf X$ with $\mathbf u\in \mathtt v_0\cup \mathtt v_1$.
We are going to show that $\mathbf u^\prime \in \mathtt v_0\cup \mathtt v_1$.
Since $M_\lambda(\mathtt a_r[\theta^\prime])\in\mathbf X$, Lemma~\ref{L: stable M_lambda(hat{a}_{0,k}[pi])} implies that $\mathbf u^\prime_{\{x,y\}}\in [\mathbf q \mathbf r]^\lambda$.
In view of Proposition~\ref{P: deduction}, we may assume without loss of generality that  either $\mathbf u \approx \mathbf u^\prime$ holds in $\mathbf M_\lambda(\mathtt a_n[\rho])$ or $\mathbf u \approx \mathbf u^\prime$ is directly deducible from $\mathbf v_0 \approx \mathbf v_1$. 
In view of the above, $\mathtt v_0\cup \mathtt v_1$ is a union of two $\FIC(\mathbf M_\lambda(\mathtt a_n[\rho]))$-classes.
Therefore, it remains to consider the case when $\mathbf u \approx \mathbf u^\prime$ is directly deducible from $\mathbf v_0 \approx \mathbf v_1$, i.e., there exist some words $\mathbf a,\mathbf b \in \mathfrak X^\ast$ and substitution $\phi\colon \mathfrak X \to \mathfrak X^\ast$ such that $(\mathbf u,\mathbf u^\prime ) = (\mathbf a\phi(\mathbf v_s)\mathbf b,\mathbf a\phi(\mathbf v_t)\mathbf b)$, where $\{s,t\}=\{0,1\}$.

First, notice that if $a$ and $b$ are distinct letters with $\{a,b\}\ne\{x,y\}$, then $ab$ occurs in $\mathbf u$ as a factor at most once.
It follows that 
\begin{itemize}
\item[\textup{($\ast$)}] for any $v\in \mul(\mathbf v_s)=\mul(\mathbf v_t)$, if $\con(\phi(v))\ne \{x,y\}$, then $\phi(v)$ is either empty or a power of letter.
\end{itemize}
If $\phi(x)=1$, then $\phi(\mathbf v_s)=\phi(\mathbf v_t)$ and so $\mathbf u=\mathbf u^\prime$, whence $\mathbf u^\prime \in \mathtt v_0\cup \mathtt v_1$.
It remains to consider the case when $\phi(x)\ne1$.
Clearly, the letters in $\{t_i^{(j)}\mid 1\le i\le n,\, 1\le j\le 6\}$ do not occur in $\phi(x)$ because these letters are simple in $\mathbf u$, while $x\in\mul(\mathbf v_s)=\mul(\mathbf v_t)$.
If $\phi(x)$ is not a power of letter, then $\con(\phi(x))=\{x,y\}$ by~($\ast$) and, therefore, all the factors of the form $\phi(x)$ in $\mathbf u$ are preceded by the letter ${_{1\mathbf u}}z_{n\rho}^{(6)}$.
In this case, the application of $\mathbf v_s\approx \mathbf v_t$ to $\mathbf u$ does not change the factor of $\mathbf u$ preceding ${_{1\mathbf u}}z_{n\rho}^{(6)}$, whence $\mathbf u^\prime\in \mathtt v_0\cup \mathtt v_1$. 
It remains to consider the case when $\phi(x)$ is a power of letter.
If $\phi(x)$ is a power of a letter in $\{z_i^{(j)}\mid 1\le i\le n,\, 1\le j\le 6\}$, then all the factors of the form $\phi(x)$ in $\mathbf u$ must be preceded by the letter $t_1^{(1)}$.
In this case, the application of $\mathbf v_s\approx \mathbf v_t$ to $\mathbf u$ does not change the factor of $\mathbf u$ preceding $t_1^{(1)}$. 
Hence $\mathbf u^\prime\in \mathtt v_0\cup \mathtt v_1$.
Therefore, we may further assume that $\phi(x)$ is a power of letter in $\{x,y\}$.
For convenience, denote this letter by $c$.
If the image of ${_{2\mathbf v_s}}x$ under $\phi$ does not contain ${_{2\mathbf u}}c$, then the application of $\mathbf v_s\approx \mathbf v_t$ to $\mathbf u$ cannot change the position of the second occurrence of $c$ in $\mathbf u$ and, therefore, $\mathbf u^\prime\in \mathtt v_0\cup \mathtt v_1$.
Therefore, $\phi({_{2\mathbf v_s}}x)$ contains ${_{2\mathbf u}}c$ in $\mathbf u$.
This is only possible when $\phi(x)=c$ and, moreover, $\phi({_{1\mathbf v_s}}x)={_{1\mathbf u}}c$ and $\phi({_{2\mathbf v_s}}x)={_{2\mathbf u}}c$.

The content of the factor of $\mathbf u$ between ${_{1\mathbf u}}y$ and ${_{2\mathbf u}}y$ consists of $3n+4$ distinct letters.
However, the factor of $\mathbf v_s$ between ${_{1\mathbf v_s}}x$ and ${_{2\mathbf v_s}}x$ contains at most $3n$ distinct letters.
These facts and~($\ast$) imply that $\phi(x)$ must not coincide with $y$.
Therefore, $\phi({_{1\mathbf v_s}}x)={_{1\mathbf u}}x$ and $\phi({_{2\mathbf v_s}}x)={_{2\mathbf u}}x$.
Then the factor of $\mathbf u$ lying between ${_{1\mathbf u}}x$ and ${_{2\mathbf u}}x$ must be the image of the factor of $\mathbf v_s$ lying between ${_{1\mathbf v_s}}x$ and ${_{2\mathbf v_s}}x$.
In other words, 
\[
\mathbf q[0;3n-3]\,y\,\mathbf q[3n-3;2-p]=\phi(\mathbf q[0;3n-3]\,y\,\mathbf q[3n-3;2-s]),
\]
where 
\[ 
p := 
\begin{cases} 
0 & \text{if }\mathbf u\in\mathtt v_0, \\ 
1 & \text{if }\mathbf u\in\mathtt v_1.
\end{cases} 
\]
There are four cases.

\smallskip

\textbf{Case 1.1}: $\mathbf u\in \mathtt v_0$ and $(s,t)=(0,1)$.
In this case, $\mathbf u^\prime \in \mathtt v_0\cup \mathtt v_1$ because the image of the letter $z_{n\rho}^{(2)}$ under $\phi$ is either empty or a letter by~($\ast$).

\smallskip

\textbf{Case 1.2}: $\mathbf u\in \mathtt v_0$ and $(s,t)=(1,0)$.
In this case, the content of the factor of $\mathbf u$ between ${_{1\mathbf u}}x$ and ${_{2\mathbf u}}x$ consists of $3n$ distinct letters, while the factor of $\mathbf v_s$ between ${_{1\mathbf v_s}}x$ and ${_{2\mathbf v_s}}x$ contains $3n-1$ distinct letters.
Therefore, the image of some letter in the latter factor under $\phi$ must contain at least two distinct letters, contradicting~($\ast$).
Therefore, this case is impossible.

\smallskip

\textbf{Case 1.3}: $\mathbf u\in \mathtt v_1$ and $(s,t)=(0,1)$.
If $\phi(z_{n\rho}^{(2)})=1$, then $\phi(\mathbf v_s)=\phi(\mathbf v_t)$ and so $\mathbf u^\prime\in \mathtt v_1$.
It remains to consider the case when $\phi(z_{n\rho}^{(2)})\ne1$.
Then $\phi(z_{n\rho}^{(2)})=z_{n\rho}^{(1)}$ by~($\ast$).
Further, it is easy to see that either $\phi(y)=y$ or $\phi(y)=1$ and, therefore, $\phi(d)=d$ for each $d\in\con(\mathbf q[0;3n-3])$.
In particular, since $n\rho<n$, we have $\phi(z_n^{(1)})=z_n^{(1)}$, contradicting
\[
\begin{aligned}
({_{2\mathbf u}}z_{n\rho}^{(1)})<({_{\ell\mathbf u}}z_{n\rho}^{(1)})&{}<({_{2\mathbf u}}z_n^{(1)})<({_{\ell\mathbf u}}z_n^{(1)}),\\
({_{2\mathbf v_s}}z_n^{(1)})<({_{\ell\mathbf v_s}}z_n^{(1)})&{}<({_{2\mathbf v_s}}z_{n\rho}^{(2)})<({_{\ell\mathbf v_s}}z_{n\rho}^{(2)}).
\end{aligned}
\]
Therefore, this case is impossible as well.

\smallskip

\textbf{Case 1.4}: $\mathbf u\in \mathtt v_1$ and $(s,t)=(1,0)$.
In this case, $\mathbf u^\prime \in \mathtt v_0\cup \mathtt v_1$ because the image of the letter $z_{n\rho}^{(2)}$ under $\phi$ is either empty or a letter by~($\ast$).

\smallskip

We see that $\mathbf u^\prime \in \mathtt v_0\cup \mathtt v_1$ in either case.
This means that $\mathtt v_0\cup \mathtt v_1$ forms a $\FIC(\mathbf X)$-class.

\smallskip

\textbf{2. The set $\mathtt v_0\cup \mathtt v_2$ forms a $\FIC(\mathbf Y)$-class.} 
The same arguments as in the proof of the fact that the set $\mathtt v_0\cup\mathtt v_1$ forms a $\FIC(\mathbf X)$-class can show that it suffices to establish that if an identity $\mathbf u\approx \mathbf u^\prime$ with $\mathbf u\in\mathtt v_0\cup \mathtt v_2$ is directly deducible from $\mathbf v_0 \approx \mathbf v_2$, i.e., there exist some words $\mathbf a,\mathbf b \in \mathfrak X^\ast$ and substitution $\phi\colon \mathfrak X \to \mathfrak X^\ast$ such that $(\mathbf u,\mathbf u^\prime ) = (\mathbf a\phi(\mathbf v_s)\mathbf b,\mathbf a\phi(\mathbf v_t)\mathbf b)$, where $\{s,t\}=\{0,2\}$, then $\mathbf u^\prime\in\mathtt v_0\cup \mathtt v_2$.

Further, arguments similar to ones from the proof of the fact that $\mathtt v_0\cup\mathtt v_1$ forms a $\FIC(\mathbf X)$-class show that~($\ast$) is true, and $\mathbf u^\prime\in\mathtt v_0\cup\mathtt v_2$ whenever the following claim does not hold:
\begin{itemize}
\item[\textup{(a)}] $\phi({_{1\mathbf v_s}}y)={_{1\mathbf u}}c$ and $\phi({_{2\mathbf v_s}}y)={_{2\mathbf u}}c$ for some $c\in\{x,y\}$.
\end{itemize}
Thus, it remains to consider the case when~(a) holds.
If $\phi(y)=x$, then the letter ${_{1\mathbf u}}y$ must be the image under $\phi$ of the first occurrence of some letter in $\{z_{(n-1)\rho}^{(4)},z_{(n-1)\rho}^{(5)},z_{(n-1)\rho}^{(6)},z_{n\rho}^{(4)}\}$ by~($\ast$).
This is only possible when the image under $\phi$ of each letter in $\{x,z_i^{(j)},t_i^{(j)}\mid 1\le i\le n,\,1\le j\le 3\}$ is the empty word.
In this case, the image under $\phi$ of the factor of $\mathbf v_s$ between ${_{1\mathbf v_s}}y$ and ${_{2\mathbf v_s}}y$ contains at most $3n$ distinct letters by~($\ast$).
Therefore, $y=\phi(z_{n\rho}^{(4)})$.
Hence $\phi(t_i^{(5)})=\phi(z_i^{(5)})=1$ for any $i=1,\dots,n$.
Then the image under $\phi$ of the factor of $\mathbf v_s$ between ${_{1\mathbf v_s}}y$ and ${_{2\mathbf v_s}}y$ contains at most $2n$ distinct letters by~($\ast$), contradicting the fact that the factor of $\mathbf u$ between ${_{1\mathbf u}}x$ and ${_{2\mathbf u}}x$ contains $3n$ distinct letters.  
Therefore, $\phi({_{1\mathbf v_s}}y)={_{1\mathbf u}}y$ and $\phi({_{2\mathbf v_s}}y)={_{2\mathbf u}}y$.
Then the factor of $\mathbf u$ lying between ${_{1\mathbf u}}y$ and ${_{2\mathbf u}}y$ must be the image of the factor of $\mathbf v_s$ lying between ${_{1\mathbf v_s}}y$ and ${_{2\mathbf v_s}}y$.
There are four cases.

\smallskip

\textbf{Case 2.1}: $\mathbf u\in \mathtt v_0$ and $(s,t)=(0,2)$.
In this case, $\mathbf u^\prime \in \mathtt v_0\cup \mathtt v_2$ because the image of the letter $z_{n\rho}^{(6)}$ under $\phi$ is either empty or a letter by~($\ast$).

\smallskip

\textbf{Case 2.2}: $\mathbf u\in \mathtt v_0$ and $(s,t)=(2,0)$.
In this case, the content of the factor of $\mathbf u$ between ${_{1\mathbf u}}y$ and ${_{2\mathbf u}}y$ consists of $3n+4$ distinct letters, while the factor of $\mathbf v_s$ between ${_{1\mathbf v_s}}y$ and ${_{2\mathbf v_s}}y$ contains $3n+3$ distinct letters.
Therefore, the image of some letter in the latter factor under $\phi$ must contain at least two distinct letters, contradicting~($\ast$).
Thus, this case is impossible.

\smallskip

\textbf{Case 2.3}: $\mathbf u\in \mathtt v_2$ and $(s,t)=(0,2)$.
If $\phi(z_{n\rho}^{(6)})=1$, then $\phi(\mathbf v_s)=\phi(\mathbf v_t)$ and so $\mathbf u^\prime\in \mathtt v_2$.
It remains to consider the case when $\phi(z_{n\rho}^{(6)})\ne1$.
Then $\phi(z_{n\rho}^{(6)})=z_{n\rho}^{(5)}$ by~($\ast$).
Since the content of the factor of $\mathbf u$ between ${_{1\mathbf u}}y$ and ${_{2\mathbf u}}y$ consists of $3n+3$ distinct letters, while the factor of $\mathbf v_s$ between ${_{1\mathbf v_s}}y$ and ${_{2\mathbf v_s}}y$ contains $3n+4$ distinct letters, it follows from~($\ast$) and the fact that $1\rho=n$ that $z_n^{(6)}\in\{\phi(z_n^{(6)}),\phi(z_{2\rho}^{(4)})\}$, contradicting
\[
\begin{aligned}
({_{2\mathbf v_s}}z_{n\rho}^{(6)})<({_{\ell\mathbf v_s}}z_{n\rho}^{(6)})<({_{2\mathbf v_s}}z_n^{(6)})&<({_{\ell\mathbf v_s}}z_n^{(6)})<({_{2\mathbf v_s}}z_{2\rho}^{(4)})<({_{\ell\mathbf v_s}}z_{2\rho}^{(4)}),\\
({_{2\mathbf u}}z_n^{(6)})<({_{\ell\mathbf u}}z_n^{(6)})&<({_{2\mathbf u}}z_{n\rho}^{(5)})<({_{\ell\mathbf u}}z_{n\rho}^{(5)}).
\end{aligned}
\]
Therefore, this case is impossible as well.

\smallskip

\textbf{Case 2.4}: $\mathbf u\in \mathtt v_2$ and $(s,t)=(2,0)$.
In this case, $\mathbf u^\prime \in \mathtt v_0\cup \mathtt v_2$ because the image of the letter $z_{n\rho}^{(6)}$ under $\phi$ is either empty or a letter by~($\ast$).

\smallskip

We see that $\mathbf u^\prime \in \mathtt v_0\cup \mathtt v_2$ in either case.
This means that $\mathtt v_0\cup \mathtt v_2$ forms a $\FIC(\mathbf Y)$-class.

\smallskip

\textbf{3. The set $\mathtt v_0$ forms a $\FIC(\mathbf Z)$-class.} 
The same arguments as in the proof of the fact that the set $\mathtt v_0\cup\mathtt v_1$ forms a $\FIC(\mathbf X)$-class can show that it suffices to establish that if an identity $\mathbf u\approx \mathbf u^\prime$ with $\mathbf u\in\mathtt v_0$ is directly deducible from some identity in $\{\mathbf v_1 \approx \mathbf v_3,\mathbf v_2 \approx \mathbf v_3\}$, i.e., there exist some words $\mathbf a,\mathbf b \in \mathfrak X^\ast$ and substitution $\phi\colon \mathfrak X \to \mathfrak X^\ast$ such that $(\mathbf u,\mathbf u^\prime ) = (\mathbf a\phi(\mathbf v_s)\mathbf b,\mathbf a\phi(\mathbf v_t)\mathbf b)$, where $\{s,t\}$ is either $\{1,3\}$ or $\{2,3\}$, then $\mathbf u^\prime\in\mathtt v_0$.

Further, arguments similar to ones from the proof of the facts that $\mathtt v_0\cup\mathtt v_1$ forms a $\FIC(\mathbf X)$-class and $\mathtt v_0\cup\mathtt v_2$ forms a $\FIC(\mathbf Y)$-class show that~($\ast$) is true and $\mathbf u^\prime\in\mathtt v_0$ whenever the following claim does not hold:
\begin{itemize}
\item[\textup{(b)}] there is a non-empty subset $A$ of $\{x,y\}$ such that for each $c\in A$, $\phi({_{1\mathbf v_s}}c)={_{1\mathbf u}}c$ and $\phi({_{2\mathbf v_s}}c)={_{2\mathbf u}}c$.
\end{itemize}
Thus, it remains to consider the case when~(b) holds.
Clearly, this is only possible when $s\ne 3$. 
Suppose that $x\in A$.
If $(s,t)=(1,3)$, then the content of the factor of $\mathbf u$ between ${_{1\mathbf u}}x$ and ${_{2\mathbf u}}x$ consists of $3n$ distinct letters, while the factor of $\mathbf v_s$ between ${_{1\mathbf v_s}}x$ and ${_{2\mathbf v_s}}x$ contains $3n-1$ distinct letters.
Therefore, the image of some letter in the latter factor under $\phi$ must contain at least two distinct letters, contradicting~($\ast$).
Hence $(s,t)=(2,3)$.
Then we apply~($\ast$) again, yielding that $\phi(y)=y$ and so $y\in A$. 
In this case, the content of the factor of $\mathbf u$ between ${_{1\mathbf u}}y$ and ${_{2\mathbf u}}y$ consists of $3n+4$ distinct letters, while the factor of $\mathbf v_s$ between ${_{1\mathbf v_s}}y$ and ${_{2\mathbf v_s}}y$ contains $3n+3$ distinct letters.
Therefore, the image of some letter in the latter factor under $\phi$ must contain at least two distinct letters, contradicting~($\ast$).
Hence, $x$ cannot belong to $A$.
By a similar argument we can show that $y\notin A$.
Thus, (b) is impossible.
It follows that $\mathbf u^\prime \in \mathtt v_0$ in either case and, therefore, $\mathtt v_0$ forms a $\FIC(\mathbf Z)$-class.

\smallskip

In view of the above, the $\delta$-class $\mathtt v_0$ forms a $\FIC(\mathbf X\vee\mathbf Y)$-class and a $\FIC(\mathbf Z)$-class.
It follows that the $\delta$-class $\mathtt v_0$ is stable with respect to $(\mathbf X\vee \mathbf Y)\wedge \mathbf Z$.
However, $\mathbf X\wedge \mathbf Z$ [respectively, $\mathbf Y\wedge \mathbf Z$] satisfies $\mathbf v_0\approx \mathbf v_1\approx\mathbf v_3$ [respectively, $\mathbf v_0\approx \mathbf v_2\approx\mathbf v_3$].
Therefore, the identity $\mathbf v_0\approx \mathbf v_3$ holds in the variety $(\mathbf X\wedge \mathbf Z)\vee(\mathbf Y\wedge \mathbf Z)$.
This means that the $\delta$-class $\mathtt v_0$ is not stable with respect to this variety, whence
\[
(\mathbf X\wedge \mathbf Z)\vee(\mathbf Y\wedge \mathbf Z)\subset (\mathbf X\vee \mathbf Y)\wedge \mathbf Z.
\]
Thus, we have proved that the lattice $\mathfrak L\left(\mathbf M(\mathtt a_n[\rho])\right)$ is not distributive.
\end{proof}

\begin{proposition}
\label{P: non-dis L(M([a_{0,k}[pi]]^lambda))}
The lattice $\mathfrak L\left(\mathbf M_\lambda([\mathbf a_{0,k}[\pi]]^\lambda)\right)$ is not distributive for any $k\ge2$ and $\pi\in S_k$.
\end{proposition}

To prove Proposition~\ref{P: non-dis L(M([a_{0,k}[pi]]^lambda))}, we need the following auxiliary result.
Its proof is quite similar to the proof of Lemma~\ref{L: stable M_lambda(hat{a}_{0,k}[pi])} and we omit it.

\begin{lemma}
\label{L: stable M_lambda(a_{0,k}[pi])}
If $k,n\ge2$, $\pi\in S_k$ and $\rho\in S_n$, then the $\lambda$-classes $xx^+yty^+$, $xzyx^+ty^+$ and $[(\mathbf a_{0,n}[\rho])_x]^\lambda$ are stable with respect to $\mathbf M_\lambda([\mathbf a_{0,k}[\pi]]^\lambda)$.\qed
\end{lemma}

\begin{proof}[Proof of Proposition~\ref{P: non-dis L(M([a_{0,k}[pi]]^lambda))}]
For any $r\in\mathbb N$ and $\tau\in S_r$, let $\mathtt a_r[\tau]:=[\mathbf a_{0,r}[\tau]]^\lambda$.
For brevity, put $n:=k+1$.
If $\rho\in S_n$ is the permutation given by~\eqref{rho= a_{o,n}[rho]}, then arguments similar to ones from the first paragraph of the proof of Proposition~\ref{P: non-dis L(M([hat{a}_{0,k}[pi]]^lambda))} imply that $M_\lambda(\mathtt a_n[\rho])\in\mathbf M_\lambda(\mathtt a_k[\pi])$.
Thus, it suffices to verify that the lattice $\mathfrak L\left(\mathbf M_\lambda(\mathtt a_n[\rho])\right)$ is not distributive.

Let now
\[
\begin{aligned}
&\mathbf v_0 := x\,\mathbf p\,z_{n\rho}^{(1)}yz_{n\rho}^{(2)}z_{n\rho}^{(3)}xz_{n\rho}^{(4)}\,\mathbf q\,z_{n\rho}^{(5)}z_{n\rho}^{(6)}z_{n\rho}^{(7)}z_{n\rho}^{(8)}y\,\mathbf r,\\
&\mathbf v_1 := x\,\mathbf p\,z_{n\rho}^{(1)}yz_{n\rho}^{(2)}xz_{n\rho}^{(3)}xz_{n\rho}^{(4)}\,\mathbf q\,z_{n\rho}^{(5)}z_{n\rho}^{(6)}z_{n\rho}^{(7)}z_{n\rho}^{(8)}y\,\mathbf r,\\
&\mathbf v_2 := x\,\mathbf p\,z_{n\rho}^{(1)}yz_{n\rho}^{(2)}z_{n\rho}^{(3)}xz_{n\rho}^{(4)}\,\mathbf q\,z_{n\rho}^{(5)}z_{n\rho}^{(6)}z_{n\rho}^{(7)}yz_{n\rho}^{(8)}y\,\mathbf r,
\end{aligned}
\]
where
\[
\begin{aligned}
&\mathbf p := \biggl(\prod_{i=1}^{n-1}  z_{i\rho}^{(1)}z_{i\rho}^{(2)}z_{i\rho}^{(3)}z_{i\rho}^{(4)}y_i^2\biggr),\\
&\mathbf q := \biggl(\prod_{i=1}^{n-1}  z_{i\rho}^{(5)}z_{i\rho}^{(6)}z_{i\rho}^{(7)}z_{i\rho}^{(8)}(y_i^\prime)^2\biggr),\\
&\mathbf r := \biggl(\prod_{j=1}^4\biggl(\prod_{i=1}^n  t_i^{(j)}z_i^{(j)}\biggr)\biggr)\biggl(\prod_{j=8}^4\biggl(\prod_{i=1}^n  t_i^{(j)}z_i^{(j)}\biggr)\biggr).
\end{aligned}
\]
Let $\mathtt v_j:=[\mathbf v_j]^\lambda$, $j=0,1,2$.
Since $\mathbf M_\lambda(\mathtt a_n[\rho])$ satisfies the identities~\eqref{xyx=xyxx}, all the words of $\mathtt v_j$ lie in the same $\FIC(\mathbf M_\lambda(\mathtt a_n[\rho]))$-class for $j=0,1,2$.
Moreover, arguments similar to ones from the proof of Proposition~\ref{P: non-dis L(M([hat{a}_{0,k}[pi]]^lambda))} imply that the $\lambda$-classes $\mathtt v_0$, $\mathtt v_1$ and $\mathtt v_2$ are $\FIC(\mathbf M_\lambda(\mathtt a_n[\rho]))$-classes.

Let 
\[
\begin{aligned}
&\mathbf X := \mathbf M_\lambda(\mathtt a_n[\rho])\wedge\var\{\mathbf v_0 \approx \mathbf v_1\},\\ 
&\mathbf Y := \mathbf M_\lambda(\mathtt a_n[\rho])\wedge\var\{\mathbf v_0 \approx \mathbf v_2\}.
\end{aligned}
\]
Using Lemma~\ref{L: stable M_lambda(a_{0,k}[pi])} instead of Lemma~\ref{L: stable M_lambda(hat{a}_{0,k}[pi])}, by arguments similar to ones from the proof of Proposition~\ref{P: non-dis L(M([hat{a}_{0,k}[pi]]^lambda))}, we can show that the set $\mathtt v_0\cup \mathtt v_1$ [respectively, $\mathtt v_0\cup \mathtt v_2$] form a $\FIC(\mathbf X)$-class [respectively, $\FIC(\mathbf Y)$-class].
It follows that the $\lambda$-class $\mathtt v_0$ is stable with respect to $\mathbf X\vee \mathbf Y$.
Put $\mathbf Z: = \mathbf M_\lambda(\mathtt v_0)$.
According to Corollary~\ref{C: M_alpha(W) in V}(ii), $\mathbf Z\subseteq \mathbf X\vee \mathbf Y$ and, therefore, $(\mathbf X\vee \mathbf Y)\wedge \mathbf Z=\mathbf Z$.
It is routine to check that $\mathbf Z$ satisfies $\mathbf v_1\approx \mathbf v_3$ and $\mathbf v_2\approx \mathbf v_3$, where 
\[
\mathbf v_3 := x^2\,\mathbf p\,z_{n\rho}^{(1)}y^2z_{n\rho}^{(2)}z_{n\rho}^{(3)}xz_{n\rho}^{(4)}\,\mathbf q\,z_{n\rho}^{(5)}z_{n\rho}^{(6)}z_{n\rho}^{(7)}z_{n\rho}^{(8)}y\,\mathbf r.
\]
Then the identity $\mathbf v_0\approx \mathbf v_3$ holds in the variety $(\mathbf X\wedge \mathbf Z)\vee(\mathbf Y\wedge \mathbf Z)$.
We see that $\mathtt v_0$ is not stable with respect to this variety, whence
\[
(\mathbf X\wedge \mathbf Z)\vee(\mathbf Y\wedge \mathbf Z)\subset \mathbf Z= (\mathbf X\vee \mathbf Y)\wedge \mathbf Z.
\]
Thus, we have proved that the lattice $\mathfrak L(\mathbf M_\lambda(\mathtt a_n[\rho]))$ is not distributive.
\end{proof}

\begin{proposition}
\label{P: non-dis L(M([a_{0,k}[pi]]^beta))}
The lattice $\mathfrak L\left(\mathbf M_\beta([\mathbf a_{0,k}[\pi]]^\beta)\right)$ is not distributive for any $k\ge2$ and $\pi\in S_k$.
\end{proposition}

To prove Proposition~\ref{P: non-dis L(M([a_{0,k}[pi]]^lambda))}, we need the following auxiliary result.
Its proof is quite similar to the proof of Lemma~\ref{L: stable M_lambda(hat{a}_{0,k}[pi])} and we omit it.

\begin{lemma}
\label{L: stable M_beta(a_{0,k}[pi])}
If $k,n\ge2$, $\pi\in S_k$ and $\rho\in S_n$, then the $\beta$-classes $xx^+yty^+$, $xzyx^+ty^+$, $xyx^+$, $xx^+y$, $yxx^+ty^+$ and $[(\mathbf a_{0,n}[\rho])_x]^\beta$ are stable with respect to $\mathbf M_\beta([\mathbf a_{0,k}[\pi]]^\beta)$.\qed
\end{lemma}

\begin{proof}[Proof of Proposition~\ref{P: non-dis L(M([a_{0,k}[pi]]^beta))}]
For any $r\ge2$ and $\tau\in S_r$, let $\mathtt a_r[\tau]:=[\mathbf a_{0,r}[\tau]]^\beta$.
For brevity, put $n:=k+2$.
Define permutation $\rho\in S_n$ as follows:
\[
i\rho := 
\begin{cases} 
1\pi+1 & \text{if $i=1$}, \\
n & \text{if $i=2$}, \\
(i-1)\pi+1 & \text{if $2<i< n-1$},\\
1 & \text{if $i=n-1$},\\
(n-2)\pi+1 & \text{if $i=n$}.
\end{cases} 
\]
Let $\mathbf a\approx \mathbf a^\prime$ be an identity of $\mathbf M_\beta(\mathtt a_k[\pi])$ with $\mathbf a\in \mathtt a_n[\rho]$.
In view of Lemma~\ref{L: stable M_beta(a_{0,k}[pi])}, the following $\beta$-classes are stable with respect to $\mathbf M_\beta(\mathtt a_k[\pi])$: $[(\mathbf a_{0,n}[\rho])_x]^\beta$, $xx^+yty^+$ and $yxx^+ty^+$.
Then $\mathbf a_x^\prime\in [(\mathbf a_{0,n}[\rho])_x]^\beta$, $({_{1\mathbf a^\prime}}x)<({_{1\mathbf a^\prime}}z_{1\rho})$ and $({_{1\mathbf a^\prime}}z_{n\rho})<({_{\ell\mathbf a^\prime}}x)<({_{\ell\mathbf a^\prime}}t_1)$.
Further, since $\psi_1(\mathbf a)\in \mathtt a_k[\pi]$, where $\psi_1\colon \mathfrak X \to \mathfrak X^\ast$ is the substitution given by
\[ 
\psi_1(v) := 
\begin{cases} 
v & \text{if }v=x \text{ or } v=y_1,\\ 
y_{i-1} & \text{if }v=y_i,\ i=3,\dots,n-2,\\
z_{i-1} & \text{if }v=z_i,\ i=2,\dots,n-1,\\
t_{i-1} & \text{if }v=t_i,\ i=2,\dots,n-1,\\
1 & \text{otherwise},
\end{cases} 
\]  
Lemma~\ref{L: M_alpha(W) in V} implies that $\psi_1(\mathbf a^\prime)\in \mathtt a_k[\pi]$.
It follows that $\mathbf a^\prime\in \mathtt a_n[\rho]$.
We see that the $\beta$-class $\mathtt a_n[\rho]$ is stable with respect to $\mathbf M_\beta(\mathtt a_k[\pi])$ and, therefore, $M_\beta(\mathtt a_n[\rho])\in\mathbf M_\beta(\mathtt a_k[\pi])$ by Lemma~\ref{L: M_beta(W) in V}.
Thus, it suffices to verify that the lattice $\mathfrak L\left(\mathbf M_\beta(\mathtt a_n[\rho])\right)$ is not distributive.

Let now
\[
\begin{aligned}
&\mathbf v_0 := x\,\mathbf q\,x\,\mathbf r,\\
&\mathbf v_1 := x\,\mathbf q[0;1]\,x\,\mathbf q[1;4n-3]\,x\,\mathbf r,\\
&\mathbf v_2 := x\,\mathbf q[0;4n-3]\,x\,\mathbf q[4n-3;1]\,x\,\mathbf r,
\end{aligned}
\]
where
\[
\mathbf q := \biggl(\prod_{i=1}^{n-1}  z_{i\rho}z_{i\rho}^{\prime}y_i^2\biggr)z_{n\rho}z_{n\rho}^{\prime}\ \text{ and }\  
\mathbf r := \biggl(\prod_{i=1}^n  t_iz_i\biggr)\biggl(\prod_{i=1}^n t_i^{\prime}z_i^{\prime} \biggr).
\]
Let $\mathtt v_j:=[\mathbf v_j]^\beta$, $j=0,1,2$.
It is easy to see that all the words of $\mathtt v_j$ lie in the same $\FIC(\mathbf M_\beta(\mathtt a_n[\rho]))$-class for $j=0,1,2$.

Let $\mathbf u\approx \mathbf v$ be an identity of $\mathbf M_\beta(\mathtt a_n[\rho])$ with $\mathbf u \in \mathtt v_1$.
It follows from Lemma~\ref{L: stable M_beta(a_{0,k}[pi])} that the $\beta$-class $[\mathbf q\mathbf r]^\beta$ is stable with respect to $\mathbf M_\beta(\mathtt a_n[\rho])$.
Hence $\mathbf v_x\in[\mathbf q\mathbf r]^\beta$.
Since $yxx^+ty^+$ and $xx^+yty^+$ are stable with respect to $\mathbf M_\beta(\mathtt a_n[\rho])$ by Lemma~\ref{L: stable M_beta(a_{0,k}[pi])}, we have $({_{1\mathbf v}}x)<({_{1\mathbf v}}z_{1\rho})$ and $({_{1\mathbf v}}z_{n\rho}^{\prime})<({_{\ell\mathbf v}}x)<({_{1\mathbf v}}t_1)$. 
Further, $\mathbf u(X)\notin \mathtt a_n[\rho]$ and $\psi_2(\mathbf u)\in\mathtt a_n[\rho]$, where $X:=\{x,y_1,z_1,t_1,\dots,y_{n-1},z_{n-1},t_{n-1},z_n,t_n\}$ and $\psi_2\colon \mathfrak X \to \mathfrak X^\ast$ is the substitution given by
\[ 
\psi_2(v) := 
\begin{cases} 
v & \text{if }v=x \text{ or } v=y_i,\ i=1,\dots,n-1,\\ 
z_i & \text{if }v=z_i^\prime,\ i=1,\dots,n,\\
t_i & \text{if }v=t_i^\prime,\ i=1,\dots,n,\\
1 & \text{otherwise}.
\end{cases} 
\]
Then $\mathbf v(X)\notin \mathtt a_n[\rho]$ and $\psi_2(\mathbf v)\in \mathtt a_n[\rho]$ by Lemma~\ref{L: M_alpha(W) in V}. 
Hence $\mathbf v\in\mathtt v_1$ and thus the $\beta$-class $\mathtt v_1$ is a $\FIC(\mathbf M_\beta(\mathtt a_n[\rho]))$-class.
By similar arguments we can show that the $\beta$-classes $\mathtt v_0$ and $\mathtt v_2$ are $\FIC(\mathbf M_\beta(\mathtt a_n[\rho]))$-classes.

Let 
\[
\mathbf X := \mathbf M_\beta(\mathtt a_n[\rho])\wedge\var\{\mathbf v_0 \approx \mathbf v_1\}\ \text{ and } \ 
\mathbf Y := \mathbf M_\beta(\mathtt a_n[\rho])\wedge\var\{\mathbf v_0 \approx \mathbf v_2\}.
\]
It is routine to check that one can take a sufficiently large integer $r$, say $r>2n$, such that   the identities $\mathbf v_0\approx\mathbf v_1\approx \mathbf v_2$ hold in $\mathbf M_\beta(\mathtt a_r[\theta])$ for any $\theta\in S_r$.
Arguments similar to ones from the first paragraph of this proof imply that there is $\theta^\prime\in S_r$ such that $M_\beta(\mathtt a_r[\theta^\prime])\in \mathbf M_\beta(\mathtt a_n[\rho])$.
Hence $M_\beta(\mathtt a_r[\theta^\prime])\in\mathbf X\wedge\mathbf Y$.

Consider an arbitrary identity $\mathbf u \approx \mathbf u^\prime$ of $\mathbf X$ with $\mathbf u\in \mathtt v_0\cup \mathtt v_1$.
We are going to show that $\mathbf u^\prime \in \mathtt v_0\cup \mathtt v_1$.
Since $M_\beta(\mathtt a_r[\theta^\prime])\in\mathbf X$, Lemma~\ref{L: stable M_beta(a_{0,k}[pi])} implies that $\mathbf u^\prime_x\in [\mathbf q \mathbf r]^\beta$.
In view of Proposition~\ref{P: deduction}, we may assume without loss of generality that  either $\mathbf u \approx \mathbf u^\prime$ holds in $\mathbf M_\beta(\mathtt a_n[\rho])$ or $\mathbf u \approx \mathbf u^\prime$ is directly deducible from $\mathbf v_0 \approx \mathbf v_1$. 
In view of the above, $\mathtt v_0\cup \mathtt v_1$ is a union of two $\FIC(\mathbf M_\beta(\mathtt a_n[\rho]))$-classes.
Therefore, it remains to consider the case when $\mathbf u \approx \mathbf u^\prime$ is directly deducible from $\mathbf v_0 \approx \mathbf v_1$, i.e., there exist some words $\mathbf a,\mathbf b \in \mathfrak X^\ast$ and substitution $\phi\colon \mathfrak X \to \mathfrak X^\ast$ such that $(\mathbf u,\mathbf u^\prime ) = (\mathbf a\phi(\mathbf v_s)\mathbf b,\mathbf a\phi(\mathbf v_t)\mathbf b)$, where $\{s,t\}=\{0,1\}$.

First, notice that if $a$ and $b$ are distinct letters, then $ab$ occurs in $\mathbf u$ as a factor at most once.
It follows that 
\begin{itemize}
\item[\textup{($\ast$)}] $\phi(v)$ is either empty or a power of letter for any $v\in \mul(\mathbf v_s)=\mul(\mathbf v_t)$.
\end{itemize}
If $\phi(x)=1$, then $\phi(\mathbf v_s)=\phi(\mathbf v_t)$ and so $\mathbf u=\mathbf u^\prime$, whence $\mathbf u^\prime \in \mathtt v_0\cup \mathtt v_1$.
It remains to consider the case when $\phi(x)\ne1$.
Clearly, the letters $t_i, t_i^\prime$ do not occur in $\phi(x)$ because these letters are simple in $\mathbf u$, while $x\in\mul(\mathbf v_s)=\mul(\mathbf v_t)$.
If $\phi(x)$ is a power of a letter in $\{z_i,z_i^\prime\mid 1\le i\le n\}$, then all the factors of the form $\phi(x)$ in $\mathbf u$ must be preceded by the letter $t_1$.
In this case, the application of $\mathbf v_s\approx \mathbf v_t$ to $\mathbf u$ does not change the factor of $\mathbf u$ preceding $t_1$, whence $\mathbf u^\prime\in \mathtt v_0\cup \mathtt v_1$.
If $\phi(x)\in y_i^+$ for some $1\le i\le n-1$, then $\mathbf u^\prime \in \mathtt v_0\cup \mathtt v_1$ because $y_i$ form exactly one island in $\mathbf u$. 
Therefore, we may further assume that $\phi(x)\in x^+$.
If the image of ${_{1\mathbf v_s}}x$ under $\phi$ is preceded by ${_{1\mathbf u}}z_{n\rho}^\prime$ in $\mathbf u$, then the application of $\mathbf v_s\approx \mathbf v_t$ to $\mathbf u$ cannot change the factor of $\mathbf u$ preceding ${_{1\mathbf u}}z_{n\rho}^\prime$ in $\mathbf u$ and thus $\mathbf u^\prime\in \mathtt v_0\cup \mathtt v_1$.
Therefore, the image of ${_{1\mathbf v_s}}x$ under $\phi$ precedes ${_{1\mathbf u}}z_{n\rho}^\prime$ in $\mathbf u$.
This is only possible when $\phi({_{1\mathbf v_s}}x)$ either precedes ${_{1\mathbf u}}z_{1\rho}$ or lies between ${_{1\mathbf u}}z_{1\rho}$ and ${_{1\mathbf u}}z_{1\rho}^\prime$ in $\mathbf u$.

If $\phi({_{1\mathbf v_s}}x)$ precedes ${_{1\mathbf u}}z_{1\rho}$ in $\mathbf u$, then $\mathbf u^\prime \in \mathtt v_0\cup \mathtt v_1$ because the image of the letter $z_{1\rho}$ under $\phi$ is either empty or a power of letter by~($\ast$).
Suppose now that $\phi({_{1\mathbf v_s}}x)$ lies between ${_{1\mathbf u}}z_{1\rho}$ and ${_{1\mathbf u}}z_{1\rho}^\prime$ in $\mathbf u$.
In this case, $\mathbf u\in\mathtt v_1$.
If $(s,t)=(1,0)$, then $\phi(z_{1\rho})\in x^\ast$, whence $\mathbf u^\prime \in \mathtt v_0\cup \mathtt v_1$.
So, it remains to consider the case when $(s,t)=(0,1)$.
Clearly, if $\phi(z_{1\rho})\in x^\ast$, then $\mathbf u^\prime \in \mathtt v_0\cup \mathtt v_1$.
Let now $\phi(z_{1\rho})\notin x^\ast$.
Then $\phi(z_{1\rho})=z_{1\rho}^\prime$ by~($\ast$).
It follows that $z_{2\rho}\in\{\phi(z_{1\rho}^\prime),\phi(z_{2\rho})\}$.
However, this is impossible because $z_{2\rho}=z_n$ and so
\[
({_{2\mathbf v_s}}z_{1\rho})<({_{\ell\mathbf v_s}}z_{1\rho})<({_{2\mathbf v_s}}z_{2\rho})<({_{\ell\mathbf v_s}}z_{2\rho})<({_{2\mathbf v_s}}z_{1\rho}^\prime)<({_{\ell\mathbf v_s}}z_{1\rho}^\prime)
\]
but 
\[
({_{2\mathbf u}}z_{2\rho})<({_{\ell\mathbf u}}z_{2\rho})<({_{2\mathbf u}}z_{1\rho}^\prime)<({_{\ell\mathbf u}}z_{1\rho}^\prime).
\]
Therefore, the case when $(s,t)=(0,1)$ is impossible and we have proved that $\mathbf u^\prime \in \mathtt v_0\cup \mathtt v_1$ in any case.

We see that the set $\mathtt v_0\cup \mathtt v_1$ forms a $\FIC(\mathbf X)$-class.
By similar arguments we can show that the set $\mathtt v_0\cup \mathtt v_2$ forms a $\FIC(\mathbf Y)$-class.
It follows that the $\beta$-class $\mathtt v_0$ is stable with respect to $\mathbf X\vee \mathbf Y$.
Put $\mathbf Z: = \mathbf M_\beta(\mathtt v_0)$.
According to Lemma~\ref{L: M_beta(W) in V}, $\mathbf Z\subseteq \mathbf X\vee \mathbf Y$ and, therefore, $(\mathbf X\vee \mathbf Y)\wedge \mathbf Z=\mathbf Z$.
It is routine to check that $\mathbf Z$ satisfies $\mathbf v_1\approx \mathbf v_3$ and $\mathbf v_2\approx \mathbf v_3$, where $\mathbf v_3:= x\,\chi(\mathbf q)\,x\,\mathbf r$.
Then the identity $\mathbf v_0\approx \mathbf v_3$ holds in the variety $(\mathbf X\wedge \mathbf Z)\vee(\mathbf Y\wedge \mathbf Z)$.
We see that $\mathtt v_0$ is not stable with respect to this variety, whence
\[
(\mathbf X\wedge \mathbf Z)\vee(\mathbf Y\wedge \mathbf Z)\subset \mathbf Z= (\mathbf X\vee \mathbf Y)\wedge \mathbf Z.
\]
Thus, we have proved that the lattice $\mathfrak L(\mathbf M_\beta(\mathtt a_n[\rho]))$ is not distributive.
\end{proof}

\begin{proposition}
\label{P: non-dis L(M_{gamma'}(a_{n,m}[pi]))}
The lattice $\mathfrak L\left(M_{\gamma^{\prime}}([\mathbf a_{k,k^\prime}^{(p)}[\pi]]^{\gamma^{\prime}})\right)$ is not distributive for any $(k,k^\prime)\in\hat{\mathbb N}_2^0$, $\pi\in S_{k,k^\prime}$ and $p=1,2,3,4$.
\end{proposition}

\begin{proof}
There are four possibilities:
\begin{itemize}
\item $k=k^\prime$ and $1\le 1\pi\le k$;
\item $k=k^\prime$ and $k+1\le 1\pi\le 2k$;
\item $k=k^\prime+1$ and so $1\le 1\pi\le k$;
\item $k=k^\prime-1$ and so $k+1\le 1\pi\le k+k^\prime$.
\end{itemize}
We will consider only the first possibility because the other ones are considered quite analogous.
In this case, $1\le i\pi\le k$ and $k+1\le (i+1)\pi\le 2k$ for any $i=1,3,\dots,2k-1$.
We also consider only the case when $p=1$; the case when $p\in\{2,3,4\}$ is very similar.

First, it is easy to see that
\begin{itemize}
\item[\textup{($\star$)}] the $\gamma^\prime$-classes $\{xzxyty\}$, $xx^+yty$ and $yxx^+ty$ are stable with respect to the variety $\mathbf M_{\gamma^\prime}([\mathbf a_{q,q}^{(1)}[\theta]]^{\gamma^\prime})$ for any $q\in\mathbb N$ and $\theta\in S_{q,q}$.
\end{itemize} 
For brevity, put $n:=k+2$.
Define permutation $\rho\in S_{n,n}$ as follows:
\[
i\rho  := 
\begin{cases} 
1\pi +1 & \text{if $i=1$}, \\
2\pi +3 & \text{if $i=2$}, \\
k+2 & \text{if $i=3$}, \\
2k+4 & \text{if $i=4$}, \\
(i-2)\pi+1 & \text{if $i=5,7,\dots,2k-1$},\\
(i-2)\pi+3 & \text{if $i=6,8,\dots,2k$},\\
1 & \text{if $i=2k+1$},\\
k+3 & \text{if $i=2k+2$},\\
(2k-1)\pi+1 & \text{if $i=2k+3$},\\
(2k)\pi+3 & \text{if $i=2k+4$}.
\end{cases} 
\]
Let $\mathbf a_{n,n}^{(1)}[\rho]\approx \mathbf a$ be an identity of $\mathbf M_{\gamma^\prime}([\mathbf a_{k,k}^{(1)}[\pi]]^{\gamma^\prime})$.
It follows from~($\star$) that $\mathbf a_x=(\mathbf a_{n,n}^{(1)}[\rho])_x$. 
Since $(\mathbf a_{n,n}^{(1)}[\rho])_X$ coincides (up to renaming of letters) with $\mathbf a_{k,k}^{(1)}[\pi]$ for 
\[
X:=\{z_i,t_i\mid i\in\{1,n,n+1,2n\}\}\cup \{y_j,z_j^\prime\mid j\in\{n+1,2n\}\},
\] 
Lemma~\ref{L: M_alpha(W) in V} implies that $\mathbf a_X\in[(\mathbf a_{n,n}^{(1)}[\rho])_X]^{\gamma^\prime}$.
We notice also that $({_{1\mathbf a}t_{n+2}}) < ({_{1\mathbf a}x})$ and $({_{\ell\mathbf a}x}) < ({_{1\mathbf a}t_{n+3}})$ by~($\star$).
It follows that $\mathbf a \in [\mathbf a_{n,n}^{(1)}[\rho]]^{\gamma^\prime}$.
We see that the $\gamma^\prime$-class $[\mathbf a_{n,n}^{(1)}[\rho]]^{\gamma^\prime}$ is stable with respect to $\mathbf M_{\gamma^\prime}([\mathbf a_{k,k}^{(1)}[\pi]]^{\gamma^\prime})$.
Then $M_{\gamma^\prime}([\mathbf a_{n,n}^{(1)}[\rho]]^{\gamma^\prime})\in\mathbf M_{\gamma^\prime}([\mathbf a_{k,k}^{(1)}[\pi]]^{\gamma^\prime})$ by Corollary~\ref{C: M_alpha(W) in V}(iv).
Thus, it suffices to verify that the lattice $\mathfrak L\left(\mathbf M_{\gamma^\prime}([\mathbf a_{n,n}^{(1)}[\rho]]^{\gamma^\prime})\right)$ is not distributive.

Let now
\[
\begin{aligned}
&\mathbf v_0 := \mathbf p\,x\,\mathbf q\,x\,\mathbf r,\\
&\mathbf v_1 := \mathbf p\,x\,\mathbf q[0;1]\,x\,\mathbf q[1;11n-1]\,x\,\mathbf r,\\
&\mathbf v_2 := \mathbf p\,x\,\mathbf q[0;11n-1]\,x\,\mathbf q[11n-1;1]\,x\,\mathbf r,
\end{aligned}
\]
where
\[
\begin{aligned}
&\mathbf p := \biggl(\prod_{i=1}^n z_i^{(1)}t_i^{(1)}\biggr)\biggl(\prod_{i=1}^n z_i^{(2)}t_i^{(2)}\biggr)\biggl(\prod_{i=1}^n z_i^{(3)}t_i^{(3)}\biggr)\biggl(\prod_{i=1}^n z_i^{(4)}t_i^{(4)}\biggr),\\
&\mathbf q := \biggl(\prod_{i=1}^n  z_{(2i-1)\rho}^{(1)}z_{(2i)\rho}^{(1)}z_{(2i-1)\rho}^{(2)}z_{(2i)\rho}^{(2)}z_{(2i-1)\rho}^{(3)}z_{(2i)\rho}^{(3)}y_{(2i)\rho}^2z_{(2i)\rho}^{(4)}z_{(2i-1)\rho}^{(4)}z_{(2i)\rho}^{(5)}\biggr),\\
&\mathbf r := \biggl(\prod_{i=n+1}^{2n} t_i^{(1)}z_i^{(1)}\biggr)\biggl(\prod_{i=n+1}^{2n} t_i^{(3)}z_i^{(3)}t_i^{(4)}z_i^{(4)}\biggr)\biggl(\prod_{i=n+1}^{2n} t_i^{(2)}z_i^{(2)}t_i^{(5)}z_i^{(5)}\biggr).
\end{aligned}
\]
Further, let $\mathtt v_0:=[\mathbf v_0]^{\gamma^\prime}$, 
\[
\begin{aligned}
&\mathtt v_1:=\{\mathbf p\,x^{k_1}\,\mathbf q[0;1]\,x^{k_2}\,\mathbf q[1;1]\,x^{k_3}\,\mathbf q[2;11n-2]\,x^{k_4}\,\mathbf r\mid k_1+k_2,k_2+k_3,k_4\ge1\},\\
&\mathtt v_2:=\{\mathbf p\,x^{k_4}\,\mathbf q[0;11n-2]\,x^{k_3}\,\mathbf q[11n-2;1]\,x^{k_2}\,\mathbf q[11n-1;1]\,x^{k_1}\,\mathbf r\mid k_1,k_2+k_3,k_4\ge1\}.
\end{aligned}
\] 
Let $\mathbf u\approx \mathbf v$ be an identity of $\mathbf M_{\gamma^\prime}([\mathbf a_{n,n}^{(1)}[\rho]]^{\gamma^\prime})$ with $\mathbf u \in \mathtt v_1$.
It follows from~($\star$) that $\mathbf v_x\in[\mathbf p\mathbf q\mathbf r]^{\gamma^\prime}$ and $({_{1\mathbf v}}z_{(2n)\rho}^{(5)})<({_{\ell\mathbf v}}x)<({_{1\mathbf v}}t_{n+1}^{(1)})$. 
Further, $\psi_1(\mathbf u)\in [\mathbf a_{n,n}^{(1)}[\rho]]^{\gamma^\prime}$, where $\psi_1\colon \mathfrak X \to \mathfrak X^\ast$ is the substitution given by
\[ 
\psi_1(v) := 
\begin{cases} 
v & \text{if }v=x \text{ or } v=y_i,\ i=n+1,\dots,2n,\\ 
z_i & \text{if }v=z_i^{(2)},\ i=1,\dots,2n,\\
t_i & \text{if }v=t_i^{(2)},\ i=1,\dots,2n,\\
z_i^\prime & \text{if }v=z_i^{(5)},\ i=n+1\dots,2n,\\
t_i^\prime & \text{if }v=t_i^{(5)},\ i=n+1,\dots,2n,\\
1 & \text{otherwise}.
\end{cases} 
\]
This fact and Lemma~\ref{L: M_alpha(W) in V} imply that $\psi_1(\mathbf v)\in [\mathbf a_{n,n}^{(1)}[\rho]]^{\gamma^\prime}$.
Hence
\[
\mathbf v=\mathbf p\,x^{r_1}\,\mathbf q[0;1]\,x^{r_2}\,\mathbf q[1;1]\,x^{r_3}\,\mathbf q[2;11n-2]\,x^{r_4}\,\mathbf r,
\]
where $r_1+r_2+r_3,r_4\ge1$.
In view of~($\star$), $({_{1\mathbf v}}x)<({_{1\mathbf v}}z_{2\rho}^{(1)})$, whence $r_1+r_2\ge1$.
Finally, $\psi_2(\mathbf u)\notin[\mathbf a_{n,n}^{(1)}[\rho]]^{\gamma^\prime}$, where $\psi_2\colon \mathfrak X \to \mathfrak X^\ast$ is the substitution given by
\[
\psi_2(v) := 
\begin{cases} 
v & \text{if }v=x \text{ or } v=y_i,\ i=n+1,\dots,2n,\\ 
z_i & \text{if }v=z_i^{(1)},\ i=1,\dots,n \text{ or }v=z_i^{(2)},\ i=n+1,\dots,2n,\\
t_i & \text{if }v=t_i^{(1)},\ i=1,\dots,n \text{ or }v=t_i^{(2)},\ i=n+1,\dots,2n,\\
z_i^\prime & \text{if }v=z_i^{(5)} \text{ or },\ i=n+1\dots,2n,\\
t_i^\prime & \text{if }v=t_i^{(5)},\ i=n+1,\dots,2n,\\
1 & \text{otherwise}.
\end{cases}
\]
Then $\psi_2(\mathbf v)\notin[\mathbf a_{n,n}^{(1)}[\rho]]^{\gamma^\prime}$ by Lemma~\ref{L: M_alpha(W) in V}. 
Hence $r_2+r_3\ge1$ and, therefore, $\mathbf v\in\mathtt v_1$.
We see that the set $\mathtt v_1$ is a union of $\FIC(\mathbf M_{\gamma^\prime}([\mathbf a_{n,n}^{(1)}[\rho]]^{\gamma^\prime}))$-classes.
By similar arguments we can show that the sets $\mathtt v_0$ and $\mathtt v_2$ are union of $\FIC(\mathbf M_{\gamma^\prime}([\mathbf a_{n,n}^{(1)}[\rho]]^{\gamma^\prime}))$-classes.

Let 
\[
\mathbf X := \mathbf M_{\gamma^\prime}([\mathbf a_{n,n}^{(1)}[\rho]]^{\gamma^\prime})\wedge\var\{\mathbf v_0 \approx \mathbf v_1\}\ \text{ and } \ 
\mathbf Y := \mathbf M_{\gamma^\prime}([\mathbf a_{n,n}^{(1)}[\rho]]^{\gamma^\prime})\wedge\var\{\mathbf v_0 \approx \mathbf v_2\}.
\]
It is routine to check that one can take a sufficiently large integer $r$, say $r>9n$, such that   the identities $\mathbf v_0\approx\mathbf v_1\approx \mathbf v_2$ hold in $\mathbf M_{\gamma^\prime}([\mathbf a_{r,r}^{(1)}[\theta]]^{\gamma^\prime})$ for any $\theta\in S_{r,r}$.
Arguments similar to ones from the first paragraph of this proof imply that there is $\theta^\prime\in S_{r,r}$ such that $M_{\gamma^\prime}([\mathbf a_{r,r}^{(1)}[\theta]]^{\gamma^\prime})\in \mathbf M_{\gamma^\prime}([\mathbf a_{n,n}^{(1)}[\rho]]^{\gamma^\prime})$.
Hence $M_{\gamma^\prime}([\mathbf a_{r,r}^{(1)}[\theta]]^{\gamma^\prime})\in\mathbf X\wedge\mathbf Y$.

Consider an arbitrary identity $\mathbf u \approx \mathbf u^\prime$ of $\mathbf X$ with $\mathbf u\in \mathtt v_0\cup \mathtt v_1$.
We are going to show that $\mathbf u^\prime \in \mathtt v_0\cup \mathtt v_1$.
Since $M_{\gamma^\prime}([\mathbf a_{r,r}^{(1)}[\theta]]^{\gamma^\prime})\in\mathbf X$, it follows from~($\star$) that $\mathbf u^\prime_x\in [\mathbf p\mathbf q \mathbf r]^{\gamma^\prime}$.
In view of Proposition~\ref{P: deduction}, we may assume without loss of generality that  either $\mathbf u \approx \mathbf u^\prime$ holds in $\mathbf M_{\gamma^\prime}([\mathbf a_{n,n}^{(1)}[\rho]]^{\gamma^\prime})$ or $\mathbf u \approx \mathbf u^\prime$ is directly deducible from $\mathbf v_0 \approx \mathbf v_1$. 
In view of the above, $\mathtt v_0\cup \mathtt v_1$ is a union of $\FIC(\mathbf M_{\gamma^\prime}([\mathbf a_{n,n}^{(1)}[\rho]]^{\gamma^\prime}))$-classes.
Therefore, it remains to consider the case when $\mathbf u \approx \mathbf u^\prime$ is directly deducible from $\mathbf v_0 \approx \mathbf v_1$, i.e., there exist some words $\mathbf a,\mathbf b \in \mathfrak X^\ast$ and substitution $\phi\colon \mathfrak X \to \mathfrak X^\ast$ such that $(\mathbf u,\mathbf u^\prime ) = (\mathbf a\phi(\mathbf v_s)\mathbf b,\mathbf a\phi(\mathbf v_t)\mathbf b)$, where $\{s,t\}=\{0,1\}$.

First, notice that if $a$ and $b$ are distinct letters, then $ab$ occurs in $\mathbf u$ as a factor at most once.
It follows that 
\begin{itemize}
\item[\textup{($\ast$)}] $\phi(v)$ is either empty or a power of letter for any $v\in \mul(\mathbf v_s)=\mul(\mathbf v_t)$.
\end{itemize}
If $\phi(x)=1$, then $\phi(\mathbf v_s)=\phi(\mathbf v_t)$ and so $\mathbf u=\mathbf u^\prime$, whence $\mathbf u^\prime \in \mathtt v_0\cup \mathtt v_1$.
It remains to consider the case when $\phi(x)\ne1$.
Clearly, the letters of the form $t_i^{(j)}$ do not occur in $\phi(x)$ because these letters are simple in $\mathbf u$, while $x\in\mul(\mathbf v_s)=\mul(\mathbf v_t)$.
Further, the letters in 
\[
\{z_i^{(1)}, z_i^{(2)},z_i^{(3)},z_i^{(4)},z_j^{(5)}\mid 1\le i\le 2n,\, n+1\le j\le 2n\}
\] also do not occur in $\phi(x)$ because the first and the second occurrences of these letters lie in different blocks in $\mathbf u$, while all the occurrences of $x$ lie in the same block in $\mathbf v_s$.
If $\phi(x)\in y_i^+$ for some $n+1\le i\le 2n$, then $\mathbf u^\prime \in \mathtt v_0\cup \mathtt v_1$ because $y_i$ forms exactly one island in $\mathbf u$. 
Therefore, we may further assume that $\phi(x)\in x^+$.
If the image of ${_{1\mathbf v_s}}x$ under $\phi$ is preceded by ${_{1\mathbf u}}z_{(2n)\rho}^{(5)}$ in $\mathbf u$, then the application of $\mathbf v_s\approx \mathbf v_t$ to $\mathbf u$ cannot change the factor of $\mathbf u$ preceding ${_{1\mathbf u}}z_{(2n)\rho}^{(5)}$ in $\mathbf u$ and thus $\mathbf u^\prime\in \mathtt v_0\cup \mathtt v_1$.
Therefore, the image of ${_{1\mathbf v_s}}x$ under $\phi$ precedes ${_{1\mathbf u}}z_{(2n)\rho}^{(5)}$ in $\mathbf u$.
This is only possible when one of the following holds:
\begin{itemize}
\item[\textup{(i)}] $\phi({_{1\mathbf v_s}}x)$ precedes ${_{2\mathbf u}}z_{1\rho}^{(1)}$ in $\mathbf u$;
\item[\textup{(ii)}] $\phi({_{1\mathbf v_s}}x)$ lies between ${_{2\mathbf u}}z_{1\rho}^{(1)}$ and ${_{1\mathbf u}}z_{2\rho}^{(1)}$ in $\mathbf u$;
\item[\textup{(iii)}] $\phi({_{1\mathbf v_s}}x)$ lies between ${_{1\mathbf u}}z_{2\rho}^{(1)}$ and ${_{2\mathbf u}}z_{1\rho}^{(2)}$ in $\mathbf u$.
\end{itemize}

If~(i) holds, then $\mathbf u^\prime \in \mathtt v_0\cup \mathtt v_1$ because the image of the letter $z_{1\rho}^{(1)}$ under $\phi$ is either empty or a power of letter by~($\ast$).
If~(ii) holds, then $\phi(z_{1\rho}^{(1)})\in x^\ast$ and so $\mathbf u^\prime \in \mathtt v_0\cup \mathtt v_1$ because $\phi(z_{1\rho}^{(1)})$ cannot contain $z_{2\rho}^{(1)}$.
Suppose now that~(iii) holds.
Clearly, if $\phi(z_{1\rho}^{(1)})\in x^\ast$, then $\mathbf u^\prime \in \mathtt v_0\cup \mathtt v_1$.
Let now $\phi(z_{1\rho}^{(1)})\notin x^\ast$.
Then $\phi(z_{1\rho}^{(1)})=z_{1\rho}^{(2)}$ by~($\ast$).
The case when $(s,t)=(1,0)$ is impossible because the are no $x$ between ${_{1\mathbf u}}z_{1\rho}^{(2)}$ and ${_{2\mathbf u}}z_{1\rho}^{(2)}$ in $\mathbf u$.
So, it remains to consider the case when $(s,t)=(0,1)$.
Now we apply~($\ast$) again, yielding that $z_{3\rho}^{(1)}\in\{\phi(z_{3\rho}^{(1)}),\phi(z_{1\rho}^{(4)})\}$.
However, this is impossible because $3\rho=n$ and so $({_{1\mathbf v_s}}z_{1\rho}^{(1)})<({_{1\mathbf v_s}}z_n^{(1)})<({_{1\mathbf v_s}}z_{1\rho}^{(4)})$ but $({_{1\mathbf u}}z_n^{(1)})<({_{1\mathbf u}}z_{1\rho}^{(2)})$.
Therefore, $\mathbf u^\prime \in \mathtt v_0\cup \mathtt v_1$ in any case.

We see that the set $\mathtt v_0\cup \mathtt v_1$ is a union of $\FIC(\mathbf X)$-classes.
By similar arguments we can show that the set $\mathtt v_0\cup \mathtt v_2$ is a union of $\FIC(\mathbf Y)$-classes.
It follows that the $\gamma^\prime$-class $\mathtt v_0$ is stable with respect to $\mathbf X\vee \mathbf Y$.
Put $\mathbf Z: = \mathbf M_{\gamma^\prime}(\mathtt v_0)$.
According to Corollary~\ref{C: M_alpha(W) in V}(iv), $\mathbf Z\subseteq \mathbf X\vee \mathbf Y$ and, therefore, $(\mathbf X\vee \mathbf Y)\wedge \mathbf Z=\mathbf Z$.
It is routine to check that $\mathbf Z$ satisfies $\mathbf v_1\approx \mathbf v_3$ and $\mathbf v_2\approx \mathbf v_3$, where $\mathbf v_3:= \mathbf p\,x\,\chi(\mathbf q)\,x\,\mathbf r$.
Then the identity $\mathbf v_0\approx \mathbf v_3$ holds in the variety $(\mathbf X\wedge \mathbf Z)\vee(\mathbf Y\wedge \mathbf Z)$.
We see that $\mathtt v_0$ is not stable with respect to this variety, whence
\[
(\mathbf X\wedge \mathbf Z)\vee(\mathbf Y\wedge \mathbf Z)\subset \mathbf Z= (\mathbf X\vee \mathbf Y)\wedge \mathbf Z.
\]
Thus, we have proved that the lattice $\mathfrak L\left(\mathbf M_{\gamma^\prime}([\mathbf a_{n,n}^{(1)}[\rho]]^{\gamma^\prime})\right)$ is not distributive.
\end{proof}

\begin{proposition}
\label{P: non-dis L(M_{gamma''}(a_{n,m}[pi]))}
The lattice $\mathfrak L\left(\mathbf M_{\gamma^{\prime\prime}}([\mathbf a_{n,m}^{(p)}[\rho]]^{\gamma^{\prime\prime}})\right)$ is not distributive for any $(n,m)\in\hat{\mathbb N}_2^0$, $\rho\in S_{n,m}$ and $p=1,2,3,4$.
\end{proposition}

\begin{proof}
There are four possibilities:
\begin{itemize}
\item $n=m$ and $1\le 1\rho\le n$;
\item $n=m$ and $n+1\le 1\rho\le 2n$;
\item $n=m+1$ and so $1\le 1\rho\le n$;
\item $n=m-1$ and so $n+1\le 1\rho\le n+m$.
\end{itemize}
We will consider only the first possibility because the other ones are considered quite analogous.
In this case, $1\le i\rho\le n$ and $n+1\le (i+1)\rho\le 2n$ for any $i=1,3,\dots,2n-1$.
We also consider only the case when $p=1$; the case when $p\in\{2,3,4\}$ is very similar.

First, it is easy to see that
\begin{itemize}
\item[\textup{($\star$)}] for any $q\in\mathbb N$ and $\theta\in S_{q,q}$, the sets $xx^+yty$ and $yxx^+ty$ are stable with respect to $\mathbf M_{\gamma^{\prime\prime}}([\mathbf a_{q,q}^{(1)}[\theta]]^{\gamma^{\prime\prime}})$; in particular, $xzxyty$ and $xyzxty$ are isoterms for $\mathbf M_{\gamma^{\prime\prime}}([\mathbf a_{q,q}^{(1)}[\theta]]^{\gamma^{\prime\prime}})$.
\end{itemize} 
Let now
\[
\mathbf v_0 := \mathbf p\,x\,\mathbf q[0;6n-6]\,y\,\mathbf q[6n-6;12]\,x\,\mathbf q[6n+6;6n-6]\,y\,\mathbf q[12n;1]\,\mathbf r,
\]
where
\[
\begin{aligned}
\mathbf p :={}&\biggl(\prod_{i=1}^n z_i^{(1)}t_i^{(1)}\biggr)\biggl(\prod_{i=1}^n z_i^{(4)}t_i^{(4)}\biggr),\\
\mathbf q :={}&z\biggl(\prod_{i=1}^{n-1}  z_{(2i-1)\rho}^{(1)}z_{(2i)\rho}^{(1)}(y_{(2i)\rho}^{(1)})^2z_{(2i)\rho}^{(2)}z_{(2i)\rho}^{(3)}\biggr) z_{1\rho}^{(4)}z_{2\rho}^{(4)}(y_{2\rho}^{(4)})^2z_{2\rho}^{(5)}z_{2\rho}^{(6)}\cdot\\
&\cdot z_{(2n-1)\rho}^{(1)}z_{(2n)\rho}^{(1)}(y_{(2n)\rho}^{(1)})^2z_{(2n)\rho}^{(2)}z_{(2n)\rho}^{(3)}\biggl(\prod_{i=2}^n  z_{(2i-1)\rho}^{(4)}z_{(2i)\rho}^{(4)}(y_{(2i)\rho}^{(4)})^2z_{(2i)\rho}^{(5)}z_{(2i)\rho}^{(6)}\biggr),\\
\mathbf r :={}&\biggl(\prod_{i=n+1}^{2n} t_i^{(1)}z_i^{(1)}t_i^{(2)}z_i^{(2)}t_i^{(3)}z_i^{(3)}\biggr)\biggl(\prod_{i=n+1}^{2n} t_i^{(4)}z_i^{(4)}t_i^{(5)}z_i^{(5)}t_i^{(6)}z_i^{(6)}\biggr)tz.
\end{aligned}
\]
Let $\mathtt v_0:=[\mathbf v_0]^{\gamma^{\prime\prime}}$ and $\mathtt v_1$ [respectively, $\mathtt v_2$] denote the set of all words $\mathbf w$ such that
\[
\begin{aligned}
&\mathtt v_1:=\{\mathbf w\mid \mathbf w_x\in[(\mathbf v_0)_x]^{\gamma^{\prime\prime}},\,\ ({_{1\mathbf w}}t_n^{(4)})<({_{1\mathbf w}}x)<({_{1\mathbf w}}z),\ ({_{1\mathbf w}}z_{(2n)\rho}^{(2)})<({_{\ell\mathbf w}}x)<({_{1\mathbf w}}z_{(2n)\rho}^{(3)})\},\\
&\mathtt v_2:=\left\{
\mathbf w\ \middle\vert
\begin{array}{l}
\mathbf w_y\in[(\mathbf v_0)_y]^{\gamma^{\prime\prime}},\\
({_{1\mathbf w}}z_{(2n-2)\rho}^{(2)})<({_{1\mathbf w}}y)<({_{1\mathbf w}}z_{(2n-2)\rho}^{(3)}),\ 
 ({_{1\mathbf w}}z_{(2n)\rho}^{(5)})<({_{\ell\mathbf w}}y)<({_{1\mathbf w}}z_{(2n)\rho}^{(6)})
\end{array}
\!\!\!\!\right\}.
\end{aligned}
\]
Arguments similar to ones from the proof of Proposition~\ref{P: non-dis L(M_{gamma'}(a_{n,m}[pi]))} can show that the set $\mathtt v_j$ is a union of $\FIC(\mathbf M_{\gamma^{\prime\prime}}([\mathbf a_{n,n}^{(1)}[\rho]]^{\gamma^{\prime\prime}}))$-classes, $j=0,1,2$.

Let 
\[
\mathbf X := \mathbf M_{\gamma^{\prime\prime}}([\mathbf a_{n,n}^{(1)}[\rho]]^{\gamma^{\prime\prime}})\wedge\var\{\mathbf v_0 \approx \mathbf v_1\}\ \text{ and } \ 
\mathbf Y := \mathbf M_{\gamma^{\prime\prime}}([\mathbf a_{n,n}^{(1)}[\rho]]^{\gamma^{\prime\prime}})\wedge\var\{\mathbf v_0 \approx \mathbf v_2\},
\]
where
\[
\begin{aligned}
&\mathbf v_1 :=  \mathbf p\,x\,\mathbf q[0;6n-6]\,y\,\mathbf q[6n-6;12]\,x^2\,\mathbf q[6n+6;6n-6]\,y\,\mathbf q[12n;1]\,\mathbf r,\\
&\mathbf v_2 :=  \mathbf p\,x\,\mathbf q[0;6n-6]\,y\,\mathbf q[6n-6;12]\,x\,\mathbf q[6n+6;6n-6]\,y^2\,\mathbf q[12n;1]\,\mathbf r.
\end{aligned}
\]
It is routine to check that one can take a sufficiently large integer $r$, say $r>9n$, such that   the identities $\mathbf v_0\approx\mathbf v_1\approx \mathbf v_2$ hold in $\mathbf M_{\gamma^{\prime\prime}}([\mathbf a_{r,r}^{(1)}[\theta]]^{\gamma^{\prime\prime}})$ for any $\theta\in S_{r,r}$.
As in the proof of Proposition~\ref{P: non-dis L(M_{gamma'}(a_{n,m}[pi]))}, using Lemma~\ref{L: M_{gamma''}(W) in V} instead of Corollary~\ref{C: M_alpha(W) in V}(iv),  we can show that there is $\theta^\prime\in S_{r,r}$ such that $M_{\gamma^{\prime\prime}}([\mathbf a_{r,r}^{(1)}[\theta]]^{\gamma^{\prime\prime}})\in \mathbf M_{\gamma^{\prime\prime}}([\mathbf a_{n,n}^{(1)}[\rho]]^{\gamma^\prime})$.
Hence $M_{\gamma^{\prime\prime}}([\mathbf a_{r,r}^{(1)}[\theta]]^{\gamma^{\prime\prime}})\in\mathbf X\wedge\mathbf Y$.

Consider an arbitrary identity $\mathbf u \approx \mathbf u^\prime$ of $\mathbf X$ with $\mathbf u\in \mathtt v_1$.
We are going to show that $\mathbf u^\prime \in \mathtt v_1$.
Since $M_{\gamma^{\prime\prime}}([\mathbf a_{r,r}^{(1)}[\theta]]^{\gamma^{\prime\prime}})\in\mathbf X$, it follows from~($\star$) that $\mathbf u^\prime_{\{x,y\}}\in [\mathbf p\mathbf q \mathbf r]^{\gamma^{\prime\prime}}$.
In view of Proposition~\ref{P: deduction}, we may assume without loss of generality that  either $\mathbf u \approx \mathbf u^\prime$ holds in $\mathbf M_{\gamma^{\prime\prime}}([\mathbf a_{n,n}^{(1)}[\rho]]^{\gamma^{\prime\prime}})$ or $\mathbf u \approx \mathbf u^\prime$ is directly deducible from $\mathbf v_0 \approx \mathbf v_1$. 
In view of the above, $\mathtt v_1$ is a union of $\FIC(\mathbf M_{\gamma^{\prime\prime}}([\mathbf a_{n,n}^{(1)}[\rho]]^{\gamma^{\prime\prime}}))$-classes.
Therefore, it remains to consider the case when $\mathbf u \approx \mathbf u^\prime$ is directly deducible from $\mathbf v_0 \approx \mathbf v_1$, i.e., there exist some words $\mathbf a,\mathbf b \in \mathfrak X^\ast$ and substitution $\phi\colon \mathfrak X \to \mathfrak X^\ast$ such that $(\mathbf u,\mathbf u^\prime ) = (\mathbf a\phi(\mathbf v_s)\mathbf b,\mathbf a\phi(\mathbf v_t)\mathbf b)$, where $\{s,t\}=\{0,1\}$.

First, notice that
\begin{itemize}
\item[\textup{($\ast$)}] if $a$ and $b$ are distinct letters and $ab$ occurs in $\mathbf u$ as a factor at least twice, then $\{a,b\}\in\{\{x,y_{2\rho}^{(4)}\},\{x,y_i^{(1)}\}\mid n+1\le i\le 2n\}$.
\end{itemize}
If $\phi(x)=1$, then $\phi(\mathbf v_s)=\phi(\mathbf v_t)$ and so $\mathbf u=\mathbf u^\prime$, whence $\mathbf u^\prime \in \mathtt v_1$.
It remains to consider the case when $\phi(x)\ne1$.
Clearly, the letters of the form $t_i^{(j)}$ do not occur in $\phi(x)$ because these letters are simple in $\mathbf u$, while $x\in\mul(\mathbf v_s)=\mul(\mathbf v_t)$.
Further, the letters in 
\[
\{z,z_i^{(1)}, z_j^{(2)}, z_j^{(3)},z_i^{(4)},z_j^{(5)},z_j^{(6)}\mid 1\le i\le 2n,\, n+1\le j\le 2n\}
\] 
also do not occur in $\phi(x)$ because the first and the second occurrences of these letters lie in different blocks in $\mathbf u$, while all the occurrences of $x$ lie in the same block in $\mathbf v_s$.
If $y_i^{(j)}\in\con(\phi(x))$, then $\mathbf u^\prime \in \mathtt v_1$ because all occurrences of $y_i^{(j)}$ lie between $_{1\mathbf u}z_i^{(j)}$ and $_{1\mathbf u}z_i^{(j+1)}$ in $\mathbf u$. 
Therefore, by~($\ast$), it remains to consider the case when either $\phi(x)\in x^+$ or $\phi(x)\in y^+$.
If $\phi(x)\in x^+$, then $\mathbf u^\prime \in \mathtt v_1$ because the application of $\mathbf v_s\approx \mathbf v_t$ do not change the positions of the first and the last occurrences of $x$.
Therefore, we may further assume that $\phi(x)\in y^+$.
This is only possible when $s=0$ and $\phi(_{i\mathbf v_0}x)={_{i\mathbf u}}y$, $i=1,2$.
Then 
\[
\phi(\mathbf q[0;6n-6]\,y\,\mathbf q[6n-6;12])=\mathbf q_1x\mathbf q_2
\]
 for some words $\mathbf q_1$, $\mathbf q_2$ with $(\mathbf q_1)_x\in[\mathbf q[6n-6;12]]^{\gamma^{\prime\prime}}$ and $\mathbf q_2\in[\mathbf q[6n+6;6n-6]]^{\gamma^{\prime\prime}}$.
Notice also that $\phi(y)=1$ in this case.
However, this contradicts~($\ast$) because the factor of $\mathbf u$ [respectively, $\mathbf v_s$] between ${_{1\mathbf u}}y$ and  ${_{2\mathbf u}}y$ [respectively, ${_{1\mathbf v_s}}x$ and ${_{2\mathbf v_s}}x$] has exactly $5n+5$ distinct letters.
Therefore, the case when $\phi(x)\in y^+$ is impossible and $\mathbf u^\prime \in \mathtt v_1$ in any case.

We see that the set $\mathtt v_1$ is a union of $\FIC(\mathbf X)$-classes.
By similar arguments we can show that the set $\mathtt v_2$ is a union of $\FIC(\mathbf Y)$-classes.
Since $\mathtt v_0=\mathtt v_1\cap \mathtt v_2$, it follows that the $\gamma^{\prime\prime}$-class $\mathtt v_0$ is stable with respect to $\mathbf X\vee \mathbf Y$.
Put $\mathbf Z: = \mathbf M_{\gamma^{\prime\prime}}(\mathtt v_0)$.
According to Lemma~\ref{L: M_{gamma''}(W) in V}, $\mathbf Z\subseteq \mathbf X\vee \mathbf Y$ and, therefore, $(\mathbf X\vee \mathbf Y)\wedge \mathbf Z=\mathbf Z$.
It is routine to check that $\mathbf Z$ satisfies $\mathbf v_1\approx \mathbf v_3$ and $\mathbf v_2\approx \mathbf v_3$, where 
\[
\mathbf v_3:= \mathbf p\,x\,\mathbf q[0;6n-6]\,y\,\mathbf q[6n-6;12]\,x^2\,\mathbf q[6n+6;6n-6]\,y^2\,\mathbf q[12n;1]\,\mathbf r.
\]
Then the identity $\mathbf v_0\approx \mathbf v_3$ holds in the variety $(\mathbf X\wedge \mathbf Z)\vee(\mathbf Y\wedge \mathbf Z)$.
We see that $\mathtt v_0$ is not stable with respect to this variety, whence
\[
(\mathbf X\wedge \mathbf Z)\vee(\mathbf Y\wedge \mathbf Z)\subset \mathbf Z= (\mathbf X\vee \mathbf Y)\wedge \mathbf Z.
\]
Thus, we have proved that the lattice $\mathfrak L\left(\mathbf M_{\gamma^{\prime\prime}}([\mathbf a_{n,n}^{(1)}[\rho]]^{\gamma^{\prime\prime}})\right)$ is not distributive.
\end{proof}

\begin{proposition}
\label{P: non-dis L(M_{alpha_1}(a_{n,n}[pi]))}
The lattice $\mathfrak L\left(\mathbf M_{\alpha_1}([\mathbf a_{k,k}[\pi]]^{\alpha_1})\right)$ is not distributive for any $k\in\mathbb N$ and $\pi\in S_{2k}^\sharp$.
\end{proposition}

\begin{proof}
First, it is easy to see that
\begin{itemize}
\item[\textup{($\star$)}] for any $q\in\mathbb N$ and $\theta\in S_{2q}^\sharp$, the sets $ytxx^+y$ and $yxx^+ty$ are stable with respect to $\mathbf M_{\alpha_1}([\mathbf a_{q,q}[\theta]]^{\alpha_1})$; in particular, the words $xyzxty$ and $xzytxy$ are isoterms for $\mathbf M_{\alpha_1}([\mathbf a_{q,q}[\theta]]^{\alpha_1})$.
\end{itemize} 
Let
\[
\mathbf v_0 := \mathbf p\,x\,\mathbf q[0;1]\,y\,\mathbf q[1;3n-2]\,x\,\mathbf q[3n-1;1]\,y\,\mathbf r,
\]
where
\[
\begin{aligned}
\mathbf p :={}&\biggl(\prod_{i=1}^{(2n)\rho-1} z_it_i\biggr)(z_{(2n)\rho}^\prime t_{(2n)\rho}^\prime z_{(2n)\rho}t_{(2n)\rho})\biggl(\prod_{i=(2n)\rho+1}^n z_it_i\biggr),\\
\mathbf q :={}&z_{1\rho}z_{1\rho}^\prime y_1^2\biggl(\prod_{i=2}^{2n-1} z_{i\rho}y_i^2\biggr) z_{(2n)\rho}z_{(2n)\rho}^\prime,\\
\mathbf r :={}&\biggl(\prod_{i=n+1}^{1\rho-1} t_iz_i\biggr)(t_{1\rho}z_{1\rho}t_{1\rho}^\prime z_{1\rho}^\prime)\biggl(\prod_{i=1\rho+1}^{2n} t_iz_i\biggr).
\end{aligned}
\]
Let now $\mathtt v_0:=[\mathbf v_0]^{\alpha_1}$ and
\[
\begin{aligned}
&\mathtt v_1:=\{\mathbf w\mid \mathbf w_x\in [(\mathbf v_0)_x]^{\alpha_1},\ ({_{1\mathbf w}}t_n^{(2)})<({_{1\mathbf w}}x)<({_{1\mathbf w}}z_{1\rho}),\ ({_{2\mathbf w}}z_{(2n)\rho})<({_{\ell\mathbf w}}x)<({_{2\mathbf w}}z_{(2n)\rho}^\prime)\},\\
&\mathtt v_2:=\{\mathbf w\mid \mathbf w_y\in[(\mathbf v_0)_y]^{\alpha_1},\ ({_{1\mathbf w}}z_{1\rho})<({_{1\mathbf w}}y)<({_{1\mathbf w}}z_{1\rho}^\prime),\ ({_{2\mathbf w}}z_{(2n)\rho}^\prime)<({_{\ell\mathbf w}}y)<({_{1\mathbf w}}t_{n+1})\}.
\end{aligned}
\]
Arguments similar to ones from the proof of Proposition~\ref{P: non-dis L(M_{gamma'}(a_{n,m}[pi]))} can show that the set $\mathtt v_j$ is a union of $\FIC(\mathbf M_{\alpha_1}([\mathbf a_{n,n}[\rho]]^{\alpha_1}))$-classes, $j=0,1,2$.

Let 
\[
\mathbf X := \mathbf M_{\alpha_1}([\mathbf a_{n,n}[\rho]]^{\alpha_1})\wedge\var\{\mathbf v_0 \approx \mathbf v_1\}\ \text{ and } \ 
\mathbf Y := \mathbf M_{\alpha_1}([\mathbf a_{n,n}[\rho]]^{\alpha_1})\wedge\var\{\mathbf v_0 \approx \mathbf v_2\},
\]
where
\[
\begin{aligned}
&\mathbf v_1 :=  \mathbf p\,x\,\mathbf q[0;1]\,y\,\mathbf q[1;3n-2]\,x^2\,\mathbf q[3n-1;1]\,y\,\mathbf r,\\
&\mathbf v_2 :=  \mathbf p\,x\,\mathbf q[0;1]\,y\,\mathbf q[1;3n-2]\,x\,\mathbf q[3n-1;1]\,y^2\,\mathbf r.
\end{aligned}
\]
It is routine to check that one can take a sufficiently large integer $r$, say $r>2n$, such that   the identities $\mathbf v_0\approx\mathbf v_1\approx \mathbf v_2$ hold in $\mathbf M_{\alpha_1}([\mathbf a_{r,r}[\theta]]^{\alpha_1})$ for any $\theta\in S_{2r}^\sharp$.
Further, as in the proof of Proposition~\ref{P: non-dis L(M_{gamma'}(a_{n,m}[pi]))}, using Lemmas~\ref{L: M_alpha(W) in V} and~\ref{L: subclasses of [a_{n,n}[tau]]^{alpha_1}} instead of Corollary~\ref{C: M_alpha(W) in V}(iv), one can show that there is $\theta^\prime\in S_{2r}^\sharp$ such that $M_{\alpha_1}([\mathbf a_{r,r}[\theta^\prime]]^{\alpha_1})\in \mathbf M_{\alpha_1}([\mathbf a_{n,n}[\rho]]^{\alpha_1})$.
Hence $M_{\alpha_1}([\mathbf a_{r,r}[\theta^\prime]]^{\alpha_1})\in\mathbf X\wedge\mathbf Y$.

Consider an arbitrary identity $\mathbf u \approx \mathbf u^\prime$ of $\mathbf X$ with $\mathbf u\in \mathtt v_1$.
We are going to show that $\mathbf u^\prime \in \mathtt v_1$.
Since $M_{\alpha_1}([\mathbf a_{r,r}[\theta]]^{\alpha_1})\in\mathbf X$, it follows from~($\star$) that $\mathbf u^\prime_{\{x,y\}}\in [\mathbf p\mathbf q \mathbf r]^{\alpha_1}$.
In view of Proposition~\ref{P: deduction}, we may assume without loss of generality that  either $\mathbf u \approx \mathbf u^\prime$ holds in $\mathbf M_{\alpha_1}([\mathbf a_{n,n}[\rho]]^{\alpha_1})$ or $\mathbf u \approx \mathbf u^\prime$ is directly deducible from $\mathbf v_0 \approx \mathbf v_1$. 
In view of the above, $\mathtt v_1$ is a union of $\FIC(\mathbf M_{\alpha_1}([\mathbf a_{n,n}[\rho]]^{\alpha_1}))$-classes.
Therefore, it remains to consider the case when $\mathbf u \approx \mathbf u^\prime$ is directly deducible from $\mathbf v_0 \approx \mathbf v_1$, i.e., there exist some words $\mathbf a,\mathbf b \in \mathfrak X^\ast$ and substitution $\phi\colon \mathfrak X \to \mathfrak X^\ast$ such that $(\mathbf u,\mathbf u^\prime ) = (\mathbf a\phi(\mathbf v_s)\mathbf b,\mathbf a\phi(\mathbf v_t)\mathbf b)$, where $\{s,t\}=\{0,1\}$.

First, notice that if $a$ and $b$ are distinct letters and at least one of them belongs to $\{y,z_{1\rho}^\prime,z_{(2n)\rho}^\prime,z_{1\rho},\dots,z_{(2n)\rho}\}$, then $ab$ occurs in $\mathbf u$ as a factor at most once.
It follows that 
\begin{itemize}
\item[\textup{($\ast$)}] for any $v\in \mul(\mathbf v_s)=\mul(\mathbf v_t)$, if $\con(\phi(v))\cap\{y,z_{1\rho}^\prime,z_{(2n)\rho}^\prime,z_{1\rho},\dots,z_{(2n)\rho}\}\ne\varnothing$, then $\phi(v)$ is a letter.
\end{itemize}
If $\phi(x)=1$, then $\phi(\mathbf v_s)=\phi(\mathbf v_t)$ and so $\mathbf u=\mathbf u^\prime$, whence $\mathbf u^\prime \in \mathtt v_1$.
It remains to consider the case when $\phi(x)\ne1$.
Clearly, the letters of the form $t_i$ or $t_i^\prime$ do not occur in $\phi(x)$ because these letters are simple in $\mathbf u$, while $x\in\mul(\mathbf v_s)=\mul(\mathbf v_t)$.
Further, the letters in $\{z_{1\rho}^\prime,z_{(2n)\rho}^\prime,z_{1\rho},\dots,z_{(2n)\rho}\}$ also do not occur in $\phi(x)$ because the first and the second occurrences of these letters lie in different blocks in $\mathbf u$, while all the occurrences of $x$ lie in the same block in $\mathbf v_s$.
If $\con(\phi(x))\subseteq\{x,y_1,\dots,y_{2n-1}\}$, then $\mathbf u^\prime \in \mathtt v_1$ because the application of $\mathbf v_s\approx \mathbf v_t$ do not change the positions of the first and last occurrences of letters in $\mathbf u$.
Therefore, it remains to consider the case when $y\in\con(\phi(x))$.
Then $\phi(x)=y$ by~($\ast$).
This is only possible when $s=0$ and $\phi(_{i\mathbf v_0}x)={_{i\mathbf u}}y$, $i=1,2$.
Now we apply~($\ast$), yielding that 
\[
\{z_{1\rho}^\prime,z_{2\rho},\dots,z_{(2n)\rho},z_{(2n)\rho}^\prime\}\subseteq\{\phi(z_{1\rho}),\phi(z_{1\rho}^\prime),\phi(z_{2\rho}),\dots,\phi(z_{(2n)\rho})\}.
\]
Since the image under $\phi$ of a non-first occurrence of letter in $\mathbf v_s$ must be a non-first occurrence of letter in $\mathbf u$, it follows that
\[
\{z_{(n+1)\rho},\dots,z_{(2n)\rho},z_{(2n)\rho}^\prime\}\subseteq\{\phi(z_{(n+1)\rho}),\dots,\phi(z_{(2n)\rho})\}.
\]
However, this contradicts the fact that the set $\{\phi(z_{(n+1)\rho}),\dots,\phi(z_{(2n)\rho})\}$ contains at most $n$ elements, while the set $\{z_{(n+1)\rho},\dots,z_{(2n)\rho},z_{(2n)\rho}^\prime\}$ contains $n+1$ elements.
Therefore, the case when $y\in\con(\phi(x))$ is impossible and $\mathbf u^\prime \in \mathtt v_1$ in any case.

We see that the set $\mathtt v_1$ is a union of $\FIC(\mathbf X)$-classes.
By similar arguments we can show that the set $\mathtt v_2$ is a union of $\FIC(\mathbf Y)$-classes.
Since $\mathtt v_0=\mathtt v_1\cap \mathtt v_2$, it follows that the $\alpha_1$-class $\mathtt v_0$ is stable with respect to $\mathbf X\vee \mathbf Y$.
According to Lemma~\ref{L: M_alpha(W) in V}, the $\alpha_1$-class $\mathtt v_0$ is stable with respect to $\mathbf Z: = \mathbf M_{\alpha_1}(\mathtt v_0)$.
Hence $\mathtt v_0$ is also stable with respect to $(\mathbf X\vee \mathbf Y)\wedge \mathbf Z$.
It is routine to check that $\mathbf Z$ satisfies $\mathbf v_1\approx \mathbf v_3$ and $\mathbf v_2\approx \mathbf v_3$, where 
\[
\mathbf v_3:= \mathbf p\,x\,\mathbf q[0;1]\,y\,\mathbf q[1;3n-2]\,x^2\,\mathbf q[3n-1;1]\,y^2\,\mathbf r.
\]
Then the identity $\mathbf v_0\approx \mathbf v_3$ holds in the variety $(\mathbf X\wedge \mathbf Z)\vee(\mathbf Y\wedge \mathbf Z)$.
We see that $\mathtt v_0$ is not stable with respect to this variety, whence
\[
(\mathbf X\wedge \mathbf Z)\vee(\mathbf Y\wedge \mathbf Z)\subset \mathbf Z= (\mathbf X\vee \mathbf Y)\wedge \mathbf Z.
\]
Thus, we have proved that the lattice $\mathfrak L\left(\mathbf M_{\alpha_1}([\mathbf a_{n,n}[\rho]]^{\alpha_1})\right)$ is not distributive.
\end{proof}

\begin{proposition}[\mdseries{\!\cite[Proposition~3.4]{Gusev-23}}]
\label{P: non-dis L(M(hat{a}_{n,m}[pi]))}
The lattice $\mathfrak L\left(\mathbf M(\hat{\mathbf a}_{n,m}[\pi])\right)$ is not distributive for any $(n,m)\in \hat{\mathbb N}_0^2$ and $\pi\in S_{n,m}$.\qed
\end{proposition}

The following statement generalizes Proposition~\ref{P: non-dis L(M(hat{a}_{n,m}[pi]))}.

\begin{proposition}
\label{P: non-dis  L(M_{zeta}(hat{a}_{n,m}[pi]))}
If $\mathbf V$ is a variety generated by a monoid in $\mathbf K$, then the lattice $\mathfrak L\left(\mathbf V\right)$ is not distributive.
\end{proposition}

To prove Proposition~\ref{P: non-dis  L(M_{zeta}(hat{a}_{n,m}[pi]))}, we need one auxiliary result and some notation.
For any $n,m\in\mathbb N_0$, $\rho\in S_{n+m}$ and $0\le p\le q\le n+m$, we put
\[
\hat{\mathbf a}_{n,m}^{p,q}[\rho]:=\biggl(\prod_{i=1}^n z_it_i\biggr)\biggl(\prod_{i=1}^p z_{i\rho}\biggr)x\biggl(\prod_{i=p+1}^q z_{i\rho}\biggr)x\biggl(\prod_{i=q+1}^{n+m} z_{i\rho}\biggr)\biggl(\prod_{i=n+1}^{n+m} t_iz_i\biggr).
\]
Denote by $\hat{\mathbf a}_{n,m}^j[\rho]$ the word obtained from $\hat{\mathbf a}_{n,m}[\rho]$ by replacing the block $z_j$ to the word $x^2z_j$.

\begin{lemma}
\label{L: x^2 is a factor or occ_x > 2}
Let $(n,m)\in\hat{\mathbb N}_0^2$, $\rho\in S_{n,m}$ and $\mathbf V$ be a subvariety of $\mathbf A$ satisfying the identities~\eqref{yxxty=xyxxty},~\eqref{ytyxx=ytxyxx},~\eqref{xyzxy=yxzxy} and~\eqref{xyzxy=xyzyx} such that $M(xyx)\in\mathbf V$.
Suppose that $\mathbf V$ satisfies a non-trivial identity $\hat{\mathbf a}_{n,m}[\rho] \approx \mathbf a$ for some $\mathbf a\in\mathfrak X^\ast$ with $\mathbf a_x=(\hat{\mathbf a}_{n,m}[\rho])_x$.
Suppose also that one of the following holds:
\begin{itemize}
\item[\textup{(i)}] $x^2$ is a factor of $\mathbf a$;
\item[\textup{(ii)}] $\occ_x(\mathbf a)>2$.
\end{itemize}
Then $\mathbf V$ satisfies either $\hat{\mathbf a}_{n,m}[\rho] \approx \hat{\mathbf a}_{n,m}^\prime[\rho]$ or $\hat{\mathbf a}_{n,m}[\rho] \approx \hat{\mathbf a}_{n,m}^j[\rho]$ for some $n+1\le j\le n+m$.
\end{lemma}

\begin{proof}
Clearly, $\mathbf a =\mathbf p\mathbf q\mathbf r$, where
\[
\mathbf p:=\biggl(\prod_{i=1}^n x^{e_i}z_ix^{f_i}t_i\biggr),\ \mathbf q:=x^{g_0}\biggl(\prod_{i=1}^{n+m} z_{i\rho}x^{g_i}\biggr),\ \mathbf r:=\biggl(\prod_{i=n+1}^{n+m} t_ix^{e_i}z_ix^{f_i}\biggr)
\]
and $g_0,e_1,f_1,g_1,\dots,e_{n+m},f_{n+m},g_{n+m}\in\mathbb N_0$.
Then $\mathbf V$ satisfies
\begin{equation}
\label{remove all z_i}
\biggl(\prod_{i=1}^n t_i\biggr)x^2\biggl(\prod_{i=n+1}^{n+m} t_i\biggr)\approx \biggl(\prod_{i=1}^n x^{e_i+f_i}t_i\biggr)x^{\sum_{i=0}^{n+m}{g_i}}\biggl(\prod_{i=n+1}^{n+m} t_ix^{e_i+f_i}\biggr).
\end{equation}

(i) Suppose first that $g_q>1$ for some $q\in\{0,\dots,n+m\}$.
Since $\mathbf V$ satisfies~\eqref{xxyx=xxyxx}, we may assume without any loss that $g_i=0$ for any $i=q+1,\dots,n+m$.
If $q=n+m$, then $\mathbf V$ satisfies
\[
\hat{\mathbf a}_{n,m}[\rho]\approx \mathbf a\stackrel{x^2\approx x^3}\approx \mathbf p\mathbf qx\mathbf r\stackrel{\hat{\mathbf a}_{n,m}[\rho]\approx \mathbf a}\approx \mathbf p_xx\mathbf q_xx^2\mathbf r_x\stackrel{\{\eqref{yxxty=xyxxty},\,\eqref{ytyxx=ytxyxx}\}}\approx \hat{\mathbf a}_{n,m}^\prime[\rho].
\]
If $q<n+m$, then using identities in
\[
\left\{\hat{\mathbf a}_{n,m}[\rho](x,z_{i\rho},t_n,t_{n+1})\approx \mathbf a(x,z_{i\rho},t_n,t_{n+1})\mid i\in\{q+1,\dots,n+m\}\right\},
\]
we can put occurrences of $x$ directly before $t_{n+1}$ in $\mathbf a$.
Then the same arguments as above imply that $\mathbf V$ satisfies $\hat{\mathbf a}_{n,m}[\rho]\approx \hat{\mathbf a}_{n,m}^\prime[\rho]$.

Thus, it remains to consider the case when $g_0,\dots,g_{n+m}\le 1$.
Then either $e_p>1$ or $f_p>1$ for some $p\in\{0,\dots,n+m\}$.
We may assume without loss of generality that $e_p>1$.
In this case, arguments similar to ones from the previous paragraph show that $\{x^2\approx x^3,\,\hat{\mathbf a}_{n,m}[\rho]\approx\mathbf a\}$ implies $\hat{\mathbf a}_{n,m}[\rho]\approx \hat{\mathbf a}_{n,m}^p[\rho]$.
If $p\le n$, then using the identity~\eqref{xxyx=xxyxx}, we put $x$ directly before $t_{n+1}$ in $\hat{\mathbf a}_{n,m}^p[\rho]$ and then apply the same arguments as in the previous paragraph.
If $p>n$, then we are done.

\smallskip

(ii) In view of Part~(i), we may assume that $x^2$ is not a factor of $\mathbf a$. 

First, we consider the case when $n+m=1$. 
If $n=1$ and $m=0$, then $\hat{\mathbf a}_{n,m}[\rho]\approx\mathbf a$ is nothing but $z_1t_1xz_1x \approx x^{e_1}z_1x^{f_1}t_1x^{g_0}z_1x^{g_1}$.
Since $x^2$ is not a factor of $\mathbf a$ and $\occ_x(\mathbf a)>2$, at least three of the numbers $e_1$, $f_1$, $g_0$, $g_1$ are equal to~$1$.
Then $\mathbf V$ satisfies
\[
\hat{\mathbf a}_{1,0}[\rho]\approx \mathbf a= x^{e_1}z_1x^{f_1}t_1x^{g_0}z_1x^{g_1} \stackrel{\{\eqref{xyzxy=yxzxy},\,\eqref{xyzxy=xyzyx}\}}\approx x^{e_1+f_1}z_1t_1x^{g_0+g_1}z_1 \stackrel{\eqref{remove all z_i}}\approx z_1t_1x^2z_1=\hat{\mathbf a}_{1,0}^\prime[\rho],
\]
and we are done.
By a similar argument we can show that if $n=0$ and $m=1$, then $\hat{\mathbf a}_{0,1}[\rho] \approx \hat{\mathbf a}_{0,1}^\prime[\rho]$ holds in $\mathbf V$.
So, we may further assume that $n,m\ge1$.

Suppose that every block of $\mathbf a$ contains at most one occurrence of $x$.
Then~\eqref{remove all z_i} implies $\mathbf a\approx \hat{\mathbf a}_{n,m}^{q,q}[\rho]$ for some $0\le q\le n+m$.
Now Part~(i) applies.
So, it remains to consider the case when $x$ is multiple in some block of $\mathbf a$.

Suppose that $e_j+f_j>1$ for some $j\in\{1,\dots,n+m\}$.
Then $e_j=f_j=1$.
We may assume without loss of generality that $j\le n$.
Clearly, $x^2$ is a factor of $\mathbf a(x,z_{n+1},t_{n+1})$ and $\hat{\mathbf a}_{n,m}[\rho](x,z_{n+1},t_{n+1})$ coincides (up to renaming of letters) with $\hat{\mathbf a}_{0,1}[\varepsilon]$.
Then Part~(i) implies that $\mathbf V$ satisfies one of the identities
\begin{align}
\label{yxxty=xyxty}
yx^2ty&\approx xyxty,\\
\label{xyxtxxy=xyxty}
xyxtx^2y&\approx xyxty.
\end{align}
Since
\[
xyxty\stackrel{\eqref{xyxtxxy=xyxty}}\approx xyxtx^2y\stackrel{\eqref{xyzxy=yxzxy}}\approx yx^2tx^2y \stackrel{\eqref{yxxty=xyxxty}}\approx xyx^2tx^2y\stackrel{\eqref{xyxtxxy=xyxty}}\approx xyx^2ty\stackrel{\eqref{yxxty=xyxxty}}\approx yx^2ty,
\]
the identity~\eqref{yxxty=xyxty} holds in $\mathbf V$ in any case.
Evidently,
\[
\mathbf a\stackrel{\eqref{yxxty=xyxty}}\approx \biggl(\prod_{i=1}^{j-1} x^{e_i}z_ix^{f_i}t_i\biggr) \cdot(z_jx^{e_j+f_j}t_j)\cdot \biggl(\prod_{i=j+1}^n x^{e_i}z_ix^{f_i}t_i\biggr)\biggl(\prod_{i=1}^{n+m} z_{i\rho}x^{g_i}\biggr)\!\biggl(\prod_{i=n+1}^{n+m} t_ix^{e_i}z_ix^{f_i}\biggr).
\]
Now Part~(i) applies again and we conclude that  $\mathbf V$ satisfies either $\hat{\mathbf a}_{n,m}[\rho] \approx \hat{\mathbf a}_{n,m}^\prime[\rho]$ or $\hat{\mathbf a}_{n,m}[\rho] \approx \hat{\mathbf a}_{n,m}^p[\rho]$ for some $n+1\le p\le n+m$.

Suppose now that $e_i+f_i\le 1$ for any $i=1,\dots,n+m$.
Then $\sum_{i=0}^{n+m}g_i\ge2$.
In this case, there are $1\le s\le r\le n+m$ such that $xz_{s\rho}z_{(s+1)\rho}\cdots z_{r\rho}x$ is a factor of $\mathbf a$.
Let 
\[
X:=\{x,t_1,\dots,t_{n+m},z_{s\rho},\dots ,z_{r\rho}\}.
\]
It is easy to see that the identity $\hat{\mathbf a}_{n,m}[\rho](X)\approx \mathbf a(X)$ implies the identity
\begin{equation}
\label{power of a}
\mathbf a \approx \biggl(\prod_{i=1}^n \mathbf e_it_i\biggr) \biggl(\prod_{i=1}^{s-1} x^{g_{i-1}}z_{i\rho}\biggr) x^{\sum_{i=0}^{s-1} g_i} \biggl(\prod_{i=s}^r z_{i\rho}\biggr) x^{\sum_{i=r}^{n+m} g_i}\biggl(\prod_{i=r+1}^{n+m} z_{i\rho}x^{g_i}\biggr)\biggl(\prod_{i=n+1}^{n+m} \mathbf e_it_i\biggr),
\end{equation}
where
\[ 
\mathbf e_i= 
\begin{cases} 
x^{2e_i}z_ix^{2f_i} & \text{if }\ s\le i\rho\le r;\\ 
x^{2e_i+f_i}z_ix^{f_i} & \text{if }\ e_i=1 \ \text{ and }\  i\rho<s\ \text{ or }\ r<i\rho;\\ 
x^{e_i}z_ix^{e_i+2f_i} & \text{if }\ e_i=0 \ \text{ and }\  i\rho<s\ \text{ or }\ r<i\rho.
\end{cases} 
\]
Evidently, if $e_j=1$ or $f_j=1$ for some $1\le j\le n+m$, then $x^2$ is a factor of the left hand-side of~\eqref{power of a}.
If $e_i=f_i=0$ for any $i=1,\dots,n+m$, then either $\sum_{i=0}^{s-1} g_i>1$ or $\sum_{i=r}^{n+m} g_i>1$ because $\occ_x(\mathbf a)>2$ and, therefore, $x^2$ is a factor of the left hand-side of~\eqref{power of a} as well.
We see that the left hand-side of~\eqref{power of a} contains the factor $x^2$ in either case.
According to Part~(i), the variety $\mathbf V$ satisfies  either $\hat{\mathbf a}_{n,m}[\rho] \approx \hat{\mathbf a}_{n,m}^\prime[\rho]$ or $\hat{\mathbf a}_{n,m}[\rho] \approx \hat{\mathbf a}_{n,m}^p[\rho]$ for some $n+1\le p\le n+m$.
\end{proof}

\begin{proof}[Proof of Proposition~\ref{P: non-dis  L(M_{zeta}(hat{a}_{n,m}[pi]))}]
There are $(n,m)\in \hat{\mathbb N}_0^2$ and  $\rho\in S_{n,m}$ such that $\mathbf V$ violates the identity $\hat{\mathbf a}_{n,m}[\rho]\approx\hat{\mathbf a}_{n,m}^\prime[\rho]$.
Take $(n,m)\in \hat{\mathbb N}_0^2$ so that $n+m$ is the minimal number with such a property.
Clearly, $n+m\ge1$.

If $M(xzxyty)\notin\mathbf V$, then the identity $\sigma_3$ holds in $\mathbf V$ by Lemma~\ref{L: swapping in linear-balanced}.
In this case, $\mathbf V$ satisfies
\[
\begin{aligned}
\hat{\mathbf a}_{n,m}[\rho]&{}=\mathbf h_1x\mathbf h_2x\mathbf h_3\stackrel{\sigma_3}\approx\mathbf h_1x\mathbf h_4\mathbf h_5x\mathbf h_3\stackrel{\sigma_3}\approx\mathbf h_1\mathbf h_4x^2\mathbf h_5\mathbf h_3\stackrel{x^2\approx x^3}\approx\mathbf h_1\mathbf h_4x^3\mathbf h_5\mathbf h_3\\
&{}\stackrel{\sigma_3}\approx \mathbf h_1\mathbf h_4x\mathbf h_5x^2\mathbf h_3\stackrel{\eqref{yxxty=xyxxty}}\approx\mathbf h_1\mathbf h_4\mathbf h_5x^2\mathbf h_3\stackrel{\sigma_3}\approx\mathbf h_1\mathbf h_2x^2\mathbf h_3=\hat{\mathbf a}_{n,m}^\prime[\rho],
\end{aligned}
\]
where
\[
\begin{aligned}
&\mathbf h_1:=\biggl(\prod_{i=1}^n z_it_i\biggr),\ \mathbf h_2:=\prod_{i=1}^{n+m} z_{i\rho},\ \mathbf h_3:=\biggl(\prod_{i=n+1}^{2n} t_iz_i\biggr),\\ 
&\mathbf h_4:=\biggl(\prod_{1\le i\rho\le n} z_{i\rho}\biggr),\ \mathbf h_5:=\biggl(\prod_{n< i\rho\le n+m} z_{i\rho}\biggr),
\end{aligned}
\]
contradicting the choice of $n,m$ and $\rho$.
Therefore, $M(xzxyty)\in\mathbf V$.

If $M(\hat{\mathbf a}_{i,j}[\pi])\in\mathbf V$ for some $(i,j)\in\hat{\mathbb N}_0^2$ and $\pi\in S_{i,j}$, then the required claim follows from Proposition~\ref{P: non-dis L(M(hat{a}_{n,m}[pi]))}.
So, we may further assume that $M(\hat{\mathbf a}_{i,j}[\rho])\notin\mathbf V$ for all $(i,j)\in\hat{\mathbb N}_0^2$ and $\pi\in S_{i,j}$.
In particular, $M(\hat{\mathbf a}_{n,m}[\rho])\notin\mathbf V$.
According to Lemma~\ref{L: M(W) in V}, the variety $\mathbf V$ satisfies a non-trivial identity $\hat{\mathbf a}_{n,m}[\rho] \approx \mathbf a$.
Since $xzxyty$ is an isoterm for $\mathbf V$ by Lemma~\ref{L: M(W) in V}, we have $\mathbf a_x=(\hat{\mathbf a}_{n,m}[\rho])_x$.
If $\occ_x(\mathbf a)>2$ or $x^2$ is a factor of $\mathbf a$, then Lemma~\ref{L: x^2 is a factor or occ_x > 2} implies that $m>0$ and $\mathbf V$ satisfies the identity $\hat{\mathbf a}_{n,m}[\rho]\approx\hat{\mathbf a}_{n,m}^r[\rho]$ for some $n+1\le r\le n+m$.
Let now $\occ_x(\mathbf a)\le2$ and $x^2$ is not a factor of $\mathbf a$.
Then $\occ_x(\mathbf a)=2$ since $x$ is an isoterm for $\mathbf V$.
If $({_{1\mathbf a}}x)<({_{1\mathbf a}}t_n)$, then $({_{\ell\mathbf a}}x)<({_{1\mathbf a}}t_n)$ because $xyx$ is an isoterm for $\mathbf V$.
In this case, $\mathbf V$ satisfies the identity $x^2y\approx yx^2$ and, by Lemma~4.8 in~\cite{Gusev-23}, the identity $\hat{\mathbf a}_{n,m}[\rho]\approx \hat{\mathbf a}_{n,m}^\prime[\rho]$, contradicting the choice of $n,m$ and $\rho$.
Hence $({_{1\mathbf a}}t_n)<({_{1\mathbf a}}x)$.
By a similar argument we can show that $({_{\ell\mathbf a}}x)<({_{1\mathbf a}}t_{n+1})$.
Then we may assume without any loss that $\mathbf a=\hat{\mathbf a}_{n,m}^{p,q}[\rho]$ for some $0<p< q\le n+m$ because $x^2$ is not a factor of $\mathbf a$ and the identity $\hat{\mathbf a}_{n,m}[\rho]\approx \mathbf a$ is non-trivial.
Let 
\[
X:=\{z_{1\rho},t_{1\rho},z_{2\rho},t_{2\rho},\dots,z_{p\rho},t_{p\rho},z_{(q+1)\rho},t_{(q+1)\rho},z_{(q+2)\rho},t_{(q+2)\rho},\dots,z_{(n+m)\rho},t_{(n+m)\rho}\}.
\]
Clearly, $(\hat{\mathbf a}_{n,m}[\rho])_X$ coincides (up to renaming of letters) with $\hat{\mathbf a}_{c,d}[\pi]$ for some $(c,d)\in\hat{\mathbb N}_0^2$ and $\pi \in S_{c,d}$ such that $c+d=q-p$.
By the choice of $n,m$ and $\rho$, the identity $\hat{\mathbf a}_{c,d}[\pi]\approx \hat{\mathbf a}_{c,d}^\prime[\pi]$ holds in the variety $\mathbf V$.
Hence this variety must satisfy $\mathbf a=\hat{\mathbf a}_{n,m}^{p,q}[\rho] \approx \hat{\mathbf a}_{n,m}^{q,q}[\rho]$.
Since $\mathbf V$ violates $\hat{\mathbf a}_{n,m}[\rho]\approx\hat{\mathbf a}_{n,m}^\prime[\rho]$, Lemma~\ref{L: x^2 is a factor or occ_x > 2}(i) implies that $\mathbf V$ satisfies the identity $\hat{\mathbf a}_{n,m}[\rho]\approx\hat{\mathbf a}_{n,m}^r[\rho]$ for some $n+1\le r\le n+m$.
Thus, we have proved that $\mathbf V$ satisfies this identity in any case.
Let $r$ be the least number such that the identity $\hat{\mathbf a}_{n,m}[\rho]\approx\hat{\mathbf a}_{n,m}^r[\rho]$ holds in $\mathbf V$.
In particular, the variety $\mathbf V$ satisfies the identity~\eqref{xxy=xxyx} but violates the identity $\hat{\mathbf a}_{n,m}[\rho]\approx\hat{\mathbf a}_{n,m}^j[\rho]$ for any $j=n+1,\dots,r$.

If $M_\lambda(xyx^+)\notin\mathbf V$, then the variety $\mathbf V$ satisfies the identity~\eqref{xyxx=xxyxx} by Lemma~\ref{L: nsub M(xyx^+)}. 
Then $\mathbf V$ also satisfies the identities
\[
\begin{aligned}
\hat{\mathbf a}_{n,m}[\rho]&{}\approx\hat{\mathbf a}_{n,m}^r[\rho]\stackrel{\eqref{xyxx=xxyxx}}\approx \mathbf h_1x\mathbf h_2x^2\biggl(\prod_{i=n+1}^{r-1} t_iz_i\biggr)(t_rx^2z_r)\biggl(\prod_{i=r+1}^{n+m} t_iz_i\biggr)\\
&{}\approx \mathbf h_1x\mathbf h_2x^2\mathbf h_3\stackrel{\{\eqref{yxxty=xyxxty},\,\eqref{ytyxx=ytxyxx}\}}\approx \hat{\mathbf a}_{n,m}^\prime[\rho],
\end{aligned}
\]
contradicting the choice of $n,m$ and $\rho$.
Therefore, $M_\lambda(xyx^+)\in\mathbf V$.

Suppose that $n+m\le2$.
Then $m=1$ and  $n\in\{0,1\}$ because $(n,m)\in\hat{\mathbb N}_0^2$. 
If $n=1$, then it is easy to see that the identity $\hat{\mathbf a}_{n,m}[\rho]\approx \hat{\mathbf a}_{n,m}^\prime[\rho]$ is a consequence of 
\[
\{\eqref{xxy=xxyx},\,\eqref{xyzxy=yxzxy},\,\eqref{ytyxx=ytxyx},\,\hat{\mathbf a}_{n,m}[\rho]\approx \hat{\mathbf a}_{n,m}^r[\rho]\}.
\]
By the choice of $n,m$, the identity~\eqref{ytyxx=ytxyx} holds in $\mathbf V$, contradicting the assumption that $\mathbf V$ violates $\hat{\mathbf a}_{n,m}[\rho]\approx \hat{\mathbf a}_{n,m}^\prime[\rho]$.
Therefore, the case when $n=1$ is impossible.
By a similar argument we can show that the case $n=0$ is impossible.
Hence $n+m>2$.

There are four possibilities:
\begin{itemize}
\item $n=m$ and $1\le 1\rho\le n$;
\item $n=m$ and $n+1\le 1\rho\le n+m$;
\item $n=m+1$ and so $1\le 1\rho\le n$;
\item $n=m-1$ and so $n+1\le 1\rho\le n+m$.
\end{itemize}
We will consider only the first possibility because the other ones are considered quite analogous.
In this case, $n=m\ge2$, $1\le i\rho\le n$ and $n+1\le (i+1)\rho\le n+m$ for any $i=1,3,\dots,n+m-1$.

Let 
\[
\begin{aligned}
&\mathbf v_0 := \mathbf p\,x\,\mathbf q[0;2]\,y\,\mathbf q[2;n+m+2]\,x\,\mathbf q[n+m+4;2]\,y\,\mathbf r,\\
&\mathbf v_1 := \mathbf p\,x\,\mathbf q[0;2]\,y\,\mathbf q[2;n+m+2]\,x^2\,\mathbf q[n+m+4;2]\,y\,\mathbf r,\\
&\mathbf v_2 := \mathbf p\,x\,\mathbf q[0;2]\,y\,\mathbf q[2;n+m+2]\,x\,\mathbf q[n+m+4;2]\,y^2\,\mathbf r,\\
&\mathbf v_3 := \mathbf p\,x\,\mathbf q[0;2]\,y\,\mathbf q[2;n+m+2]\,x^2\,\mathbf q[n+m+4;2]\,y^2\,\mathbf r,
\end{aligned}
\]
where
\[
\begin{aligned}
\mathbf p :={}&\psi_{n,m}^{(n+m-1)\rho}\biggl(\psi_{n,m}^{1\rho}\biggl(zt\biggl(\prod_{i=1}^n z_it_i\biggr)\biggr)\biggr),\\
\mathbf q :={}&z_{1\rho}^\prime z_{2\rho}^\prime z_{1\rho}z_{2\rho}zz^\prime \biggl(\prod_{i=3}^{n+m} z_{i\rho}\biggr)z_{(n+m-1)\rho}^\prime z_{(n+m)\rho}^\prime,\\
\mathbf r :={}&\psi_{n,m}^{(n+m)\rho}\biggl(\psi_{n,m}^{2\rho}\biggl(\biggl(\prod_{i=n+1}^{n+m} t_iz_i\biggr)t^\prime z^\prime\biggr)\biggr),
\end{aligned}
\]
and $\psi_{n,m}^i\colon \mathfrak X \to \mathfrak X^\ast$ is the substitution defined by
\[
\psi_{n,m}^i(v):=
\begin{cases} 
t_iz_i^\prime t_i^\prime & \text{if $v=t_i$ and $1\le i\le n$}, \\
t_i^\prime z_i^\prime t_i & \text{if $v=t_i$ and $n+1\le i\le n+m$}, \\
v & \text{otherwise}.
\end{cases} 
\]
Let
\[
\mathbf X := \mathbf V\wedge\var\{\mathbf v_0 \approx \mathbf v_1\},\ 
\mathbf Y := \mathbf V\wedge\var\{\mathbf v_0 \approx \mathbf v_2\},\
\mathbf Z := \mathbf V\wedge\var\{\mathbf v_1 \approx \mathbf v_2\approx \mathbf v_3\}.
\]

Let now $\mathtt v_1$ denote the set of all words $\mathbf w$ such that
\begin{itemize}
\item[\textup{(a)}] $\mathbf w_{\{x,y\}}=\mathbf p\mathbf q\mathbf r$;
\item[\textup{(b)}] the first two occurrences of both $x$ and $y$ in $\mathbf w$ lie in the same block of $\mathbf w$ as $_{2\mathbf w}z$;
\item[\textup{(c)}] $(_{2\mathbf w}z_{1\rho}^\prime)<(_{1\mathbf w}y)<(_{2\mathbf w}z_{1\rho})$ and $(_{1\mathbf w}z_{(n+m)\rho}^\prime)<(_{2\mathbf w}y)$;
\item[\textup{(d)}] there are no $y$ between $_{2\mathbf w}y$ and $_{1\mathbf w}t_r$.
\end{itemize}
We are going to show that the set $\mathtt v_1$ is stable with respect to $\mathbf X$.
Take an identity $\mathbf u \approx \mathbf u^\prime$ of $\mathbf X$ with $\mathbf u\in \mathtt v_1$.
We need to show $\mathbf u^\prime\in \mathtt v_1$.
In view of Lemma~\ref{L: M(W) in V}  and Corollary~\ref{C: M_alpha(W) in V}(ii), the sets $\{xzxyty\}$ and $xyx^+$ are stable with respect to $\mathbf V$.
It is easy to see that these two sets also are stable with respect to $\var\{\mathbf v_0\approx\mathbf v_1\}$.
Hence $\{xzxyty\}$ and $xyx^+$ are stable with respect to $\mathbf X$.
It follows that $\mathbf u^\prime_{\{x,y\}}=\mathbf p\mathbf q \mathbf r$ and the first two occurrences of both $x$ and $y$ in $\mathbf u^\prime$ lie in the same block of $\mathbf u^\prime$ as $_{2\mathbf u^\prime}z$.
In view of Proposition~\ref{P: deduction}, we may assume without loss of generality that  either $\mathbf u \approx \mathbf u^\prime$ holds in $\mathbf V$ or $\mathbf u \approx \mathbf u^\prime$ is directly deducible from $\mathbf v_0 \approx \mathbf v_1$. 

Suppose that $\mathbf u \approx \mathbf u^\prime$ holds in $\mathbf V$.
Let $\mathbf w_1:=\mathbf u(X_1)$ and $\mathbf w_1^\prime:=\mathbf u^\prime(X_1)$, where
\[
X_1:=\{y,z_{i\rho},t_{i\rho}\mid 1\le i\le n+m\}.
\]
Since $\mathbf u\in\hat{\mathtt v}_1$, the word $\ini_2(\mathbf w_1)$ coincides (up to renaming of letters) with $\hat{\mathbf a}_{n,m}[\rho]$ and there are no $y$ between $_{2\mathbf w_1}y$ and $_{1\mathbf w_1}t_r$.
Then the identity $\ini_2(\mathbf w_1)\approx \mathbf w_1$ follows from the identity $\hat{\mathbf a}_{n,m}[\rho]\approx \hat{\mathbf a}_{n,m}^r[\rho]$.
Hence $\mathbf V$ satisfies $\ini_2(\mathbf w_1)\approx \mathbf w_1^\prime$.
By the choice of $n$, $m$, $\rho$ and $r$, it follows that $\ini_2(\mathbf w_1^\prime)=\ini_2(\mathbf w_1)$ and  there are no $y$ between $_{2\mathbf w_1^\prime}y$ and $_{1\mathbf w_1^\prime}t_r$. 
Hence $(_{1\mathbf u^\prime}y)<(_{2\mathbf u^\prime}z_{1\rho})$ and there are no $y$ between $_{2\mathbf u^\prime}y$ and $_{1\mathbf u^\prime}t_r$.
By a similar argument we can show that $\ini_2(\mathbf u(X_2))=\ini_2(\mathbf u^\prime(X_2))$ and $\ini_2(\mathbf u(X_3))=\ini_2(\mathbf u^\prime(X_3))$, where
\[
\begin{aligned}
&X_2:=\{y,z_{1\rho}^\prime,t_{1\rho}^\prime,z_{i\rho},t_{i\rho}\mid 2\le i\le n+m\},\\
&X_3:=\{y,z_{(n+m)\rho}^\prime,t_{(n+m)\rho}^\prime,z_{i\rho},t_{i\rho}\mid 1\le i\le n+m-1\}.
\end{aligned}
\]
Hence $(_{2\mathbf u^\prime}z_{1\rho})<(_{1\mathbf u^\prime}y)$ and $(_{1\mathbf u^\prime}z_{(n+m)\rho}^\prime)<(_{2\mathbf u^\prime}y)$.
Therefore, $\mathbf u^\prime\in\mathtt v_1$.

Suppose now that $\mathbf u \approx \mathbf u^\prime$ is directly deducible from $\mathbf v_0 \approx \mathbf v_1$, i.e., there exist some words $\mathbf a,\mathbf b \in \mathfrak X^\ast$ and substitution $\phi\colon \mathfrak X \to \mathfrak X^\ast$ such that $(\mathbf u,\mathbf u^\prime ) = (\mathbf a\phi(\mathbf v_s)\mathbf b,\mathbf a\phi(\mathbf v_t)\mathbf b)$, where $\{s,t\}=\{0,1\}$.
First, notice that if $a$ and $b$ are distinct letters in 
\[
X_4:=\{z,z^\prime,z_{1\rho}^\prime,z_{2\rho}^\prime,z_{(n+m-1)\rho}^\prime,z_{(n+m)\rho}^\prime,z_{1\rho},\dots,z_{(n+m)\rho}\},
\] 
then $ab$ occurs in $\mathbf u$ as a factor at most once.
It follows that 
\begin{itemize}
\item[\textup{($\ast$)}] for any $v\in \mul(\mathbf v_s)=\mul(\mathbf v_t)$, if $\con(\phi(v))\cap X_4\ne\varnothing$, then $(\phi(v))_{\{x,y\}}$ is a letter.
\end{itemize}
If $y\notin\con(\phi(x))$, then $\mathbf u^\prime \in \hat{\mathtt v}_1$ because the application of $\mathbf v_s\approx \mathbf v_t$ do not change the positions of the occurrences of $y$ in $\mathbf u$.
It remains to consider the case when $y\in\con(\phi(x))$.
Clearly, the letters in 
\[
\{t,t^\prime,t_{1\rho}^\prime,t_{2\rho}^\prime,t_{(n+m-1)\rho}^\prime,t_{(n+m)\rho}^\prime,t_{1\rho},\dots,t_{(n+m)\rho}\}
\] 
do not occur in $\phi(x)$ because these letters are simple in $\mathbf u$, while $x\in\mul(\mathbf v_s)=\mul(\mathbf v_t)$.
Further, the letters in $X_4$ also do not occur in $\phi(x)$ because the first and the second occurrences of these letters lie in different blocks in $\mathbf u$, while all the occurrences of $x$ lie in the same block in $\mathbf v_s$.
Thus, $\{y\}\subseteq\con(\phi(x))\subseteq\{x,y\}$.
If the image under $\phi$ of $_{1\mathbf v_s}x$ does not contain ${_{1\mathbf u}}y$, then $\mathbf u^\prime \in \hat{\mathtt v}_1$ because the application of $\mathbf v_s\approx \mathbf v_t$ do not change the positions of the first three occurrences of $y$ in $\mathbf u$.
Assume now that the image under $\phi$ of $_{1\mathbf v_s}x$ contains ${_{1\mathbf u}}y$.
This is only possible when $s=0$ and the image under $\phi$ of $_{2\mathbf v_s}x$ contains ${_{2\mathbf u}}y$.
Now we apply~($\ast$), yielding that $X_4\setminus\{z_{1\rho}^\prime,z_{2\rho}^\prime\}$ is a subset of 
\[
\left\{(\phi(z)))_x\mid z\in X_4\setminus\{z_{(n+m-1)\rho}^\prime,z_{(n+m)\rho}^\prime\}\right\}.
\]
Since these two sets are of the same power, this is only possible when these two sets are equal.
Hence 
\[
(\phi(z_{1\rho}^\prime))_x=z_{1\rho},  (\phi(z_{2\rho}^\prime))_x=z_{2\rho}, (\phi(y))_x=1, (\phi(z_{1\rho}))_x=z, (\phi(z_{2\rho}))_x=z^\prime, (\phi(z))_x=z_{3\rho},
\]
contradicting $({_{1\mathbf v_s}}z)<({_{1\mathbf v_s}}z_{1\rho})$ and $({_{1\mathbf u}}z)<({_{1\mathbf u}}z_{3\rho})$.
Therefore, the case when $y\in\con(\phi(x))$ is impossible and $\mathbf u^\prime \in \mathtt v_1$.

We see that the set $\mathtt v_1$ is stable with respect to $\mathbf X$.
By similar arguments we can show that the set $\mathtt v_2$ of all words $\mathbf w$ such that~(a) and~(b) hold and
\begin{itemize}
\item[\textup{(c$^\prime$)}] $(_{1\mathbf w}x)<(_{2\mathbf w}z_{1\rho}^\prime)$ and $(_{2\mathbf w}x)<(_{1\mathbf w}z_{(n+m)\rho}^\prime)$;
\item[\textup{(d$^\prime$)}] there are no $x$ between $_{2\mathbf w}x$ and $_{1\mathbf w}t_r$;
\end{itemize}
is stable with respect to $\mathbf Y$.
Then the set $\mathtt v_1\cap \mathtt v_2$ is stable with respect to $\mathbf X\vee \mathbf Y$.
The same arguments as above imply that $\mathtt v_1\cap \mathtt v_2$ is also stable with respect to $\mathbf Z$.
Therefore, $\mathtt v_1\cap \mathtt v_2$ is stable with respect to $(\mathbf X\vee \mathbf Y)\wedge \mathbf Z$.
Clearly, the identity $\mathbf v_0\approx \mathbf v_3$ holds in the variety $(\mathbf X\wedge \mathbf Z)\vee(\mathbf Y\wedge \mathbf Z)$.
We see that $\mathtt v_0$ is not stable with respect to this variety, whence
\[
(\mathbf X\wedge \mathbf Z)\vee(\mathbf Y\wedge \mathbf Z)\subset \mathbf Z= (\mathbf X\vee \mathbf Y)\wedge \mathbf Z.
\]
Thus, we have proved that the lattice $\mathfrak L(\mathbf V)$ is not distributive.
\end{proof}

\subsection{Infinite series of varieties induced by $\mathbf c_{n,m,k}[\rho]$}

\begin{proposition}
\label{P: non-mod M([c_{0,0,n}[rho]]^lambda)}
The lattice $\mathfrak L\left(\mathbf M_\lambda([\mathbf c_{0,0,n}[\rho]]^\lambda)\right)$ is not modular for any $n\in \mathbb N$ and $\rho\in S_n$.
\end{proposition}

\begin{proof}
First, it is easy to see that
\begin{itemize}
\item[\textup{($\star$)}] the sets $xzyx^+ty^+$ and $xx^+yty^+$ are stable with respect to $\mathbf M_\lambda([\mathbf c_{0,0,q}[\theta]]^\lambda)$ for any $q\in\mathbb N$ and $\theta\in S_q$.
\end{itemize} 

For any $p\in\mathbb N$, $\theta\in S_p$ and $q,r\in\{1,\dots, p+1\}$, we define the permutation $\theta_{q,r}\in S_{p+1}$ as follows: 
\begin{itemize}
\item for any $i=1,\dots,q-1$,
\[
i\theta_{q,r} := 
\begin{cases} 
i\theta & \text{if }i\theta< r, \\ 
i\theta+1 & \text{if }r\le i\theta, 
\end{cases} 
\]
\item $q\theta_{q,r} := r$;
\item for any $i=q+1,\dots,p+1$,
\[
i\theta_{q,r} := 
\begin{cases} 
(i-1)\theta & \text{if }(i-1)\theta< r, \\ 
(i-1)\theta+1 & \text{if }r\le (i-1)\theta. 
\end{cases} 
\]
\end{itemize}
For any $\theta\in S_q$, define $\mathbf c_q[\theta]:=(\mathbf c_{0,0,q}[\theta])_{\{y_1,y_{q-1}\}}$, $\mathbf c_q^\prime[\theta]:=(\mathbf c_{0,0,q}^\prime[\theta])_{\{y_1,y_{q-1}\}}$ and $\mathtt c_q[\theta]:=[\mathbf c_q[\theta]]^\lambda$.

For brevity, put $k:=n+2$, $\pi:=(\rho_{2,n+1})_{n+1,1}$ and $\tau:=(\rho_{2,1})_{n+1,n+2}$.
Let $\mathbf u\approx \mathbf v$ be an identity of $\mathbf M_\lambda([\mathbf c_{0,0,n}[\rho]]^\lambda)$ with $\mathbf u\in \mathtt c_k[\pi]$.
Then~($\star$) implies that 
\[
\mathbf v\in\mathbf ctx^+z_{1\pi}\biggl(\prod_{i=2}^{k-2} z_{i\pi}y_iy_i^+\biggr)z_{(k-1)\pi}z_{k\pi}y^+\biggl(\prod_{i=1}^k t_iz_i\biggr),
\]
where $\mathbf c\in\{xy,yx\}$.
By definition, $(\mathbf c_{0,0,k}[\pi])_{X_1}$ coincides (up to renaming of letters) with $\mathbf c_{0,0,n}[\rho]$ for $X_1:=\{y_2,y_{k-2},z_1,z_k,t_2,t_k\}$.
Then Lemma~\ref{L: M_alpha(W) in V} implies that $\mathbf v_{X_1}\in[\mathbf u_{X_1}]^\lambda$.
In particular, $\mathbf c=xy$.
Thus, $\mathbf v\in \mathtt c_k[\pi]$. 
We see that $\mathtt c_k[\pi]$ is stable with respect to $\mathbf M_\lambda([\mathbf c_{0,0,n}[\rho]]^\lambda)$.
By similar arguments we can show that $\mathtt c_k[\tau]$ is stable with respect to $\mathbf M_\lambda([\mathbf c_{0,0,n}[\rho]]^\lambda)$.
Then $M_\lambda(\mathtt c_k[\pi])$ and $M_\lambda(\mathtt c_k[\tau])$ lie in $\mathbf M_\lambda([\mathbf c_{0,0,n}[\rho]]^\lambda)$.
It is routine to check that the monoids $M_\lambda(\mathtt c_k[\pi])$ and $M_\lambda(\mathtt c_k[\tau])$ satisfy the identities $\mathbf c_k[\tau]\approx \mathbf c_k^\prime[\tau]$ and $\mathbf c_k[\pi]\approx \mathbf c_k^\prime[\pi]$, respectively.
So, we have proved that the variety $\mathbf M_\lambda([\mathbf c_{0,0,n}[\rho]]^\lambda)$ contains two incomparable subvarieties $\mathbf M_\lambda(\mathtt c_k[\pi])$ and $\mathbf M_\lambda(\mathtt c_k[\tau])$.
In view of this fact and Lemma~\ref{L: M_alpha(W_1) vee M_alpha(W_2)}, it suffices to show that the lattice $\mathfrak L\left(\mathbf M_\lambda(\mathtt c_k[\pi],\mathtt c_k[\tau])\right)$ is not modular.

For any $\xi,\eta \in S_2$, we define the word:
\[
\mathbf v_{\xi,\eta}:= a_1b_1\, x_{1\xi}x_{2\xi}\, x_{1\eta}^\prime x_{2\eta}^\prime\, b_2a_2\,t\mathbf r\mathbf s,
\]
where
\[
\begin{aligned}
\mathbf r :={}&x_1z_{1\pi}^\prime a_1\biggl(\prod_{i=2}^{k-2} z_{i\pi}^\prime (y_i^\prime)^2\biggr) z_{(k-1)\pi}^\prime b_2z_{k\pi}^\prime x_2\, z_{1\tau}\biggl(\prod_{i=2}^{k-2}z_{i\tau}y_i^2\biggr)z_{(k-1)\tau}z_{k\tau}\cdot\\
&\cdot x_1^\prime z_{1\pi}^{\prime\prime} b_1\biggl(\prod_{i=2}^{k-2} z_{i\pi}^{\prime\prime}(y_i^{\prime\prime})^2\biggr)z_{(k-1)\pi}^{\prime\prime}a_2z_{k\pi}^{\prime\prime} x_2^\prime,\\
\mathbf s :={}&\biggl(\prod_{i=1}^k t_iz_i\biggr)\biggl(\prod_{i=1}^k t_i^\prime z_i^\prime\biggr)\biggl(\prod_{i=1}^k t_i^{\prime\prime} z_i^{\prime\prime}\biggr).
\end{aligned}
\]
Put $\mathtt v_{\xi,\eta}:=[\mathbf v_{\xi,\eta}]^\lambda$ for any $\xi,\eta \in S_2$.
Let $\varepsilon$ [respectively, $\upsilon$] denote the identity [respectively, unique non-identity] element in $S_2$.
We need the following two auxiliary facts.

\begin{lemma}
\label{L: FIC(M(c_{0,0,k}[tau]))-class}
The set $\mathtt v_{\varepsilon,\varepsilon}\cup\mathtt v_{\upsilon,\varepsilon}\cup\mathtt v_{\varepsilon,\upsilon}\cup\mathtt v_{\upsilon,\upsilon}$ forms a $\FIC(\mathbf M_\lambda(\mathtt c_k[\tau]))$-class.
\end{lemma}

\begin{proof}
It is routine to check that $M_\lambda(\mathtt c_k[\tau])$ satisfies $\mathbf u\approx\mathbf u^\prime$ whenever $\mathbf u,\mathbf u^\prime\in\mathtt v_{\varepsilon,\varepsilon}\cup\mathtt v_{\upsilon,\varepsilon}\cup\mathtt v_{\varepsilon,\upsilon}\cup\mathtt v_{\upsilon,\upsilon}$.
Let now $\mathbf v_{\varepsilon,\varepsilon}\approx \mathbf v$ be an identity of $M(\mathbf c_{n,m,k+2}[\tau])$.
Then~($\star$) implies that $\mathbf v\in[\mathbf v^\prime \mathbf q\mathbf r]^\lambda$, where $\mathbf v^\prime$ is a linear word with $\con(\mathbf v^\prime)=\{a_1,a_2,b_1,b_2,x_1,x_2,x_1^\prime,x_2^\prime\}$.
Since $(\mathbf v_{\varepsilon,\varepsilon})(X_2)$ coincides (up to renaming of letters) with $\mathbf c_k[\tau]$ for 
\[
X_2:=\{a_1,b_1,t\}\cup\{y_j,z_i,t_i\mid 1\le i\le k,\, 2\le j\le k-2\},
\] 
Lemma~\ref{L: M_alpha(W) in V} implies that $(_{1\mathbf v}a_1)<(_{1\mathbf v}b_1)$.
By a similar argument we can show that 
\[
\begin{aligned}
&(_{1\mathbf v}b_1)<(_{1\mathbf v}x_1),\ (_{1\mathbf v}b_1)<(_{1\mathbf v}x_2),\ (_{1\mathbf v}x_1)<(_{1\mathbf v}x_1^\prime),\\
&(_{1\mathbf v}x_1)<(_{1\mathbf v}x_2^\prime),\ (_{1\mathbf v}x_2)<(_{1\mathbf v}x_1^\prime),\ (_{1\mathbf v}x_2)<(_{1\mathbf v}x_2^\prime),\\
&(_{1\mathbf v}x_1^\prime)<(_{1\mathbf v}b_2),\ (_{1\mathbf v}x_2^\prime)<(_{1\mathbf v}b_2),\
(_{1\mathbf v}b_2)<(_{1\mathbf v}a_2).
\end{aligned}
\]
It follows that $\mathbf v \in\mathtt v_{\xi,\eta}$ for some $\xi,\eta\in S_2$ as required.
\end{proof}

\begin{lemma}
\label{L: v_{xi,eta}}
Let $\xi,\eta,\xi_1,\eta_1,\xi_2,\eta_2\in S_2$, $\mathbf u\in \mathtt v_{\xi,\eta}$ and $\mathbf v\notin \mathtt v_{\xi,\eta}$. 
If the identity $\mathbf u \approx \mathbf v$ is directly deducible from the identity $\mathbf v_{\xi_1,\eta_1} \approx \mathbf v_{\xi_2,\eta_2}$, then $\{\mathtt v_{\xi,\eta},[\mathbf v]^\lambda\}=\{\mathtt v_{\xi_1,\eta_1},\mathtt v_{\xi_2,\eta_2}\}$.
\end{lemma}

\begin{proof}
Since $\mathbf u \approx \mathbf v$ is directly deducible from $\mathbf v_{\xi_1,\eta_1} \approx \mathbf v_{\xi_2,\eta_2}$, there are $\mathbf a,\mathbf b\in \mathfrak X^\ast$ and $\phi\colon \mathfrak X \to \mathfrak X^\ast$ such that $(\mathbf u,\mathbf v)=(\mathbf a\phi(\mathbf v_{\xi_1,\eta_1})\mathbf b,\mathbf a\phi(\mathbf v_{\xi_2,\eta_2})\mathbf b)$.
Now Lemma~\ref{L: FIC(M(c_{0,0,k}[tau]))-class} implies that $\mathbf v\in\mathtt v_{\xi^\prime,\eta^\prime}$ for some $\xi^\prime,\eta^\prime\in S_2$.
Since $\mathbf u\in \mathtt v_{\xi,\eta}$ and $\mathbf v\notin \mathtt v_{\xi,\eta}$, we have $(\xi,\eta)\ne(\xi^\prime,\eta^\prime)$.
By symmetry, we may assume that $\xi\ne\xi^\prime$.
Then the first occurrences of $x_1$ and $x_2$ appear in different order in the words $\mathbf u$ and $\mathbf v$.
This is only possible when one of the following holds:
\begin{itemize}
\item $\xi_1\ne\xi_2$ and the image of ${_{1\mathbf v_{\xi_1,\eta_1}}}x_{i\xi_1}$ under $\phi$ contains ${_{1\mathbf u}}x_{i\xi}$, $i=1,2$;
\item $\eta_1\ne\eta_2$ and the image of ${_{1\mathbf v_{\xi_1,\eta_1}}}x_{i\eta_1}^\prime$ under $\phi$ contains ${_{1\mathbf u}}x_{i\xi}$, $i=1,2$.
\end{itemize}
We notice that if $a$ and $b$ are distinct letters, then $ab$ occurs in $\mathbf u$ as a factor at most once.
It follows that
\begin{itemize}
\item[\textup{($\ast$)}] $\phi(c)$ is either the empty word or a power of letter for any $c\in\mul(\mathbf v_{\xi_1,\eta_1})$.
\end{itemize}
Since the first block of $\mathbf u$ is a linear word, this fact implies that one of the following holds:
\begin{itemize}
\item[\textup{(a)}] $\xi_1\ne\xi_2$, $\phi(x_{1\xi_1})=\phi(x_{2\xi_2})=x_{1\xi}$ and $\phi(x_{2\xi_1})=\phi(x_{1\xi_2})=x_{2\xi}$;
\item[\textup{(b)}] $\eta_1\ne\eta_2$, $\phi(x_{1\eta_1}^\prime)=\phi(x_{2\eta_2}^\prime)=x_{1\xi}$ and $\phi(x_{2\eta_1}^\prime)=\phi(x_{1\eta_2}^\prime)=x_{2\xi}$.
\end{itemize}

Suppose that~(a) holds.
In this case, the second occurrences of $x_{1\xi_1}$ and $x_{2\xi_1}$ must appear in $\mathbf v_{\xi_1,\eta_1}$ in the same order as the second occurrences of $x_{1\xi}$ and $x_{2\xi}$ in $\mathbf u$.
This implies that $(1\xi_1,2\xi_1)=(1\xi,2\xi)$, whence $\xi_1=\xi$.
Then
\[
\phi\biggl(z_{1\pi}^\prime a_1\biggl(\prod_{i=2}^{k-2} z_{i\pi}^\prime (y_i^\prime)^2\biggr) z_{(k-1)\pi}^\prime b_2z_{k\pi}^\prime\biggr)\in z_{1\pi}^\prime a_1a_1^+\biggl(\prod_{i=2}^{k-2} z_{i\pi}^\prime y_i^\prime(y_i^\prime)^+\biggr) z_{(k-1)\pi}^\prime b_2b_2^+z_{k\pi}^\prime.
\]
It follows from~($\ast$) and the fact that the first block of $\mathbf u$ is a linear word that $\phi(a_1)=a_1$, $\phi(b_2)=b_2$ and $\phi(z_{i\pi}^\prime)=z_{i\pi}^\prime$ for any $i=1,\dots,k$.
Then 
\[
\phi(b_1x_{1\xi_1}x_{2\xi_1}x_{1\eta_1}^\prime x_{2\eta_1}^\prime)=b_1x_{1\xi}x_{2\xi}x_{1\eta}^\prime x_{2\eta}^\prime.
\]
Now we apply~($\ast$) again, yielding that $\phi(b_1)=b_1$ and $\phi(x_{i\eta_1}^\prime)=x_{i\eta}^\prime$ for any $i=1,2$.
Clearly, the second occurrences of $x_{1\eta_1}^\prime$ and $x_{2\eta_1}^\prime$ must appear in $\mathbf v_{\xi_1,\eta_1}$ in the same order as the second occurrences of $x_{1\eta}^\prime$ and $x_{2\eta}^\prime$ in $\mathbf u$.
This implies that $(1\eta_1,2\eta_1)=(1\eta,2\eta)$, whence $\eta_1=\eta$.
Then $\mathbf u\in \mathtt v_{\xi_1,\eta_1}$, $\mathbf v\in \mathtt v_{\xi_2,\eta_2}$ and so $\{\mathtt v_{\xi,\eta},[\mathbf v]^\lambda\}=\{\mathtt v_{\xi_1,\eta_1},\mathtt v_{\xi_2,\eta_2}\}$.

Suppose that~(b) holds.
In this case, the second occurrences of $x_{1\eta_1}^\prime$ and $x_{2\eta_1}^\prime$ must appear in $\mathbf v_{\xi_1,\eta_1}$ in the same order as the second occurrences of $x_{1\xi}$ and $x_{2\xi}$ in $\mathbf u$.
This implies that $(1\eta_1,2\eta_1)=(1\xi,2\xi)$, whence $\eta_1=\xi$.
Then
\[
\phi\biggl( z_{1\pi}^{\prime\prime} b_1\biggl(\prod_{i=2}^{k-2} z_{i\pi}^{\prime\prime}(y_i^{\prime\prime})^2\biggr)z_{(k-1)\pi}^{\prime\prime}a_2z_{k\pi}^{\prime\prime}\biggr)\in z_{1\pi}^\prime a_1a_1^+\biggl(\prod_{i=2}^{k-2} z_{i\pi}^\prime y_i^\prime(y_i^\prime)^+\biggr) z_{(k-1)\pi}^\prime b_2b_2^+z_{k\pi}^\prime.
\]
It follows from~($\ast$) and the fact that the first block of $\mathbf u$ is a linear word that $\phi(b_1)=a_1$ and $\phi(a_2)=b_2$.
Then 
\[
\phi(x_{1\xi_1}x_{2\xi_1}x_{1\eta_1}^\prime x_{2\eta_1}^\prime b_2)=b_1x_{1\xi}x_{2\xi}x_{1\eta}^\prime x_{2\eta}^\prime.
\]
Now we apply~($\ast$) again and obtain that $\phi(x_{2\eta_1}^\prime)=x_{1\eta}^\prime$ contradicting~(b).
Therefore,~(b) is impossible.

Lemma~\ref{L: v_{xi,eta}} is proved.
\end{proof}

One can return to the proof of Proposition~\ref{P: non-mod M([c_{0,0,n}[rho]]^lambda)}.
Let 
\[
\begin{aligned}
&\mathbf X := \mathbf M_\lambda(\mathtt c_k[\pi],\mathtt c_k[\tau])\wedge\var\{\mathbf v_{\varepsilon,\varepsilon} \approx \mathbf v_{\upsilon,\varepsilon},\,\mathbf v_{\varepsilon,\upsilon} \approx \mathbf v_{\upsilon,\upsilon}\},\\
&\mathbf Y := \mathbf M_\lambda(\mathtt c_k[\pi],\mathtt c_k[\tau])\wedge\var\{\mathbf v_{\upsilon,\varepsilon} \approx \mathbf v_{\upsilon,\upsilon}\},\\
&\mathbf Z := \mathbf M_\lambda(\mathtt c_k[\pi],\mathtt c_k[\tau])\wedge\var\{\mathbf v_{\varepsilon,\varepsilon} \approx \mathbf v_{\varepsilon,\upsilon},\,\mathbf v_{\upsilon,\varepsilon} \approx \mathbf v_{\upsilon,\upsilon}\}.
\end{aligned}
\]
Consider an identity $\mathbf u \approx \mathbf u^\prime$ of $\mathbf X$ with $\mathbf u\in \mathtt v_{\varepsilon,\varepsilon}\cup\mathtt v_{\upsilon,\varepsilon}$.
We are going to show that $\mathbf u^\prime \in \mathtt v_{\varepsilon,\varepsilon}\cup\mathtt v_{\upsilon,\varepsilon}$.
In view of Proposition~\ref{P: deduction}, we may assume without loss of generality that  either $\mathbf u \approx \mathbf u^\prime$ holds in $\mathbf M_\lambda(\mathtt c_k[\pi],\mathtt c_k[\tau])$ or $\mathbf u \approx \mathbf u^\prime$ is directly deducible from $\mathbf v_{\varepsilon,\varepsilon} \approx \mathbf v_{\upsilon,\varepsilon}$ or $\mathbf v_{\varepsilon,\upsilon} \approx \mathbf v_{\upsilon,\upsilon}$. 
According to Lemma~\ref{L: v_{xi,eta}}, $\mathbf u \approx \mathbf u^\prime$ cannot be directly deducible from $\mathbf v_{\varepsilon,\upsilon} \approx \mathbf v_{\upsilon,\upsilon}$ and if $\mathbf u \approx \mathbf u^\prime$ is directly deducible from $\mathbf v_{\varepsilon,\varepsilon} \approx \mathbf v_{\upsilon,\varepsilon}$, then $\{[\mathbf u]^\lambda, [\mathbf u^\prime]^\lambda\}\subseteq \{\mathtt v_{\varepsilon,\varepsilon}, \mathtt v_{\upsilon,\varepsilon}\}$.
Therefore, it remains to consider the case when $\mathbf u \approx \mathbf u^\prime$ is satisfied by $\mathbf M_\lambda(\mathtt c_k[\pi],\mathtt c_k[\tau])$.
It follows from Lemma~\ref{L: FIC(M(c_{0,0,k}[tau]))-class} that $\mathbf u^\prime\in\mathtt v_{\varepsilon,\varepsilon}\cup\mathtt v_{\upsilon,\varepsilon}\cup\mathtt v_{\varepsilon,\upsilon}\cup\mathtt v_{\upsilon,\upsilon}$.
If $\mathbf u^\prime\in\mathtt v_{\varepsilon,\upsilon}\cup\mathtt v_{\upsilon,\upsilon}$, then the identity $\mathbf u(X)\approx \mathbf u^\prime(X)$ is equivalent modulo~\eqref{xyx=xyxx} to the identity $\mathbf c_k[\pi]\approx \mathbf c_k^\prime[\pi]$, where
\[
X=\{x_1^\prime,x_2^\prime,t,y_j^{\prime\prime},z_i^{\prime\prime},t_i^{\prime\prime}\mid1\le i\le k, 2\le j\le k-2\}.
\]
But this is impossible because $M_\lambda(\mathtt c_k[\pi],\mathtt c_k[\tau])$ violates this identity by Lemma~\ref{L: M_alpha(W) in V}.
We see that $\mathbf u^\prime\in \mathtt v_{\varepsilon,\varepsilon}\cup\mathtt v_{\upsilon,\varepsilon}$ in either case and so the set $\mathtt v_{\varepsilon,\varepsilon}\cup\mathtt v_{\upsilon,\varepsilon}$ forms a $\FIC(\mathbf X)$-class.
By similar arguments we can show that the sets $\mathtt v_{\varepsilon,\varepsilon}$ and $\mathtt v_{\varepsilon,\varepsilon}\cup\mathtt v_{\varepsilon,\upsilon}$ form $\FIC(\mathbf Y)$-class and $\FIC(\mathbf Z)$-class, respectively.
This implies that the $\lambda$-class $\mathtt v_{\varepsilon,\varepsilon}$ is stable with respect to $(\mathbf X\vee \mathbf Z)\wedge \mathbf Y$.
Clearly, both $\mathbf X\wedge \mathbf Y$ and $\mathbf Z$ satisfy the identity $\mathbf v_{\varepsilon,\varepsilon}\approx \mathbf v_{\varepsilon,\upsilon}$.
Therefore, the $\lambda$-class $\mathtt v_{\varepsilon,\varepsilon}$ is not stable with respect to $(\mathbf X\wedge \mathbf Y)\vee\mathbf Z$.
Since $\mathbf Z\subseteq\mathbf Y$, we have
\[
(\mathbf X\wedge \mathbf Y)\vee\mathbf Z\subset (\mathbf X\vee \mathbf Z)\wedge \mathbf Y.
\]
It follows that the lattice $\mathfrak L\left(\mathbf M_\lambda([\mathbf c_{0,0,n}[\rho]]^\lambda)\right)$ is not modular.
\end{proof}

\begin{proposition}[\mdseries{\!\cite[Propositions~3.5 and~3.8]{Gusev-Vernikov-21}}]
\label{P: non-mod M(hat{c}_{n,m,n+m+1}[rho]),M(hat{c}_{n,m,0}[rho])}
The lattices $\mathfrak L\left(\mathbf M(\hat{\mathbf c}_{n,m,n+m+1}[\rho])\right)$ and $\mathfrak L\left(\mathbf M(\hat{\mathbf c}_{n,m,0}[\tau])\right)$ are not modular for any $n,m\in \mathbb N_0$, $\rho\in S_{n+m,n+m+1}$ and $\tau\in S_{n+m}$.\qed
\end{proposition}

\begin{proposition}
\label{P: non-mod M([c_n^{(i)}[rho]]^{gamma'})}
The lattices 
\[
\begin{aligned}
&\mathfrak L\left(\mathbf M_{\gamma^\prime}([\mathbf c_n^{(1)}[\pi_1,\tau]]^{\gamma^\prime})\right),\ 
\mathfrak L\left(\mathbf M_{\gamma^\prime}([\mathbf c_n^{(2)}[\pi_2,\tau]]^{\gamma^\prime})\right),\\
&\mathfrak L\left(\mathbf M_{\gamma^\prime}([\mathbf c_n^{(3)}[\pi_3,\tau]]^{\gamma^\prime})\right),\ 
\mathfrak L\left(\mathbf M_{\gamma^\prime}([\mathbf c_n^{(4)}[\pi_4,\tau]]^{\gamma^\prime})\right)
\end{aligned}
\]
are not modular for any $n\in \mathbb N$, $\pi_1\in S_{4n+1}$, $\pi_2\in S_{n+1}$, $\pi_3,\pi_4\in S_{2n+1}$ and $\tau\in S_{2n}$.
\end{proposition}

\begin{proof}
We will prove only the fact that the variety $\mathbf M_{\gamma^\prime}([\mathbf c_n^{(1)}[\pi,\tau]]^{\gamma^\prime})$ has a non-modular lattice of subvarieties; the proof for the other three varieties is very similar and we omit it.

Arguments similar to ones from the proof of Proposition~\ref{P: non-mod M([c_{0,0,n}[rho]]^lambda)} can show that the variety $\mathbf M_{\gamma^\prime}([\mathbf c_n^{(1)}[\pi,\tau]]^{\gamma^\prime})$ contains two incomparable subvarieties $\mathbf M_{\gamma^\prime}([\mathbf c_k^{(1)}[\pi_1,\tau_1]]^{\gamma^\prime})$ and $\mathbf M_{\gamma^\prime}([\mathbf c_k^{(1)}[\pi_2,\tau_2]]^{\gamma^\prime})$, where $k:=n+2$, $\pi_1,\pi_2\in S_{4k+1}$ and $\tau_1,\tau_2\in S_{2k}$.
In view of this fact and Lemma~\ref{L: M_alpha(W_1) vee M_alpha(W_2)}, it suffices to show that the lattice $\mathfrak L\left(\mathbf M_{\gamma^\prime}([\mathbf c_k^{(1)}[\pi_1,\tau_1]]^{\gamma^\prime},[\mathbf c_k^{(1)}[\pi_2,\tau_2]]^{\gamma^\prime})\right)$ is not modular.

For any $\xi,\eta \in S_2$, we define the word:
\[
\mathbf v_{\xi,\eta}:= \mathbf p\,a_1b_1\, x_{1\xi}x_{2\xi}\, x_{1\eta}^\prime x_{2\eta}^\prime\, b_2a_2\,\mathbf q\mathbf r\mathbf s,
\]
where
\[
\begin{aligned}
\mathbf p :={}&\biggl(\prod_{i=1}^k t_i^{(2)}z_i^{(2)}\biggr)\biggl(\prod_{i=1}^k t_i^{(4)} z_i^{(4)}\biggr)\biggl(\prod_{i=1}^k t_i^{(6)} z_i^{(6)}\biggr),\\
\mathbf q :={}&t\biggl(\prod_{i=k+1}^{2k}z_i^{(2)} t_i^{(2)}\biggr)\biggl(\prod_{i=k+1}^{2k} z_i^{(4)}t_i^{(4)} \biggr)\biggl(\prod_{i=k+1}^{2k} z_i^{(6)} t_i^{(6)}\biggr),\\
\mathbf r :={}&x_1 z_{1\pi_1}^{(3)}a_1\hat{z}_{1\pi_1}^{(3)}(y_1^{(3)})^2z_{2\pi_1}^{(3)}z_{1\tau_1}^{(4)} \biggl(\prod_{i=2}^{2k} z_{(2i-1)\pi_1}^{(3)}(y_i^{(3)})^2z_{(2i)\pi_1}^{(3)}z_{i\tau_1}^{(4)}\biggr)\hat{z}_{(4k+1)\pi_1}^{(3)}b_2z_{(4k+1)\pi_1}^{(3)} x_2\cdot \\
&\cdot \biggl(\prod_{i=1}^{2k} z_{(2i-1)\pi_2}^{(1)}(y_i^{(1)})^2z_{(2i)\pi_2}^{(1)}z_{i\tau_2}^{(2)}\biggr)z_{(4k+1)\pi_2}^{(1)}\cdot\\
&\cdot x_1^\prime z_{1\pi_1}^{(5)}b_1\hat{z}_{1\pi_1}^{(5)}(y_1^{(5)})^2z_{2\pi_1}^{(5)}z_{1\tau_1}^{(6)} \biggl(\prod_{i=2}^{2k} z_{(2i-1)\pi_1}^{(5)}(y_i^{(5)})^2z_{(2i)\pi_1}^{(5)}z_{i\tau_1}^{(6)}\biggr)\hat{z}_{(4k+1)\pi_1}^{(5)}a_2z_{(4k+1)\pi_1}^{(5)} x_2^\prime,\\
\mathbf s :={}&\biggl(\prod_{i=1}^{4k+1} t_i^{(1)}z_i^{(1)}\biggr)\cdot\\
&\cdot\biggl(\prod_{i=1}^{k^\prime-1} t_i^{(3)}z_i^{(3)}\biggr)(t_{k^\prime}^{(3)}z_{k^\prime}^{(3)}\hat{t}_{k^\prime}^{(3)}\hat{z}_{k^\prime}^{(3)})\biggl(\prod_{i=k^\prime+1}^{k^{\prime\prime}-1} t_i^{(3)}z_i^{(3)}\biggr)(t_{k^{\prime\prime}}^{(3)}z_{k^{\prime\prime}}^{(3)}\hat{t}_{k^{\prime\prime}}^{(3)}\hat{z}_{k^{\prime\prime}}^{(3)})\biggl(\prod_{i=k^{\prime\prime}+1}^{4k+1} t_i^{(3)}z_i^{(3)}\biggr)\cdot\\
&\cdot\biggl(\prod_{i=1}^{k^\prime-1} t_i^{(5)}z_i^{(5)}\biggr)(t_{k^\prime}^{(5)}z_{k^\prime}^{(5)}\hat{t}_{k^\prime}^{(5)}\hat{z}_{k^\prime}^{(5)})\biggl(\prod_{i=k^\prime+1}^{k^{\prime\prime}-1} t_i^{(5)}z_i^{(5)}\biggr)(t_{k^{\prime\prime}}^{(5)}z_{k^{\prime\prime}}^{(5)}\hat{t}_{k^{\prime\prime}}^{(5)}\hat{z}_{k^{\prime\prime}}^{(5)})\biggl(\prod_{i=k^{\prime\prime}+1}^{4k+1} t_i^{(5)}z_i^{(5)}\biggr)
\end{aligned}
\]
and $k^\prime:=\min(1\pi_1,(4k+1)\pi_1)$, $k^{\prime\prime}:=\max(1\pi_1,(4k+1)\pi_1)$.
Define also three varieties:
\[
\begin{aligned}
&\mathbf X := \mathbf M_{\gamma^\prime}([\mathbf c_k^{(1)}[\pi_1,\tau_1]]^{\gamma^\prime},[\mathbf c_k^{(1)}[\pi_2,\tau_2]]^{\gamma^\prime})\wedge\var\{\mathbf v_{\varepsilon,\varepsilon} \approx \mathbf v_{\upsilon,\varepsilon},\,\mathbf v_{\varepsilon,\upsilon} \approx \mathbf v_{\upsilon,\upsilon}\},\\
&\mathbf Y := \mathbf M_{\gamma^\prime}([\mathbf c_k^{(1)}[\pi_1,\tau_1]]^{\gamma^\prime},[\mathbf c_k^{(1)}[\pi_2,\tau_2]]^{\gamma^\prime})\wedge\var\{\mathbf v_{\upsilon,\varepsilon} \approx \mathbf v_{\upsilon,\upsilon}\},\\
&\mathbf Z := \mathbf M_{\gamma^\prime}([\mathbf c_k^{(1)}[\pi_1,\tau_1]]^{\gamma^\prime},[\mathbf c_k^{(1)}[\pi_2,\tau_2]]^{\gamma^\prime})\wedge\var\{\mathbf v_{\varepsilon,\varepsilon} \approx \mathbf v_{\varepsilon,\upsilon},\,\mathbf v_{\upsilon,\varepsilon} \approx \mathbf v_{\upsilon,\upsilon}\},
\end{aligned}
\]
where $\varepsilon$ [respectively, $\upsilon$] is the identity [respectively, unique non-identity] element in $S_2$
Then using arguments as in the proof of Proposition~\ref{P: non-mod M([c_{0,0,n}[rho]]^lambda)}, we can show that 
\[
(\mathbf X\wedge \mathbf Y)\vee\mathbf Z\subset (\mathbf X\vee \mathbf Z)\wedge \mathbf Y
\]
and thus the lattice $\mathfrak L\left(\mathbf M_{\gamma^\prime}([\mathbf c_n^{(1)}[\pi,\tau]]^{\gamma^\prime})\right)$ is not modular.
\end{proof}

\section{Identities defining varieties}
\label{Sec: identities}

\subsection{Short lists of identities}

Following~\cite{Lee-23}, we call an identity of the form
\[
\mathbf u_0\biggl(\prod_{i=1}^r t_i\mathbf u_i\biggr)\approx\mathbf v_0\biggl(\prod_{i=1}^r t_i\mathbf v_i\biggr),
\]
where $\{t_1,\dots,t_r\}=\simple(\mathbf u_0t_1\mathbf u_1\cdots t_r\mathbf u_r)=\simple(\mathbf v_0t_1\mathbf v_1\cdots t_r\mathbf v_r)$, \textit{efficient} if $\mathbf u_i\mathbf v_i\ne1$ for any $i=0,\dots,r$. 
The following statement was first established in~\cite[Proposition~4.1]{Lee-12}.

\begin{proposition}[\mdseries{\!\cite[Lemma~12.1 and Remark~12.3]{Lee-23}}]
\label{P: var{sigma_2,sigma_3} subvarieties}
Each non-commutative variety satisfying the identities $\sigma_2$ and $\sigma_3$ can be defined by the identities $\sigma_2$ and $\sigma_3$ together with some of the following identities:
\begin{equation}
\label{one letter in a block}
x^{e_0}\biggl(\prod_{i=1}^r t_ix^{e_i}\biggr) \approx x^{f_0}\biggl(\prod_{i=1}^r t_ix^{f_i}\biggr),
\end{equation}
where $r,e_0,f_0,\dots,e_r,f_r\in\mathbb N_0$ and $\sum_{i=0}^re_i,\sum_{i=0}^rf_i\ge2$; and
\begin{equation}
\label{two letters in a block}
x^{e_0}y^{f_0}\biggl(\prod_{i=1}^r t_ix^{e_i}y^{f_i}\biggr)\approx y^{f_0}x^{e_0}\biggl(\prod_{i=1}^r t_ix^{e_i}y^{f_i}\biggr),
\end{equation}
where $r\in\mathbb N_0$, $e_0,f_0\in\mathbb N$, $e_1,f_1,\dots,e_r,f_r\in\mathbb N_0$, $\sum_{i=0}^r e_i\ge 2$ and $\sum_{i=0}^r f_i\ge 2$.
The identities~\eqref{one letter in a block} and~\eqref{two letters in a block} can be chosen to be efficient.\qed
\end{proposition}

\begin{proposition}
\label{P: Phi,Phi_1,Phi_2 subvarieties}
Each non-commutative variety satisfying $\{\Phi,\,\Phi_1,\,\Phi_2\}$ can be defined by the identities in $\{\Phi,\,\Phi_1,\,\Phi_2\}$ together with some of the following identities:~\eqref{one letter in a block}, where
\begin{equation}
\label{one letter in a block conditions}
r\in\mathbb N,\ e_0,f_0,\dots,e_r,f_r\in\{0,1,2\},\ \sum_{i=0}^re_i,\sum_{i=0}^rf_i\ge2;
\end{equation}
and
\begin{align}
\label{two letters in a block middle 1}
\biggl(\prod_{i=1}^k a_i^{g_i}t_i\biggr) xy \biggl(\prod_{i=k+1}^{k+\ell} t_i a_i^{g_i}\biggr)&{}\approx\biggl(\prod_{i=1}^k a_i^{g_i}t_i\biggr) yx  \biggl(\prod_{i=k+1}^{k+\ell} t_ia_i^{g_i}\biggr),\\
\label{two letters in a block middle 2}
\biggl(\prod_{i=1}^k a_i^{g_i}t_i\biggr) x^2y \biggl(\prod_{i=k+1}^{k+\ell} t_i a_i^{g_i}\biggr)&{}\approx\biggl(\prod_{i=1}^k a_i^{g_i}t_i\biggr) xyx  \biggl(\prod_{i=k+1}^{k+\ell} t_ia_i^{g_i}\biggr),\\
\label{two letters in a block middle 3}
\biggl(\prod_{i=1}^k a_i^{g_i}t_i\biggr) yx^2 \biggl(\prod_{i=k+1}^{k+\ell} t_i a_i^{g_i}\biggr)&{}\approx\biggl(\prod_{i=1}^k a_i^{g_i}t_i\biggr) xyx  \biggl(\prod_{i=k+1}^{k+\ell} t_ia_i^{g_i}\biggr),
\end{align}
where
\begin{equation}
\label{two letters in a block middle conditions}
k,\ell\in\mathbb N_0,\ g_1,\dots,g_{k+\ell}\in\{1,2\}\ \text{ and }\ a_1,\dots,a_{k+\ell}\in\{x,y\}.
\end{equation}
The identities~\eqref{one letter in a block} can be chosen to be efficient.
\end{proposition}

To prove Proposition~\ref{P: Phi,Phi_1,Phi_2 subvarieties}, we need three auxiliary results.

\begin{lemma}
\label{L: from pxx_1..x_kxr to px_1x..x_kxr}
Let $\mathbf V$ be a variety satisfying $\Phi_2$. 
If $\mathbf w:=\mathbf px\mathbf q_1\mathbf q_2x\mathbf r$ and $\con(\mathbf q_1\mathbf q_2)\subseteq\mul(\mathbf w)$, then $\mathbf V$ satisfies the identity $\mathbf w\approx\mathbf px\mathbf q_1x\mathbf q_2x\mathbf r$.
\end{lemma}

\begin{proof}
If, for any $a\in\mul(\mathbf q_1\mathbf q_2)$, each island formed by $a$ is of length $>1$, then one can find $k,m\in\mathbb N_0$ and $\rho\in S_{k+m}$ such that $\mathbf a_{m,k}[\rho]\approx \overline{\mathbf a}_{m,k}[\rho]$ implies $\mathbf w\approx\mathbf px\mathbf q_1x\mathbf q_2x\mathbf r$.

So, it remains to consider the case when some multiple letter forms an island of length~$1$ in $\mathbf q_1\mathbf q_2$.
Then there is a letter $y_1\in\con(\mathbf q_1\mathbf q_2)$ such that $\mathbf q_1\mathbf q_2=\mathbf v_1y_1\mathbf v_2y_1\mathbf v_3$ and, for any $a\in\mul(\mathbf v_2)$, each island formed by $a$ in $\mathbf v_2$ is of length $>1$.
By the arguments in the previous paragraph, $\mathbf w=\mathbf px\mathbf v_1y_1\mathbf v_2y_1\mathbf v_3x\mathbf r\stackrel{\Phi_2}\approx\mathbf px\mathbf v_1y_1^2\mathbf v_2y_1^2\mathbf v_3x\mathbf r$.
In other words, using the identities in $\Phi_2$, we may replace the discussed occurrences of $y_1$ to $y_1^2$ in $\mathbf q_1\mathbf q_2$. 
Repeating these considerations, we may replace all occurrences of multiple letters to their squares in $\mathbf q_1\mathbf q_2$, resulting some word $\mathbf q_1^\prime\mathbf q_2^\prime$. 
Thus, $\mathbf V$ satisfies the identity $\mathbf w\approx\mathbf px\mathbf q_1^\prime\mathbf q_2^\prime x\mathbf r$. 
By the arguments in the previous paragraph, $\Phi_2$ implies $\mathbf px\mathbf q_1^\prime\mathbf q_2^\prime x\mathbf r\approx \mathbf px\mathbf q_1^\prime x\mathbf q_2^\prime x\mathbf r$.
It remains to remove the added squares in $\mathbf q_1^\prime x\mathbf q_2^\prime$ using the identities in $\Phi_2$, resulting the word $\mathbf q_1x\mathbf q_2$. 
As a result, we obtain that $\mathbf V$ satisfies the identity $\mathbf w\approx\mathbf px\mathbf q_1x\mathbf q_2x\mathbf r$ as required.
\end{proof}

\begin{lemma}
\label{L: two letters in a block middle 2,3 <-> 4}
If the condition~\eqref{two letters in a block middle conditions} holds, then the identity
\begin{equation}
\label{two letters in a block middle 4}
\biggl(\prod_{i=1}^k a_i^{g_i}t_i\biggr) x^2y \biggl(\prod_{i=k+1}^{k+\ell} t_i a_i^{g_i}\biggr)\approx\biggl(\prod_{i=1}^k a_i^{g_i}t_i\biggr) yx^2  \biggl(\prod_{i=k+1}^{k+\ell} t_ia_i^{g_i}\biggr),
\end{equation}
is equivalent modulo $\Phi_2$ to $\{\eqref{two letters in a block middle 2},\,\eqref{two letters in a block middle 3}\}$.
\end{lemma}

\begin{proof}
Taking into account Lemma~\ref{L: from pxx_1..x_kxr to px_1x..x_kxr}, we have
\[
\begin{aligned}
&\mathbf h\, x^2y\, \mathbf t\stackrel{\eqref{two letters in a block middle 2}}\approx \mathbf h\, xyx\, \mathbf t\stackrel{\eqref{two letters in a block middle 3}}\approx \mathbf h\, yx^2\, \mathbf t, \\ 
&\mathbf h\, x^2y\, \mathbf t\stackrel{\Phi_2}\approx \mathbf h\, x^3y\, \mathbf t\stackrel{\eqref{two letters in a block middle 4}}\approx \mathbf h\, xyx^2\, \mathbf t\stackrel{\Phi_2}\approx\mathbf h\, xyx\, \mathbf t,\\
&\mathbf h\, yx^2\, \mathbf t\stackrel{\Phi_2}\approx \mathbf h\, yx^3\, \mathbf t\stackrel{\eqref{two letters in a block middle 4}}\approx \mathbf h\, x^2yx\, \mathbf t\stackrel{\Phi_2}\approx\mathbf h\, xyx\, \mathbf t,
\end{aligned}
\]
where
\[
\mathbf h:=\biggl(\prod_{i=1}^k a_i^{g_i}t_i\biggr)\ \text{ and }\ \mathbf t:=\biggl(\prod_{i=k+1}^{k+\ell} t_ia_i^{g_i}\biggr).
\]
This means that the identity~\eqref{two letters in a block middle 4} is equivalent modulo $\Phi_2$ to $\{\eqref{two letters in a block middle 2},\,\eqref{two letters in a block middle 3}\}$ as required.
\end{proof}

\begin{lemma}
\label{L: from pxyqxrys to pyxqxrys}
Let $\mathbf V$ be a variety satisfying 
\[
\{\Phi_2,\,\mathbf c_{n,m,k}[\rho] \approx\mathbf c_{n,m,k}^\prime[\rho]\mid n,m,k\in\mathbb N_0,\,\rho\in S_{n+m+k}\}.
\] 
If $\mathbf w:=\mathbf pxy\mathbf qx\mathbf ry\mathbf s$ and $\con(\mathbf r)\subseteq\mul(\mathbf w)$, then $\mathbf V$ satisfies the identity $\mathbf w\approx\mathbf pyx\mathbf qx\mathbf ry\mathbf s$.
\end{lemma}

\begin{proof}
In view of Lemma~\ref{L: from pxx_1..x_kxr to px_1x..x_kxr}, the identities in $\Phi_2$ imply the identity  $\mathbf pxy\mathbf qx\mathbf ry\mathbf s\approx\mathbf pxy\mathbf qx\mathbf r^\prime y\mathbf s$, where $\mathbf r^\prime$ is obtained from $\mathbf r$ by replacing  all multiple letters to their squares.
Further, one can find $n,k,m\in\mathbb N_0$ and $\rho\in S_{n+k+m}$ such that $\mathbf pxy\mathbf qx\mathbf r^\prime y\mathbf s\approx\mathbf pyx\mathbf qx\mathbf r^\prime y\mathbf s$ follows from $\mathbf c_{n,m,k}[\rho]\approx\mathbf c_{n,m,k}^\prime[\rho]$.
Now Lemma~\ref{L: from pxx_1..x_kxr to px_1x..x_kxr} applies again, yielding that $\mathbf pyx\mathbf qx\mathbf r^\prime y\mathbf s\stackrel{\Phi_2}\approx\mathbf pyx\mathbf qx\mathbf r y\mathbf s$.
Hence, $\mathbf V$ satisfies $\mathbf w\approx\mathbf pyx\mathbf qx\mathbf ry\mathbf s$ as required.
\end{proof}

\begin{proof}[Proof of Proposition~\ref{P: Phi,Phi_1,Phi_2 subvarieties}]
Denote by $\Sigma$ the set of all identities of form~\eqref{one letter in a block},~\eqref{two letters in a block middle 1} and~\eqref{two letters in a block middle 2} such that~\eqref{one letter in a block conditions} and~\eqref{two letters in a block middle conditions} hold.
Let $\mathbf V$ be a non-commutative subvariety of 
\[
\mathbf O:=\var\{\Phi,\,\Phi_1,\,\Phi_2\}.
\]
Take an arbitrary identity $\mathbf u \approx \mathbf u^\prime$ of $\mathbf V$.
We need to verify that $\mathbf u \approx \mathbf u^\prime$ is equivalent within $\mathbf O$ to some subset of $\Sigma$.

According to Lemma~\ref{L: nsub M(xy)}, $M(xy)\in\mathbf V$.
Let $t_0\mathbf u_0\cdots t_m\mathbf u_m$ be the decomposition of $\mathbf u$. 
Lemma~\ref{L: identities of M(xy)} implies that the decomposition of $\mathbf u^\prime$ has the form $t_0\mathbf u_0^\prime\cdots t_m\mathbf u_m^\prime$.
According to Lemma~\ref{L: from pxx_1..x_kxr to px_1x..x_kxr}, the identities in $\Phi_2$ can be used to convert the words $\mathbf u$ and $\mathbf u^\prime$ into some words $\mathbf v$ and $\mathbf v^\prime$, respectively, such that the following hold:
\begin{itemize}
\item the decompositions of $\mathbf v$ and $\mathbf v^\prime$ are of the form $t_0\mathbf v_0\cdots t_m\mathbf v_m$ and $t_0\mathbf v_0^\prime\cdots t_m\mathbf v_m^\prime$, respectively;
\item $\occ_x(\mathbf v_i),\occ_x(\mathbf v_i^\prime)\le 2$ for any $x\in\mathfrak X$ and $i=0,\dots,m$.
\end{itemize}
In view of this fact, it suffices to show that the identity $\mathbf v \approx \mathbf v^\prime$ is equivalent within $\mathbf V$ to some subset of $\Sigma$.

Take an arbitrary $x\in\mul(\mathbf v)=\mul(\mathbf v^\prime)$.
Evidently, 
\[
\mathbf v_i=\mathbf p_ix^{c_i}\mathbf q_ix^{d_i}\mathbf r_i\ \text{ and }\  \mathbf v_i^\prime=\mathbf p_i^\prime x^{c_i^\prime}\mathbf q_i^\prime x^{d_i^\prime}\mathbf r_i^\prime
\]
for some words $\mathbf p_i,\mathbf p_i^\prime,\mathbf q_i,\mathbf q_i^\prime,\mathbf r_i,\mathbf r_i^\prime\in \mathfrak X^\ast$ not containing the letter $x$ and
\[
\begin{aligned}
&c_i:=
\begin{cases} 
1 & \text{if $\occ_x(\mathbf v_i)>0$}, \\
0 & \text{if $\occ_x(\mathbf v_i)=0$},
\end{cases}\ \ \ 
d_i:=
\begin{cases} 
1 & \text{if $\occ_x(\mathbf v_i)=2$}, \\
0 & \text{if $\occ_x(\mathbf v_i)\le 1$},
\end{cases}
\\ 
&c_i^\prime:=
\begin{cases} 
1 & \text{if $\occ_x(\mathbf v_i^\prime)>0$}, \\
0 & \text{if $\occ_x(\mathbf v_i^\prime)=0$},
\end{cases}\ \ \ 
d_i^\prime:=
\begin{cases} 
1 & \text{if $\occ_x(\mathbf v_i^\prime)=2$}, \\
0 & \text{if $\occ_x(\mathbf v_i^\prime)\le 1$},
\end{cases}
\end{aligned}
\]
$i=0,\dots,m$. 
According to Lemma~\ref{L: from pxx_1..x_kxr to px_1x..x_kxr}, the identity
\begin{equation}
\label{v(x,t_1,..,t_m)=v'(x,t_1,..,t_m)}
\mathbf v(x,t_1,\dots,t_m)\approx \mathbf v^\prime(x,t_1,\dots,t_m)
\end{equation}
together with $\Phi_2$ imply the identities
\[
\begin{aligned}
\mathbf v&{}\stackrel{\Phi_2}\approx \biggl(\prod_{i=0}^mt_i\mathbf p_ix^{\occ_x(\mathbf v_i)+d_i}\mathbf q_ix^{d_i}\mathbf r_i\biggr)\stackrel{\eqref{v(x,t_1,..,t_m)=v'(x,t_1,..,t_m)}}\approx \biggl(\prod_{i=0}^mt_i\mathbf p_ix^{\occ_x(\mathbf v_i^\prime)+d_i}\mathbf q_ix^{d_i}\mathbf r_i\biggr)
\stackrel{\Phi_2}\approx \biggl(\prod_{i=0}^mt_i\mathbf p_ix^{h_i}\mathbf q_ix^{d_i}\mathbf r_i\biggr),\\
\mathbf v^\prime&{}\stackrel{\Phi_2}\approx \biggl(\prod_{i=0}^mt_i\mathbf p_i^\prime x^{\occ_x(\mathbf v_i^\prime)}\mathbf q_i^\prime x^{d_i^\prime}\mathbf r_i^\prime\biggr)
\stackrel{\eqref{v(x,t_1,..,t_m)=v'(x,t_1,..,t_m)}}\approx \biggl(\prod_{i=0}^mt_i\mathbf p_i^\prime x^{\occ_x(\mathbf v_i)}\mathbf q_i^\prime x^{d_i^\prime}\mathbf r_i\biggr)\\
&{}\stackrel{\Phi_2}\approx \biggl(\prod_{i=0}^mt_i\mathbf p_i^\prime x^{\occ_x(\mathbf v_i)+2d_i}\mathbf q_i^\prime x^{d_i^\prime}\mathbf r_i\biggr)
\stackrel{\eqref{v(x,t_1,..,t_m)=v'(x,t_1,..,t_m)}}\approx\biggl(\prod_{i=0}^mt_i\mathbf p_i^\prime x^{\occ_x(\mathbf v_i^\prime)+2d_i}\mathbf q_i^\prime x^{d_i^\prime}\mathbf r_i\biggr)
\stackrel{\Phi_2}\approx\biggl(\prod_{i=0}^mt_i\mathbf p_i^\prime x^{h_i^\prime}\mathbf q_i^\prime x^{d_i^\prime}\mathbf r_i\biggr),
\end{aligned}
\]
where
\[
h_i:=
\begin{cases} 
\occ_x(\mathbf v_i^\prime) & \text{if $d_i=0$}, \\
1 & \text{if $d_i=1$},
\end{cases}\ \ \ 
h_i^\prime:=
\begin{cases} 
\occ_x(\mathbf v_i^\prime) & \text{if $d_i=d_i^\prime=0$}, \\
1 & \text{if $d_i^\prime=1$},\\
2 & \text{if $d_i^\prime=0$, $d_i=1$},
\end{cases}
\]
$i=0,\dots,m$. 
Notice also that $h_i+d_i=h_i^\prime+d_i^\prime$ for all $i=0,\dots,m$.
Since the letter $x$ is arbitrary, the set 
\begin{equation}
\label{set of rigid identities}
\{\mathbf v(x,t_1,\dots,t_m)\approx \mathbf v^\prime(x,t_1,\dots,t_m)\mid x\in \mul(\mathbf v)=\mul(\mathbf v^\prime)\}
\end{equation}
of identities together with $\Phi_2$ can be used to convert the words $\mathbf v$ and $\mathbf v^\prime$ into some words $\mathbf w$ and $\mathbf w^\prime$, respectively, such that the following hold:
\begin{itemize}
\item the decompositions $\mathbf w$ and $\mathbf w^\prime$ have the form $t_0\mathbf w_0\cdots t_m\mathbf w_m$ and $t_0\mathbf w_0^\prime\cdots t_m\mathbf w_m^\prime$, respectively;
\item $\occ_x(\mathbf w_i)=\occ_x(\mathbf w_i^\prime)\le2$ for any $x\in\mathfrak X$ and $i=0,\dots,m$.
\end{itemize}
Evidently, every identity from~\eqref{set of rigid identities} is of the form~\eqref{one letter in a block}.
Hence $\mathbf V\{\mathbf u \approx \mathbf u^\prime\}=\mathbf V\{\Gamma,\,\mathbf w\approx \mathbf w^\prime\}$ for some $\Gamma\subseteq\Sigma$.
In view of this fact, it remains to show that the identity $\mathbf w \approx \mathbf w^\prime$ is equivalent within $\mathbf V$ to a subset of $\Sigma$.

We call an identity $\mathbf c\approx\mathbf d$ 1-\textit{invertible} if $\mathbf c=\mathbf e^\prime\, xy\,\mathbf e^{\prime\prime}$ and $\mathbf d=\mathbf e^\prime\, yx\,\mathbf e^{\prime\prime}$ for some words $\mathbf e^\prime,\mathbf e^{\prime\prime}\in\mathfrak X^\ast$ and letters $x,y\in\con(\mathbf e^\prime\mathbf e^{\prime\prime})$. 
Let $j>1$. 
An identity $\mathbf c\approx\mathbf d$ is called $j$-\textit{invertible} if there is a sequence of words $\mathbf c=\mathbf w_0,\dots,\mathbf w_j=\mathbf d$ such that the identity $\mathbf w_i\approx\mathbf w_{i+1}$ is 1-invertible for each $i=0,\dots,j-1$ and $j$ is the least number with such a property. 
For convenience, we will call the trivial identity 0-\textit{invertible}. 

Notice that the identity $\mathbf w \approx \mathbf w^\prime$ is $r$-invertible for some $r\in\mathbb N_0$ because $\occ_x(\mathbf w_i)=\occ_x(\mathbf w_i^\prime)$ for any $x\in\mathfrak X$ and $i=0,\dots,m$. 
We will use induction by $r$.

\smallskip

\textit{Induction base}. 
If $r=0$, then $\mathbf w=\mathbf w^\prime$, whence $\mathbf V\{\mathbf w\approx\mathbf w^\prime\}=\mathbf V\{\varnothing\}$.

\smallskip

\textit{Induction step}. 
Let $r>0$. 
Obviously, $\mathbf w_s\ne\mathbf w_s^\prime$ for some $s\in\{0,\dots,m\}$. 
Then there are letters $x$ and $y$ such that $\mathbf w_s=\mathbf a_s\, {_{q\mathbf w_s}}y\ {_{p\mathbf w_s}}x\,\mathbf b_s$ for some $p,q\in\{1,2\}$ and $\mathbf a_s,\mathbf b_s\in\mathfrak X^\ast$, while the letter ${_{p\mathbf w_s^\prime}}x$ precedes the letter ${_{q\mathbf w_s^\prime}}y$ in $\mathbf w_s^\prime$. 
We denote by $\hat{\mathbf w}$ the word obtained from $\mathbf w$ by swapping of the letters ${_{p\mathbf w_s}}x$ and ${_{q\mathbf w_s}}y$ in the block $\mathbf w_s$. 

To complete the proof, it suffices to show that $\mathbf w \approx \hat{\mathbf w}$ holds in $\mathbf O\{\mathbf w \approx \mathbf w^\prime\}$ and $\mathbf O\{\mathbf w \approx \hat{\mathbf w}\}=\mathbf O\Gamma^\prime$ for some $\Gamma^\prime\subseteq\Sigma$.
Indeed, in this case, the identity $\hat{\mathbf w}\approx\mathbf w^\prime$ is \mbox{$(r-1)$}-invertible. 
By the induction assumption, $\mathbf O\{\hat{\mathbf w}\approx\mathbf w\}=\mathbf O\Gamma^{\prime\prime}$ for some $\Gamma^{\prime\prime}\subseteq\Sigma$, whence  
\[
\mathbf O\{\mathbf w\approx\mathbf w^\prime\}=\mathbf O\{\mathbf w\approx \hat{\mathbf w},\hat{\mathbf w}\approx \mathbf w^\prime\}=\mathbf O\{\Gamma^\prime,\,\Gamma^{\prime\prime}\},
\] and we are done.

Suppose that $\occ_x(\mathbf w_s),\occ_y(\mathbf w_s)\ge2$.
Then, using Lemma~\ref{L: from pxx_1..x_kxr to px_1x..x_kxr}, we obtain that $\mathbf O$ satisfies the identities
\[
\mathbf w= \mathbf a\, yx\,\mathbf b\stackrel{\Phi_2}\approx \mathbf a\, y^2x^2\,\mathbf b \stackrel{\Phi}\approx \mathbf a\, x^2y^2\,\mathbf b\stackrel{\Phi_2}\approx \mathbf a\, xy\,\mathbf b=\hat{\mathbf w},
\]
where
\[
\mathbf a:=\biggl(\prod_{i=0}^{s-1}t_i\mathbf w_i\biggr)\,t_s\mathbf a_s \ \text{ and } \ \mathbf b:=\mathbf b_s\,\biggl(\prod_{i=s+1}^mt_i\mathbf w_i\biggr)
\]
So, it remains to consider the case when either $\occ_x(\mathbf w_s)=1$ or $\occ_y(\mathbf w_s)=1$.
By symmetry, we may assume that $\occ_y(\mathbf w_s)=1$.

If $x,y\in\con(\mathbf w_{s^\prime})=\con(\mathbf w_{s^\prime})$ for some $s^\prime\ne s$, then Lemma~\ref{L: from pxyqxrys to pyxqxrys} or the statement dual to it implies that $\mathbf O$ satisfies the identity $\mathbf w\approx\hat{\mathbf w}$. 
Thus, we may further assume that at most one of the letters $x$ and $y$ occurs in $\con(\mathbf w_i)$ for any $i\ne s$.
Denote this letter by $a_i$ (if $x,y\notin\con(\mathbf w_i)$, then $a_i$ denote the empty word).
Then $\mathbf w_i=\mathbf a_ia_i^{b_i}\mathbf b_ia_i^{b_i^\prime}\mathbf c_i$ for some $\mathbf a_i,\mathbf b_i,\mathbf c_i\in \mathfrak X^\ast$ and
\[
b_i:=
\begin{cases} 
1 & \text{if $\occ_{a_i}(\mathbf w_i)>0$}, \\
0 & \text{if $\occ_{a_i}(\mathbf w_i)=0$},
\end{cases}\ \ \ 
b_i^\prime:=
\begin{cases} 
1 & \text{if $\occ_{a_i}(\mathbf w_i)=2$}, \\
0 & \text{if $\occ_{a_i}(\mathbf w_i)\le 1$},
\end{cases}
\]
$i=0,\dots,s-1,s+1,\dots,m$.

Further, the identity $\mathbf w(x,y,t_1,\dots,t_m)\approx \mathbf w^\prime(x,y,t_1,\dots,t_m)$ coincides with either the identity
\begin{equation}
\label{two letters in a block middle mod}
\begin{aligned}
{}&\biggl(\prod_{i=0}^{s-1} t_ia_i^{\occ_{a_i}(\mathbf w_i)}\biggr)\, t_sx^{\occ_x(\mathbf a_s)}yx^{1+\occ_x(\mathbf b_s)}\, \biggl(\prod_{i=s+1}^m t_i a_i^{\occ_{a_i}(\mathbf w_i)}\biggr)\\
\approx{}&\biggl(\prod_{i=0}^{s-1} t_ia_i^{\occ_{a_i}(\mathbf w_i)}\biggr)\, t_sx^{1+\occ_x(\mathbf a_s)}yx^{\occ_x(\mathbf b_s)}  \biggl(\prod_{i=s+1}^m t_ia_i^{\occ_{a_i}(\mathbf w_i)}\biggr)
\end{aligned}
\end{equation}
or the identity
\[
\biggl(\prod_{i=0}^{s-1} t_ia_i^{\occ_{a_i}(\mathbf w_i)}\biggr)\, t_syx^2\, \biggl(\prod_{i=s+1}^m t_i a_i^{\occ_{a_i}(\mathbf w_i)}\biggr)
\approx\biggl(\prod_{i=0}^{s-1} t_ia_i^{\occ_{a_i}(\mathbf w_i)}\biggr)\, t_sx^2y \biggl(\prod_{i=s+1}^m t_ia_i^{\occ_{a_i}(\mathbf w_i)}\biggr).
\]
In view of Lemma~\ref{L: two letters in a block middle 2,3 <-> 4}, the latter identity together with $\Phi_2$ imply~\eqref{two letters in a block middle mod}.
Therefore, $\{\Phi_2,\,\mathbf w(x,y,t_1,\dots,t_m)\approx \mathbf w^\prime(x,y,t_1,\dots,t_m)\}$ implies~\eqref{two letters in a block middle mod} in either case.
Clearly, the identity $\mathbf w(x,y,t_1,\dots,t_m)\approx \hat{\mathbf w}(x,y,t_1,\dots,t_m)$ is nothing but the identity~\eqref{two letters in a block middle mod}. 
By Lemma~\ref{L: from pxx_1..x_kxr to px_1x..x_kxr},
\[
\begin{aligned}
\mathbf w=\mathbf a\, yx\,\mathbf b\stackrel{\Phi_2}\approx\mathbf a^\prime\, x^{\occ_x(\mathbf a_s)}yx^{1+\occ_x(\mathbf b_s)}\,\mathbf b^\prime
\stackrel{\eqref{two letters in a block middle mod}}\approx\mathbf a^\prime\, x^{1+\occ_x(\mathbf a_s)}yx^{\occ_x(\mathbf b_s)}\,\mathbf b^\prime
\stackrel{\Phi_2}\approx\mathbf a\, xy\,\mathbf b=\hat{\mathbf w},
\end{aligned}
\]
where
\[
\mathbf a^\prime:=\biggl(\prod_{i=0}^{s-1}t_i\mathbf a_ia_i^{\occ_{a_i}(\mathbf w_i)}\mathbf b_ia_i^{b_i^\prime}\mathbf c_i\biggr)\,t_s\mathbf a_s \ \text{ and } \ \mathbf b^\prime:=\mathbf b_s\,\biggl(\prod_{i=s+1}^mt_i\mathbf a_ia_i^{\occ_{a_i}(\mathbf w_i)}\mathbf b_ia_i^{b_i^\prime}\mathbf c_i\biggr).
\]
Therefore, the identity~\eqref{two letters in a block middle mod} is equivalent within $\mathbf O$ to the identity $\mathbf w\approx \hat{\mathbf w}$. 
It remains to notice that the identity~\eqref{two letters in a block middle mod} is of the form~\eqref{two letters in a block middle 1}--\eqref{two letters in a block middle 3} such that~\eqref{two letters in a block middle conditions} holds.
Finally, it is easy to see that every identity of the form~\eqref{one letter in a block} or~\eqref{two letters in a block middle 1}--\eqref{two letters in a block middle 3} is equivalent to an efficient identity of the same form.
\end{proof}

\subsection{Identities formed by words with one multiple letter}

Recall that $\mathbf A$ is the variety defined by the identities $x^2\approx x^3$ and~\eqref{xxyx=xxyxx}.
In this subsection, we find a minimal list of identities of the form~\eqref{one letter in a block} which can be used to define the subvarieties of $\mathbf A$.

The following statement readily follows from Lemma~\ref{L: nsub M(yxx^+)}, the claim dual to it and Lemma~\ref{L: nsub M(x^+yzx^+)}.

\begin{lemma}
\label{L: zero exponent}
Let $r,e_0,f_0,e_1,f_1,\dots,e_r,f_r\in\mathbb N_0$ with $\sum_{i=0}^re_i,\sum_{i=0}^rf_i\ge2$.
\begin{itemize}
\item[\textup{(i)}] If one of the numbers $e_0$ or $f_0$ is zero, while the other one is not, then $\mathbf A\{\eqref{one letter in a block}\}$ satisfies the identity~\eqref{yxx=xyxx}.
\item[\textup{(ii)}] If one of the numbers $e_r$ or $f_r$ is zero, while the other one is not, then $\mathbf A\{\eqref{one letter in a block}\}$ satisfies the identity~\eqref{xxy=xxyx}.
\item[\textup{(iii)}] If, for some $k\in\{1,\dots,r-1\}$, one of the numbers $e_k$ or $f_k$ is zero, while the other one is not, then $\mathbf A\{\eqref{one letter in a block}\}$ satisfies the identity~\eqref{xxyzx=xxyxzx}.\qed
\end{itemize}
\end{lemma}

For any $n\in\mathbb N$, we fix notation for the following identities:
\[
\begin{aligned}
\beta_n&:\enskip xt_1x\cdots t_nx\approx xt_1x\cdots xt_nx^2,\\
\gamma_n&:\enskip xt_1x\cdots  t_{n-1}x^2t_nx\approx xt_1x\cdots t_{n-1}xt_nx^2.
\end{aligned}
\]
For the rest of the article, let
\[
\Delta:=\{\eqref{yxx=xyxx},\,\eqref{xxy=xxyx},\,\eqref{xxyzx=xxyxzx},\,\beta_n,\gamma_n\mid n\in\mathbb N\}.
\]

\begin{lemma}
\label{L: one letter in block reduction}
Let $r,e_0,f_0,e_1,f_1,\dots,e_r,f_r\in\mathbb N_0$  with $\sum_{i=0}^re_i,\sum_{i=0}^rf_i\ge2$.
Then the identity~\eqref{one letter in a block} is equivalent within $\mathbf A$ to some identities in $\Delta$. 
\end{lemma}

\begin{proof}
Evidently, any identity of the form~\eqref{one letter in a block} is equivalent to an efficient identity of the same form.
Further, it is easy to see that any identity of the form~\eqref{one letter in a block} is equivalent modulo $\{x^2\approx x^3,\,\eqref{xxyx=xxyxx}\}$ to an identity of the same form such that both hand sides of this identity have at most one island formed by $x$ of length $>1$.
Thus, we may further assume that the identity~\eqref{one letter in a block} is efficient and both hand sides of this identity have at most one island $x^2$, while the other islands formed by $x$ are of length $1$.
If the identity~\eqref{one letter in a block} is trivial, then there is nothing to prove. 
So, we assume below that the identity~\eqref{one letter in a block} is non-trivial. 
For convenience, denote the left-hand [right-hand] side of the identity~\eqref{one letter in a block} by $\mathbf u$ [respectively, $\mathbf v$]. 
Further considerations are divided into two cases.

\smallskip

\textbf{Case 1}: $e_i,f_i>0$ for all $i=0,\dots,r$.
Since the identity~\eqref{one letter in a block} is non-trivial and efficient, either $\mathbf u$ or $\mathbf v$ must contain an island of the form $x^2$.
By symmetry, we may assume without loss of generality that $f_k=2$ for some $k$.
Then $f_i=1$ for all $i=0,\dots,k-1,k+1,\dots,r$.

Assume that $e_i=1$ for all $i=0,\dots,r$.
If $k=r$, then~\eqref{one letter in a block} is nothing but $\beta_r$, and we are done.
Suppose now that $k<r$.
Then
\[
\begin{aligned}
&\mathbf u(x,t_1,\dots,t_{k+1})x=xt_1x\cdots  t_kxt_{k+1}x^{r-k+1},\\
&\mathbf v(x,t_1,\dots,t_{k+1})x=xt_1x\cdots  t_kx^2t_{k+1}x^{r-k+1}.
\end{aligned}
\]
Hence $\mathbf u(x,t_1,\dots,t_{k+1})x\approx \mathbf v(x,t_1,\dots,t_{k+1})x$ is equivalent modulo $\{x^2\approx x^3,\,\eqref{xxyx=xxyxx}\}$ to $\gamma_{k+1}$.
Then $\beta_r$ follows from~\eqref{one letter in a block} because
\[
xt_1x\cdots t_rx\stackrel{\eqref{one letter in a block}}\approx xt_1x\cdots t_kx^2t_{k+1}x\cdots t_rx\stackrel{\gamma_{k+1}}\approx xt_1x\cdots t_{r-1}xt_rx^2.
\]
Finally, since
\[
\mathbf u=xt_1x\cdots t_rx\stackrel{\beta_r}\approx xt_1x\cdots t_{r-1}xt_rx^2\stackrel{\gamma_{k+1}}\approx xt_1x\cdots t_kx^2t_{k+1}x\cdots t_rx=\mathbf v,
\]
we have $\mathbf A\{\eqref{one letter in a block}\}=\mathbf A\{\beta_r,\,\gamma_{k+1}\}$.

Assume now that $e_m=2$ for some $m$.
Then $e_i=1$ for all $i=0,\dots,m-1,m+1,\dots,r$.
Since the identity~\eqref{one letter in a block} is non-trivial and efficient, $k\ne m$.
We may assume without any loss that $m<k$.
Then
\[
\begin{aligned}
&\mathbf u(x,t_1,\dots,t_{m+1})=xt_1x\cdots  t_{m-1}xt_mx^2t_{m+1}x^{r-m},\\
&\mathbf v(x,t_1,\dots,t_{m+1})=xt_1x\cdots  t_mxt_{m+1}x^{r-m+1}.
\end{aligned}
\]
Hence $\mathbf u(x,t_1,\dots,t_{m+1})\approx \mathbf v(x,t_1,\dots,t_{m+1})$ is equivalent modulo $\{x^2\approx x^3,\,\eqref{xxyx=xxyxx}\}$ to $\gamma_{m+1}$.
Clearly, $\gamma_{m+1}$ implies~\eqref{one letter in a block}, whence $\mathbf A\{\eqref{one letter in a block}\}=\mathbf A\{\gamma_{m+1}\}$.

\smallskip

\textbf{Case 2}: at least one of the exponents  $e_0,f_0,\dots,e_r,f_r$ is zero. 
In this case, $r>0$. 
We will use induction on $r$.
There are two subcases.

\textbf{Case 2.1}: $e_0,f_0>0$.
Let $k$ be the least number such that either $e_k$ or $f_k$ is zero.
By symmetry, we may assume that $e_k=0$.

\textbf{Induction base}: $r=1$.
Then $\mathbf u=x^2t_1$ and $f_0,f_1>0$.
According to Lemma~\ref{L: zero exponent}(ii), the variety $\mathbf A\{\eqref{one letter in a block}\}$ satisfies the identity~\eqref{xxy=xxyx}.
By Case~1, $\mathbf A\{\mathbf ux\approx \mathbf v\}=\mathbf A\Delta^\prime$ for some $\Delta^\prime\subseteq\Delta$.
Then
\[
\mathbf A\{\eqref{one letter in a block}\}=\mathbf A\{\eqref{one letter in a block},\,\eqref{xxy=xxyx}\}=\mathbf A\{\mathbf ux\approx \mathbf v,\,\eqref{xxy=xxyx}\}=\mathbf A\{\Delta^\prime,\,\eqref{xxy=xxyx}\}
\]
as required.

\textbf{Induction step}: $r>1$.
We have
\[
\begin{aligned}
&\mathbf u_{t_k}=x^{e_0}t_1x^{e_1}\cdots  t_{k-1}x^{e_{k-1}}t_{k+1}x^{e_{k+1}}\cdots t_rx^{e_r},\\
&\mathbf v_{t_k}=x^{f_0}t_1x^{f_1}\cdots  t_{k-1}x^{f_{k-1}+f_k}t_{k+1}x^{f_{k+1}}\cdots t_rx^{f_r}.
\end{aligned}
\]
By the induction assumption or Case~1, $\mathbf A\{\mathbf u_{t_k}\approx \mathbf v_{t_k}\}=\mathbf A\Delta_1^\prime$ for some $\Delta_1^\prime\subseteq\Delta$.
Let $\sigma$ denote the identity~\eqref{xxyzx=xxyxzx} whenever $e_r,f_r>0$ or the identity~\eqref{xxy=xxyx} otherwise.
In view of Lemma~\ref{L: zero exponent}(ii),(iii), the identity $\sigma$ holds in the variety $\mathbf A\{\eqref{one letter in a block}\}$.
Put $\Delta_1:=\Delta_1^\prime\cup\{\sigma\}$.
Then $\mathbf A\Delta_1$ and so $\mathbf A\{\eqref{one letter in a block}\}$ satisfy the identities
\[
\begin{aligned}
\mathbf u&{}\approx x^{f_0}t_1x^{f_1}\cdots  t_{k-1}x^{f_{k-1}+f_k}t_kt_{k+1}x^{f_{k+1}}\cdots t_rx^{f_r}&&\textup{by }\mathbf u_{t_k}\approx \mathbf v_{t_k}\\
&\approx x^{f_0}t_1x^{f_1}\cdots  t_{k-1}x^2t_kx\cdots t_rx&&\textup{by }\sigma\\
&=:\mathbf u^\prime.
\end{aligned}
\]
Now consider the identity $\mathbf u^\prime\approx \mathbf v$.
Arguments similar to the above ones show that there is a subset $\Delta_2$ of $\Delta$ holding in $\mathbf A\{\mathbf u^\prime\approx \mathbf v\}$ such that $\mathbf A\Delta_2$ satisfies an identity 
\[
\mathbf v\approx x^{f_0^\prime}t_1x^{f_1^\prime}\cdots t_rx^{f_r^\prime} =:\mathbf v^\prime,
\]
where $f_0^\prime,\dots,f_r^\prime>0$.
In view of the above, $\mathbf u^\prime\approx \mathbf v$ is satisfied by $\mathbf A\{\eqref{one letter in a block}\}$, whence $\Delta_2$ and so $\mathbf u^\prime\approx\mathbf v^\prime$ hold in $\mathbf A\{\mathbf u\approx \mathbf v\}$.
Finally, by Case~1, $\mathbf A\{\mathbf u^\prime\approx\mathbf v^\prime\}=\mathbf A\Delta_3$ for some $\Delta_3\subseteq\Delta$.
We see that $\mathbf A\{\eqref{one letter in a block}\}\subseteq\mathbf A\{\Delta_1,\Delta_2,\Delta_3\}$.
In fact, the reverse inclusion is valid because $\mathbf u\stackrel{\mathbf A\Delta_1}\approx\mathbf u^\prime\stackrel{\mathbf A\Delta_3}\approx\mathbf v^\prime\stackrel{\mathbf A\Delta_2}\approx\mathbf v$.
Hence $\mathbf A\{\eqref{one letter in a block}\}=\mathbf A\{\Delta_1,\Delta_2,\Delta_3\}$, and we are done.

\textbf{Case 2.2}: at least one of the numbers $e_0$ or $f_0$ is zero.
In this case, $\mathbf A\{\eqref{one letter in a block}\}$ satisfies~\eqref{yxx=xyxx} by Lemma~\ref{L: zero exponent}(i).
By symmetry, we may assume without any loss that $e_0=0$.
Then $f_0>0$ since the identity~\eqref{one letter in a block} is efficient.

\textbf{Induction base}: $r=1$.
Then $\mathbf u=t_1x^2$.
By Case~1 or Case~2.1, $\mathbf A\{x\mathbf u\approx \mathbf v\}=\mathbf A\Delta^\prime$ for some $\Delta^\prime\subseteq\Delta$.
Then
\[
\mathbf A\{\eqref{one letter in a block}\}=\mathbf A\{\eqref{one letter in a block},\,\eqref{yxx=xyxx}\}=\mathbf A\{x\mathbf u\approx \mathbf v,\,\eqref{yxx=xyxx}\}=\mathbf A\{\Delta^\prime,\,\eqref{yxx=xyxx}\}
\]
as required.

\textbf{Induction step}: $r>1$.
Suppose that $f_1>0$.
Then $f_0+f_1\ge2$ and
\[
\mathbf u_{t_1}=x^{e_1}t_2x^{e_2}\cdots t_rx^{e_r},\ \ \ \mathbf v_{t_1}=x^{f_0+f_1}t_2x^{f_2}\cdots t_rx^{f_r}.
\]
By the induction assumption or Case~1 or Case~2.1, $\mathbf A\{\mathbf u_{t_1}\approx \mathbf v_{t_1}\}=\mathbf A\Delta_1$ for some $\Delta_1\subseteq\Delta$.
Clearly, $\mathbf u_{t_1}\approx \mathbf v_{t_1}$ together with~\eqref{yxx=xyxx} imply
\[
\mathbf u\approx xt_1x^2t_2x^{f_2}\cdots t_rx^{f_r}=:\mathbf u^\prime.
\]
According to Case~1 or Case~2.1, $\mathbf A\{\mathbf u^\prime\approx\mathbf v\}=\mathbf A\Delta_2$ for some $\Delta_2\subseteq\Delta$.
Then
\[
\mathbf A\{\eqref{one letter in a block}\}=\mathbf A\{\eqref{one letter in a block},\,\mathbf u_{t_1}\approx \mathbf v_{t_1},\,\eqref{yxx=xyxx}\}=\mathbf A\{\mathbf u^\prime\approx\mathbf v,\,\mathbf u_{t_1}\approx \mathbf v_{t_1},\,\eqref{yxx=xyxx} \}=\mathbf A\{\Delta_1,\Delta_2,\,\eqref{yxx=xyxx} \},
\]
and we are done.

Suppose now that $f_1=0$.
Then
\[
\mathbf u_{t_2}=t_1x^{e_1+e_2}t_3x^{e_3}\cdots t_rx^{e_r},\ \ \ 
\mathbf v_{t_2}=x^{f_0}t_1x^{f_2}t_3x^{f_3}\cdots t_rx^{f_r}.
\]
By the induction assumption, $\mathbf A\{\mathbf u_{t_2}\approx \mathbf v_{t_2}\}=\mathbf A\Delta_3$ for some $\Delta_3\subseteq\Delta$.
Clearly, $\mathbf u_{t_2}\approx \mathbf v_{t_2}$ implies
\[
\mathbf v\approx t_1t_2x^{e_1+e_2}t_3x^{e_3}\cdots t_rx^{e_r}=:\mathbf v^\prime.
\]
By the induction assumption, $\mathbf A\{\mathbf u_{t_1}\approx \mathbf v_{t_1}^\prime\}=\mathbf A\Delta_4$ for some $\Delta_4\subseteq\Delta$.
Since $\mathbf u\approx \mathbf v^\prime$ is equivalent to $\mathbf u_{t_1}\approx \mathbf v_{t_1}^\prime$, we have
\[
\begin{aligned}
\mathbf A\{\eqref{one letter in a block}\}&{}=\mathbf A\{\eqref{one letter in a block},\,\mathbf u_{t_2}\approx \mathbf v_{t_2}\}=\mathbf A\{\mathbf u\approx\mathbf v^\prime,\,\mathbf u_{t_2}\approx \mathbf v_{t_2}\}\\
&{}=\mathbf A\{\mathbf u_{t_1}\approx \mathbf v_{t_1}^\prime,\,\mathbf u_{t_2}\approx \mathbf v_{t_2}\}=\mathbf A\{\Delta_3,\Delta_4\},
\end{aligned}
\]
and we are done.
\end{proof}

Put $\mu_n:=\FIC(\var\{x^2\approx x^3,\,\gamma_n\})$ for any $n\in\mathbb N$.
Notice that $\mu_2$ is nothing but $\mu$.

\begin{corollary}
\label{C: nsub M(xt_1x...t_nx),M(xt_1x...t_nx^+)}
Let $\mathbf V$ be a subvariety of $\mathbf A$ such that $M(xy)\in \mathbf V$.
\begin{itemize}
\item[\textup{(i)}] If $M(xt_1x\cdots t_nx)\notin \mathbf V$, then $\mathbf V$ satisfies the identity $\beta_n$.
\item[\textup{(ii)}] If $M_{\mu_n}(xt_1x\cdots t_nx^+)\notin \mathbf V$, then $\mathbf V$ satisfies the identity $\gamma_n$.
\end{itemize}
\end{corollary}

\begin{proof}
(i) In view of Lemma~\ref{L: M(W) in V}, the word $xt_1x\cdots t_nx$ is not an isoterm for $\mathbf V$.
This means that $\mathbf V$ satisfies a non-trivial identity $xt_1x\cdots t_nx\approx \mathbf w$ for some $\mathbf w\in \mathfrak X^\ast$.
The inclusion $\mathbf M(xy)\subseteq \mathbf V$ and Lemma~\ref{L: identities of M(xy)} imply that $\mathbf w(\simple(\mathbf w))=t_1\cdots t_n$ and $\mul(\mathbf w)=\{x\}$.
Then, by Lemma~\ref{L: one letter in block reduction},
$
\mathbf A\{xt_1x\cdots t_nx\approx \mathbf w\}=\mathbf A\Delta^\prime
$
for some $\Delta^\prime\subseteq\Delta$.
It is easy to see that $xt_1x\cdots t_nx$ is an isoterm for $\mathbf A\wedge\var\,(\Delta\setminus\{\beta_1,\dots,\beta_n\})$.
Hence $\beta_k\in\Delta^\prime$ for some $k\le n$.
Clearly, this identity implies $\beta_n$, and we are done.

\smallskip 

(ii) Using Lemma~\ref{L: le_alpha}, it is routine to check that
\[
\{xt_1x\cdots t_nx^+\}^{\le_{\mu_n}}=
\left\{\!\!\!\!
\begin{array}{l}
1,  x, x^et_kx\cdots t_mx^f, x^et_kx\cdots t_nxx^+,\\
t_1x\cdots t_nx, xt_1x\cdots t_{n-1}xt_n,\\
t_1x\cdots t_nxx^+, xt_1x\cdots t_nx^+
\end{array}
\middle\vert
\begin{array}{l}
e,f\in\{0,1\},\\ 
1\le k\le m\le n,\\ 
m-k<n
\end{array}
\!\!\!\!\right\}.
\]
One can deduce from this fact that the $\mu_n$-class $xt_1x\cdots t_nx^+$ is stable with respect to a variety if and only if every $\mu_n$-class in $\{xt_1x\cdots t_nx^+\}^{\le_{\mu_n}}$ is stable respect to this variety.
This fact and Lemma~\ref{L: M_alpha(W) in V} imply that the $\mu_n$-class $xt_1x\cdots t_nx^+$ is not stable with respect to $\mathbf V$.
This means that $\mathbf V$ satisfies a non-trivial identity $\mathbf u\approx \mathbf v$  such that $\mathbf u\in xt_1x\cdots t_nx^+$ but  $\mathbf v\notin xt_1x\cdots t_nx^+$.
The inclusion $\mathbf M(xy)\subseteq \mathbf V$ and Lemma~\ref{L: identities of M(xy)} imply that $\mathbf w(\simple(\mathbf v))=t_1\cdots t_n$ and $\mul(\mathbf w)=\{x\}$.
Then, by Lemma~\ref{L: one letter in block reduction}, $\mathbf A\{\mathbf u\approx \mathbf v\}=\mathbf A\Delta^\prime$ for some $\Delta^\prime\subseteq\Delta$.
It is routine to check that the set $xt_1x\cdots t_nx^+$ is stable with respect to $\mathbf A\wedge\var\,(\Delta\setminus\{\beta_1,\dots,\beta_{n-1},\gamma_1,\dots,\gamma_n\})$.
Hence the set $\Delta^\prime\cap \{\beta_1,\dots,\beta_{n-1},\gamma_1,\dots,\gamma_n\}$ is not empty.
This implies that $\gamma_n$ is a consequence of $\Delta^\prime$ and so holds in $\mathbf V$.
\end{proof}

\begin{corollary}
\label{C: Psi in X wedge Y}
Let $\mathbf X$ and $\mathbf Y$ be subvarieties of $\mathbf A\{\eqref{xxyy=yyxx}\}$.
If the variety $\mathbf X\wedge\mathbf Y$ satisfies an identity $\sigma\in\Delta$, then $\sigma$ holds in either $\mathbf X$ or $\mathbf Y$.
\end{corollary}

\begin{proof}
If $M(xy)\notin\mathbf X$, then $\mathbf X$ is commutative by Lemma~\ref{L: nsub M(xy)}. 
Then $\mathbf X$ satisfies all the identities in $\Delta$ (and, in particular, $\sigma$) because $\Delta$ is a consequence of $\{x^2\approx x^3,\,xy\approx yx\}$.
By a similar argument we can show that if $M(xy)\notin\mathbf Y$, then $\sigma$ holds in $\mathbf Y$.
So, we may further assume that $M(xy)\in\mathbf X\wedge\mathbf Y$.

Assume that the identity $\sigma$ coincides with the identity~\eqref{xxyzx=xxyxzx}.
Then $M_\gamma(x^+yzx^+)\notin \mathbf X\wedge\mathbf Y$ by Lemma~\ref{L: M_alpha(W) in V}.
This means that either $\mathbf X$ or $\mathbf Y$ does not contain $M_\gamma(x^+yzx^+)$.
Now Lemma~\ref{L: nsub M(x^+yzx^+)} applies, yielding that $\sigma$ holds in either $\mathbf X$ or $\mathbf Y$.
Using Lemma~\ref{L: nsub M(yxx^+)} or the dual to it instead of Lemma~\ref{L: nsub M(x^+yzx^+)}, by similar arguments we can show that $\sigma$ holds in either $\mathbf X$ or $\mathbf Y$ whenever $\sigma\in\{\eqref{yxx=xyxx},\,\eqref{xxy=xxyx}\}$.

Further, if $\sigma=\beta_n$ for some $n\in\mathbb N$, then $M(xt_1x\cdots t_nx)\notin\mathbf X\wedge\mathbf Y$ by Lemma~\ref{L: M(W) in V}.
This means that either $\mathbf X$ or $\mathbf Y$ does not contain $M(xt_1x\cdots t_nx)$.
Now Corollary~\ref{C: nsub M(xt_1x...t_nx),M(xt_1x...t_nx^+)}(i) applies, yielding that $\sigma$ holds in either $\mathbf X$ or $\mathbf Y$.
Finally, using Lemma~\ref{L: M_alpha(W) in V} instead of Lemma~\ref{L: M(W) in V} and Part~(ii) instead of Part~(i) in Corollary~\ref{C: nsub M(xt_1x...t_nx),M(xt_1x...t_nx^+)}, one can show that $\sigma$ holds in either $\mathbf X$ or $\mathbf Y$ whenever $\sigma=\gamma_n$.

Corollary~\ref{C: Psi in X wedge Y} is thus proved.
\end{proof}

\subsection{Identities formed by words with two multiple letters}

\begin{lemma}
\label{L: xy.. = yx.. in V}
Let $\mathbf V$ be a subvariety of $\mathbf A\{\sigma_3,\,\beta_2,\,\eqref{xyzxxtyy=yxzxxtyy},\,\eqref{xxyy=yyxx}\}$ and $r\in\mathbb N_0$, $e_0,f_0\in\mathbb N$, $e_1,f_1,\dots,e_r,f_r\in\mathbb N_0$ such that $\sum_{i=0}^r e_i,\sum_{i=0}^r f_i\ge 2$. 
Suppose that either $\sum_{i=0}^r e_i>2$ or $e_0>1$ and either $\sum_{i=0}^r f_i>2$ or $f_0>1$.
Then the identity~\eqref{two letters in a block} holds in $\mathbf V$.
\end{lemma}

\begin{proof}
If $e_0>1$ or $f_0>1$, then the identity~\eqref{two letters in a block} is equivalent modulo $x^2\approx x^3$ to an identity of the same form  such that $x$ and $y$ occur in its hand-sides at least trice.
So, we may further assume that $\sum_{i=0}^r e_i,\sum_{i=0}^r f_i>2$.
Obviously,~\eqref{two letters in a block} follows from $\Phi$ whenever $e_0,f_0>1$.
Assume now that at least one of the numbers $e_0$ or $f_0$, say $f_0$, is $1$.
Let $k$ and $m$ denote the greatest numbers such that $e_k,f_m>0$.
Since $f_0=1$, we have $m>0$.
If $k=0$, then $\mathbf V$ satisfies
\[
x^{e_0}y^{f_0}\mathbf p\stackrel{x^2\approx x^3}\approx x^3y^{f_0}\mathbf p\stackrel{\sigma_3}\approx xy^{f_0}x^2\mathbf p\stackrel{\beta_2}\approx xy^{f_0}x^2\mathbf p^\prime \stackrel{\eqref{xyzxxtyy=yxzxxtyy}}\approx y^{f_0}x^3\mathbf p^\prime\stackrel{\beta_2}\approx y^{f_0}x^3\mathbf p\stackrel{x^2\approx x^3}\approx y^{f_0}x^{e_0}\mathbf p,
\]
where
\[
\mathbf p:=\biggl(\prod_{i=1}^r t_ix^{e_i}y^{f_i}\biggr)\ \text{ and }\  \mathbf p^\prime:=\biggl(\prod_{i=1}^{m-1} t_ix^{e_i}y^{f_i}\biggr) (t_mx^{e_m}y^2) \biggl(\prod_{i=m+1}^r t_ix^{e_i}y^{f_i}\biggr).
\]
Let now $k>0$.
If $k<m$, then $\mathbf V$ satisfies
\[
x^{e_0}y^{f_0}\mathbf p\stackrel{\beta_2}\approx x^{e_0}y^{f_0}\mathbf q\stackrel{\eqref{xyzxxtyy=yxzxxtyy}}\approx y^{f_0}x^{e_0}\mathbf q\stackrel{\beta_2}\approx y^{f_0}x^{e_0}\mathbf p,
\]
where
\[
\mathbf q:=\biggl(\prod_{i=1}^{k-1} t_ix^{e_i}y^{f_i}\biggr) (t_kx^2y^{f_k}) \biggl(\prod_{i=k+1}^{m-1} t_ix^{e_i}y^{f_i}\biggr) (t_mx^{e_m}y^2) \biggl(\prod_{i=m+1}^r t_ix^{e_i}y^{f_i}\biggr).
\]
By a similar argument we can show that~\eqref{two letters in a block} holds in $\mathbf V$ whenever $m\le k$.
\end{proof}

\begin{lemma}
\label{L: xy.. = yx.. in X wedge Y}
Let $\mathbf X$ and $\mathbf Y$ be non-commutative subvarieties of $\mathbf A\{\sigma_2,\,\sigma_3,\,\eqref{xxyy=yyxx}\}$. 
Suppose that $\mathbf X\wedge\mathbf Y$ satisfies the identity~\eqref{two letters in a block} with $r\in\mathbb N_0$, $e_0,f_0\in\mathbb N$, $e_1,f_1,\dots,e_r,f_r\in\mathbb N_0$, $\sum_{i=0}^r e_i\ge 2$ and $\sum_{i=0}^r f_i\ge 2$.
Then this identity is true in either $\mathbf X$ or $\mathbf Y$ whenever one of the following holds:
\begin{itemize}
\item[\textup{(i)}] the words $x^{e_0}t_1x^{e_1}\cdots t_rx^{e_r}$ and $y^{f_0}t_1y^{f_1}\cdots t_ry^{f_r}$ are isoterms for $\mathbf X\wedge\mathbf Y$;
\item[\textup{(ii)}] $\mathbf X\vee\mathbf Y$ satisfies~\eqref{xyzxxtyy=yxzxxtyy} and $M(xyx)\notin\mathbf X\wedge\mathbf Y$;
\item[\textup{(iii)}] $\mathbf X\vee\mathbf Y$ satisfies~\eqref{xxyty=yxxty}, $M_\lambda(xyx^+)\notin\mathbf X\wedge\mathbf Y$ and, for some $j$, either $e_j>1$ or $f_j>1$;
\end{itemize} 
\end{lemma}

\begin{proof}
For convenience, denote by $\mathbf u$ [respectively, $\mathbf v$] the left-hand [right-hand] side of the identity~\eqref{two letters in a block}. 

\smallskip

(i) In view of Proposition~\ref{P: deduction}, there exists a finite sequence $\mathbf u = \mathbf w_0, \dots, \mathbf w_m = \mathbf v$ of words such that each identity $\mathbf w_i \approx \mathbf w_{i+1}$ holds in either $\mathbf X$ or $\mathbf Y$.
Since the words $x^{e_0}t_1x^{e_1}\cdots t_rx^{e_r}$ and $y^{f_0}t_1y^{f_1}\cdots t_ry^{f_r}$ are isoterms for $\mathbf X\wedge\mathbf Y$, we have $(\mathbf w_i)_y=x^{e_0}t_1x^{e_1}\cdots t_rx^{e_r}$ and $(\mathbf w_i)_x=y^{f_0}t_1y^{f_1}\cdots t_ry^{f_r}$ for all $i=0,\dots,m$.
Evidently, there is $j\in\{0,\dots,m-1\}$ such that $(_{1\mathbf w_j}x)<(_{1\mathbf w_j}y)$ but $(_{1\mathbf w_{j+1}}y)<(_{1\mathbf w_{j+1}}x)$.
Then the identity $\mathbf w_j\approx \mathbf w_{j+1}$ is equivalent modulo $\{x^2\approx x^3,\,\sigma_2,\,\sigma_3\}$ to~\eqref{two letters in a block}.
Hence~\eqref{two letters in a block} holds in either $\mathbf X$ or $\mathbf Y$ as required.

\smallskip

(ii) According to Corollary~\ref{C: nsub M(xt_1x...t_nx),M(xt_1x...t_nx^+)}(i), $\beta_1$ holds in either $\mathbf X$ or $\mathbf Y$. 
Now Lemma~\ref{L: xy.. = yx.. in V} applies.

\smallskip

(iii) By symmetry, we may assume that $e_j>1$.
According to Lemmas~\ref{L: nsub M(xy)} and~\ref{L: nsub M(xyx^+)}, the identity~\eqref{xyxx=xxyxx} holds in either $\mathbf X$ or $\mathbf Y$. 
If $\sum_{i=1}^rf_i>0$, then either $\mathbf X$ or $\mathbf Y$ satisfies
\[
x^{e_0}y^{f_0}\mathbf p\stackrel{\{x^2\approx x^3,\,\eqref{xyxx=xxyxx}\}}\approx x^2y^{f_0}\mathbf p\stackrel{\eqref{xxyty=yxxty}}\approx y^{f_0}x^2\mathbf p\stackrel{\{x^2\approx x^3,\,\eqref{xyxx=xxyxx}\}}\approx y^{f_0}x^{e_0}\mathbf p,
\]
where $\mathbf p:=t_1x^{e_1}y^{f_1}\cdots t_rx^{e_r}y^{f_r}$.
If $\sum_{i=1}^rf_i=0$, then $f_0>1$, whence either $\mathbf X$ or $\mathbf Y$ satisfies
\[
x^{e_0}y^{f_0}\mathbf p\stackrel{\{x^2\approx x^3,\,\eqref{xyxx=xxyxx}\}}\approx x^2y^2\mathbf p\stackrel{\eqref{xxyy=yyxx}}\approx y^2x^2\mathbf p\stackrel{\{x^2\approx x^3,\,\eqref{xyxx=xxyxx}\}}\approx y^{f_0}x^{e_0}\mathbf p.
\]
Hence~\eqref{two letters in a block} holds in either $\mathbf X$ or $\mathbf Y$ in any case.
\end{proof}

The proof of the following statement is quite similar to the proof of Lemma~\ref{L: xy.. = yx.. in X wedge Y} and so we omit it.

\begin{lemma}
\label{L: xy.. = yx.. in X wedge Y 2}
Let $\mathbf X$ and $\mathbf Y$ be subvarieties of $\mathbf A\{\sigma_1,\,\sigma_3,\,\eqref{xxyy=yyxx}\}$. 
Suppose that $\mathbf X\wedge\mathbf Y$ satisfies the identity
\begin{equation}
\label{two letters in a block dual}
\biggl(\prod_{i=r}^1 x^{e_i}y^{f_i}t_i\biggr)x^{e_0}y^{f_0}\approx \biggl(\prod_{i=r}^1 x^{e_i}y^{f_i}t_i\biggr)y^{f_0}x^{e_0}
\end{equation}
with $r\in\mathbb N_0$, $e_0,f_0\in\mathbb N$, $e_1,f_1,\dots,e_r,f_r\in\mathbb N_0$, $\sum_{i=0}^r e_i\ge 2$ and $\sum_{i=0}^r f_i\ge 2$.
Then this identity is true in either $\mathbf X$ or $\mathbf Y$ whenever one of the following holds:
\begin{itemize}
\item[\textup{(i)}] the words $x^{e_0}t_1x^{e_1}\cdots t_rx^{e_r}$ and $y^{f_0}t_1y^{f_1}\cdots t_ry^{e_r}$ are isoterms for $\mathbf X\wedge\mathbf Y$;
\item[\textup{(ii)}] $M(xyx)\notin\mathbf X\wedge\mathbf Y$.\qed
\end{itemize} 
\end{lemma}

\section{Proof of Theorem~\ref{T: A_com}}
\label{Sec: proof}

\textbf{Necessity}. 
Let $\mathbf V$ be a distributive subvariety of $\mathbf A_\mathsf{com}$.
Then $\mathbf V$ satisfies the identities $x^n\approx x^{n+1}$ and $x^ny^n\approx y^nx^n$ for some $n\in\mathbb N$.
If $\mathbf V$ is a variety of monoids with central idempotents, then, by~\cite[ Theorem~1.1]{Gusev-23}, $\mathbf V$ is contained in one of the varieties $\mathbf P_n$, $\mathbf Q_n$, $\mathbf Q_n^\delta$, $\mathbf R_n$ or $\mathbf R_n^\delta$, and we are done.
Suppose that $\mathbf V$ is not a variety of monoids with central idempotents.
In this case, $\mathbf V$ violates $x^ny\approx yx^n$.
In particular, the variety $\mathbf V$ cannot be commutative.
Then $M(xy)\in \mathbf V$ by Lemma~\ref{L: nsub M(xy)}.
Now Lemma~\ref{L: nsub M(yxx^+)} and the dual to it apply, yielding that at least one of the monoids $M_\gamma(yxx^+)$ or $M_\gamma(xx^+y)$ lies in $\mathbf V$.
It follows from Lemma~2 in~\cite{Gusev-20b}, the dual statement and the results of~\cite{Sapir-21} that the lattices $\mathfrak L\left(\mathbf M(x^2)\vee\mathbf M_\gamma(yxx^+)\right)$ and $\mathfrak L\left(\mathbf M(x^2)\vee\mathbf M_\gamma(xx^+y)\right)$ are not modular.
Hence $M(x^2)\notin \mathbf V$.
Using Lemma~\ref{L: M(W) in V}, one can easily deduce from this that $\mathbf V$ satisfies $x^2\approx x^3$.
It follows that $\mathbf V$ satisfies~\eqref{xxyy=yyxx} as well. 
Further, according to Proposition~\ref{P: sporadic}(iv), $\mathbf V$ does not contain at least one of the varieties $\mathbf M_\lambda(xyx^+)$ or $\mathbf M_{\overline{\lambda}}(x^+yx)$.
By symmetry, we may assume that $\mathbf M_{\overline{\lambda}}(x^+yx)\nsubseteq\mathbf V$.
Then $\mathbf V$ satisfies~\eqref{xxyx=xxyxx} by the dual to Lemma~\ref{L: nsub M(xyx^+)}.

Suppose that $M(xyx)\notin\mathbf V$.
Then the identity~\eqref{xyx=xyxx} holds in $\mathbf V$ by Corollary~\ref{C: nsub M(xt_1x...t_nx),M(xt_1x...t_nx^+)}(i).
If $M_\lambda(xyx^+)\notin\mathbf V$, then, by Lemma~\ref{L: nsub M(xyx^+)}, $\mathbf V$ satisfies the identity~\eqref{xyxx=xxyxx} and so the identity~\eqref{xyx=xxyx} because $xyx\stackrel{\eqref{xyx=xyxx}}\approx xyx^2\stackrel{\eqref{xyxx=xxyxx}}\approx x^2yx^2\stackrel{\eqref{xxyx=xxyxx}}\approx x^2yx$.
In this case, $\Phi_2$ is satisfied by $\mathbf V$ since, for any $k,\ell\in\mathbb N$ and $\rho\in S_{k+\ell}$,
\[
\begin{aligned}
\mathbf a_{k,\ell}[\rho]&{}\stackrel{\{\eqref{xyx=xyxx},\,\eqref{xyx=xxyx}\}}\approx \mathbf h\,x^{2(k+\ell)}\biggl(\prod_{i=1}^{k+\ell-1} z_{i\rho}^2y_i^2\biggr)z_{(k+\ell)\rho}x\,\mathbf t\\
&{}\ \ \ \ \ \, \stackrel{\eqref{xxyy=yyxx}}\approx\ \ \ \ \ \mathbf h\,x^2\biggl(\prod_{i=1}^{k+\ell-1} z_{i\rho}^2y_i^2x^2\biggr)z_{(k+\ell)\rho}x\,\mathbf t\stackrel{\{\eqref{xyx=xyxx},\,\eqref{xyx=xxyx}\}}\approx \overline{\mathbf a}_{k,\ell}[\rho],
\end{aligned}
\]
where $\mathbf h:=z_1t_1\cdots z_kt_k$ and $\mathbf t:=t_{k+1}z_{k+1}\cdots t_{k+\ell}z_{k+\ell}$.
By a similar argument, we can show that $\Phi_1$ holds in $\mathbf V$.
Hence $\mathbf V\subseteq\mathbf D_2$.
Assume now that $M_\lambda(xyx^+)\in\mathbf V$.
If $M_\gamma(xx^+y)\notin\mathbf V$, then $\mathbf V$ satisfies~\eqref{xxy=xxyx} by the statement dual to Lemma~\ref{L: nsub M(yxx^+)}, whence $\mathbf V\subseteq\mathbf D_1$.
If $M_\gamma(xx^+y)\in\mathbf V$, then $\mathbf H\nsubseteq \mathbf V$ by Proposition~\ref{P: sporadic}(i).
In this case, Lemma~\ref{L: nsub H} implies that $\mathbf V$ satisfies~\eqref{xyzxxyy=yxzxxyy}.
Then the identity~\eqref{xyzxy=yxzxy} is satisfied by $\mathbf V$ because
\[
xyzxy \stackrel{\eqref{xyx=xyxx}}\approx xyzx^2y^2\stackrel{\eqref{xyzxxyy=yxzxxyy}}\approx yxzx^2y^2\stackrel{\eqref{xyx=xyxx}}\approx yxzxy.
\]
According to Propositions~\ref{P: non-dis L(M([hat{a}_{0,k}[pi]]^lambda))},~\ref{P: non-dis L(M([a_{0,k}[pi]]^lambda))} and~\ref{P: non-dis L(M([a_{0,k}[pi]]^beta))}, the monoids $M_\lambda([\hat{\mathbf a}_{0,k}[\pi]]^\lambda)$, $M_\lambda([\mathbf a_{0,k}[\pi]]^\lambda)$ and $M_\beta([\mathbf a_{0,k}[\pi]]^\beta)$ do not lie in $\mathbf V$ for any $k\ge2$ and $\pi\in S_k$.
Now Lemma~\ref{L: nsub M([hat{a}_{0,n}[rho]]^lambda),M([a_{0,n}[rho]]^lambda),M([a_{0,n}[rho]]^beta)} applies, yielding that $\mathbf V$ satisfies $\Phi_2$.
Further, Proposition~\ref{P: non-mod M([c_{0,0,n}[rho]]^lambda)} implies that $M_\lambda([\mathbf c_{0,0,k}[\pi]]^\lambda)\notin\mathbf V$ for any $k\in\mathbb N$ and $\pi\in S_k$ .
Then, by Lemma~\ref{L: nsub M([c_{0,0,n}[rho]]^lambda)}, the variety $\mathbf V$ satisfies $\mathbf c_{k,\ell,m}[\rho]\approx \mathbf c_{k,\ell,m}^\prime[\rho]$ for all $k,\ell,m\in\mathbb N$ and $\rho\in S_{k+\ell+m}$.
Finally, the identity $\mathbf d_{k,\ell,m}[\rho]\approx \mathbf d_{k,\ell,m}^\prime[\rho]$ is also satisfied by $\mathbf V$ because this identity follows from $\{\eqref{xyx=xyxx},\,\eqref{xxyy=yyxx}\}$.
We see that $\Phi_1$ holds in $\mathbf V$.
Hence $\mathbf V\subseteq\mathbf D_2$ again.
Thus, it remains to consider the case when $M(xyx)\in\mathbf V$.
Three cases are possible.

\medskip

\textbf{Case 1:} $\mathbf N\subseteq\mathbf V$.
Then Parts~(xxiii) and~(xxiv) of Proposition~\ref{P: sporadic} and Lemma~\ref{L: swapping in linear-balanced} imply that $\mathbf V$ satisfies $\sigma_2$ and $\sigma_3$.

If $M_\gamma(yxx^+)\notin\mathbf V$, then $\mathbf V$ satisfies~\eqref{yxx=xyxx} by Lemma~\ref{L: nsub M(yxx^+)}.
In this case, $\mathbf V\subseteq\mathbf D_6$ because $xyzx^2ty\stackrel{\eqref{yxx=xyxx}}\approx yxzx^2ty$, $xyzytx^2\stackrel{\eqref{yxx=xyxx}}\approx yxzytx^2$ and $x^2yty\stackrel{x^2\approx x^3}\approx x^3yty \stackrel{\sigma_3}\approx xyx^2ty \stackrel{\eqref{yxx=xyxx}}\approx yx^2ty$.
Thus, we may further assume that $M_\gamma(yxx^+)\in\mathbf V$.

Assume now that $M_\nu([yx^2ty]^\nu)\in \mathbf V$.
Then $\mathbf V$ does not contain the varieties $\mathbf M_\gamma(xx^+y)$, $\mathbf M_\lambda(xyx^+)$ and $\mathbf M(xyxzx)$ by Parts~(ii),~(v) and~(xv) of Proposition~\ref{P: sporadic}.
Now the dual to Lemma~\ref{L: nsub M(yxx^+)}, Lemma~\ref{L: nsub M(xyx^+)} and Corollary~\ref{C: nsub M(xt_1x...t_nx),M(xt_1x...t_nx^+)}(i) apply, yielding that the identities~\eqref{xxy=xxyx},~\eqref{xyxx=xxyxx} and $\beta_2$ are satisfied by $\mathbf V$.
Since the identity~\eqref{xyxx=xxyxx} is equivalent modulo~\eqref{xxy=xxyx} to the identity~\eqref{xyxx=xxyx}, we have $\mathbf V\subseteq\mathbf D_3$.
Therefore, we may further assume that $M_\nu([yx^2ty]^\nu)\notin \mathbf V$.
In view of Lemma~\ref{L: nsub M([yx^2zy]),M(xx^+yty)}(i), $\mathbf V$ satisfies the identity~\eqref{yxxtxxyxx=xxyxxtxxyxx}.
Clearly, this identity is equivalent modulo $\{x^2\approx x^3,\,\eqref{xxyx=xxyxx},\,\sigma_2,\,\sigma_3\}$ to~\eqref{xxytxy=yxxtxy}.
Two subcases are possible.

\smallskip

\textbf{Case 1.1:} $M_\lambda(xyx^+)\notin\mathbf V$. 
Then, by Lemma~\ref{L: nsub M(xyx^+)}, $\mathbf V$ satisfies~\eqref{xyxx=xxyxx} and, therefore,~\eqref{xyxx=xxyx}.
If $M(xyxzx)\notin\mathbf V$, then $\mathbf V$ satisfies $\beta_2$ by Corollary~\ref{C: nsub M(xt_1x...t_nx),M(xt_1x...t_nx^+)}(i).
In this case, $\mathbf V\subseteq\mathbf D_5$ because $xyzx^2y\stackrel{\eqref{xyxx=xxyxx}}\approx  x^2yzx^2y\stackrel{\eqref{xxytxy=yxxtxy}}\approx yx^2zx^2y\stackrel{\eqref{xyxx=xxyxx}}\approx yxzx^2y$.
So, we may further assume that $M(xyxzx)\in\mathbf V$. 
In this case, Proposition~\ref{P: sporadic}(xvi) implies that $M_{\gamma^\prime}(yxx^+ty)\notin \mathbf V$.
Then $\mathbf V$ satisfies~\eqref{yxxty=xyxxty} by Corollary~\ref{C: nsub M(yxx^+ty)}.
Hence the identity~\eqref{xxyty=yxxty} holds in $\mathbf V$ because 
\begin{equation}
\label{xxyty=xxxyty=xyxxty=yxxty}
x^2yty\stackrel{x^2\approx x^3}\approx x^3yty\stackrel{\sigma_3}\approx xyx^2ty\stackrel{\eqref{yxxty=xyxxty}}\approx yx^2ty.
\end{equation}
Finally, since
\begin{equation}
\label{xyzxxty=..=yxzxxty,xyzytxx=..=yxzytxx}
\begin{aligned}
xyzx^2ty\stackrel{\eqref{xyxx=xxyx}}\approx x^2yzxty\stackrel{\eqref{xxyty=yxxty}}\approx yx^2zxty\stackrel{\eqref{xyxx=xxyx}}\approx yxzx^2ty,\\
xyzytx^2\stackrel{\eqref{xyxx=xxyx}}\approx x^2yzytx\stackrel{\eqref{xxyty=yxxty}}\approx yx^2zytx\stackrel{\eqref{xyxx=xxyx}}\approx yxzytx^2,
\end{aligned}
\end{equation}
we have $\mathbf V\subseteq\mathbf D_6$.

\smallskip

\textbf{Case 1.2:} $M_\lambda(xyx^+)\in\mathbf V$. 
Then Parts~(vi) and~(xii) of Proposition~\ref{P: sporadic} imply that $\mathbf V$ does not contain the varieties $\mathbf M_{\gamma^\prime}(yxx^+ty)$ and $\mathbf H$.
In view of Corollary~\ref{C: nsub M(yxx^+ty)} and Lemma~\ref{L: nsub H}, $\mathbf V$ satisfies the identities~\eqref{yxxty=xyxxty} and~\eqref{xyzxxyy=yxzxxyy}.
Then the identity~\eqref{xxyty=yxxty} holds in $\mathbf V$ because~\eqref{xxyty=xxxyty=xyxxty=yxxty}.

Assume first that $M_\eta([xyzx^2ty]^\eta)\in \mathbf V$.
Clearly, $M_\eta([xyzx^2ty]^\eta)$ violates~\eqref{xyzytxx=yxzytxx}.
This fact and Lemma~\ref{L: nsub M([xyzx^2ty]),M(xyzytxx^+),M(yzyxtxx^+)}(ii) imply that $\mathbf M_{\lambda^\prime}(xyzytxx^+)\subset \mathbf M_\eta([xyzx^2ty]^\eta)$.
Then $\mathbf V$ does not contain the varieties $\mathbf M_\gamma(xx^+y)$, $\mathbf M(xyxzx)$ and $\mathbf M_\mu(xyxzx^+)$ by Parts~(iii),~(xix) and~(xxii) of Proposition~\ref{P: sporadic}.
Now the dual to Lemma~\ref{L: nsub M(yxx^+)} and Corollary~\ref{C: nsub M(xt_1x...t_nx),M(xt_1x...t_nx^+)} apply, yielding that the identities~\eqref{xxy=xxyx}, $\beta_2$ and $\gamma_2$ are satisfied by $\mathbf V$.
Finally, the identities
\[
xyzx^2ty^2\stackrel{\eqref{xxy=xxyx}}\approx xyzx^2tx^2y^2\stackrel{\eqref{xyzxxyy=yxzxxyy}}\approx yxzx^2tx^2y^2\stackrel{\eqref{xxy=xxyx}}\approx yxzx^2ty^2
\]
hold in $\mathbf V$.
Hence $\mathbf V\subseteq\mathbf D_4$.
Therefore, we may further assume that $M_\eta([xyzx^2ty]^\eta)\notin \mathbf V$.
In view of Lemma~\ref{L: nsub M([xyzx^2ty]),M(xyzytxx^+),M(yzyxtxx^+)}(i), $\mathbf V$ satisfies~\eqref{xyzxxtxyx=yxzxxtxyx}.
Clearly, this identity is equivalent modulo $\{x^2\approx x^3,\,\sigma_2\}$ to~\eqref{xyzxxy=yxzxxy}.
Two subcases are possible.

\smallskip

\textbf{Case 1.2.1:} $M(xyxzx)\in\mathbf V$.
In this case, Parts (xviii) and~(xix) of Proposition~\ref{P: sporadic} imply that the monoids $M_{\lambda^\prime}(xyzxx^+ty)$ and $M_{\lambda^\prime}(xyzytxx^+)$ do not lie in $\mathbf V$.
In view of Lemma~\ref{L: nsub M([xyzx^2ty]),M(xyzytxx^+),M(yzyxtxx^+)}(ii) and Corollary~\ref{C: nsub M(xyzxx^+ty)}, $\mathbf V$ satisfies the identities~\eqref{xyzytxx=yxzytxx} and~\eqref{xyzxxty=yxzxxty}.
Hence $\mathbf V\subseteq\mathbf D_6$.

\smallskip

\textbf{Case 1.2.2:} $M(xyxzx)\notin\mathbf V$.
Then $\mathbf V$ satisfies $\beta_2$ by Corollary~\ref{C: nsub M(xt_1x...t_nx),M(xt_1x...t_nx^+)}(i).
According to Proposition~\ref{P: sporadic}(xiii), $M_\lambda(xyzx^+ty^+)\notin\mathbf V$.
Then $\mathbf V$ satisfies~\eqref{xyzxxtyy=yxzxxtyy} by Corollary~\ref{C: nsub M(xyzx^+ty^+)}.
Further, if $M_\mu(xyxzx^+)\notin\mathbf V$, then the identity $\gamma_2$ holds in $\mathbf V$ by Corollary~\ref{C: nsub M(xt_1x...t_nx),M(xt_1x...t_nx^+)}(ii), whence $\mathbf V\subseteq\mathbf D_7$. 
If $M_\mu(xyxzx^+)\in\mathbf V$, then Proposition~\ref{P: sporadic}(xxi) and Lemma~\ref{L: nsub M(xyzx^+tysx^+)} imply that $\mathbf V$ satisfies~\eqref{xyzxxtysx=yxzxxtysx}, whence $\mathbf V\subseteq\mathbf D_8$.

\medskip

\textbf{Case 2:} $\mathbf N^\delta\subseteq\mathbf V$.
Then Part~(xxiii), the dual to Part~(xxiv) of Proposition~\ref{P: sporadic} and Lemma~\ref{L: swapping in linear-balanced} imply that $\mathbf V$ satisfies $\sigma_1$ and $\sigma_3$.

Assume first that $M_{\overline{\nu}}([ytx^2y]^{\overline{\nu}})\in \mathbf V$.
Then $\mathbf V$ does not contain the varieties $\mathbf M_\gamma(yxx^+)$ and $\mathbf M(xyxzx)$ by Parts~(ii) and~(xv) of the statement dual to Proposition~\ref{P: sporadic}.
Now Lemma~\ref{L: nsub M(yxx^+)} and Corollary~\ref{C: nsub M(xt_1x...t_nx),M(xt_1x...t_nx^+)}(i) apply, yielding that the identities~\eqref{yxx=xyxx} and $\beta_2$ are satisfied by $\mathbf V$.
Since the identity~\eqref{xxyx=xxyxx} together with~\eqref{yxx=xyxx} imply~\eqref{xyxx=xxyx}, we have $\mathbf V\subseteq\mathbf D_3^\delta$.
Therefore, we may further assume that $M_{\overline{\nu}}([ytx^2y]^{\overline{\nu}})\notin \mathbf V$.
In view of the dual to Lemma~\ref{L: nsub M([yx^2zy]),M(xx^+yty)}(i), $\mathbf V$ satisfies the identity $x^2yx^2tx^2y\approx x^2yx^2tx^2yx^2$.
Clearly, this identity is equivalent modulo $\{x^2\approx x^3,\,\eqref{xxyx=xxyxx},\,\sigma_3\}$ to~\eqref{xxytxy=xxytyx}.

Further, if $M_\gamma(xx^+y)\notin\mathbf V$, then $\mathbf V$ satisfies~\eqref{xxy=xxyx} by the dual to Lemma~\ref{L: nsub M(yxx^+)}.
Clearly, this identity together with $\sigma_3$ imply
\begin{equation}
\label{ytxxy=ytyxx}
ytx^2y\approx ytyx^2,
\end{equation}
whence $\mathbf V\subseteq\mathbf D_9$.
So, we may further assume that $M_\gamma(xx^+y)\in\mathbf V$.
If $M(xyxzx)\in\mathbf V$, then the dual to Proposition~\ref{P: sporadic}(xvi) implies $M_{\gamma^\prime}(ytxx^+y)\notin\mathbf V$.
If $M_\gamma(yxx^+)\notin\mathbf V$, then, by Lemma~\ref{L: nsub M(yxx^+)}, $\mathbf V$ satisfies~\eqref{yxx=xyxx} and so $ytx^2y\stackrel{\eqref{yxx=xyxx}}\approx x^2ytx^3y \stackrel{\eqref{xxytxy=xxytyx}}\approx x^2ytx^2yx \stackrel{\eqref{yxx=xyxx}}\approx ytx^2yx$.
If $M_\gamma(yxx^+)\in\mathbf V$, then the dual to Lemma~\ref{L: nsub M(yxx^+ty)} applies, yielding that
\begin{equation}
\label{ytxxy=ytxxyx}
ytx^2y\approx ytx^2yx
\end{equation} 
is satisfied by $\mathbf V$.
Since $ytx^2yx\stackrel{\sigma_3}\approx ytyx^3\stackrel{x^2\approx x^3}\approx ytyx^2$, we have $\mathbf V\subseteq\mathbf D_9$ again.
So, it remains to consider the case when $M(xyxzx)\notin\mathbf V$.
Then $\mathbf V$ satisfies $\beta_2$ by Corollary~\ref{C: nsub M(xt_1x...t_nx),M(xt_1x...t_nx^+)}(i).
If $M_\lambda(xyx^+)\notin\mathbf V$, then Lemma~\ref{L: nsub M(xyx^+)} implies that $\mathbf V$ satisfies~\eqref{xyxx=xxyxx} and, therefore,~\eqref{xyxx=xxyx}, whence $\mathbf V\subseteq\mathbf D_5^\delta$.
Assume now that $M_\lambda(xyx^+)\in\mathbf V$.
Then $\mathbf M_{\gamma^\prime}(ytyxx^+),M_{\gamma^\prime}([x^2zytx^2y]^{\gamma^\prime})\notin\mathbf V$ by Parts~(vii) and~(xi) of Proposition~\ref{P: sporadic}.
According to the dual to Lemma~\ref{L: nsub M([yx^2zy]),M(xx^+yty)}(ii), the identity~\eqref{ytyxx=ytxyxx} holds in $\mathbf V$.
Now Lemma~\ref{L: nsub M([x^2zytx^2y])} applies, yielding that $\mathbf V$ satisfies the identity~\eqref{xxzytxy=xxzytyx}.
Hence $\mathbf V\subseteq\mathbf D_{10}$.

\medskip

\textbf{Case 3:} $\mathbf N,\mathbf N^\delta\nsubseteq\mathbf V$.
Propositions~\ref{P: non-mod M(hat{c}_{n,m,n+m+1}[rho]),M(hat{c}_{n,m,0}[rho])},~\ref{P: non-mod M([c_n^{(i)}[rho]]^{gamma'})}, Lemma~\ref{L: nsub M(c_{n,m,k}[rho]) 2} and the dual statements imply that $\mathbf V$ satisfies $\Phi_1$.
In particular, the identities~\eqref{xyzxy=yxzxy} and~\eqref{xyzxy=xyzyx} are satisfied by $\mathbf V$.

Two subcases are possible.

\smallskip

\textbf{Case 3.1:} $M_\lambda(xyx^+)\notin\mathbf V$. 
Then, by Lemma~\ref{L: nsub M(xyx^+)}, $\mathbf V$ satisfies~\eqref{xyxx=xxyxx} and, therefore,~\eqref{xyxx=xxyx}.
Further, $\mathbf V$ satisfies $\Phi_2$ by Lemma~\ref{L: nsub M(a_{n,m}[rho]) 2}, Propositions~\ref{P: non-dis L(M_{gamma'}(a_{n,m}[pi]))},~\ref{P: non-dis L(M_{gamma''}(a_{n,m}[pi]))},~\ref{P: non-dis L(M_{alpha_1}(a_{n,n}[pi]))},~\ref{P: non-dis  L(M_{zeta}(hat{a}_{n,m}[pi]))} and the dual to Proposition~\ref{P: non-dis  L(M_{zeta}(hat{a}_{n,m}[pi]))}.
If $M(xyxzx)\notin\mathbf V$, then $\mathbf V$ satisfies $\beta_2$ by Corollary~\ref{C: nsub M(xt_1x...t_nx),M(xt_1x...t_nx^+)}(i), whence $\mathbf V\subseteq\mathbf D_{11}$.
So, we may further assume that $M(xyxzx)\in\mathbf V$. 
Then Parts~(xvi) and~(xvii) of Proposition~\ref{P: sporadic} imply that the monoids $M_{\gamma^\prime}(xx^+yty)$ and $M_{\gamma^\prime}(yxx^+ty)$ do not lie in $\mathbf V$.
Then $\mathbf V$ satisfies~\eqref{xxyty=xxyxty} and~\eqref{yxxty=xyxxty} by Lemma~\ref{L: nsub M([yx^2zy]),M(xx^+yty)}(ii) and Corollary~\ref{C: nsub M(yxx^+ty)}.
In this case, the identity~\eqref{xxyty=yxxty} holds in $\mathbf V$ because 
\begin{equation}
\label{xxyty=xxyxxty=yxxty}
x^2yty\stackrel{\eqref{xxyty=xxyxty}}\approx x^2yx^2ty\stackrel{\eqref{yxxty=xyxxty}}\approx yx^2ty.
\end{equation}
By the dual arguments we can show that $\mathbf V$ satisfies~\eqref{ytxxy=ytyxx}.
Finally, since~\eqref{xyzxxty=..=yxzxxty,xyzytxx=..=yxzytxx} and 
\[
yzyxtx^2\stackrel{\eqref{xyxx=xxyx}}\approx yzyx^2tx\stackrel{\eqref{ytxxy=ytyxx}}\approx yzx^2ytx\stackrel{\eqref{xyxx=xxyx}}\approx yzxytx^2,
\]
we have $\mathbf V\subseteq\mathbf D_{12}$.

\smallskip

\textbf{Case 3.2:} $M_\lambda(xyx^+)\in\mathbf V$.
Parts (vi) and~(vii) of Proposition~\ref{P: sporadic} imply that the monoids $M_{\gamma^\prime}(yxx^+ty)$ and $M_{\gamma^\prime}(ytyxx^+)$  do not lie in $\mathbf V$.
In view of the dual to Lemma~\ref{L: nsub M([yx^2zy]),M(xx^+yty)}(ii) and Corollary~\ref{C: nsub M(yxx^+ty)}, $\mathbf V$ satisfies the identities~\eqref{yxxty=xyxxty} and~\eqref{ytyxx=ytxyxx}.
Further, $\mathbf V$ satisfies $\Phi_2$ by Lemma~\ref{L: nsub M(a_{n,m}[rho]) 1} and Proposition~\ref{P: non-dis  L(M_{zeta}(hat{a}_{n,m}[pi]))}.
Two subcases are possible.

\smallskip

\textbf{Case 3.2.1:} $M(xyxzx)\in\mathbf V$.
Parts (xvii)--(xx) of Proposition~\ref{P: sporadic} and the statement dual to Proposition~\ref{P: sporadic}(xvi) imply that the following monoids do not lie in $\mathbf V$:
\[
M_{\gamma^\prime}(ytxx^+y),\, M_{\gamma^\prime}(xx^+yty),\,M_{\lambda^\prime}(xyzxx^+ty),\,M_{\lambda^\prime}(xyzytxx^+),\,M_{\lambda^\prime}(yzyxtxx^+).
\]
Then $\mathbf V$ satisfies the identities~\eqref{ytxxy=ytxxyx}, \eqref{xxyty=xxyxty}, \eqref{xyzytxx=yxzytxx}, \eqref{yzxytxx=yzyxtxx} and~\eqref{xyzxxty=yxzxxty} by the dual to Corollary~\ref{C: nsub M(yxx^+ty)}, Lemma~\ref{L: nsub M([yx^2zy]),M(xx^+yty)}(ii), Lemma~\ref{L: nsub M([xyzx^2ty]),M(xyzytxx^+),M(yzyxtxx^+)}(ii),(iii) and Corollary~\ref{C: nsub M(xyzxx^+ty)}.
Since~\eqref{xxyty=xxyxxty=yxxty} and $ytyx^2\stackrel{\eqref{ytyxx=ytxyxx}}\approx ytx^2yx^2\stackrel{\eqref{ytxxy=ytxxyx}}\approx ytx^2y$, we have $\mathbf V\subseteq\mathbf D_{12}$.

\smallskip

\textbf{Case 3.2.2:} $M(xyxzx)\notin\mathbf V$.
Then $\mathbf V$ satisfies $\beta_2$ by Corollary~\ref{C: nsub M(xt_1x...t_nx),M(xt_1x...t_nx^+)}(i).
Further, Parts~(viii)--(xiv) of Proposition~\ref{P: sporadic} imply that $\mathbf V$ does not contain the following varieties:
\[
\begin{aligned}
&\mathbf H,\,
\mathbf M_\lambda(xyzx^+ty^+),\,
\mathbf M_\lambda(xx^+yty^+),\\
&\mathbf M_{\gamma^\prime}([x^2yzytx^2]^{\gamma^\prime}),\,
\mathbf M_{\gamma^\prime}([x^2yzx^2ty]^{\gamma^\prime}),\,
\mathbf M_{\gamma^\prime}([yzx^2ytx^2]^{\gamma^\prime}),\,
\mathbf M_{\gamma^\prime}([x^2zytx^2y]^{\gamma^\prime}).
\end{aligned}
\]
Then Lemmas~\ref{L: nsub M([x^2yzytx^2]),M([x^2yzx^2ty])},~\ref{L: nsub M([yzx^2ytx^2])},~\ref{L: nsub M([x^2zytx^2y])} and Corollaries~\ref{C: nsub M(xx^+yty^+)},~\ref{C: nsub M(xyzx^+ty^+)} imply that $\mathbf V$ satisfies the following identities: \eqref{xxyzytx=yxxzytx}, \eqref{xxyzxty=yxxzxty}, \eqref{yzxxytx=yzyxxtx}, \eqref{xxzytxy=xxzytyx}, \eqref{xxytyy=yxxtyy} and~\eqref{xyzxxtyy=yxzxxtyy}.
Further, if $M_\mu(xyxzx^+)\notin\mathbf V$, then the identity $\gamma_2$ holds in $\mathbf V$ by Corollary~\ref{C: nsub M(xt_1x...t_nx),M(xt_1x...t_nx^+)}(ii), whence $\mathbf V\subseteq\mathbf D_{13}$. 
If $M_\mu(xyxzx^+)\in\mathbf V$, then Proposition~\ref{P: sporadic}(xxi) and Lemma~\ref{L: nsub M(xyzx^+tysx^+)} imply that $\mathbf V$ satisfies~\eqref{xyzxxtysx=yxzxxtysx}, whence $\mathbf V\subseteq\mathbf D_{14}$.

\medskip

\textbf{Sufficiency}.
By symmetry, it suffices to show that the varieties $\mathbf D_1$--$\mathbf D_{14}$, $\mathbf P_n$, $\mathbf Q_n$ and $\mathbf R_n$ are distributive.
In view of~\cite[Theorem~1.1]{Gusev-23}, varieties $\mathbf P_n$, $\mathbf Q_n$ and $\mathbf R_n$ have this property.
The variety $\mathbf D_1$ is also distributive by Lemma~\ref{L: L(D1)}.
So, it remains to prove distributivity of the varieties $\mathbf D_2$--$\mathbf D_{14}$.
The following lemma is the key tool of the remained part of the proof.

\begin{lemma}
\label{L: smth imply distributivity}
Let $\mathbf V$ be a monoid variety satisfying $\Phi$.
Suppose that there is a set $\Sigma$ of identities such that:
\begin{itemize}
\item[\textup{(i)}] if $\mathbf U$ is a non-commutative subvariety of $\mathbf V$, then $\mathbf U=\mathbf V\Sigma^\prime$ for some subset $\Sigma^\prime$ of $\Sigma$;
\item[\textup{(ii)}] if $\mathbf U,\mathbf U^\prime$ are non-commutative subvarieties of $\mathbf V$ and $\mathbf U\wedge\mathbf U^\prime$ satisfies an identity $\sigma\in\Sigma$, then $\sigma$ holds in either $\mathbf U$ or $\mathbf U^\prime$.
\end{itemize}
Then the the lattice $\mathfrak L(\mathbf V)$ is distributive.
\end{lemma}

\begin{proof}
In view of Lemma~\ref{L: nsub M(xy)}, a subvariety of $\mathbf V$ is commutative if and only if it does not contain $M(xy)$.
This fact and~\cite[Corollary~6.1.5]{Almeida-94} imply that the lattice $\mathfrak L(\mathbf V)$ is a disjoint union of the lattice $\mathfrak L(\mathbf M(x))$ and the interval $[\mathbf M(xy),\mathbf V]$.
According to Lemma~\ref{L: L(M(xzyx^+ty^+))}, the lattice $\mathfrak L(\mathbf M(x))$ is distributive.
The interval $[\mathbf M(xy),\mathbf V]$ is also distributive by~\cite[Lemma~2.19]{Gusev-Vernikov-21}.
Lemma~\ref{L: smth imply distributivity} is thus proved.
\end{proof}

\smallskip

To prove distributivity of the varieties $\mathbf D_2$--$\mathbf D_{14}$, we use the following general scheme.
For each $2\le n\le 14$, we construct a set $\Delta_n$ of identities such that: 
\begin{itemize}
\item[\textup{(i)}] each non-commutative subvariety of $\mathbf D_n$ can be defined within $\mathbf D_n$ by some identities in $\{\Delta,\,\Delta_n\}$;
\item[\textup{(ii)}] if $\mathbf U,\mathbf U^\prime$ are non-commutative subvarieties of $\mathbf D_n$ and $\mathbf U\wedge\mathbf U^\prime$ satisfies an identity $\sigma\in\Delta_n$, then $\sigma$ holds in either $\mathbf U$ or $\mathbf U^\prime$.
\end{itemize}
The variety $\mathbf D_n$ is thus distributive by Lemma~\ref{L: smth imply distributivity} and Corollary~\ref{C: Psi in X wedge Y}.
If $n=2$ or $11\le n\le 14$, then it follows from Proposition~\ref{P: Phi,Phi_1,Phi_2 subvarieties} and Lemma~\ref{L: one letter in block reduction} that each non-commutative subvariety of $\mathbf D_n$ can be defined by identities in $\Delta$ together with some identities of the form~\eqref{two letters in a block middle 1}--\eqref{two letters in a block middle 3} such that the condition~\eqref{two letters in a block middle conditions} holds.
In this case, we build the set $\Delta_n$ by refining the form of identities~\eqref{two letters in a block middle 1}--\eqref{two letters in a block middle 3}.
If $3\le n\le 8$, then Proposition~\ref{P: var{sigma_2,sigma_3} subvarieties} and Lemma~\ref{L: one letter in block reduction} imply that each non-commutative subvariety of $\mathbf D_n$ can be defined by identities in $\Delta$ together with some efficient identities of the form~\eqref{two letters in a block}, where $r\in\mathbb N_0$, $e_0,f_0\in\mathbb N$, $e_1,f_1,\dots,e_r,f_r\in\mathbb N_0$, $\sum_{i=0}^r e_i\ge 2$ and $\sum_{i=0}^r f_i\ge 2$.
In this case, to construct the set $\Delta_n$, we refine the form of identity~\eqref{two letters in a block}.
Finally, if $n=9$ or $n=10$, then the dual to Proposition~\ref{P: var{sigma_2,sigma_3} subvarieties} and Lemma~\ref{L: one letter in block reduction} imply that each non-commutative subvariety of $\mathbf D_n$ can be given by identities in $\Delta$ together with some efficient identities of the form~\eqref{two letters in a block dual}, where $r\in\mathbb N_0$, $e_0,f_0\in\mathbb N$, $e_1,f_1,\dots,e_r,f_r\in\mathbb N_0$, $\sum_{i=0}^r e_i\ge 2$ and $\sum_{i=0}^r f_i\ge 2$.
In this case, to build the set $\Delta_n$, we build the set $\Delta_n$ by refining the form of identities~\eqref{two letters in a block dual}.

\smallskip

\textbf{Distributivity of $\mathfrak L(\mathbf D_2)$}.
Take arbitrary $k,\ell\in\mathbb N_0$, $g_1,\dots,g_{k+\ell}\in\{1,2\}$, $a_1,\dots,a_{k+\ell}\in\{x,y\}$.
Consider the identity~\eqref{two letters in a block middle 1}.
If $k>0$, then, by symmetry, we may assume that $x\in\con(a_1\cdots a_k)$.
Then~\eqref{two letters in a block middle 1} is equivalent modulo~\eqref{xyx=xyxx} to~\eqref{two letters in a block middle 4}. 
Thus, the identity~\eqref{two letters in a block middle 1} is equivalent within $\mathbf D_2$ to either
\begin{equation}
\label{two letters in a block middle 1 D_2}
xy\biggl(\prod_{i=k+1}^{k+\ell} t_i a_i\biggr)\approx yx\biggl(\prod_{i=k+1}^{k+\ell} t_i a_i\biggr)
\end{equation}
or $\{\eqref{two letters in a block middle 2},\eqref{two letters in a block middle 3}\}$ by Lemma~\ref{L: two letters in a block middle 2,3 <-> 4}.
Further, suppose that the identity~\eqref{two letters in a block middle 2} does not hold in $\mathbf D_2$.
It is easy to see that $y\notin\con(a_1\cdots a_k)$ because~\eqref{two letters in a block middle 2} follows from~$\{\eqref{xyx=xyxx},\,\eqref{xxyy=yyxx}\}$ otherwise.
Then it is routine to check that~\eqref{two letters in a block middle 2} is equivalent modulo $\{\Phi_2,\,\eqref{xyx=xyxx}\}$ to
\begin{equation}
\label{two letters in a block middle 2 D_2}
x^2y\biggl(\prod_{i=k+1}^{k+\ell} t_i a_i\biggr)\approx xyx\biggl(\prod_{i=k+1}^{k+\ell} t_i a_i\biggr).
\end{equation}
By a similar argument we can show that if the identity~\eqref{two letters in a block middle 3} does not hold in $\mathbf D_2$, then it is equivalent modulo $\{\Phi_2,\,\eqref{xyx=xyxx}\}$ to
\begin{equation}
\label{two letters in a block middle 3 D_2}
x^gt\,yx^2\biggl(\prod_{i=k+1}^{k+\ell} t_i a_i\biggr)\approx x^gt\,xyx\biggl(\prod_{i=k+1}^{k+\ell} t_i a_i\biggr),
\end{equation}
 where
\[
g:=
\begin{cases} 
0 & \text{if $k=0$},\\
1 & \text{if $k=1$ and $g_1=1$}, \\
2 & \text{if $k>1$ or $g_1>1$}.
\end{cases} 
\]
Denote by $\Delta_2$ the set of all identities of the form~\eqref{two letters in a block middle 1 D_2}--\eqref{two letters in a block middle 3 D_2}.
It remains to show that if $\mathbf X$ and $\mathbf Y$ are non-commutative subvarieties of $\mathbf D_2$ and $\mathbf X\wedge\mathbf Y$ satisfies some identity $\mathbf u\approx\mathbf v\in\Delta_2$, then $\mathbf u\approx\mathbf v$ holds in either $\mathbf X$ or $\mathbf Y$.
By Lemma~\ref{L: nsub M(xy)}, $M(xy)\in\mathbf X\wedge\mathbf Y$ .

If $M_\lambda(xyx^+)\notin\mathbf X\wedge\mathbf Y$, then the identity~\eqref{xyxx=xxyxx} holds in either $\mathbf X$ or $\mathbf Y$ by Lemma~\ref{L: nsub M(xyx^+)}.
In this case, the identity $\mathbf u\approx\mathbf v$ is satisfied by either $\mathbf X$ or $\mathbf Y$ because this identity follows from $\{\eqref{xyxx=xxyxx},\,\eqref{xyx=xyxx},\,\eqref{xxyy=yyxx}\}$.
If $M_\lambda(xx^+y)\notin\mathbf X\wedge\mathbf Y$, then the identity~\eqref{xxy=xxyx} holds in either $\mathbf X$ or $\mathbf Y$ by the dual to Lemma~\ref{L: nsub M(yxx^+)}.
In this case, the identity $\mathbf u\approx\mathbf v$ is satisfied by either $\mathbf X$ or $\mathbf Y$ because $\mathbf u\stackrel{\{\eqref{xxy=xxyx},\,\eqref{xyx=xyxx}\}}\approx \mathbf uxy\stackrel{\eqref{xyzxy=yxzxy}}\approx \mathbf vxy\stackrel{\{\eqref{xxy=xxyx},\,\eqref{xyx=xyxx}\}}\approx \mathbf v$.
So, we may further assume that $\mathbf M_\lambda(xyx^+)\vee \mathbf M_\gamma(xx^+y)\subseteq\mathbf X\wedge\mathbf Y$.

We consider only the case when $\mathbf u\approx\mathbf v$ has the form~\eqref{two letters in a block middle 3 D_2}.
The case when $\mathbf u\approx\mathbf v$ is of the form~\eqref{two letters in a block middle 1 D_2} or~\eqref{two letters in a block middle 2 D_2} is quite similar and we omit it.

Assume first that there is $k+1\le r\le k+\ell$ such that $a_{k+1}=\cdots=a_{r-1}=y$ and $a_r=\cdots=a_{k+\ell}=y$.
In view of Proposition~\ref{P: deduction}, there exists a finite sequence $\mathbf u = \mathbf w_0, \dots, \mathbf w_m = \mathbf v$ of words such that each identity $\mathbf w_j \approx \mathbf w_{j+1}$ holds in either $\mathbf X$ or $\mathbf Y$.
Since the sets $xx^+y$ and $xyx^+$ are stable with respect to $\mathbf X\wedge\mathbf Y$ by Lemma~\ref{L: M_alpha(W) in V}, we have
\[
\mathbf w_j\in x^{g_j} t\, \mathbf a_j \,\biggl(\prod_{i=k+1}^{k+\ell} t_i a_i^+\biggr),
\]
where $\con(\mathbf a_j)=\{x,y\}$,  $\mathbf w_j(x,t_{k+1})\in xx^+t_{k+1}x^\ast$, $\mathbf w_j(y,t_{k+1})\in yt_{k+1}y^+$, $\mathbf w_j(y,t_r)\in yt_ry^+$,  and $g_j=g$ whenever $g\le 1$, and $g_j\ge2$ whenever $g\ge2$ for any $j=0,\dots,m$. 
Then there is $s\in\{0,\dots,m-1\}$ such that $(_{1\mathbf a_s}y)<(_{1\mathbf a_s}x)$ but $(_{1\mathbf a_{s+1}}x)<(_{1\mathbf a_{s+1}}y)$.
Hence the identity $\mathbf w_s\approx \mathbf w_{s+1}$ is equivalent modulo $\{\Phi_2,\,\eqref{xyx=xyxx}\}$ to either~\eqref{two letters in a block middle 3 D_2} or
\[
x^gt\,x^2y\biggl(\prod_{i=k+1}^{k+\ell} t_i a_i\biggr)\approx x^gt\,yx^2\biggl(\prod_{i=k+1}^{k+\ell} t_i a_i\biggr).
\]
It is easy to see that the latter identity together with~\eqref{xyx=xyxx} imply~\eqref{two letters in a block middle 3 D_2}.
It follows that~\eqref{two letters in a block middle 3 D_2} holds in either $\mathbf X$ or $\mathbf Y$.

Assume now that there is $k+1\le r<k+\ell$ such that $a_r=y$ and $a_{r+1}=x$.
If $\mathbf X\wedge\mathbf Y$ does not contain the monoid $M_\gamma(x^+yzx^+)$, then, by Lemma~\ref{L: nsub M(x^+yzx^+)}, either $\mathbf X$ or $\mathbf Y$ satisfies~\eqref{xxyzx=xxyxzx} and so 
\[
\mathbf u\stackrel{\eqref{xxyzx=xxyxzx}}\approx x^gt\,yx^2\biggl(\prod_{i=k+1}^{r-1} t_i a_i\biggr)\, t_rxy\,\biggl(\prod_{i=r+1}^{k+\ell} t_i a_i\biggr)\stackrel{\eqref{xyzxy=yxzxy}}\approx x^gt\,xyx\biggl(\prod_{i=k+1}^{r-1} t_i a_i\biggr)\, t_rxy\,\biggl(\prod_{i=r+1}^{k+\ell} t_i a_i\biggr)\stackrel{\eqref{xxyzx=xxyxzx}}\approx \mathbf v.
\]
If $M_\gamma(x^+yzx^+)\in\mathbf X\wedge\mathbf Y$, then arguments similar to ones from the previous paragraph imply that this identity holds in either $\mathbf X$ or $\mathbf Y$.
Then either $\mathbf X$ or $\mathbf Y$ satisfies~\eqref{two letters in a block middle 3 D_2} as required.

\smallskip

\textbf{Distributivity of $\mathfrak L(\mathbf D_3)$}.
Take an arbitrary efficient identity of the form~\eqref{two letters in a block}, where $r\in\mathbb N_0$, $e_0,f_0\in\mathbb N$, $e_1,f_1,\dots,e_r,f_r\in\mathbb N_0$ and $e:=\sum_{i=0}^r e_i,f:=\sum_{i=0}^r f_i\ge 2$, which does not hold in $\mathbf D_3$.
Since 
\begin{equation}
\label{xyzxxtyy=xxyyzxxtyy=yyxxzxxtyy=yxzxxtyy}
xyzx^2ty^2\stackrel{\eqref{xyxx=xxyx}}\approx x^2y^2zxty\stackrel{\eqref{xxyy=yyxx}}\approx y^2x^2zxty\stackrel{\eqref{xyxx=xxyx}}\approx yxzx^2ty^2,
\end{equation} 
Lemma~\ref{L: xy.. = yx.. in V} implies that either $(e,e_0)=(2,1)$ or $(f,f_0)=(2,1)$. 
By symmetry, we may assume without loss of generality that $(f,f_0)=(2,1)$.
Then it is routine to check that
\begin{itemize}
\item if $(e,e_0)=(2,1)$, then $r\in\{1,2\}$ and so~\eqref{two letters in a block} coincides (up to renaming of letters) with either $\sigma_1$ or~\eqref{xyzxy=yxzxy};
\item if either $e_0>1$ or $e>2$, then~\eqref{two letters in a block} is equivalent modulo $\{\eqref{xxy=xxyx},\,\eqref{xyxx=xxyx},\,\beta_2\}$ to~\eqref{xxyty=yxxty}.
\end{itemize}
Thus,~\eqref{two letters in a block} is equivalent within $\mathbf D_3$ to an identity in $\Delta_3:=\{\sigma_1,\,\eqref{xyzxy=yxzxy},\,\eqref{xxyty=yxxty}\}$. 
It remains to show that if $\mathbf X$ and $\mathbf Y$ are non-commutative subvarieties of $\mathbf D_3$ and $\mathbf X\wedge\mathbf Y$ satisfies some identity $\sigma\in\Delta_3$, then $\sigma$ holds in either $\mathbf X$ or $\mathbf Y$.

According to Lemma~\ref{L: xy.. = yx.. in X wedge Y}(ii), the required claim is true whenever $M(xyx)\notin\mathbf X\wedge\mathbf Y$.
So, we may assume that $M(xyx)\in\mathbf X\wedge\mathbf Y$.
Further, if $\sigma\in\{\sigma_1,\,\eqref{xyzxy=yxzxy}\}$, then $\sigma$ holds in either $\mathbf X$ or $\mathbf Y$ by Lemma~\ref{L: xy.. = yx.. in X wedge Y}(i).
Let now $\sigma$ is~\eqref{xxyty=yxxty}.
In view of Proposition~\ref{P: deduction}, there exists a finite sequence $x^2yty = \mathbf w_0, \dots, \mathbf w_m = yx^2ty$ of words such that each identity $\mathbf w_i \approx \mathbf w_{i+1}$ holds in either $\mathbf X$ or $\mathbf Y$.
Since $xyx$ is an isoterm for $\mathbf X\wedge\mathbf Y$ by Lemma~\ref{L: M(W) in V}, we have $\mathbf w_i(y,t)=yty$ for all $i=0,\dots,m$.
Evidently, there is $j\in\{0,\dots,m-1\}$ such that $(_{1\mathbf w_j}x)<(_{1\mathbf w_j}y)$ but $(_{1\mathbf w_{j+1}}y)<(_{1\mathbf w_{j+1}}x)$.
Then the identity $\phi(\mathbf w_j)\approx \phi(\mathbf w_{j+1})$, where $\phi\colon \mathfrak X \to \mathfrak X^\ast$ is the substitution given by
\begin{equation}
\label{phi(v) D_3}
\phi(v) := 
\begin{cases} 
x^2t & \text{if }v=t, \\ 
v & \text{otherwise},
\end{cases} 
\end{equation}
is equivalent modulo~$\{\eqref{xxy=xxyx},\,\eqref{xyxx=xxyx}\}$ to~\eqref{xxyty=yxxty}.
Hence~\eqref{xxyty=yxxty} holds in either $\mathbf X$ or $\mathbf Y$ as required.

\smallskip

\textbf{Distributivity of $\mathfrak L(\mathbf D_4)$}.
Take an arbitrary efficient identity of the form~\eqref{two letters in a block}, where $r\in\mathbb N_0$, $e_0,f_0\in\mathbb N$, $e_1,f_1,\dots,e_r,f_r\in\mathbb N_0$ and $e:=\sum_{i=0}^r e_i,f:=\sum_{i=0}^r f_i\ge 2$, which does not hold in $\mathbf D_4$.
Then either $(e,e_0)=(2,1)$ or $(f,f_0)=(2,1)$ by Lemma~\ref{L: xy.. = yx.. in V}. 
By symmetry, we may assume without loss of generality that $(f,f_0)=(2,1)$.
Evidently, $e_0=1$ because~\eqref{two letters in a block} holds in $\mathbf D_4$ otherwise. 
Then it is routine to check that
\begin{itemize}
\item if $e=2$, then~\eqref{two letters in a block} coincides (up to renaming of letters) with either $\sigma_1$ or~\eqref{xyzxy=yxzxy};
\item if $e>2$ and $e_1=0$, then~\eqref{two letters in a block} is equivalent modulo $\{\eqref{xxy=xxyx},\,\beta_2,\,\gamma_2\}$ to~\eqref{xyzytxx=yxzytxx};
\item if $e>2$ and $e_1>0$, then~\eqref{two letters in a block} is equivalent modulo $\{\eqref{xxy=xxyx},\,\beta_2,\,\gamma_2\}$ to~\eqref{xyzxxty=yxzxxty}.
\end{itemize}
Thus,~\eqref{two letters in a block} is equivalent within $\mathbf D_4$ to an identity in $\Delta_4:=\{\sigma_1,\,\eqref{xyzxy=yxzxy},\,\eqref{xyzytxx=yxzytxx},\,\eqref{xyzxxty=yxzxxty}\}$. 
It remains to show that if $\mathbf X$ and $\mathbf Y$ are non-commutative subvarieties of $\mathbf D_4$ and $\mathbf X\wedge\mathbf Y$ satisfies some identity $\sigma\in\Delta_4$, then $\sigma$ holds in either $\mathbf X$ or $\mathbf Y$.

According to Lemma~\ref{L: xy.. = yx.. in X wedge Y}(ii), the required claim is true whenever $M(xyx)\notin\mathbf X\wedge\mathbf Y$.
So, we may assume that $M(xyx)\in\mathbf X\wedge\mathbf Y$.
Then it follows from Lemma~\ref{L: xy.. = yx.. in X wedge Y}(i) that the required claim is true whenever $\sigma\in\{\sigma_1,\,\eqref{xyzxy=yxzxy}\}$.
Thus, it remains to consider the case when $\sigma\in\{\eqref{xyzytxx=yxzytxx},\,\eqref{xyzxxty=yxzxxty}\}$.
If $M_\lambda(xyx^+)\notin\mathbf X\wedge\mathbf Y$, then $\sigma$ holds in either $\mathbf X$ or $\mathbf Y$ by Lemma~\ref{L: xy.. = yx.. in X wedge Y}(iii).
So, we may assume that $M_\lambda(xyx^+)\in\mathbf X\wedge\mathbf Y$.
Suppose that $\sigma$ is~\eqref{xyzxxty=yxzxxty}.
In view of Proposition~\ref{P: deduction}, there exists a finite sequence $xyzx^2ty = \mathbf w_0, \dots, \mathbf w_m = yxzx^2ty$ of words such that each identity $\mathbf w_i \approx \mathbf w_{i+1}$ holds in either $\mathbf X$ or $\mathbf Y$.
Since the sets $xyx^+$ and $\{xyx\}$ are stable with respect to $\mathbf X\wedge\mathbf Y$ by Lemmas~\ref{L: M(W) in V} and~\ref{L: M_alpha(W) in V}, we have $\mathbf w_i(x,z)\in xzxx^+$ and $\mathbf w_i(y,t,z)=yzty$ for all $i=0,\dots,m$.
Hence $\mathbf w_i\in xyzx^+tx^\ast y x^\ast\cup yxzx^+tx^\ast y x^\ast$.
Evidently, there is $j\in\{0,\dots,m-1\}$ such that $(_{1\mathbf w_j}x)<(_{1\mathbf w_j}y)$ but $(_{1\mathbf w_{j+1}}y)<(_{1\mathbf w_{j+1}}x)$.
Then the identity $\phi(\mathbf w_j)\approx \phi(\mathbf w_{j+1})$, where $\phi\colon \mathfrak X \to \mathfrak X^\ast$ is the substitution given by~\eqref{phi(v) D_3}, is equivalent modulo~\eqref{xxy=xxyx} to~\eqref{xyzxxty=yxzxxty}.
Hence~\eqref{xyzxxty=yxzxxty} holds in either $\mathbf X$ or $\mathbf Y$ as required.
By a similar argument we can establish the required claim when $\sigma$ is ~\eqref{xyzytxx=yxzytxx}.

\smallskip

\textbf{Distributivity of $\mathfrak L(\mathbf D_5)$}.
Take an arbitrary efficient identity of the form~\eqref{two letters in a block}, where $r\in\mathbb N_0$, $e_0,f_0\in\mathbb N$, $e_1,f_1,\dots,e_r,f_r\in\mathbb N_0$ and $e:=\sum_{i=0}^r e_i,f:=\sum_{i=0}^r f_i\ge 2$, which does not hold in $\mathbf D_5$.
In view of~\eqref{xyzxxtyy=xxyyzxxtyy=yyxxzxxtyy=yxzxxtyy}, we may apply Lemma~\ref{L: xy.. = yx.. in V}, yielding that $(e,e_0)=(2,1)$ or $(f,f_0)=(2,1)$. 
By symmetry, we may assume without loss of generality that $(f,f_0)=(2,1)$.
There is $p\in\{1,\dots,r\}$ such that $f_p=1$, while $f_1=\cdots=f_{p-1}=f_{p+1}=\cdots =f_r=0$.
Since the identity~\eqref{two letters in a block} is efficient, we have $e_1,\dots,e_{p-1},e_{p+1},\dots,e_r>0$.
If $(e,e_0)=(2,1)$, then~\eqref{two letters in a block} coincides (up to renaming of letters) with either $\sigma_1$ or~\eqref{xyzxy=yxzxy}.
Let now either $e>2$ or $e_0>1$.
Then $e_p=0$ because~\eqref{two letters in a block} follows from $\{\beta_2,\,\eqref{xyzxxy=yxzxxy},\,\eqref{xyxx=xxyx}\}$ otherwise.
Then it is routine to check that
\begin{itemize}
\item if $p<r$, then~\eqref{two letters in a block} is equivalent modulo $\{\sigma_3,\,\beta_2,\,\eqref{xyxx=xxyx}\}$ to~\eqref{xyzytxx=yxzytxx};
\item if $p=r$, then~\eqref{two letters in a block} is equivalent modulo $\{\sigma_3,\,\beta_2,\,\eqref{xyxx=xxyx}\}$ to~\eqref{xxyty=yxxty}.
\end{itemize}
Thus,~\eqref{two letters in a block} is equivalent within $\mathbf D_5$ to an identity in $\Delta_5:=\{\sigma_1,\,\eqref{xyzxy=yxzxy},\,\eqref{xyzytxx=yxzytxx},\,\eqref{xxyty=yxxty}\}$. 
It remains to show that if $\mathbf X$ and $\mathbf Y$ are non-commutative subvarieties of $\mathbf D_5$ and $\mathbf X\wedge\mathbf Y$ satisfies some identity $\sigma\in\Delta_5$, then $\sigma$ holds in either $\mathbf X$ or $\mathbf Y$.

According to Lemma~\ref{L: xy.. = yx.. in X wedge Y}(ii), the required claim is true whenever $M(xyx)\notin\mathbf X\wedge\mathbf Y$.
So, we may assume that $M(xyx)\in\mathbf X\wedge\mathbf Y$.
Then it follows from Lemma~\ref{L: xy.. = yx.. in X wedge Y}(i) that the required claim is true whenever $\sigma\in\{\sigma_1,\,\eqref{xyzxy=yxzxy}\}$.

Suppose that $\sigma$ is~\eqref{xyzytxx=yxzytxx}.
If $M_\gamma(x^+yzx^+)\notin\mathbf X\wedge\mathbf Y$, then, by Lemmas~\ref{L: nsub M(xy)} and~\ref{L: nsub M(x^+yzx^+)}, either $\mathbf X$ or $\mathbf Y$ satisfies~\eqref{xxyzx=xxyxzx} and so
\[
xyzytx^2\stackrel{\eqref{xyxx=xxyx}}\approx x^2yzytx\stackrel{\eqref{xxyzx=xxyxzx}}\approx x^2yzx^2ytx\stackrel{\eqref{xyzxxy=yxzxxy}}\approx yx^2zx^2ytx\stackrel{\eqref{xxyzx=xxyxzx}}\approx yx^2zytx\stackrel{\eqref{xyxx=xxyx}}\approx yxzytx^2.
\]
Thus, we may further assume that $M_\gamma(x^+yzx^+)\in\mathbf X\wedge\mathbf Y$.
In view of Proposition~\ref{P: deduction}, there exists a finite sequence $xyzytx^2 = \mathbf w_0, \dots, \mathbf w_m = yxzytx^2$ of words such that each identity $\mathbf w_i \approx \mathbf w_{i+1}$ holds in either $\mathbf X$ or $\mathbf Y$.
Since the sets $x^+yzx^+$ and $\{xyx\}$ are stable with respect to $\mathbf X\wedge\mathbf Y$ by Lemmas~\ref{L: M(W) in V} and~\ref{L: M_alpha(W) in V}, we have $\mathbf w_i(x,z,t)\in x^+ztx^+$, $\mathbf w_i(x,z,t)\ne xztx$ and $\mathbf w_i(y,z,t)=yzyt$ for all $i=0,\dots,m$.
Evidently, there is $j\in\{0,\dots,m-1\}$ such that $(_{1\mathbf w_j}x)<(_{1\mathbf w_j}y)$ but $(_{1\mathbf w_{j+1}}y)<(_{1\mathbf w_{j+1}}x)$.
Then the identity $\mathbf w_j\approx \mathbf w_{j+1}$ is equivalent modulo $\{x^2\approx x^3,\,\eqref{xyxx=xxyx},\,\sigma_3\}$ to~\eqref{xyzytxx=yxzytxx}.
Hence~\eqref{xyzytxx=yxzytxx} holds in either $\mathbf X$ or $\mathbf Y$ as required.

Suppose now that $\sigma$ is~\eqref{xxyty=yxxty}.
If $M_\gamma(xx^+y)\notin\mathbf X\wedge\mathbf Y$, then, by Lemma~\ref{L: nsub M(xy)} and the dual to Lemma~\ref{L: nsub M(yxx^+)}, either $\mathbf X$ or $\mathbf Y$ satisfies~\eqref{xxy=xxyx} and so
\[
x^2yty\stackrel{\eqref{xxy=xxyx}}\approx x^2ytx^2y\stackrel{\eqref{xyzxxy=yxzxxy}}\approx yx^2tx^2y\stackrel{\eqref{xxy=xxyx}}\approx yx^2ty.
\]
Thus, we may further assume that $M_\gamma(xx^+y)\in\mathbf X\wedge\mathbf Y$.
In view of Proposition~\ref{P: deduction}, there exists a finite sequence $x^2yty = \mathbf w_0, \dots, \mathbf w_m = yx^2ty$ of words such that each identity $\mathbf w_i \approx \mathbf w_{i+1}$ holds in either $\mathbf X$ or $\mathbf Y$.
Since the sets $xx^+y$ and $\{xyx\}$ are stable with respect to $\mathbf X\wedge\mathbf Y$ by Lemmas~\ref{L: M(W) in V} and~\ref{L: M_alpha(W) in V}, we have $\mathbf w_i(x,t)\in xx^+t$ and $\mathbf w_i(y,t)=yty$ for all $i=0,\dots,m$.
Hence $\mathbf w_i\in xx^+ y x^\ast ty\cup x^+yx^+ty\cup yxx^+ty$.
Evidently, there is $j\in\{0,\dots,m-1\}$ such that $(_{1\mathbf w_j}x)<(_{1\mathbf w_j}y)$ but $(_{1\mathbf w_{j+1}}y)<(_{1\mathbf w_{j+1}}x)$.
Then the identity $\mathbf w_j\approx \mathbf w_{j+1}$ is equivalent modulo $\{x^2\approx x^3,\,\sigma_3\}$ to~\eqref{xxyty=yxxty}.
Hence~\eqref{xxyty=yxxty} holds in either $\mathbf X$ or $\mathbf Y$ as required.

\smallskip

\textbf{Distributivity of $\mathfrak L(\mathbf D_6)$}.
Take an arbitrary efficient identity of the form~\eqref{two letters in a block}, where $r\in\mathbb N_0$, $e_0,f_0\in\mathbb N$, $e_1,f_1,\dots,e_r,f_r\in\mathbb N_0$ and $e:=\sum_{i=0}^r e_i,f:=\sum_{i=0}^r f_i\ge 2$, which does not hold in $\mathbf D_6$.
Then $e_i,f_i\le1$ for all $i=0,\dots,m$ because~\eqref{two letters in a block} follows from the identitties in $\{\Phi,\,\eqref{xyzytxx=yxzytxx},\,\eqref{xyzxxty=yxzxxty},\,\eqref{xxyty=yxxty}\}$ otherwise.
Hence~\eqref{two letters in a block} belongs to
\[
\Delta_6:=\{xyt_1\mathbf c_1\cdots t_i\mathbf c_i\approx yxt_1\mathbf c_1\cdots t_i\mathbf c_i\mid i\in\mathbb N,\ \mathbf c_1,\dots,\mathbf c_i\in\{x,y,xy\}\}.
\]
It remains to show that if $\mathbf X$ and $\mathbf Y$ are non-commutative subvarieties of $\mathbf D_6$ and $\mathbf X\wedge\mathbf Y$ satisfies some identity $\mathbf u\approx \mathbf v\in\Delta_6$, then $\mathbf u\approx \mathbf v$ holds in either $\mathbf X$ or $\mathbf Y$.
According to Lemma~\ref{L: xy.. = yx.. in X wedge Y}(i), the required claim is true whenever the words $xt_1x\cdots t_kx$ and $yt_1y\cdots t_m y$, where $k:=\occ_x(\mathbf u)-1=\occ_x(\mathbf v)-1$ and $m:=\occ_y(\mathbf u)-1=\occ_y(\mathbf v)-1$, are isoterms for $\mathbf X\wedge\mathbf Y$.
So, we may assume that at least one of these words, say $xt_1x\cdots t_kx$, is not an isoterm for at least one of the varieties $\mathbf X$ or $\mathbf Y$, say $\mathbf X$.
Now Lemma~\ref{L: M(W) in V} and Corollary~\ref{C: nsub M(xt_1x...t_nx),M(xt_1x...t_nx^+)}(i) apply, yielding that $\mathbf X$ satisfies $\beta_k$.
Notice that $\mathbf u\approx \mathbf v$ follows from $\{\beta_k,\,\eqref{xyzytxx=yxzytxx},\,\eqref{xyzxxty=yxzxxty}\}$.
Hence $\mathbf u\approx \mathbf v$ is satisfied by $\mathbf X$ as required.

\smallskip

\textbf{Distributivity of $\mathfrak L(\mathbf D_7)$}.
Take an arbitrary efficient identity of the form~\eqref{two letters in a block}, where $r\in\mathbb N_0$, $e_0,f_0\in\mathbb N$, $e_1,f_1,\dots,e_r,f_r\in\mathbb N_0$ and $e:=\sum_{i=0}^r e_i,f:=\sum_{i=0}^r f_i\ge 2$, which does not hold in $\mathbf D_7$.
Then either $(e,e_0)=(2,1)$ or $(f,f_0)=(2,1)$ by Lemma~\ref{L: xy.. = yx.. in V}. 
By symmetry, we may assume without loss of generality that $(f,f_0)=(2,1)$.
There is $p\in\{1,\dots,r\}$ such that $f_p=1$, while $f_1=\cdots=f_{p-1}=f_{p+1}=\cdots =f_r=0$.
Since the identity~\eqref{two letters in a block} is efficient, we have $e_1,\dots,e_{p-1},e_{p+1},\dots,e_r>0$.
Evidently, $e_0=1$ because~\eqref{two letters in a block} follows from $\{x^2\approx x^3,\,\eqref{xxyty=yxxty}\}$ otherwise.
If $e=2$, then~\eqref{two letters in a block} coincides (up to renaming of letters) with either $\sigma_1$ or~\eqref{xyzxy=yxzxy}.
Assume now that $e>2$.
In this case, $e_p=0$ becasue~\eqref{two letters in a block} follows from $\{x^2\approx x^3,\,\eqref{xyzxxy=yxzxxy},\,\beta_2,\,\gamma_2\}$ otherwise.
Then it is routine to check that
\begin{itemize}
\item if $p=1$, then~\eqref{two letters in a block} is equivalent modulo $\{x^2\approx x^3,\,\beta_2\}$ to~\eqref{xyzytxx=yxzytxx};
\item if $p=r$, then~\eqref{two letters in a block} is equivalent modulo $\{x^2\approx x^3,\,\beta_2\}$ to~\eqref{xyzxxty=yxzxxty};
\item if $1<p<r$, then~\eqref{two letters in a block} is equivalent modulo $\{x^2\approx x^3,\,\beta_2,\,\gamma_2\}$ to~\eqref{xyzxtysx=yxzxtysx}.
\end{itemize}
Thus,~\eqref{two letters in a block} is equivalent within $\mathbf D_7$ to an identity in 
\[
\Delta_7:=\{\sigma_1,\,\eqref{xyzxy=yxzxy},\,\eqref{xyzytxx=yxzytxx},\,\eqref{xyzxxty=yxzxxty},\,\eqref{xyzxtysx=yxzxtysx}\}.
\] 
It remains to show that if $\mathbf X$ and $\mathbf Y$ are non-commutative subvarieties of $\mathbf D_7$ and $\mathbf X\wedge\mathbf Y$ satisfies some identity $\sigma\in\Delta_7$, then $\sigma$ holds in either $\mathbf X$ or $\mathbf Y$.
This fact can be established by similar arguments as in the proof of distributivity of the variety $\mathbf D_4$ and we omit the corresponding considerations.

\smallskip

\textbf{Distributivity of $\mathfrak L(\mathbf D_8)$}. 
Take an arbitrary efficient identity of the form~\eqref{two letters in a block}, where $r\in\mathbb N_0$, $e_0,f_0\in\mathbb N$, $e_1,f_1,\dots,e_r,f_r\in\mathbb N_0$ and $e:=\sum_{i=0}^r e_i,f:=\sum_{i=0}^r f_i\ge 2$, which does not hold in $\mathbf D_8$.
Then either $(e,e_0)=(2,1)$ or $(f,f_0)=(2,1)$ by Lemma~\ref{L: xy.. = yx.. in V}. 
By symmetry, we may assume without loss of generality that $(f,f_0)=(2,1)$.
There is $p\in\{1,\dots,r\}$ such that $f_p=1$, while $f_1=\cdots=f_{p-1}=f_{p+1}=\cdots =f_r=0$.
Since the identity~\eqref{two letters in a block} is efficient, we have $e_1,\dots,e_{p-1},e_{p+1},\dots,e_r>0$.
Evidently, $e_0=1$ because~\eqref{two letters in a block} follows from $\{x^2\approx x^3,\,\eqref{xxyty=yxxty}\}$ otherwise.
If $e=2$, then it is easy to see that~\eqref{two letters in a block} coincides (up to renaming of letters) with either $\sigma_1$ or~\eqref{xyzxy=yxzxy}.
Assume now that $e>2$.
In this case, $e_p\le 1$ becasue~\eqref{two letters in a block} follows from~\eqref{xyzxxy=yxzxxy} otherwise.
Then it is routine to check that
\begin{itemize}
\item if $e_p=0$ and $p=1$, then~\eqref{two letters in a block} is equivalent modulo $\{x^2\approx x^3,\,\beta_2\}$ to~\eqref{xyzytxx=yxzytxx};
\item if $e_p=0$ and $p=r$, then~\eqref{two letters in a block} is equivalent modulo $\{x^2\approx x^3,\,\beta_2\}$ to~\eqref{xyzxxty=yxzxxty};
\item if $e_p=0$ and $1<p<r$, then $p=2$ and $e_1=1$ because~\eqref{two letters in a block} follows from $\{\eqref{xyzxxtysx=yxzxxtysx},\,\beta_2\}$ otherwise and so~\eqref{two letters in a block} is equivalent modulo $\beta_2$ to~\eqref{xyzxtysx=yxzxtysx};
\item if $e_p=1$, then $p=1$ because~\eqref{two letters in a block} follows from $\{\eqref{xyzxxy=yxzxxy},\,\beta_2\}$ otherwise and so~\eqref{two letters in a block} is equivalent modulo $\beta_2$ to
\begin{equation}
\label{xyzxytx=yxzxytx}
xyzxytx\approx yxzxytx.
\end{equation}
\end{itemize}
Thus,~\eqref{two letters in a block} is equivalent within $\mathbf D_8$ to an identity in 
\[
\Delta_8:=\{\sigma_1,\,\eqref{xyzxy=yxzxy},\,\eqref{xyzytxx=yxzytxx},\,\eqref{xyzxxty=yxzxxty},\,\eqref{xyzxtysx=yxzxtysx},\,\eqref{xyzxytx=yxzxytx}\}.
\] 
It remains to show that if $\mathbf X$ and $\mathbf Y$ are non-commutative subvarieties of $\mathbf D_8$ and $\mathbf X\wedge\mathbf Y$ satisfies some identity $\sigma\in\Delta_8$, then $\sigma$ holds in either $\mathbf X$ or $\mathbf Y$.
This fact can be established by similar arguments as in the proof of distributivity of $\mathbf D_4$ and we omit the corresponding considerations.

\smallskip

\textbf{Distributivity of $\mathfrak L(\mathbf D_9)$}. 
Take an arbitrary efficient identity of the form~\eqref{two letters in a block dual}, where $r\in\mathbb N_0$, $e_0,f_0\in\mathbb N$, $e_1,f_1,\dots,e_r,f_r\in\mathbb N_0$ and $e:=\sum_{i=0}^r e_i,f:=\sum_{i=0}^r f_i\ge 2$, which does not hold in $\mathbf D_9$. Then $e_i,f_i\le1$ for all $i=0,\dots,m$ because~\eqref{two letters in a block dual} follows from $\{\Phi,\,\eqref{xxyx=xxyxx},\,\eqref{ytxxy=ytyxx}\}$ otherwise.
Hence~\eqref{two letters in a block dual} belongs to
\[
\Delta_9:=\{\mathbf c_it_i\cdots \mathbf c_1t_1xy\approx \mathbf c_it_i\cdots \mathbf c_1t_1yx\mid i\in\mathbb N,\ \mathbf c_1,\dots,\mathbf c_i\in\{x,y,xy\}\}.
\]
It remains to show that if $\mathbf X$ and $\mathbf Y$ are non-commutative subvarieties of $\mathbf D_9$ and $\mathbf X\wedge\mathbf Y$ satisfies some identity $\mathbf u\approx \mathbf v\in\Delta_9$, then $\mathbf u\approx \mathbf v$ holds in either $\mathbf X$ or $\mathbf Y$.
According to Lemma~\ref{L: xy.. = yx.. in X wedge Y 2}(i), the required claim is true whenever the words $xt_1x\cdots t_kx$ and $yt_1y\cdots t_m y$, where $k:=\occ_x(\mathbf u)-1=\occ_x(\mathbf v)-1$ and $m:=\occ_y(\mathbf u)-1=\occ_y(\mathbf v)-1$, are isoterms for $\mathbf X\wedge\mathbf Y$.
So, we may assume that at least one of these words, say $xt_1x\cdots t_kx$, is not an isoterm for at least one of the varieties $\mathbf X$ or $\mathbf Y$, say $\mathbf X$.
Now Lemma~\ref{L: M(W) in V} and Corollary~\ref{C: nsub M(xt_1x...t_nx),M(xt_1x...t_nx^+)}(i) apply, yielding that $\mathbf X$ satisfies $\beta_k$.
Notice that $\mathbf u\approx \mathbf v$ follows from $\{\beta_k,\,\eqref{ytxxy=ytyxx}\}$.
Hence $\mathbf u\approx \mathbf v$ is satisfied by $\mathbf X$ as required.

\smallskip

\textbf{Distributivity of $\mathfrak L(\mathbf D_{10})$}. 
Take an arbitrary efficient identity of the form~\eqref{two letters in a block dual}, where $r\in\mathbb N_0$, $e_0,f_0\in\mathbb N$, $e_1,f_1,\dots,e_r,f_r\in\mathbb N_0$ and $e:=\sum_{i=0}^r e_i,f:=\sum_{i=0}^r f_i\ge 2$, which does not hold in $\mathbf D_{10}$.
Then either $(e,e_0)=(2,1)$ or $(f,f_0)=(2,1)$ because~\eqref{two letters in a block dual} follows from $\{\Phi,\,\beta_2\}$ otherwise. 
By symmetry, we may assume without loss of generality that $(f,f_0)=(2,1)$.
There is $p\in\{1,\dots,r\}$ such that $f_p>0$, while $f_1=\cdots=f_{p-1}=f_{p+1}=\cdots =f_r=0$.
Since the identity~\eqref{two letters in a block dual} is efficient, we have $e_1,\dots,e_{p-1},e_{p+1},\dots,e_r>0$.
Clearly, $e_p,\dots,e_r\le1$ and $\sum_{i=p}^re_i\le2$ because~\eqref{two letters in a block dual} follows from $\{\beta_2,\,\eqref{xxzytxy=xxzytyx}\}$ otherwise.
If $(e,e_0)=(2,1)$, then~\eqref{two letters in a block dual} coincides (up to renaming of letters) with either $\sigma_2$ or~\eqref{xyzxy=xyzyx}.
Let now $e>2$ or $e_0>1$.
In this case, it is routine to check that
\begin{itemize}
\item if $(e_p,\sum_{i=p}^re_i)=(0,0)$, then~\eqref{two letters in a block dual} is equivalent modulo $\{x^2\approx x^3,\,\sigma_3,\,\beta_2,\,\eqref{xyxx=xxyx}\}$ to~\eqref{ytxxy=ytyxx};
\item if $(e_p,\sum_{i=p}^re_i)=(0,1)$, then~\eqref{two letters in a block dual} is equivalent modulo $\{x^2\approx x^3,\,\sigma_3,\,\beta_2,\,\eqref{xyxx=xxyx}\}$ to~\eqref{xzytxxy=xzytyxx};
\item if $(e_p,\sum_{i=p}^re_i)=(0,2)$, then~\eqref{two letters in a block dual} is equivalent modulo $\{x^2\approx x^3,\,\sigma_3,\,\beta_2,\,\eqref{xyxx=xxyx}\}$ to
\begin{equation}
\label{xzxtysxy=xzxtysyx}
xzxtysxy\approx xzxtysyx;
\end{equation}
\item if $(e_p,\sum_{i=p}^re_i)=(1,1)$, then~\eqref{two letters in a block dual} is equivalent modulo $\{x^2\approx x^3,\,\sigma_3,\,\beta_2,\,\eqref{xyxx=xxyx}\}$ to
\begin{equation}
\label{xyzxxy=xyzyxx}
xyzx^2y\approx xyzyx^2;
\end{equation}
\item if $(e_p,\sum_{i=p}^re_i)=(1,2)$, then~\eqref{two letters in a block dual} is equivalent modulo $\{x^2\approx x^3,\,\sigma_3,\,\beta_2,\,\eqref{xyxx=xxyx}\}$ to
\begin{equation}
\label{xzxytxy=xzxytyx}
xzxytxy\approx xzxytyx.
\end{equation}
\end{itemize}
Thus,~\eqref{two letters in a block dual} is equivalent within $\mathbf D_{10}$ to an identity in 
\[
\Delta_{10}:=\{\sigma_2,\,\eqref{xyzxy=xyzyx},\,\eqref{ytxxy=ytyxx},\,\eqref{xzytxxy=xzytyxx},\,\eqref{xzxtysxy=xzxtysyx},\,\eqref{xyzxxy=xyzyxx},\,\eqref{xzxytxy=xzxytyx}\}.
\]
It remains to show that if $\mathbf X$ and $\mathbf Y$ are non-commutative subvarieties of $\mathbf D_{10}$ and $\mathbf X\wedge\mathbf Y$ satisfies some identity $\sigma\in\Delta_{10}$, then $\sigma$ holds in either $\mathbf X$ or $\mathbf Y$.

According to Lemma~\ref{L: xy.. = yx.. in X wedge Y 2}(ii), the required claim is true whenever $M(xyx)\notin\mathbf X\wedge\mathbf Y$.
So, we may assume that $M(xyx)\in\mathbf X\wedge\mathbf Y$.
Then it follows from Lemma~\ref{L: xy.. = yx.. in X wedge Y 2}(i) that the required claim is true whenever $\sigma\in\{\sigma_2,\,\eqref{xyzxy=xyzyx}\}$.

Thus, it remains to consider the case when $\sigma\in\{\eqref{ytxxy=ytyxx},\,\eqref{xzytxxy=xzytyxx},\,\eqref{xzxtysxy=xzxtysyx},\,\eqref{xyzxxy=xyzyxx},\,\eqref{xzxytxy=xzxytyx}\}$.
Suppose that $\sigma$ is~\eqref{xzytxxy=xzytyxx}.
If $M_\lambda(xyx^+)\notin \mathbf X\wedge\mathbf Y$, then, by Lemma~\ref{L: nsub M(xyx^+)}, either $\mathbf X$ or $\mathbf Y$ satisfies~\eqref{xyxx=xxyxx} and so $xzytx^2y \stackrel{\eqref{xyxx=xxyxx}}\approx x^2zytx^2y\stackrel{\eqref{xxzytxy=xxzytyx}}\approx x^2zytyx^2\stackrel{\eqref{xyxx=xxyxx}}\approx xzytyx^2$.
Let now $M_\lambda(xyx^+)\in \mathbf X\wedge\mathbf Y$. 
In view of Proposition~\ref{P: deduction}, there exists a finite sequence $xzytx^2y = \mathbf w_0, \dots, \mathbf w_m = xzytyx^2$ of words such that each identity $\mathbf w_i \approx \mathbf w_{i+1}$ holds in either $\mathbf X$ or $\mathbf Y$.
Since the sets $xyx^+$ and $\{xyx\}$ are stable with respect to $\mathbf X\wedge\mathbf Y$ by Lemmas~\ref{L: M(W) in V} and~\ref{L: M_alpha(W) in V}, we have $\mathbf w_i(x,z,t)\in xztxx^+$ and $\mathbf w(y,z,t)=zyty$ for all $i=0,\dots,m$.
Evidently, there is $j\in\{0,\dots,m-1\}$ such that $(_{\ell\mathbf w_j}x)<(_{\ell\mathbf w_j}y)$ but $(_{\ell\mathbf w_{j+1}}y)<(_{\ell\mathbf w_{j+1}}x)$.
Then the identity $\mathbf w_j\approx \mathbf w_{j+1}$ is equivalent modulo $\{x^2\approx x^3,\,\sigma_3\}$ to~\eqref{xzytxxy=xzytyxx}.
Hence~\eqref{xzytxxy=xzytyxx} holds in either $\mathbf X$ or $\mathbf Y$ as required.
By a similar argument we can establish the required claim when $\sigma$ is in $\Delta_{10}\setminus\{\sigma_2,\,\eqref{xyzxy=xyzyx},\,\eqref{xzytxxy=xzytyxx}\}$.

\smallskip

\textbf{Distributivity of $\mathfrak L(\mathbf D_{11})$}. 
Take arbitrary $k,\ell\in\mathbb N_0$, $g_1,\dots,g_{k+\ell}\in\{1,2\}$, $a_1,\dots,a_{k+\ell}\in\{x,y\}$.
If $\occ_x(a_1^{g_1}\cdots a_{k+\ell}^{g_{k+\ell}})>1$, then the identity~\eqref{two letters in a block middle 1} is equivalent modulo~$\{\eqref{xyxx=xxyx},\,\beta_2\}$ to~\eqref{two letters in a block middle 4}.
It follows that the identity~\eqref{two letters in a block middle 1} either coincides (up to renaming of letters) with an identity in $\{\sigma_1,\,\sigma_2,\,\sigma_3\}$ or is equivalent within $\mathbf D_{11}$ to $\{\eqref{two letters in a block middle 2},\eqref{two letters in a block middle 3}\}$ by Lemma~\ref{L: two letters in a block middle 2,3 <-> 4}.
If $\occ_y(a_1^{g_1}\cdots a_{k+\ell}^{g_{k+\ell}})>1$, then we can rename the letters $x$, $y$ and apply analogous arguments.

Further, if the identity~\eqref{two letters in a block middle 2} does not hold in $\mathbf D_{11}$, then $\occ_y(a_1^{g_1}\cdots a_{k+\ell}^{g_{k+\ell}})=1$ because the identity~\eqref{two letters in a block middle 2} follows from~$\{\Phi,\,\beta_2,\,\eqref{xyxx=xxyx}\}$ otherwise.
In this case, there is a unique $p\in\{1,\dots,k+\ell\}$ such that $a_p=y$ and $g_p=1$.
Then it is routine to check that
\begin{itemize}
\item if $p=1$ and $\ell=0$, then~\eqref{two letters in a block middle 2} is equivalent modulo $\Phi_2$ to
\begin{equation}
\label{ytxxy=ytxyx}
ytx^2y\approx ytxyx; 
\end{equation}
\item if $p=1$ and $\ell>1$, then~\eqref{two letters in a block middle 2} is equivalent modulo $\{\Phi_2,\,\eqref{xyxx=xxyx}\}$ to
\begin{equation}
\label{yzxxytx=yzxyxtx}
yzx^2ytx\approx yzxyxtx; 
\end{equation} 
\item if $1<p\le k$ and $\ell=0$, then~\eqref{two letters in a block middle 2} is equivalent modulo $\{\Phi_2,\,\eqref{xyxx=xxyx}\}$ to~\eqref{xzytxxy=xzytxyx};
\item if $1<p\le k$ and $\ell>0$, then~\eqref{two letters in a block middle 2} is equivalent modulo $\{\Phi_2,\,\eqref{xyxx=xxyx}\}$ to
\begin{equation}
\label{xzytxxysx=xzytxyxsx}
xzytx^2ysx\approx xzytxyxsx; 
\end{equation} 
\item if $p=k+1$ and $\ell=1$, then~\eqref{two letters in a block middle 2} is equivalent modulo $\Phi_2$ to~\eqref{xxyty=xyxty};
\item if $p=k+1$ and $\ell>1$, then~\eqref{two letters in a block middle 2} is equivalent modulo $\{\Phi_2,\,\eqref{xyxx=xxyx}\}$ to
\begin{equation}
\label{xxyzytx=xyxzytx}
x^2yzytx\approx xyxzytx; 
\end{equation}
\item if $k+1<p<k+\ell$, then~\eqref{two letters in a block middle 2} is equivalent modulo $\{\Phi_2,\,\eqref{xyxx=xxyx}\}$ to
\begin{equation}
\label{xxyzxtysx=xyxzxtysx}
x^2yzxtysx\approx xyxzxtysx;
\end{equation}
\item if $p=k+\ell$ and $\ell>1$, then~\eqref{two letters in a block middle 2} is equivalent modulo $\{\Phi_2,\,\eqref{xyxx=xxyx}\}$ to
\begin{equation}
\label{xxyzxty=xyxzxty}
x^2yzxty\approx xyxzxty. 
\end{equation} 
\end{itemize}

Denote by $\Delta_{11}$ the set consisting of all identities listed in the previous paragraph, namely, the identities
\[
\eqref{ytxxy=ytxyx},\,\eqref{yzxxytx=yzxyxtx},\,\eqref{xzytxxy=xzytxyx},\,\eqref{xzytxxysx=xzytxyxsx},\,\eqref{xxyty=xyxty},\,\eqref{xxyzytx=xyxzytx},\,\eqref{xxyzxtysx=xyxzxtysx},\,\eqref{xxyzxty=xyxzxty},
\] 
as well as the identities dual to them and the identities $\sigma_1$, $\sigma_2$, $\sigma_3$.
It remains to show that if $\mathbf X$ and $\mathbf Y$ are non-commutative subvarieties of $\mathbf D_{11}$ and $\mathbf X\wedge\mathbf Y$ satisfies some identity $\sigma\in\Delta_{11}$, then $\sigma$ holds in either $\mathbf X$ or $\mathbf Y$.

If $M(xyx)\notin\mathbf X\wedge\mathbf Y$, then either $\mathbf X$ or $\mathbf Y$, say $\mathbf X$, satisfies~\eqref{xyx=xyxx} by Corollary~\ref{C: nsub M(xt_1x...t_nx),M(xt_1x...t_nx^+)}(i).
In this case, the identity $\sigma$ holds in $\mathbf X$ because this identity follows from $\{\eqref{xxyy=yyxx},\,\eqref{xyx=xyxx},\,\eqref{xyxx=xxyx}\}$.
Thus, we may further assume that $M(xyx)\in\mathbf X\wedge\mathbf Y$.
Then it follows from Lemma~\ref{L: swapping in linear-balanced} that the required claim is true whenever $\sigma\in\{\sigma_1,\,\sigma_2,\,\sigma_3\}$.

Thus, it remains to consider the case when $\sigma\in\Delta_{11}\setminus\{\sigma_1,\,\sigma_2,\,\sigma_3\}$.
Suppose that $\sigma$ is~\eqref{ytxxy=ytxyx}. 
If $M_\lambda(yxx^+)\notin\mathbf X\wedge\mathbf Y$, then, by Lemmas~\ref{L: nsub M(xy)} and~\ref{L: nsub M(yxx^+)}, either $\mathbf X$ or $\mathbf Y$ satisfies~\eqref{yxx=xyxx} and so
\[
ytx^2y\stackrel{\eqref{yxx=xyxx}}\approx xytx^3y\stackrel{\Phi_1}\approx xytx^2yx\stackrel{\eqref{yxx=xyxx}}\approx ytx^2yx\stackrel{\Phi_2}\approx ytxyx.
\]
Let now $M_\lambda(yxx^+)\in\mathbf X\wedge\mathbf Y$.
In view of Proposition~\ref{P: deduction}, there exists a finite sequence $ytx^2y = \mathbf w_0, \dots, \mathbf w_m = ytxyx$ of words such that each identity $\mathbf w_i \approx \mathbf w_{i+1}$ holds in either $\mathbf X$ or $\mathbf Y$.
Since the sets $yxx^+$ and $\{xyx\}$ are stable with respect to $\mathbf X\wedge\mathbf Y$ by Lemmas~\ref{L: M(W) in V} and~\ref{L: M_alpha(W) in V}, we have $\mathbf w_i(x,t)\in txx^+$ and $\mathbf w_i(y,t)=yty$ for all $i=0,\dots,m$.
Hence $\mathbf w_i\in ytxx^+ yx^\ast\cup ytx^+ yx^+\cup ytx^\ast yxx^+$.
Evidently, there is $j\in\{0,\dots,m-1\}$ such that $(_{\ell\mathbf w_j}x)<(_{\ell\mathbf w_j}y)$ but $(_{\ell\mathbf w_{j+1}}y)<(_{\ell\mathbf w_{j+1}}x)$.
Then the identity $\mathbf w_j\approx \mathbf w_{j+1}$ is equivalent modulo $\Phi_2$ to either~\eqref{ytxxy=ytyxx} or~\eqref{ytxxy=ytxyx}.
Since~\eqref{ytxxy=ytyxx} together with $\{\Phi,\,\Phi_2\}$ imply~\eqref{ytxxy=ytxyx} by Lemma~\ref{L: two letters in a block middle 2,3 <-> 4}, the identity~\eqref{ytxxy=ytxyx} holds in either $\mathbf X$ or $\mathbf Y$ in any case.
By a similar argument we can establish the required claim when $\sigma\in\Delta_{11}\setminus\{\sigma_1,\,\sigma_2,\,\sigma_3,\,\eqref{ytxxy=ytxyx}\}$.

\smallskip

\textbf{Distributivity of $\mathfrak L(\mathbf D_{12})$}. 
Take arbitrary $k,\ell\in\mathbb N_0$, $g_1,\dots,g_{k+\ell}\in\{1,2\}$, $a_1,\dots,a_{k+\ell}\in\{x,y\}$.
First, notice that the identities~\eqref{two letters in a block middle 2} and~\eqref{two letters in a block middle 3} follow from $\{\Phi_2,\,\eqref{xxyty=yxxty},\,\eqref{ytxxy=ytyxx}\}$ and so hold in $\mathbf D_{12}$. 

Suppose that the identity~\eqref{two letters in a block middle 1} does not hold in $\mathbf D_{12}$. 
Then $g_1=\cdots=g_{k+\ell}=1$ because the identity~\eqref{two letters in a block middle 1} follows from $\{x^2\approx x^3,\,\eqref{xxyx=xxyxx},\,\eqref{xyzytxx=yxzytxx},\,\eqref{yzxytxx=yzyxtxx},\,\eqref{xyzxxty=yxzxxty},\,\eqref{xxyty=yxxty},\,\eqref{ytxxy=ytyxx}\}$ otherwise. 
Hence~\eqref{two letters in a block middle 1} has the form
\[
\biggl(\prod_{i=1}^k a_it_i\biggr)xy\biggl(\prod_{i=k+1}^{k+\ell} t_i a_i\biggr)\approx \biggl(\prod_{i=1}^k a_it_i\biggr)yx\biggl(\prod_{i=k+1}^{k+\ell} t_i a_i\biggr).
\]
Denote by $\Delta_{12}$ the set of all identities of such a form.
It remains to show that if $\mathbf X$ and $\mathbf Y$ are non-commutative subvarieties of $\mathbf D_{12}$ and $\mathbf X\wedge\mathbf Y$ satisfies some identity $\mathbf u\approx \mathbf v\in\Delta_{12}$, then $\mathbf u\approx \mathbf v$ holds in either $\mathbf X$ or $\mathbf Y$.

According to Lemma~\ref{L: swapping in linear-balanced}, the required claim is true whenever $\mathbf X\wedge\mathbf Y$ contains both the monoids $M(xt_1x\cdots t_px)$ and $M(yt_1y\cdots t_q y)$, where $p:=\occ_x(\mathbf u)-1=\occ_x(\mathbf v)-1$ and $q:=\occ_y(\mathbf u)-1=\occ_y(\mathbf v)-1$.
So, we may assume that at least one of these monoids, say $M(xt_1x\cdots t_px)$, does not lie in at least one of the varieties $\mathbf X$ or $\mathbf Y$, say $\mathbf X$.
Now Corollary~\ref{C: nsub M(xt_1x...t_nx),M(xt_1x...t_nx^+)}(i) applies, yielding that $\mathbf X$ satisfies $\beta_p$.
Notice that $\mathbf u\approx \mathbf v$ follows from $\{x^2\approx x^3,\,\beta_p,\,\eqref{xyzytxx=yxzytxx},\,\eqref{yzxytxx=yzyxtxx},\,\eqref{xyzxxty=yxzxxty},\,\eqref{xxyty=yxxty},\,\eqref{ytxxy=ytyxx}\}$.
Hence $\mathbf u\approx \mathbf v$ is satisfied by $\mathbf X$ as required.

\smallskip

\textbf{Distributivity of $\mathfrak L(\mathbf D_{13})$}. 
Take arbitrary $k,\ell\in\mathbb N_0$, $g_1,\dots,g_{k+\ell}\in\{1,2\}$, $a_1,\dots,a_{k+\ell}\in\{x,y\}$ and let $e:=\occ_x(a_1^{g_1}\cdots a_{k+\ell}^{g_{k+\ell}})$ and $f:=\occ_y(a_1^{g_1}\cdots a_{k+\ell}^{g_{k+\ell}})$.
First, notice that the identity~\eqref{two letters in a block middle 3} follows from $\{\Phi_2,\,\eqref{yxxty=xyxxty},\,\eqref{ytyxx=ytxyxx}\}$ and so holds in $\mathbf D_{13}$. 

Suppose that the identity~\eqref{two letters in a block middle 1} does not hold in $\mathbf D_{13}$. 
Then $e=1$ or $f=1$ because~\eqref{two letters in a block middle 1} follows from~$\{\Phi,\,\eqref{xxytyy=yxxtyy},\,\eqref{xyzxxtyy=yxzxxtyy},\,\beta_2\}$ otherwise. 
By symmetry, we may assume that $f=1$.
Then there is a unique $p\in\{1,\dots,k+\ell\}$ such that $a_p=y$ and $g_p=1$.
If $e=1$, then~\eqref{two letters in a block middle 1} coincides (up to renaming of letters) with an identity in $\{\sigma_1,\,\sigma_2,\,\sigma_3\}$.
Let now $e>1$.
If $\occ_x(a_1^{g_1}\cdots a_k^{g_k})>0$, then~\eqref{two letters in a block middle 1} is equivalent modulo~$\{\beta_2,\,\gamma_2\}$ to~\eqref{two letters in a block middle 4}. 
The latter identity is equivalent within $\mathbf D_{13}$ to~\eqref{two letters in a block middle 2} by the above and Lemma~\ref{L: two letters in a block middle 2,3 <-> 4}.
If $\occ_x(a_1^{g_1}\cdots a_k^{g_k})=0$, then it is routine to check that  
\begin{itemize}
\item if $k=p=1$, then~\eqref{two letters in a block middle 1} is equivalent modulo $x^2\approx x^3$ to~\eqref{yzxytxx=yzyxtxx};
\item if $k=0$ and $p=1$, then~\eqref{two letters in a block middle 1} is equivalent modulo $x^2\approx x^3$ to~\eqref{xyzytxx=yxzytxx};
\item if $k=0$ and $1<p<\ell$, then~\eqref{two letters in a block middle 1} is equivalent modulo $\{\beta_2,\,\gamma_2\}$ to~\eqref{xyzxtysx=yxzxtysx};
\item if $k=0$ and $p=\ell$, then~\eqref{two letters in a block middle 1} is equivalent modulo $x^2\approx x^3$ to~\eqref{xyzxxty=yxzxxty}.
\end{itemize}

Suppose that the identity~\eqref{two letters in a block middle 2} does not hold in $\mathbf D_{13}$. 
Then $f=1$ because~\eqref{two letters in a block middle 2} follows from~$\{\Phi,\,\Phi_2,\,\eqref{xxytyy=yxxtyy},\,\beta_2\}$ otherwise. 
Then there is a unique $q\in\{1,\dots,k+\ell\}$ such that $a_q=y$ and $g_q=1$.
Clearly, $\occ_x(a_{k+1}^{g_{k+1}}\cdots a_{k+\ell}^{g_{k+\ell}})=0$ because~\eqref{two letters in a block middle 2} follows from the identities in $\{\Phi_2,\,\eqref{xxyzytx=yxxzytx},\,\eqref{xxyzxty=yxxzxty},\,\eqref{yzxxytx=yzyxxtx}\}$ otherwise.  
It is routine to check that  
\begin{itemize}
\item if $q=1$, then~\eqref{two letters in a block middle 2} is equivalent modulo $\Phi_2$ to~\eqref{ytxxy=ytxyx};
\item if $1<q\le k$, then $q=2$ and $g_1=1$ because~\eqref{two letters in a block middle 2} follows from $\{\beta_2,\,\gamma_2,\,\eqref{xxzytxy=xxzytyx}\}$ otherwise and so~\eqref{two letters in a block middle 2} is equivalent modulo $\Phi_2$ to~\eqref{xzytxxy=xzytxyx};
\item if $q=k+1$, then~\eqref{two letters in a block middle 2} is equivalent modulo $\Phi_2$ to~\eqref{xxyty=xyxty}.
\end{itemize}
Let
\[
\Delta_{13}:=\{\sigma_1,\,\sigma_2,\,\sigma_3,\,\eqref{yzxytxx=yzyxtxx},\,\eqref{xyzytxx=yxzytxx},\,\eqref{xyzxtysx=yxzxtysx},\,\eqref{xyzxxty=yxzxxty},\,\eqref{ytxxy=ytxyx},\,\eqref{xzytxxy=xzytxyx},\,\eqref{xxyty=xyxty}\}.
\]
It remains to show that if $\mathbf X$ and $\mathbf Y$ are non-commutative subvarieties of $\mathbf D_{13}$ and $\mathbf X\wedge\mathbf Y$ satisfies some identity $\sigma\in\Delta_{13}$, then $\sigma$ holds in either $\mathbf X$ or $\mathbf Y$.

Suppose that $\sigma$ is~\eqref{yzxytxx=yzyxtxx}. 
If $M(xyx)\notin\mathbf X\wedge\mathbf Y$, then, by Corollary~\ref{C: nsub M(xt_1x...t_nx),M(xt_1x...t_nx^+)}(i), either $\mathbf X$ or $\mathbf Y$ satisfies~\eqref{xyx=xyxx} and so
\[
yzxytx^2\stackrel{\eqref{xyx=xyxx}}\approx yzxy^2tx^2 \stackrel{\eqref{xxytyy=yxxtyy}}\approx yzy^2xtx^2 \stackrel{\eqref{xyx=xyxx}}\approx yzyxtx^2.
\]
If $M_\lambda(xyx^+)\notin\mathbf X\wedge\mathbf Y$, then, by Lemmas~\ref{L: nsub M(xy)} and~\ref{L: nsub M(xyx^+)}, either $\mathbf X$ or $\mathbf Y$ satisfies~\eqref{xyxx=xxyxx} and so
\[
yzxytx^2\stackrel{\eqref{xyxx=xxyxx}}\approx yzx^2ytx^2\stackrel{\eqref{yzxxytx=yzyxxtx}}\approx yzyx^2tx^2\stackrel{\eqref{xyxx=xxyxx}}\approx yzyxtx^2.
\]
Thus, we may further assume that $\mathbf M_\lambda(xyx^+)\vee\mathbf M(xyx)\subseteq\mathbf X\wedge\mathbf Y$.
In view of Proposition~\ref{P: deduction}, there exists a finite sequence $yzxytx^2 = \mathbf w_0, \dots, \mathbf w_m = yzyxtx^2$ of words such that each identity $\mathbf w_i \approx \mathbf w_{i+1}$ holds in either $\mathbf X$ or $\mathbf Y$.
Since the sets $xyx^+$ and $\{xyx\}$ are stable with respect to $\mathbf X\wedge\mathbf Y$ by Lemmas~\ref{L: M(W) in V} and~\ref{L: M_alpha(W) in V}, we have $\mathbf w_i(x,z,t)\in zxtxx^+$ and $\mathbf w_i(y,z,t)=yzyt$ for all $i=0,\dots,m$.
Hence $\mathbf w_i\in yzxytxx^+\cup yzyxtxx^+$.
Evidently, there is $j\in\{0,\dots,m-1\}$ such that $\mathbf w_j\in yzxytxx^+$ but $\mathbf w_{j+1}\in yzyxtxx^+$.
Then the identity $\mathbf w_j\approx \mathbf w_{j+1}$ is equivalent modulo $x^2\approx x^3$ to~\eqref{yzxytxx=yzyxtxx}.
Hence~\eqref{yzxytxx=yzyxtxx} holds in either $\mathbf X$ or $\mathbf Y$ as required.
By a similar argument we can establish the required claim when $\sigma\in\Delta_{13}\setminus\{\eqref{yzxytxx=yzyxtxx}\}$.

\smallskip

\textbf{Distributivity of $\mathfrak L(\mathbf D_{14})$}. 
Take arbitrary $k,\ell\in\mathbb N_0$, $g_1,\dots,g_{k+\ell}\in\{1,2\}$, $a_1,\dots,a_{k+\ell}\in\{x,y\}$ and let $e:=\occ_x(a_1^{g_1}\cdots a_{k+\ell}^{g_{k+\ell}})$ and $f:=\occ_y(a_1^{g_1}\cdots a_{k+\ell}^{g_{k+\ell}})$.
First, notice that the identity~\eqref{two letters in a block middle 3} follows from $\{\Phi_2,\,\eqref{yxxty=xyxxty},\,\eqref{ytyxx=ytxyxx}\}$ and so holds in $\mathbf D_{14}$. 

Suppose that the identity~\eqref{two letters in a block middle 1} does not hold in $\mathbf D_{14}$. 
Then $e=1$ or $f=1$ because~\eqref{two letters in a block middle 1} follows from~$\{\Phi,\,\eqref{xxytyy=yxxtyy},\,\eqref{xyzxxtyy=yxzxxtyy},\,\beta_2\}$ otherwise. 
By symmetry, we may assume that $f=1$.
Then there is a unique $p\in\{1,\dots,k+\ell\}$ such that $a_p=y$ and $g_p=1$.
If $e=1$, then~\eqref{two letters in a block middle 1} coincides (up to renaming of letters) with an identity in $\{\sigma_1,\,\sigma_2,\,\sigma_3\}$.
Let now $e>1$.
If $\occ_x(a_1^{g_1}\cdots a_k^{g_k})>1$, then~\eqref{two letters in a block middle 1} is equivalent modulo~$\beta_2$ to~\eqref{two letters in a block middle 4}. 
The latter identity is equivalent within $\mathbf D_{14}$ to~\eqref{two letters in a block middle 2} by the above and Lemma~\ref{L: two letters in a block middle 2,3 <-> 4}.
If $\occ_x(a_1^{g_1}\cdots a_k^{g_k})=0$, then it is routine to check that  
\begin{itemize}
\item if $k=p=1$, then~\eqref{two letters in a block middle 1} is equivalent modulo $\{x^2\approx x^3,\,\beta_2\}$ to~\eqref{yzxytxx=yzyxtxx};
\item if $k=0$ and $p=1$, then~\eqref{two letters in a block middle 1} is equivalent modulo $\{x^2\approx x^3,\,\beta_2\}$ to~\eqref{xyzytxx=yxzytxx};
\item if $k=0$ and $1<p<\ell$, then $p=2$ and $g_1=1$ because the identity~\eqref{two letters in a block middle 1} follows from $\{\eqref{xyzxxtysx=yxzxxtysx},\,\beta_2\}$ otherwise and so~\eqref{two letters in a block middle 1} is equivalent modulo $\beta_2$ to~\eqref{xyzxtysx=yxzxtysx};
\item if $k=0$ and $p=\ell$, then~\eqref{two letters in a block middle 1} is equivalent modulo $x^2\approx x^3$ to~\eqref{xyzxxty=yxzxxty}.
\end{itemize}
If $\occ_x(a_1^{g_1}\cdots a_k^{g_k})=1$, then it is routine to check that 
\begin{itemize}
\item if $p=1$, then~\eqref{two letters in a block middle 1} is equivalent modulo $\beta_2$ to
\begin{equation}
\label{yzxtxysx=yzxtyxsx}
yzxtxysx\approx yzxtyxsx;
\end{equation}
\item if $p=k=2$, then~\eqref{two letters in a block middle 1} is equivalent modulo $\beta_2$ to
\begin{equation}
\label{xzytxysx=xzytyxsx}
xzytxysx\approx xzytyxsx;
\end{equation}
\item if $k=1$ and $1<p<\ell+1$, then $p=2$ and $g_1=1$ because~\eqref{two letters in a block middle 1} follows from $\{\eqref{xyzxxtysx=yxzxxtysx},\,\beta_2\}$ otherwise and so~\eqref{two letters in a block middle 1} is equivalent modulo $\beta_2$ to
\begin{equation}
\label{xzxytysx=xzyxtysx}
xzxytysx\approx xzyxtysx;
\end{equation}
\item if $k=1$ and $2<p=\ell+1$, then~\eqref{two letters in a block middle 1} is equivalent modulo $\beta_2$ to
\begin{equation}
\label{xzxytxsy=xzyxtxsy}
xzxytxsy\approx xzyxtxsy.
\end{equation}
\end{itemize} 

Suppose that the identity~\eqref{two letters in a block middle 2} does not hold in $\mathbf D_{14}$. 
Then $f=1$ because~\eqref{two letters in a block middle 2} follows from~$\{\Phi,\,\Phi_2,\,\eqref{xxytyy=yxxtyy},\,\beta_2\}$ otherwise. 
Then there is a unique $q\in\{1,\dots,k+\ell\}$ such that $a_q=y$ and $g_q=1$.
Clearly, $\occ_x(a_{k+1}^{g_{k+1}}\cdots a_{k+\ell}^{g_{k+\ell}})=0$ because~\eqref{two letters in a block middle 2} follows from the identities in $\{\Phi_2,\,\eqref{xxyzytx=yxxzytx},\,\eqref{xxyzxty=yxxzxty},\,\eqref{yzxxytx=yzyxxtx}\}$ otherwise.  
It is routine to check that  
\begin{itemize}
\item if $q=1$, then~\eqref{two letters in a block middle 2} is equivalent modulo $\Phi_2$ to~\eqref{ytxxy=ytxyx};
\item if $1<q\le k$, then $2\le q\le3$ and $g_1=g_{q-1}=1$ because~\eqref{two letters in a block middle 2} follows from $\{\beta_2,\,\eqref{xxzytxy=xxzytyx}\}$ otherwise and so~\eqref{two letters in a block middle 2} is equivalent modulo $\Phi_2$ to either~\eqref{xzytxxy=xzytxyx} or
\begin{equation}
\label{xzxtysxxy=xzxtysxyx}
xzxtysx^2y\approx xzxtysxyx;
\end{equation}
\item if $q=k+1$, then~\eqref{two letters in a block middle 2} is equivalent modulo $\Phi_2$ to~\eqref{xxyty=xyxty}.
\end{itemize}
Let
\[
\Delta_{14}:=
\left\{
\begin{array}{l}
\sigma_1,\,\sigma_2,\,\sigma_3,\,\eqref{yzxytxx=yzyxtxx},\,\eqref{xyzytxx=yxzytxx},\,\eqref{xyzxtysx=yxzxtysx},\,\eqref{xyzxxty=yxzxxty},\,\eqref{yzxtxysx=yzxtyxsx},\,\eqref{xzytxysx=xzytyxsx},\\ 
\eqref{xzxytysx=xzyxtysx},\,\eqref{xzxytxsy=xzyxtxsy},\,\eqref{ytxxy=ytxyx},\,\eqref{xzytxxy=xzytxyx},\,\eqref{xzxtysxxy=xzxtysxyx},\,\eqref{xxyty=xyxty}
\end{array}
\right\}.
\]
It remains to show that if $\mathbf X$ and $\mathbf Y$ are non-commutative subvarieties of $\mathbf D_{14}$ and $\mathbf X\wedge\mathbf Y$ satisfies some identity $\sigma\in\Delta_{14}$, then $\sigma$ holds in either $\mathbf X$ or $\mathbf Y$.
This fact can be established by the same arguments as in the proof of distributivity of $\mathbf D_{13}$ and we omit it.

\subsection*{Acknowledgments.} 
The author is grateful to Olga Sapir for her remarks and constructive suggestions.

\newpage

\section*{APPENDIX. Identities labelled by numbers}

\begin{table}[tbh]
\begin{center}
\begin{tabular}{|c|c||c|c|}
\hline
Number & Identity & Number & Identity \\
\hline
\eqref{xyxx=xxyxx} &  $xyx^2\approx x^2yx^2$ & \eqref{xxyxty=xyxty} & $x^2yxty\approx xyxty$\\
\eqref{yxxtxxyxx=xxyxxtxxyxx} & $yx^2tx^2yx^2\approx x^2yx^2tx^2yx^2$ & \eqref{xxyty=xyxty} & $x^2yty\approx xyxty$\\
\eqref{xxyty=xxyxty} & $x^2yty\approx x^2yxty$ & \eqref{xyzxy=xyzyx} & $xyzxy\approx xyzyx$\\
\eqref{yxx=xyxx} & $yx^2\approx xyx^2$ & \eqref{xyxx=xxyx} & $xyx^2\approx x^2yx$\\
\eqref{xxy=xxyx} & $x^2y\approx x^2yx$ & \eqref{ytyxx=ytxyx} & $ytyx^2\approx ytxyx$\\
\eqref{yxxty=xyxxty} & $yx^2ty\approx xyx^2ty$ & \eqref{xxyty=yxxty} & $x^2yty\approx yx^2ty$\\
\eqref{xxytxy=yxxtxy} & $x^2ytxy\approx yx^2txy$ & \eqref{xzytxxy=xzytyxx} & $xzytx^2y\approx xzytyx^2$\\
\eqref{xxyx=xxyxx} & $x^2yx\approx x^2yx^2$ & \eqref{xzytxxy=xzytxyx} & $xzytx^2y\approx xzytxyx$\\
\eqref{xxyzx=xxyxzx} & $x^2yzx\approx x^2yxzx$ & \eqref{xyzxtysx=yxzxtysx} & $xyzxtysx\approx yxzxtysx$\\
\eqref{xxyzytx=yxxzytx} & $x^2yzytx\approx yx^2zytx$ &\eqref{xyx=xxyx} & $xyx\approx x^2yx$\\
\eqref{xxyzxty=yxxzxty} & $x^2yzxty\approx yx^2zxty$ &\eqref{yxxty=xyxty} & $yx^2ty\approx xyxty$\\
\eqref{ytyxx=ytxyxx} & $ytyx^2\approx ytxyx^2$ &\eqref{xyxtxxy=xyxty} & $xyxtx^2y\approx xyxty$\\
\eqref{yzxxytx=yzyxxtx} & $yzx^2ytx\approx yzyx^2tx$ &\eqref{ytxxy=ytyxx} & $ytx^2y\approx ytyx^2$\\
\eqref{xxytxy=xxytyx} & $x^2ytxy\approx x^2ytyx$ &\eqref{ytxxy=ytxxyx} & $ytx^2y\approx ytx^2yx$\\
\eqref{xxzytxy=xxzytyx} & $x^2zytxy\approx x^2zytyx$ & \eqref{xyzxytx=yxzxytx} & $xyzxytx\approx yxzxytx$\\
\eqref{xyzxxtxyx=yxzxxtxyx} & $xyzx^2txyx\approx yxzx^2txyx$ &\eqref{xzxtysxy=xzxtysyx} & $xzxtysxy\approx xzxtysyx$\\
\eqref{xyzytxx=yxzytxx} & $xyzytx^2\approx yxzytx^2$ &\eqref{xyzxxy=xyzyxx} & $xyzx^2y\approx xyzyx^2$\\
\eqref{yzxytxx=yzyxtxx} & $yzxytx^2\approx yzyxtx^2$ &\eqref{xzxytxy=xzxytyx} & $xzxytxy\approx xzxytyx$\\ 
\eqref{xyzxxyy=yxzxxyy} & $xyzx^2y^2\approx yxzx^2y^2$ & \eqref{ytxxy=ytxyx} & $ytx^2y\approx ytxyx$\\
\eqref{xxytyy=xxyxtyy} & $x^2yty^2\approx x^2yxty^2$ & \eqref{yzxxytx=yzxyxtx} & $yzx^2ytx\approx yzxyxtx$\\
\eqref{xxytyy=yxxtyy} & $x^2yty^2\approx yx^2ty^2$ & \eqref{xzytxxysx=xzytxyxsx} & $xzytx^2ysx\approx xzytxyxsx$\\ 
\eqref{xyzxxtyy=yxzxxtyy} & $xyzx^2ty^2\approx yxzx^2ty^2$ & \eqref{xxyzytx=xyxzytx} & $x^2yzytx\approx xyxzytx$\\ 
\eqref{xyzxxy=yxzxxy} & $xyzx^2y\approx yxzx^2y$ &\eqref{xxyzxtysx=xyxzxtysx} & $x^2yzxtysx\approx xyxzxtysx$\\
\eqref{xyzxxty=yxzxxty} & $xyzx^2ty\approx yxzx^2ty$ &\eqref{xxyzxty=xyxzxty} & $x^2yzxty\approx xyxzxty$\\
\eqref{xyzxxtysx=yxzxxtysx} & $xyzx^2tysx\approx yxzx^2tysx$ &\eqref{yzxtxysx=yzxtyxsx} & $yzxtxysx\approx yzxtyxsx$\\
\eqref{xyx=xyxx} & $xyx\approx xyx^2$ &\eqref{xzytxysx=xzytyxsx} & $xzytxysx\approx xzytyxsx$\\
\eqref{xxyy=yyxx} & $x^2y^2\approx y^2x^2$ & \eqref{xzxytysx=xzyxtysx} & $xzxytysx\approx xzyxtysx$\\
\eqref{xyzxy=yxzxy} & $xyzxy\approx yxzxy$ & \eqref{xzxytxsy=xzyxtxsy} & $xzxytxsy\approx xzyxtxsy$\\
\eqref{xzyxty=xzxyxty} & $xzyxty\approx xzxyxty$ & & \\
\hline
\end{tabular}
\end{center}
\label{numbers}
\end{table}
\end{document}